\documentclass[10pt, a4paper]{article}
\usepackage{amsmath,amsthm,amsfonts,amssymb}
\usepackage{authblk}
\usepackage{multirow}
\usepackage{array}
\usepackage{tikz}
\usepackage{centernot}
\usepackage{pdflscape}
\usepackage{wedn}
\usepackage{calligra}
\usepackage{multicol}
\usepackage{pdfpages}

\usepackage{arydshln}
\newtheorem{theorem}{Theorem}[section]
\newtheorem{proposition}[theorem]{Proposition}
\newtheorem{lemma}[theorem]{Lemma}
\newtheorem{corollary}[theorem]{Corollary}

\theoremstyle{definition}
\newtheorem{definition}[theorem]{Definition}
\newtheorem{example}[theorem]{Example}

\theoremstyle{remark}
\newtheorem{remark}[theorem]{Remark}

\title{On the classification of $3$-dimensional complex hom-Lie algebras}

\makeatletter
\def\blfootnote{\xdef\@thefnmark{}\@footnotetext}
\makeatother

\makeatletter
\renewcommand\@date
{{%

\vspace{-1.5cm}%
\large\centering

  \begin{tabular}{@{}c@{}}
    Edison Alberto \textsc{Fern\'andez-Culma} \textsuperscript{1,3,$\ddagger$}\blfootnote{${}^{\ddagger}$ Partially supported by CONICET and SECyT-UNC} \\
    \texttt{\small efernandez@famaf.unc.edu.ar}
  \end{tabular}%

  and

  \begin{tabular}{@{}c@{}}
   Nadina  \textsc{Rojas} \textsuperscript{1,2,$\ddagger$} \\
     \texttt{\small nadina.rojas@unc.edu.ar}
  \end{tabular}

  \textsuperscript{1}{\small Centro de Investigaci\'on y Estudios en Matem\'atica de C\'ordoba, CONICET}\par
  \textsuperscript{2}{\small Facultad de Ciencias Exactas, F\'isicas y Naturales, UNC}\par
  \textsuperscript{3}{\small Facultad de Matem\'atica, Astronom\'ia, F\'isica y Computaci\'on, UNC}\par
  \textsuperscript{ }{\small C\'ordoba, Argentina}\par
}}
\makeatother

\author{}


\begin{document}

\maketitle

\begin{abstract}
We classify hom-Lie structures with nilpotent twisting map on $3$-dimensional complex Lie algebras, up to isomorphism, and classify all
degenerations in such family. The ideas and techniques presented here can be easily extrapolated to study similar problems in other
algebraic structures and provide different perspectives from where one can tackle classical open problems of interest in rigid Lie algebras.
\end{abstract}

\section{Introduction}

The concept of hom-Lie algebra has its origin in \textit{Quantum Calculus}. In the discrete context,
when one replace the usual way of differentiating functions by a $\sigma$-derivation, it is expected that similar properties, identities and
formulas to the ones in differential calculus appear (e.g. \textit{Twisted Leibniz rule}). In the case of the notion of \textit{hom-Lie algebra}, such definition was introduced in the Daniel Larsson's doctoral thesis, \cite{Larsson}, (see also publications resulting from this thesis, as \cite{HLS})
and it is motivated by some examples of \textit{quantum deformations} of algebras of vector fields which satisfied a \textit{twisted Jacobi identity}
with respect to an appropriate (Lie) bracket operation. Potential applications of such notion were considered in the mentioned thesis
and subsequently in \cite{Larsson2}. Over the last fifteen years, the hom-Lie algebras have become a research topic in algebra and
many papers on such algebraic structures have been dedicated to obtain results about (regular) hom-Lie algebras generalizing notions and ideas well-known
from Lie algebra theory (in some cases, \textit{mutatis mutandis}).

\newpage

Our motivation to study hom-Lie algebras comes from classical open problems in Lie theory, such as a conjecture due to Mich\`{e}le Vergne, which states
that there are no rigid complex nilpotent Lie algebras in the algebraic set of all $n$-dimensional complex Lie algebras (\cite[p. 83]{V}), or more generally the Grunewald-O'Halloran Conjecture: every complex nilpotent Lie algebra is the degeneration of another nonisomorphic Lie algebra.
Both of these problems are relevant for the understanding of the geometry and algebraic properties of Nilpotent Lie groups. For reasons that we will explain in
\textit{Remark} \ref{vergne}, \textit{regular hom-Lie} algebra structures on a Lie algebra can be used to deform it linearly. It is also of interest to study
how much we can say of a Lie algebra about its different hom-Lie algebra structures; which correspond to the solution set of a linear equation system
depending on the Lie algebra. A precursor result about this kind of ideas is the well known Jacobson's theorem concerning Lie algebras
admitting non-singular derivations and its very nice generalization due to Wolfgang Moens (see \cite{Moens}), or more generally, the work of Alice Fialowski, Abror Khudoyberdiyev and Bakhrom Omirov for the Leibniz algebra case in \cite{FialowskiKhudoyberdiyevOmirov}.

In this paper, we tackle the problem of classifying hom-Lie structures on $3$-dimension complex Lie algebras and their degenerations.
We give more importance to the techniques and ideas for doing so, we then focus our attention to present only the case when the twisting map
is a nilpotent operator, although the general case can be obtained in the same way. Of the different ways to obtain the classification, we
have preferred a more practical approach, while avoiding symbolic computations or any no easy verifiable result. As far as degenerations
of the mentioned structures are concerned, we use the classification of orbit closures of $3$-dimensional complex Lie algebras to reduce
our exposition and do a more substantial presentation of the results and techniques, which can be implemented in many other problems of a similar nature
(\cite{F,Rojas}); in fact, we are guided by the forgotten ancient proverb that says

\begin{quote}
    \textit{``Algebras are like persons; to known how an algebra is, we need to observe what your algebra does, likes ...
    and how your algebra behaves towards all things around it''.}
\end{quote}

Clearly results such as \'{E}lie Cartan's criterion for semisimplicity (for solvability), the Jacobson-Moens theorem mentioned earlier,
the work of Dietrich Burde and Manuel Ceballos in \cite{B3}, the results obtained by Valerii Filippov in \cite{Filippov1,Filippov2} or by George Leger and Eugene Luks in \cite{LegerLuks} can be considered as illustrative examples of this saying. Here, we can also mention the  Ji\v{r}\'{i} Hrivn\'{a}k's doctoral thesis \cite{Hrivnak} (and its publications \cite{HJ,NH}) which introduced and studied some Lie algebra invariants defined by slightly modifying the linear equations corresponding to the notions of \textit{derivation} of a Lie algebra and \textit{cocycle} of the adjoint representation, and which have been used effectively to re-examine the classification of Lie algebras and their degenerations in low dimensions.

Originally, the study of degenerations of algebraic structures, especially in the context of Lie algebras or associative algebras,
was largely driven by  applications in Mathematical physics (\cite{NP}), Geometry and Algebra: examples include the solution by Peter Gabriel to a question
of Maurice Auslander concerning associative algebras of finite representation type (\cite{Gabriel}), results by Johannes Grassberger, Alastair King and Paulo Tirao on the homology of $2$-step nilpotent Lie algebras in \cite[Section 4]{GrassbergerKingTirao}, and in Geometry, we might mention Ernst Heintze's result
\cite[Theorem 1]{Heintze} and the Theorem 2.5 in the Milnor's seminal paper \cite{Milnor13}; both of these results use (implicitly) degenerations of Lie algebras to prove existence the Riemannian metrics on Lie groups with certain curvature properties. More recently, Yuri Nikolayevsky and Yurii Nikonorov
have given results on the curvature of solvable Lie groups by using explicitly degenerations (see \cite{NN13}).

At present, there is an overwhelmed proliferation of papers dealing with degenerations of different kinds of (linear) algebraic structures (including some families of Lie algebras), some of these ones are motivated by abstruse incentives and purposes. We expect our techniques would motivate more developments and
different strategies to attempt solving the above-mentioned problems in Lie theory or new potential applications of the notion of hom-Lie algebra in Lie theory.

\section{Preliminaries}

\subsection{Classification of $3$-dimensional complex Lie algebras and their automorphisms}

The classification of low dimensional complex Lie algebras was started earlier by Sophus Lie himself (see \cite[Abtheilung VI, Kap. 28, \S 136]{Lie}). 
Three-dimensional real Lie algebras have been investigated and completely classified by Luigi Bianchi in \cite[\S 198-199]{Bianchi} 
and have distinguished role not only in algebra and geometry, but also in cosmology; which was first noted by the Kurt G\"odel in 1949
(see \cite{Jantzen}). Since then, the classification of the Lie algebras mentioned above have been revised in several different ways by
using methods from representation theory, Lie theory or algebraic geometry (see \cite{Popovych2}, which is an arXiv extended version of  \cite{Popovych1}), or can be studied as an (guided) exercise in a usual Lie algebra course (see for instance  \cite[\textsc{Chapter} I, Section 15, Problems 28-35]{Knapp}).

\begin{theorem}[{\cite[Lemma 2]{B2}}]\label{3Lie}
Every complex Lie algebra of dimension $3$ is isomorphic to one and only one of the Lie algebras in the following list:
  \begin{itemize}
    \item $\mathfrak{L}_{0}$:= $\mathfrak{a}_{3}(\mathbb{C})$, the $3$-dimensional abelian Lie algebra.
    \item $\mathfrak{L}_{1}$:= $\mathfrak{n}_{3}(\mathbb{C})$, the $3$-dimensional Heisenberg Lie algebra: $[e_1,e_2]=e_3$.
    \item $\mathfrak{L}_{2}$:= $\mathfrak{r}_{3}(\mathbb{C})$ : $[e_1,e_2]=e_2,\,[e_1,e_3]=e_2+e_3$.
    \item $\mathfrak{L}_{3}$:= $\mathfrak{r}_{3,1}(\mathbb{C})$: $[e_1,e_2]=e_2,\,[e_1,e_3]=e_3$.
    \item $\mathfrak{L}_{4}$:= $\mathfrak{r}_{3,-1}(\mathbb{C})$, the Poincar\'e algebra $\mathfrak{p}(1,1)$: $[e_1,e_2]=e_2,\,[e_1,e_3]=-e_3$.
    \item $\mathfrak{L}_{5}(z)$:= $\mathfrak{r}_{3,z}(\mathbb{C})$ with $0<|z|<1$ or, $|z|=1$ and $\mathfrak{Im}(z)>0$: $[e_1,e_2]=e_2,\,[e_1,e_3]=ze_3$.
    \item $\mathfrak{L}_{6}$:= $\mathfrak{r}_{3,0}(\mathbb{C}) = \mathfrak{r}_{2}(\mathbb{C}) \oplus \mathbb{C}$: $[e_1,e_2]=e_2$.
    \item $\mathfrak{L}_{7}$:= $\mathfrak{sl}_{2}(\mathbb{C}) \cong \mathfrak{so}_{3}(\mathbb{C})$: $[e_1,e_2]=e_3,\,[e_2,e_3]=e_1,\,[e_3,e_1]=e_2$.
  \end{itemize}
\end{theorem}

\begin{remark}
  It is important to note that $\mathfrak{r}_{3,z} \cong \mathfrak{r}_{3,\widetilde{z}}$
  if and only if $(z-\widetilde{z})(z\widetilde{z}-1) = 0$.
\end{remark}

\begin{definition}
An \textit{almost Abelian Lie algebra} is a Lie algebra with a codimension one Abelian ideal.
\end{definition}

It is shown by Dietrich Burde and Manuel Ceballos in \cite[Proposition 3.1]{B3} that
a Lie algebra $\mathfrak{g}$  is almost Abelian if and only if $\mathfrak{g}$ has a codimension one Abelian \textbf{subalgebra}.
Any $n+1$-dimensional almost Abelian Lie algebra is in particular a solvable Lie algebra and can be identified with a square matrix
of size $n$. In fact, if $\mathfrak{g}$ is an almost Abelian Lie algebra, $\mathfrak{h}$ is a codimension $1$ Abelian ideal of $\mathfrak{g}$ and
$\operatorname{span}_{\mathbb{C}}(v_0)$ is a linear complement of $\mathfrak{h}$, then $\mathfrak{g}$ is determined by
the linear endomorphism $A=\operatorname{ad}(v_0)|_{\mathfrak{h}}$ ($\mathfrak{g} \cong \operatorname{Span}_{\mathbb{C}}(v_0) \ltimes \mathfrak{h}$). Conversely, given a linear endomorphism $A$ of $\mathbb{C}^{n}$, it is defined the almost Abelian Lie algebra $\mathfrak{r}_{A}$ by
the Lie algebra structure on the vector space $\mathbb{C}A \oplus \mathbb{C}^{n}$ satisfying $[A,v]=Av$
for any $v \in \mathbb{C}^{n}$ and $[v,w]=0$ for any $v,w$ in $\mathbb{C}^{n}$.

The following result is a rewrite of Theorem \ref{3Lie} and will be used throughout the paper.

\begin{proposition}\label{almost3Lie}
  Let $\mathfrak{g}$ be a $3$-dimensional complex Lie algebra. Then $\mathfrak{g}$ is a simple Lie algebra
  or $\mathfrak{g}$ is a solvable Lie algebra. Moreover, if $\mathfrak{g}$ is simple then $\mathfrak{g}$
  is isomorphic to $\mathfrak{so}(3,\mathbb{C})$ ($\cong \mathfrak{sl}(2,\mathbb{C})$), and
  if $\mathfrak{g}$ is solvable then $\mathfrak{g}$ is isomorphic to an almost Abelian Lie algebra
  $\mathfrak{r}_{A}$, which is isomorphic to
  \begin{itemize}
   \item the $3$-dimensional Abelian Lie algebra if and only if $A=0$,
   \item the Heisenberg Lie algebra if and only if $A\neq0$, $\operatorname{Det}(A) =0$ and $\operatorname{Trace}(A)=0$,
   \item $\mathfrak{r}_{3,1}(\mathbb{C})$ if and only if $A = t\operatorname{Id}$ for some $t \in \mathbb{C}^{\star}$.
   \item $\mathfrak{r}_{3}(\mathbb{C})$ if and only if $A \neq t\operatorname{Id}$ for all $t \in \mathbb{C}$,
   $\operatorname{Trace}(A)^2 = 4 \operatorname{Det}(A)$ and $\operatorname{Trace}(A)\neq 0$,
   \item $\mathfrak{r}_{3,-1}(\mathbb{C})$ if and only if $\operatorname{Trace}(A) = 0$ and $ \operatorname{Det}(A) \neq0$,
   \item $\mathfrak{r}_{2}(\mathbb{C}) \times \mathbb{C}$ if and only if $\operatorname{Trace}(A) \neq 0$ and $ \operatorname{Det}(A)  = 0$,
   \item $\mathfrak{r}_{3,z}(\mathbb{C})$ for some $z \in \mathbb{C}$ such that $z(z^2-1)\neq0$ if and only if $\operatorname{Trace}(A)^2 \neq 4 \operatorname{Det}(A)$,           $\operatorname{Trace}(A) \neq 0$ and $ \operatorname{Det}(A) \neq0$.
  \end{itemize}

\end{proposition}

The automorphisms of all three-dimensional real Lie algebras, as far as we know, were first investigated by Alex Harvey in \cite{Harvey}, where it is given
the identity component of each group.

\begin{proposition}\label{autoLie3}
The automorphism group of each Lie algebra given in Theorem \ref{3Lie} is given in the following list,
by identifying an automorphism with its matrix representation in the ordered basis $\{e_1, e_2, e_3\}$:

\begin{center}
    \begin{tabular}{c : p{3cm} : c : p{3cm}}
     Lie algebra & Automorphism& Lie algebra & \multicolumn{1}{c}{Automorphism}\\
    \hdashline
      $\mathfrak{a}_{3}(\mathbb{C})$ &  any $g \in \operatorname{GL}(3,\mathbb{C})$  &
      $\mathfrak{so}_{3}(\mathbb{C})$ &  any $g \in \operatorname{SO}(3,\mathbb{C})$\\
    \hdashline
      $\mathfrak{n}_{3}(\mathbb{C})$ & $\left(\begin{array}{cc | c} a & b & 0 \\ c & d & 0 \\ \hline x & y & ad-bc \\ \end{array}\right)$ &
    $\mathfrak{r}_{3,1}(\mathbb{C})$ &  $\left(\begin{array}{c | cc} 1 & 0 & 0 \\ \hline x & a & b \\ y & c & d \\ \end{array}\right)$\\
    \hdashline
      $\mathfrak{r}_{3}(\mathbb{C})$ &  $\left(\begin{array}{c| cc} 1 & 0 & 0 \\ \hline x & a & w \\ y & 0 & a \\ \end{array}\right)$ &
      \begin{tabular}{c}
        $\mathfrak{r}_{3,z}(\mathbb{C})$ \\ $(z^2-1)\neq0$ \end{tabular} &  $\left(\begin{array}{c | cc} 1 & 0 & 0 \\ \hline x & a & 0 \\ y & 0 & b \\ \end{array}\right)$\\
    \hdashline
      $\mathfrak{r}_{3,-1}(\mathbb{C})$ & \multicolumn{3}{l}{  $\left(\begin{array}{c | cc} 1 & 0 & 0 \\ \hline x & a & 0 \\ y & 0 & b \\ \end{array}\right)$
      or
      $\left(\begin{array}{c | cc} -1 & 0 & 0 \\ \hline x & 0 & a \\ y & b & 0 \\ \end{array}\right)$ }
      \\
    \hdashline
    \end{tabular}
\end{center}
In all cases the complex matrix must be non singular.

\end{proposition}

\subsection{hom-Lie algebras}

\begin{definition}[{\cite[Definition 14]{HLS}}, {\cite[Definition 1.3]{MS1}} and {\cite[Definition 2.1]{Sheng}}]

  A complex \textit{hom-Lie algebra} is a triple $(V,\cdot, A)$ where $V$ is a complex vector space,
  $(V,\cdot)$ is a \textit{skew-symmetric algebra} (which is the same as an \textit{anti-commutative algebra}), i.e.
  \begin{eqnarray*}
    x\cdot y & = & - y \cdot x,\, \forall x,y \in V
  \end{eqnarray*}
   and $A:V\rightarrow V$ is a linear transformation, called the \textit{twisting map}, satisfying the \textit{hom-Jacobi identity}:
   \begin{eqnarray}
   \operatorname{Jac}_{(\cdot,A)}(x_1,x_2,x_3) & = & \sum_{\sigma \in S_{3}} \operatorname{sign}(\sigma)Ax_{\sigma(1)}\cdot (x_{\sigma(2)} \cdot x_{\sigma(3)})\\
\nonumber                              & = & 0.
   \end{eqnarray}
   for any $x_1, x_2$ and $x_3$ in $V$. Here, $S_3$ is the permutation group of degree $3$ and $\operatorname{sign}( \sigma)$ is the signum of a permutation $\sigma$.
   If $A$ is also an algebra homomorphism of $(V,\cdot)$, i.e.
   \begin{eqnarray*}
    A(x\cdot y) = Ax\cdot Ay,\, \forall x,y \in V,
   \end{eqnarray*}
   then the hom-Lie algebra $(V,\cdot,A)$ is called \textit{multiplicative}. A \textit{regular hom-Lie algebra} is one for which
   the twisting map is an algebra automorphism.
\end{definition}

\begin{remark}
  Note that any Lie algebra defines a hom-Lie algebra by taking the twisting map to be the identity map.
  More generally, as it is noted in \cite[Corollary 2.6]{Yau}, it is easy to check that if $( V , \cdot)$ is a skew-symmetric algebra and $A$ is
  an algebra automorphism of $( V , \cdot)$, then $( V , \cdot , A)$ is a hom-Lie algebra
  if and only if $(V,[-,-])$ is a Lie algebra where $[-,-]$ is defined by
  \begin{eqnarray}
\label{asociada}    [x,y] & := & A^{-1}(x\cdot y),\, \forall x,y \in V.
  \end{eqnarray}

  Note that any regular hom-Lie algebra $(V,\cdot,A)$ and its associated Lie algebra $(V,[-,-])$, as it is defined in (\ref{asociada}),
  share so many algebraic properties. For instance, the most obvious observations are that $A$ is an algebra automorphism
  of both algebras $(V,\cdot)$ and $(V,[-,-])$ and also, such algebras have the same \textit{lower central series}
  and the \textit{derived series}.
 \end{remark}

\begin{remark}\label{vergne}
  Our primary interest in hom-Lie structures on Lie algebras is mainly motivated by
  the apparent relation between the hom-Jacobi identity and \textit{deformations of Lie algebras}.
  In fact, given a Lie algebra $(V,[-,-])$, we can naturally consider the vector space
  \begin{eqnarray}
     Z:=\left \{A \in \mathfrak{gl}(V) : \sum_{\sigma \in S_{3}} \operatorname{sign}(\sigma) [x_{\sigma(1)}, A[x_{\sigma(2)} , x_{\sigma(3)}]] = 0 \right\}.
  \end{eqnarray}
  It is easy to check that any $A\in Z$ defines a deformation of the Lie algebra $(V,[-,-])$ given by $\varphi:=A[-,-]$. Even better,
  the skew-symmetric bilinear map $\varphi$ is a \textit{trivial solution} to the \textit{Maurer-Cartan (deformation) equation} of the Lie algebra $(V,[-,-]=:\mu)$
  \begin{eqnarray}
  \delta_{\mu} \varphi + \frac{1}{2}[\![\varphi,\varphi ]\!]=0,
   \end{eqnarray}
  since $\varphi$ is a Lie bracket on $V$ and is a 2-cocycle for the adjoint representation of the Lie algebra $(V,\mu)$;
  and so $\mu + t\varphi$ is a Lie bracket on $V$ for all $t\in \mathbb{C}$.

  If $A$ is an algebra automorphism of $(V,[-,-]=\mu)$, then $A\in Z$ if and only if $(V,[-,-],A^{-1})$ is a regular hom-Lie algebra.
\end{remark}

\begin{definition}\label{derivation}
A \textit{derivation} of a hom-Lie algebra $(V,\cdot,A)$ is a linear transformation $D:V\rightarrow V$ such that
$D$ is a usual derivation of the algebra $(V,\cdot)$, i.e.
\begin{eqnarray}
  D (x\cdot y) & = &(Dx)\cdot y + x\cdot (Dy),\, \forall x,y \in V,
\end{eqnarray}
and also, $D$ commutes with the twisting map $A$. We denote by $\operatorname{Der}(V,\cdot,A)$
the Lie algebra of all derivations of the hom-Lie algebra $(V,\cdot,A)$.
\end{definition}

\begin{remark}
  In \cite[Definition 1.7]{Pasha}, it is introduced a well-founded notion of \textit{$A$-derivation of a hom-Lie algebra} $(V,\cdot,A)$;
  among others substantiated definitions so that these notions can be analogous concepts from Lie algebra theory.
  Thus, in \cite[Section 5]{Pasha}, the authors ask whether the Jacobson's result about invertible derivations
  is true for hom-Lie algebras. In fact, it is easy to check that if a regular hom-Lie algebra $(V,\cdot,A)$ admits an invertible $A$-derivation $D$,
  then $(V,\cdot,A)$ must be nilpotent (in the sense described in \cite[Definition 3.1]{Pasha}), since $A^{-1}D$ is a invertible
  derivation of the Lie algebra $(V,[-,-])$ with $[-,-]$ as it is defined in Equation (\ref{asociada}). But it is not true in general:
  consider the Lie algebra $\mathfrak{r}_{2} \times \mathbb{C}$ with multiplicative hom-Lie algebra structure
  given by the twisting map $A = \operatorname{diag}(1,0,1)$ and $A$-derivation defined by $D = \operatorname{diag}(1,1,1)$;
  here $A$ and $D$ are written with respect to the ordered basis $\{e_1,e_2,e_3\}$ given in Theorem \ref{3Lie}.

  In our case, the definition given in \ref{derivation} is  different from that of  \cite[Definition 1.7]{Pasha}
  and is related to what could be the \textit{automorphism group} of a hom-Lie algebra.
\end{remark}

\begin{definition}[{\cite[Pag. 331]{HLS}},{\cite[Section 2]{Sheng}}]\label{morfismo}
If $(V,\cdot,A)$ and $(U,\ast,B)$ are hom-Lie algebras,
\textit{a homomorphism of hom-Lie algebras from $(V,\cdot,A)$ to $(U,\ast,B)$ }
is a linear transformation $g:V \rightarrow U$ such that $g$ is an algebra homomorphism
between $(V,\cdot)$ and $(U,\ast)$, i.e.
\begin{eqnarray*}
  g(x\cdot y) = (gx)\ast(gy),\, \forall x\,y \in V,
\end{eqnarray*}
and also
\begin{eqnarray*}
  g \circ A = B \circ g.
\end{eqnarray*}
An invertible hom-Lie algebra homomorphism is called a \textit{hom-Lie algebra isomorphism}.
\end{definition}

From now on, we identify our $n$-dimensional complex vector space $V$ with $\mathbb{C}^{n}$
and we denote by $\{e_1,e_2,\ldots, e_n\}$ the canonical basis of $\mathbb{C}^{n}$.

\begin{definition}
Let $W$ be the vector space $C^{2}\times C^{1}$ where $C^{k}$ (with $k\geq 1$) is the vector space
\begin{eqnarray*}
\left \{ \varphi: \mathbb{C}^{n}  \times \ldots \times \mathbb{C}^{n} \rightarrow \mathbb{C}^{n} : \varphi \mbox{ is an alternating multilinear map of $k$ variables}  \right\},
\end{eqnarray*}
i.e. $C^{k} \cong \operatorname{Hom}(\Lambda^{k}\mathbb{C}^{n},\mathbb{C}^n) \cong \Lambda^{k}\mathbb{C}^{n} \otimes \mathbb{C}^n$.

We denote by $\operatorname{hom}\text{-}\mathcal{L}_{n}(\mathbb{C})$  the subset of $W$ of all hom-Lie algebra structures on $\mathbb{C}^{n}$.
\end{definition}
Note that $\operatorname{hom}\text{-}\mathcal{L}_{n}(\mathbb{C})$ is an affine algebraic subset of $W$. If we consider
structure constants of a element $(\mu,A)$ in $W$, $\mu(e_i,e_j) = \sum c_{i,j}^{k}e_k$ and $Ae_l = \sum a_{h,l}e_h$,
we have $(\mu,A) \in \operatorname{hom}\text{-}\mathcal{L}_{n}(\mathbb{C})$ if and only if  $\{c_{i,j}^{k}\} \cup \{a_{h,l}\}$
satisfy the polynomial relations determined by the hom-Jacobi identity
\begin{eqnarray*}
  \sum_{p,q =1}^{n} (a_{p,i}c_{j,k}^{q} + a_{p,j}c_{k,i}^{q}+a_{p,k}c_{i,j}^{q})c_{p,q}^{l} = 0.
\end{eqnarray*}
for all $1\leq i,j,k,l \leq n$. In practice, it is very easy to determine whether a given pair $(\mu,A) \in W$ is a hom-Lie algebra structure on $\mathbb{C}^{n}$, since the map $\operatorname{Jac}_{\mu,A}$ is an alternating tensor, and so we only need
to check that $\operatorname{Jac}_{\mu,A}(e_i,e_j,e_k) = 0$ for any $i,j,k$ with $1\leq i<j<k \leq n$ (one straightforward verification is enough in dimension $3$).

Recall that the complex general linear group $\operatorname{GL}(n,\mathbb{C})$ acts on $C^{k}$ ($k\geq 1$)
via \textit{change of basis}, i.e., given $\varphi \in C^{k}$
\begin{eqnarray}
  g \bullet \varphi(x_1,x_2,\ldots,x_k) := g\varphi(g^{-1}x_1, g^{-1}x_2,\ldots,g^{-1}x_k)
\end{eqnarray}
for $g \in \operatorname{GL}(n,\mathbb{C})$ and $x_1,x_2,\ldots, x_k \in \mathbb{C}^{n}$.
And so, $\operatorname{GL}(n,\mathbb{C})$ acts on the vector space $W$ in a natural way and $\operatorname{hom}\text{-}\mathcal{L}_{n}(\mathbb{C})$
is a $\operatorname{GL}(n,\mathbb{C})$-invariant subset of $W$. The \textit{orbit} of a point $(\mu,A)$ in $W$,
denoted by $O(\mu,A)$ (or by $\operatorname{GL}(n,\mathbb{C}) \bullet  (\mu,A)$) is the set of images of $(\mu,A)$ under the action by elements of $\operatorname{GL}(n,\mathbb{C})$:
\begin{eqnarray*}
  O(\mu,A) &: =& \left\{ g\bullet (\mu,A) : g \in \operatorname{GL}(n,\mathbb{C}) \right\}.
\end{eqnarray*}
Given a hom-Lie algebra structure $(\mu,A)$ in $\operatorname{hom}\text{-}\mathcal{L}_{n}(\mathbb{C})$,
it follows from Definition \ref{morfismo} that the orbit $ O(\mu,A)$ corresponds to the isomorphism class
of the hom-Lie algebra $(\mathbb{C}^{n},\mu,A)$.

For each $(\mu,A) \in \operatorname{hom}\text{-}\mathcal{L}_{n}(\mathbb{C})$,
the \textit{isotropy group} of  $(\mu,A)$ is the set of elements of $\operatorname{GL}(n,\mathbb{C})$
that fix $(\mu,A)$, i.e., $g \in \operatorname{GL}(n,\mathbb{C})$ such that $g\bullet \mu = \mu$
and $g \bullet A = A$ (or equivalently, $g A g^{-1} = A $). We can think of such group
as the \textit{automorphism group} of the hom-Lie algebra $(\mathbb{C}^{n},\mu,A)$, $\operatorname{Aut}(\mathbb{C}^{n},\mu,A)$;
in this way $\operatorname{Der}(\mathbb{C}^{n},\mu,A)$ (see Definition \ref{derivation}) is the Lie algebra of $\operatorname{Aut}(\mathbb{C}^{n},\mu,A)$.

The \textit{Zariski closure} of an arbitrary subset $Y \subseteq W$, denoted by $\overline{Y}^{\operatorname{Z}}$, is the smallest \textit{Zariski closed set} in $W$ containing $Y$; which is $\mathcal{V}(\mathcal{I}(Y))$.
In the case $Y=O(\mu,A)$ with $(\mu,A)$ a hom-Lie algebra structure on $\mathbb{C}^{n}$,
since $\operatorname{hom}\text{-}\mathcal{L}_{n}(\mathbb{C})$ is a Zariski closed set of $W$,
$\overline{O(\mu,A)}^{\operatorname{Z}} \subseteq \operatorname{hom}\text{-}\mathcal{L}_{n}(\mathbb{C})$.
Even though an orbit $O(v)$ in a $G$-variety is not necessarily a Zariski closed set, the \textit{dimension of an orbit}
is by definition the dimension of its Zariski closure and it is well-known that the isotropy group of $v$, $G_{v}$, is
an algebraic subgroup of $G$ and
\begin{eqnarray*}
  \operatorname{dim}(O(v)) & = & \operatorname{dim}(G) - \operatorname{dim}(G_{v})\\
                           & = & \operatorname{dim}(\operatorname{Lie}G) -\operatorname{dim}(\operatorname{Lie}G_{v});
\end{eqnarray*}
see for instance \cite[Sections 21.4, 23.2]{Tauvel}.

\begin{definition}[Degeneration]
  Let $(\mu,A)$ and $(\lambda,B)$ be two hom-Lie structures on $\mathbb{C}^{n}$. It is said that
  \textit{$(\mu,A)$ degenerates to $(\lambda,B)$} if $(\lambda,B) \in \overline{O(\mu,A)}^{\operatorname{Z}}$
  and this is denoted by $ (\mu,A) \xrightarrow{\text{\,\,deg\,\,}} (\lambda,B)$.
  The degeneration $ (\mu,A) \xrightarrow{\text{\,\,deg\,\,}} (\lambda,B)$ is called \textit{proper}
  if $(\lambda,B)$ is in the boundary of the orbit $O(\mu,A)$, or equivalently,
  if $(\mathbb{C}^{n},\mu,A)$ and $(\mathbb{C}^{n},\lambda,B)$ are not isomorphic hom-Lie algebras.
\end{definition}

The following proposition is due to Armand Borel (see \cite[Proposition 15.4]{Borel1}) and is very important to study degenerations of linear (algebraic) structures:

\begin{proposition}[{\cite[\textit{Closed orbit lemma}]{Borel2}}, {\cite[Proposition 21.4.5]{Tauvel}}]\label{orden}
Let $G$ be an algebraic group acting morphically on a non-empty
variety $\mathcal{Z}$. Then each orbit is a smooth variety which is open in its closure in
$\mathcal{Z}$. Its boundary is a union of orbits of strictly lower dimension. In particular,
the orbits of minimal dimension are closed.
\end{proposition}

An immediate consequence of the proposition above is that the dimension
of the algebra of derivations of a hom-Lie algebra is an \textit{obstruction}
to study its degenerations.

\begin{corollary}\label{niceDer}
  If $ (\mu,A) \xrightarrow{\text{\,\,deg\,\,}} (\lambda,B)$ is a proper degeneration
  then
  \begin{eqnarray*}
  \operatorname{dim}(\operatorname{Der}) (\mathbb{C}^{n},\mu,A) &<& \operatorname{dim}(\operatorname{Der}) (\mathbb{C}^{n},\lambda,B).
  \end{eqnarray*}
\end{corollary}

Although it is generally quite difficult to work with the Zariski topology,
the following observation that has originally been proved by Jean-Pierre Serre
in {\cite[Proposition 5.]{Serre}} (see also {\cite[Section 2.2]{Borel3}}, {\cite[Pag. 165]{Weil}} or
the exposition given by David Mumford of \textit{Complex varieties} in \cite[\S 10]{M})
shows how to study Zariski closure of orbits by using tools from elementary calculus.

\begin{proposition}\label{EucZar}
  Let $\mathcal{Z}$ be a complex algebraic variety and Let $U$ be a Zariski open and Zariski dense subset of $\mathcal{Z}$.
  Then $U$ is dense in $\mathcal{Z}$ with respect to Euclidean subspace topology, and, in consequence,
  the closures of $U$ with respect to both topologies are the same.
\end{proposition}


Other basic concept to consider is related with the notion of \textit{rigidity}

\begin{definition}[Rigid hom-Lie algebra]
A hom-Lie algebra $(\mathbb{C}^{n},\mu,A)$ is called rigid (in $\operatorname{hom}\text{-}\mathcal{L}_{n}(\mathbb{C})$)
if the $\operatorname{GL}(n,\mathbb{C})$-orbit of $(\mu,A)$ is a Zariski open set of $\operatorname{hom}\text{-}\mathcal{L}_{n}(\mathbb{C})$.
\end{definition}

A consequence of this definition is that a rigid hom-Lie algebra is not a proper
degeneration of other hom-Lie algebra. In fact, a rigid hom-Lie algebra
is \textit{responsible} for one of the irreducibles components of $\operatorname{hom}\text{-}\mathcal{L}_{n}(\mathbb{C})$; i.e.,
the Zariski closure of a rigid hom-Lie algebra is an irreducible component in the mentioned algebraic set.
Therefore, there exists only a finite number of rigid hom-Lie algebra in each dimension.

\begin{remark}
A hom-Lie algebra $(\mathbb{C}^{n},\mu,A)$ with $(\mathbb{C}^{n},\mu)$ a Lie algebra cannot be rigid in $\operatorname{hom}\text{-}\mathcal{L}_{n}(\mathbb{C})$ because $(\mathbb{C}^{n},\mu, A + t\operatorname{Id})$ (with $t$  sufficiently close to zero) is a one-parameter family of hom-Lie algebras which are not
isomorphic to $(\mathbb{C}^{n},\mu, A)$. If $(\mathbb{C}^{n},\mu,A)$ is a hom-Lie algebra with $A$ a non-nilpotent transformation, then
it can be shown quickly that $(\mathbb{C}^{n},\mu, A)$ is not rigid because $(\mathbb{C}^{n},\mu,(1+t)A)$ is a one-parameter family of hom-Lie algebras which are not
isomorphic to $(\mathbb{C}^{n},\mu, A)$.
\end{remark}

Other interesting consequence of the Proposition \ref{orden} is worth noting that the notion
of degeneration defines a \textit{partial order} on the orbit space $\operatorname{hom}\text{-}\mathcal{L}_{n}(\mathbb{C}) / \operatorname{GL}(n,\mathbb{C})$ given
by $\operatorname{GL}(n,\mathbb{C})\bullet (\lambda,B) \leq \operatorname{GL}(n,\mathbb{C})\bullet (\mu,A)$ if $(\mu,A)$ degenerates to $(\lambda,B)$.
And so, once we have an ``enumeration" of $n$-dimensional hom-Lie algebras (with certain properties), the next natural step would be to keep ``order" in such set. This kind of ideas originally go back to works of Albert Nijenhuis and Roger Wolcott Richardson in \cite{NR1, NR2},
Murray Gerstenhaber (see \cite{Gerstenhaber}) and Peter Gabriel \cite{Gabriel}. Obtain
a classification of orbit closures in a variety of algebraic structures of certain type is occasionally
called a \textit{geometric classification} (term introduced by Guerino Mazzola in \cite{mazzola} inspired by \cite{Gabriel}).

\subsubsection{Other varieties of hom-Lie algebras}

One can consider other algebraic sets related with hom-Lie algebras in the same way as we have introduced the set $\operatorname{hom}\text{-}\mathcal{L}_{n}(\mathbb{C})$
and can also study the notions of rigidity and degenerations;
for instance, the algebraic set of $n$-dimensional complex multiplicative hom-Lie algebras, $\operatorname{hom}^{\operatorname{m}}\text{-}\mathcal{L}_{n}(\mathbb{C})$ and if the twisting map is prescripted, say $A$, it induces the algebraic subset of $C^{2}$, $\operatorname{hom}_{A}\text{-}\mathcal{L}_{n}(\mathbb{C})$, given by the skew-symmetric bilinear maps $\mu \in C^{2} $ such that $(\mathbb{C}^{n},\mu,A)$ is a hom-Lie algebra (or $\operatorname{hom}_{A}^{\operatorname{m}}\text{-}\mathcal{L}_{n}(\mathbb{C})$
for the multiplicative case). Whereas $\operatorname{GL}(n,\mathbb{C})$  naturally acts on $\operatorname{hom}\text{-}\mathcal{L}_{n}(\mathbb{C})$,
we only have the action of $\operatorname{GL}(n,\mathbb{C})_{A}$, the set of invertible matrices that commute with $A$, in the set $\operatorname{hom}_{A}\text{-}\mathcal{L}_{n}(\mathbb{C})$.

By following \cite[Section 24]{NR2}, we have sufficient conditions for a hom-Lie algebra to be rigid in each algebraic set.
In fact, let $(\mathbb{C}^{n},\mu,A)$ be a (multiplicative) hom-Lie algebra and let us denote by $\mathfrak{gl}(n,\mathbb{C})\bullet (\mu,A)$
and $\mathfrak{gl}_{A}(n,\mathbb{C})\bullet\mu$ the first approximations to the ``Zariski tangent'' to $\operatorname{GL}(n,\mathbb{C})\bullet(\mu,A)$ and $\operatorname{GL}_{A}(n,\mathbb{C})\bullet \mu$ at $(\mu,A)$ and $\mu$ respectively. More precisely, this means
\begin{eqnarray*}
  \mathfrak{gl}(n,\mathbb{C})\bullet (\mu,A)  &:= &  \left\{ (\lambda,B)\in  C^{2}\times C^{1} :
 \begin{array}{l}
  \lambda = \delta_{\mu}(X) \mbox{ and }  B = AX-XA, \\
   \mbox{ with } X\in \mathfrak{gl}(n,\mathbb{C})
  \end{array}
  \right\}
  \\
  \mathfrak{gl}_{A}(n,\mathbb{C})\bullet\mu  & := &
  \left\{
  \lambda \in  C^{2} :
   \begin{array}{l}
   \lambda = \delta_{\mu}(X) \mbox{ with } X\in \mathfrak{gl}_{A}(n,\mathbb{C})
  \end{array}
  \right\}
\end{eqnarray*}
where
\begin{eqnarray*}
   \delta_{\mu}(X)(y,z)&:=& X\mu(y,z) - \mu(X y,z) - \mu(y, Xz),\, \forall y,z \in \mathbb{C}^{n}
\end{eqnarray*}
and $\mathfrak{gl}(n,\mathbb{C})_{A}$ is the set of complex matrices $n\times n$ that commute with $A$.

Similarly, let $T_{1}(\mu,A)$, $T_{2}(\mu,A)$, $T_{3}(\mu)$ and $T_{4}(\mu)$ be denote the first approximations to the \textit{Zariski tangent} to
$\operatorname{hom}\text{-}\mathcal{L}_{n}(\mathbb{C})$, $\operatorname{hom}^{\operatorname{m}}\text{-}\mathcal{L}_{n}(\mathbb{C})$,
$\operatorname{hom}_{A}\text{-}\mathcal{L}_{n}(\mathbb{C})$ and $\operatorname{hom}_{A}^{\operatorname{m}}\text{-}\mathcal{L}_{n}(\mathbb{C})$
at $(\mu,A)$ and $\mu$ respectively: for $(\mu,A)$ we have
\begin{eqnarray*}
  T_{1}(\mu,A) & := &
  \left\{ (\lambda,B) \in  C^{2}\times C^{1} :( \operatorname{d}\operatorname{Jac})_{(\mu,A)}(\lambda,B) =0
   \right\}\\
  T_{2}(\mu,A)  & := &
  \left\{
  (\lambda,B) \in  C^{2}\times C^{1} : ( \operatorname{d}\operatorname{Jac})_{(\mu,A)}(\lambda,B) =0 \mbox{ and } ( \operatorname{d}\operatorname{m})_{(\mu,A)}(\lambda,B) =0
  \right\},
\end{eqnarray*}
where
\begin{eqnarray}
  \begin{array}{lll}
    (\operatorname{d}\operatorname{Jac})_{(\mu,A)}(\lambda,B)(x_1,x_2,x_3) & = &
    \sum_{\sigma \in S_{3}} \operatorname{sign}(\sigma) \mu     (Ax_{\sigma(1)} , \lambda(x_{\sigma(2)} , x_{\sigma(3)})) + \\
    & & \sum_{\sigma \in S_{3}} \operatorname{sign}(\sigma) \lambda (Ax_{\sigma(1)} , \mu(x_{\sigma(2)} , x_{\sigma(3)})) +\\
    & & \sum_{\sigma \in S_{3}} \operatorname{sign}(\sigma) \mu (Bx_{\sigma(1)} , \mu(x_{\sigma(2)} , x_{\sigma(3)}))
 \end{array}
\end{eqnarray}
and
\begin{eqnarray}
  \begin{array}{lll}
    (\operatorname{d}\operatorname{m})_{(\mu,A)}(\lambda,B)(y,z) & = & \left( A\lambda(y,z) -\lambda(Ay,Az) \right) +\\
     & & \left( B\mu(y,z) - \mu(Ay,Bz) - \mu(By,Az) \right),\, \forall y,z \in \mathbb{C}^{n}.
 \end{array}
\end{eqnarray}

With respect to the sets $\operatorname{hom}_{A}\text{-}\mathcal{L}_{n}(\mathbb{C})$ and $\operatorname{hom}_{A}^{\operatorname{m}}\text{-}\mathcal{L}_{n}(\mathbb{C})$ we have
\begin{eqnarray*}
  T_{3}(\mu) & := &
  \left\{
  \lambda \in  C^{2} :( \operatorname{d}\operatorname{Jac}_{A})_{\mu}\lambda =0
   \right\}\\
  T_{4}(\mu)  & := &
  \left\{
  \lambda \in  C^{2} :( \operatorname{d}\operatorname{Jac}_{A})_{\mu}\lambda =0 \mbox{ and } ( \operatorname{d}\operatorname{m})_{A}\lambda =0
  \right\},
\end{eqnarray*}
where
\begin{eqnarray}
\begin{array}{lcl}
    (\operatorname{d}\operatorname{Jac}_{A})_{\mu}\lambda(x_1,x_2,x_3) & = &
    \sum_{\sigma \in S_{3}} \operatorname{sign}(\sigma) \mu     (Ax_{\sigma(1)} , \lambda(x_{\sigma(2)} , x_{\sigma(3)})) + \\
    & & \sum_{\sigma \in S_{3}} \operatorname{sign}(\sigma) \lambda (Ax_{\sigma(1)} , \mu(x_{\sigma(2)} , x_{\sigma(3)}))
\end{array}
\end{eqnarray}
and
\begin{eqnarray}
(\operatorname{d}\operatorname{m})_{A}\lambda(y,z) & = &  A\lambda(y,z) -\lambda(Ay,Az),\, \forall y,z \in \mathbb{C}^{n}.
\end{eqnarray}
If $\mathfrak{gl}(n,\mathbb{C})\bullet (\mu,A)$ coincides with  $T_{1}(\mu,A)$ or $\mathfrak{gl}_{A}(n,\mathbb{C})\bullet\mu$
or $\mathfrak{gl}_{A}(n,\mathbb{C})\bullet\mu$ with $T_{3}(\mu)$, then $(\mu,A)$ is rigid in $\operatorname{hom}\text{-}\mathcal{L}_{n}(\mathbb{C})$,
respectively $\mu$ is rigid in $\operatorname{hom}_{A}\text{-}\mathcal{L}_{n}(\mathbb{C})$ (and analogously in the multiplicative case).

It is to be expected that the above conditions are not necessary for the rigidity of a hom-Lie algebra
in such algebraic sets; since a rigid hom-Lie algebra can satisfy additional (intrinsic) identities which are ``independent'' of
the hom-Jacobi condition. For instance, it is easy to give a linear transformation $A$ of $\mathbb{C}^{3}$ such that $\operatorname{hom}_{A}^{\operatorname{m}}\text{-}\mathcal{L}_{3}(\mathbb{C})$ has only two $\operatorname{GL}_{A}(3,\mathbb{C})$-orbits,
one of which corresponds to a rigid hom-Lie algebra not satisfying the above sufficient condition (consider $A$ defined by
$Ae_{1} = 0$, $Ae_2 = 2 e_2 + e_3$ and $Ae_3 = e_3$).

\section{Classification}\label{homLie3}

Our purpose in this section is to give the classification of the hom-Lie algebra structures on
$3$-dimensional complex Lie algebras with nilpotent twisting map, up to isomorphism of
hom-Lie algebras. Let $(V,\cdot)$ be a skew-symmetric algebra and let us consider the vector
space of hom-Lie structures on $(V,\cdot)$
\begin{eqnarray*}
 \operatorname{hom-Lie}(V,\cdot) & := & \left\{ A \in \operatorname{End}(V) : \operatorname{Jac}_{(\cdot,A)} = 0\right\}.
\end{eqnarray*}
It follows from Definition \ref{morfismo} that two hom-Lie algebras $(V,\cdot,A)$ and $(V,\cdot,B)$ are
isomorphic if and only if $A$ and $B$ are conjugate with respect to an automorphism of the algebra $(V,\cdot)$.
Therefore, for each Lie algebra $\mathfrak{g}$ in Theorem \ref{3Lie}, we need to study $\operatorname{Aut}(\mathfrak{g})$-conjugacy classes in
$\operatorname{hom-Lie}(\mathfrak{g})$. And so, for instance, by the \textit{cyclic decomposition of a nilpotent operator}, the $3$-dimensional complex abelian
Lie algebra has only three hom-Lie algebra structures with nilpotent twisting map, up to isomorphism:

\begin{proposition}
  The hom-Lie algebra structures with nilpotent twisting map on the $3$-dimensional complex abelian Lie algebra are given by (up to isomorphism):
  \begin{enumerate}
    \item $\mathfrak{L}_{0}^{0}$: $(\mathfrak{a}_{3}(\mathbb{C}),A_0)$ with $A_0$ the zero map,
    \item $\mathfrak{L}_{0}^{1}$: $(\mathfrak{a}_{3}(\mathbb{C}),A_1)$ with $A_1 e_1=0$, $A_1 e_2 =0$ and $A_1 e_3 =e_2$,
    \item $\mathfrak{L}_{0}^{2}$: $(\mathfrak{a}_{3}(\mathbb{C}),A_2)$ with $A_2 e_1=0$, $A_2 e_2 =e_1$ and $A_2 e_3 =e_2$.
  \end{enumerate}
\end{proposition}

\subsection{hom-Lie structures on $\mathfrak{so}_{3}(\mathbb{C})$}

As we mentioned above, it is important to determine the vector space $\operatorname{hom-Lie}(\mathfrak{so}_{3}(\mathbb{C}))$.
It is straightforward to show that:

\begin{lemma}
An endomorphism $A$ of $\mathbb{C}^{3}$ defines a hom-Lie algebra structure on $\mathfrak{so}_{3}(\mathbb{C})$
if and only if $A$ is a symmetric operator for the Killing form of $\mathfrak{so}_{3}(\mathbb{C})$,
or equivalently the matrix of $A$ with respect to the ordered basis $\{e_1,e_2,e_3\}$ is a symmetric complex matrix.
\end{lemma}

Therefore, since $\operatorname{Aut}(\mathfrak{so}_{3}(\mathbb{C})) = \operatorname{SO}(3,\mathbb{C})$, we need to get
\textit{canonical forms} for complex symmetric matrices under complex orthogonal similarity.
Nigel Scott presents a solution to this problem in \cite{Scott} (or see also the references given there),
which can be used to classify, up to isomorphism, the hom-Lie algebra structures on $\mathfrak{so}(3,\mathbb{C})$
(not just the ones corresponding to nilpotent twisting maps). In our case, we have:

\begin{proposition}
 Any hom-Lie algebra $(\mathfrak{so}(3,\mathbb{C}), A)$ with $A$ nilpotent operator is isomorphic
   to some hom-Lie algebra $\mathfrak{L}_{7}^{i}:= (\mathfrak{so}(3,\mathbb{C}) , A_i)$ where $A_0$ is the Zero map and
   the matrix of $A_i$ with respect to the ordered basis $\{e_1,e_2,e_3\}$ is:
   \begin{center}
  \begin{tabular}{ c : c}
    $A_1$ & $A_2$   \\
    $\left(
      \begin{array}{c cc}
        0 & 0 & 0 \\
        0 & 1 & \sqrt{-1} \\
        0 & \sqrt{-1} & -1 \\
      \end{array}
    \right)$
    &
    $\left(
      \begin{array}{ccc}
        0 & 1 & \sqrt{-1} \\
        1 & 0 & 0 \\
        \sqrt{-1} & 0 & 0 \\
      \end{array}
    \right)$

  \end{tabular}
   \end{center}
Two hom-Lie algebras $(\mathfrak{so}(3,\mathbb{C}),A_i)$ and $(\mathfrak{so}(3,\mathbb{C}),A_j)$ are isomorphic if and only if $A_i = A_j$.

\end{proposition}

\subsection{hom-Lie structures on $\mathfrak{r}_{3,z}(\mathbb{C})$}

It is easy to see that a linear map $A$ defines a hom-Lie algebra structure on $\mathfrak{r}_{2}(\mathbb{C}) \times \mathbb{C}$
if and only if $A e_3 \in \operatorname{Span}_{\mathbb{C}}\{e_2,e_3\}$. With respect to $\mathfrak{r}_{3,z}(\mathbb{C})$ ($z\neq 0$),
its hom-Lie algebras structures are exactly those linear maps that leave its derived algebra invariant.  On the other hand,
the Lie algebras $\mathfrak{r}_{3,z}(\mathbb{C})$, for all $z \in \mathbb{C}$ have \textit{almost} the same automorphism group;
$\mathfrak{r}_{3,-1}$ has \textit{one} more automorphism and $\mathfrak{r}_{3,1}$ has others more.

So, we first study the action by conjugation of the group
$$
G:=\left\{
\left[
\begin {array}{c | c c}
1&0&0\\
\hline
x&a&0\\
y&0&b\end {array}
\right]
\in M(3,\mathbb{C})
:
a,b \in \mathbb{C}^{*} \right\}
$$
on the set
$$
\mathcal{N}_{1}:=\left\{
A \in M(3,\mathbb{C})
:
\begin{array}{l}
A \mbox{ is a nilpotent matrix and } \\
Ae_3 \in \operatorname{Span}_{\mathbb{C}}\{e_2,e_3\}
\end{array}
\right\}
$$
which contains the $G$-invariant set
$$
\mathcal{N}_{2}:=\left\{
A \in M(3,\mathbb{C})
:
\begin{array}{l}
A \mbox{ is a nilpotent matrix and } \\
 A\operatorname{Span}_{\mathbb{C}}\{e_2,e_3\} \subseteq \operatorname{Span}_{\mathbb{C}}\{e_2,e_3\}
\end{array}
\right\}
$$

\subsubsection{$G$-conjugacy classes in $\mathcal{N}_{2}$ }\label{classesN2}
Let us denote by $\mathfrak{k}$ the vector space $\operatorname{span}_{\mathbb{C}}\{e_2,e_3\}$. Note that $\mathfrak{k}$ is a codimension-$1$ invariant subspace for any $A$ in $\mathcal{N}_{2}$, and so $\operatorname{Im}A \subseteq \mathfrak{k}$. Also, $A |_\mathfrak{k} $ is obviously a nilpotent operator.

First consider the case in which $A |_\mathfrak{k}$ is the zero map. If $A$ is
non-zero then the degree of nilpotency of $A$ is $2$. Since $Ae_{1} \in \mathfrak{k}$
is easy to see that $A$ is $G$-similar to some of the following nilpotent matrices of $\mathcal{N}_{2}$ given by
\begin{itemize}
  \item $A_1 e_1 = e_2$, $A_1 e_2 = 0$, $A_1 e_3 = 0$,
  \item $A_2 e_1 = e_3$, $A_2 e_2 = 0$, $A_2 e_3 = 0$,
  \item $A_3 e_1 = e_2+e_3$, $A_3 e_2 = 0$, $A_3 e_3 = 0$.
\end{itemize}
For example, if $A e_1 = a e_2 +b e_3$ with $a,b \neq 0$, let $g$ be the matrix of $G$
defined by $ge_1 = e_1$, $g e_2 = ae_2$ and $g e_3 = b e_3$. Since $g (e_2 + e_3) = A e_1$,
we have $g^{-1} A g = A_3$.

The case in which $A |_\mathfrak{k}$  has nilpotency degree equal to $2$,
according to how $\operatorname{Ker} A |_\mathfrak{k}$ is positioned in $\mathfrak{k}$,
we have three types of possible $\widetilde{G}$-conjugacy classes for $A |_\mathfrak{k}$:

\begin{itemize}
  \item[$\star$] $B_1 e_2 = 0$, $B_1 e_3 = e_2$,
  \item[$\star$] $B_2 e_2 = e_3$, $B_2 e_3 = e_0$,
  \item[$\star$] $B_3 e_2 = B_3 e_3 = \lambda (e_2 - e_3)$  with $\lambda \in \mathbb{C}^{\star}$.
\end{itemize}

Here, $\widetilde{G}$ is the group $\{ g: \mathfrak{k} \rightarrow \mathfrak{k} : ge_2=ae_2, ge_3=be_3 \mbox{ wiht }a,b \in \mathbb{C}^{\star}\}$.
For example, if both $e_2$ and $e_3$ do not belong to the $\operatorname{Ker} A |_\mathfrak{k} $,
we have $A |_\mathfrak{k}$ sends $e_2$ to a nonzero vector $v_0$, and $e_3$ to $ t v_0$ with $t\neq 0$
($\operatorname{dim} (\operatorname{Im} (A |_\mathfrak{k})) = 1$). Let $g$ be the linear transformation of $\widetilde{G}$
given by $g e_2 = e_2$ and $g e_3 = \tfrac{1}{t} e_3$, we have $g^{-1} A |_\mathfrak{k} g$
is of the form described in the third type.

If $A \in \mathcal{N}_{2}$ is a nilpotent matrix of degree $2$ with $A|_{\mathfrak{k}}$ nonzero,
we have a vector $v_0$ in $\operatorname{Ker}(A)$ of the form $e_1 + w_0$ with $w_0 \in \mathfrak{k}$, and so $A$ is $G$-similar to
a matrix $\widetilde{A}$ in $\mathcal{N}_{2}$ with $Ae_1 =0$; consider $g e_1 =v_0$, $ge_2=e_2 $ and $ge_3=e_3$,
we see that $g\in G$ and $g^{-1}Age_1 = 0$. In turn  $\widetilde{A}$ is $G$-similar to
a matrix $A_{i}$ ($4 \leq i\leq 6$) such that $A_i e_1 =0$ and $A_i | \mathfrak{k}$ coincides
with $B_{i-3}$; as we observed above.

If $A \in \mathcal{N}_{2}$ is a nilpotent matrix of degree $3$, then $A|_{\mathfrak{k}}$ is nonzero
and $A$ is $G$-similar to a matrix $\widetilde{A}$ in $\mathcal{N}_{2}$ with
$A|_{\mathfrak{k}}$ equal to some $B_i$. When $i=2$ or $3$, since $\operatorname{Im}(\widetilde{A}) = \mathfrak{k}$,
there exists a vector $v_1$ of the form $e_1 + w_1$, with $w_1\in \mathfrak{k}$ such that $\widetilde{A} v_1 = t e_2$, $t \neq 0$ . By considering
the linear map $g \in G$ given by $ge_1 = v_1$, $ge_2=t e_2$ and $ge_3= t e_3$, we have $\widetilde{A}$ is $G$-similar to $g^{-1} \widetilde{A} g$:

\begin{itemize}
  \item $A_7 e_1 = e_2$, $A_7 e_2 =e_3$, $A_7 e_3 = 0$.
  \item $A_9 e_1 = e_2$, $A_9 e_2 = A_9 e_3 = \lambda(e_2-e_3)$.
\end{itemize}

In the case when $i=1$, by a similar argument as above (with $e_3$ playing the role of $e_2$), we have $\widetilde{A}$ is $G$-similar to a matrix of the form

\begin{itemize}
  \item $A_8 e_1 = e_3$, $A_8 e_2 =0$, $A_8 e_3 = e_2$.
\end{itemize}

The above discussion can be rephrased in terms of hom-Lie algebra structures on $\mathfrak{r}_{3,z}$ ($z(z^2-1)\neq0$)
with nilpotent twisting map:

\begin{proposition}
   Any hom-Lie algebra $(\mathfrak{r}_{3,z} , A)$ with $z(z^2-1)\neq0$ and $A$ nilpotent operator is isomorphic
   to some hom-Lie algebra $\mathfrak{L}_{5}^{i}(z):= (\mathfrak{r}_{3,z} , A_i)$ where $A_0$ is the zero map and
   the matrix of $A_i$ with respect to the ordered basis $\{e_1,e_2,e_3\}$ is:
   \begin{center}
  \begin{tabular}{c : c : c}
    $A_1$ & $A_2$ & $A_3$ \\
    $\left(
      \begin{array}{c | cc}
        0 & 0 & 0 \\
        \hline
        1 & 0 & 0 \\
        0 & 0 & 0 \\
      \end{array}
    \right)$
    &
    $\left(
      \begin{array}{c | cc}
        0 & 0 & 0 \\
        \hline
        0 & 0 & 0 \\
        1 & 0 & 0 \\
      \end{array}
    \right)$
    &
    $\left(
      \begin{array}{c | cc}
        0 & 0 & 0 \\
        \hline
        1 & 0 & 0 \\
        1 & 0 & 0 \\
      \end{array}
    \right)$\\
 \hdashline
    $A_4$ & $A_5$ & $A_6(\lambda)$ \\
    $\left(
      \begin{array}{c | cc}
        0 & 0 & 0 \\
        \hline
        0 & 0 & 1 \\
        0 & 0 & 0 \\
      \end{array}
    \right)$
    &
    $\left(
      \begin{array}{c | cc}
        0 & 0 & 0 \\
        \hline
        0 & 0 & 0 \\
        0 & 1 & 0 \\
      \end{array}
    \right)$
    &
    $\left(
      \begin{array}{c | cc}
        0 & 0 & 0 \\
        \hline
        0 & \lambda &  \lambda \\
        0 & - \lambda & - \lambda \\
      \end{array}
    \right)$\\
  \hdashline
    $A_7$ & $A_8$ & $A_9(\lambda)$ \\
    $\left(
      \begin{array}{c | cc}
        0 & 0 & 0 \\
        \hline
        1 & 0 & 0 \\
        0 & 1 & 0 \\
      \end{array}
    \right)$
    &
    $\left(
      \begin{array}{c | cc}
        0 & 0 & 0 \\
        \hline
        0 & 0 & 1 \\
        1 & 0 & 0 \\
      \end{array}
    \right)$
    &
    $\left(
      \begin{array}{c | cc}
        0 & 0 & 0 \\
        \hline
        1 &  \lambda &  \lambda \\
        0 & - \lambda & - \lambda \\
      \end{array}
    \right)$
  \end{tabular}
   \end{center}
Two hom-Lie algebras $(\mathfrak{r}_{3,z},A_i)$ and $(\mathfrak{r}_{3,z},A_j)$ are isomorphic if and only if $A_i = A_j$.
\end{proposition}

The last assertion follows by using the invariants given in the Tables \ref{tab:table1}, \ref{tab:tableder} and \ref{tab:table3}, or alternatively, by taking advantage of the fact that we have all elements of $\operatorname{Aut}(\mathfrak{r}_{3,z})$ expressed explicitly and by showing the affirmation with
straightforward computations.

Now, by considering the group $\operatorname{Aut}(\mathfrak{r}_{3,-1} )$ , which contains the group $G$
and has (essentially) one element more: $g e_1 =-e_1$, $g e_2 = e_3$ and $g e_3 = e_2$, it is easy
to check that the linear maps $A_1$ and $A_2$ given in the above discussion are $\operatorname{Aut}(\mathfrak{r}_{3,-1} )$-similar,
and the same is true for the pairs of maps $A_4$ and $A_5$, $A_7$ and $A_8$.

\begin{proposition}
   Any hom-Lie algebra $(\mathfrak{r}_{3,-1} , A)$ with nilpotent twisting map is isomorphic
   to some hom-Lie algebra $\mathfrak{L}_{4}^{i}:= (\mathfrak{r}_{3,-1} , A_i)$ where $A_0$ is the Zero map and
   the matrix of $A_i$ with respect to the ordered basis $\{e_1,e_2,e_3\}$ is :
   \begin{center}
  \begin{tabular}{c : c}
    $A_1$ & $A_2$ \\
    $\left(
      \begin{array}{c | cc}
        0 & 0 & 0 \\
        \hline
        1 & 0 & 0 \\
        0 & 0 & 0 \\
      \end{array}
    \right)$
    &
    $\left(
      \begin{array}{c | cc}
        0 & 0 & 0 \\
        \hline
        1 & 0 & 0 \\
        1 & 0 & 0 \\
      \end{array}
    \right)$\\
  \hdashline
    $A_4$ & $A_5(\lambda)$ \\
    $\left(
      \begin{array}{c | cc}
        0 & 0 & 0 \\
        \hline
        0 & 0 & 1 \\
        0 & 0 & 0 \\
      \end{array}
    \right)$
    &
    $\left(
      \begin{array}{c | cc}
        0 & 0 & 0 \\
        \hline
        0 & \lambda &  \lambda \\
        0 & - \lambda & - \lambda \\
      \end{array}
    \right)$\\
    & $\lambda \in \mathbb{R}_{>0}$ or $\mathfrak{Im}(\lambda)>0$ \\
  \hdashline
    $A_5$ & $A_6(\lambda)$ \\
    $\left(
      \begin{array}{c | cc}
        0 & 0 & 0 \\
        \hline
        1 & 0 & 0 \\
        0 & 1 & 0 \\
      \end{array}
    \right)$
    &
    $\left(
      \begin{array}{c | cc}
        0 & 0 & 0 \\
        \hline
        1 &  \lambda &  \lambda \\
        0 & - \lambda & - \lambda \\
      \end{array}
    \right)$\\
 & $\lambda \in \mathbb{R}_{>0}$ or $\mathfrak{Im}(\lambda)>0$ \\
  \end{tabular}
   \end{center}
Two hom-Lie algebras $(\mathfrak{r}_{3,-1},A_i)$ and $(\mathfrak{r}_{3,-1},A_j)$ are isomorphic if and only if $A_i = A_j$.
\end{proposition}

With respect to the hom-Lie algebra structures on $\mathfrak{r}_{3,1}$ with nilpotent twisting map,
by using that $\operatorname{Aut}(\mathfrak{r}_{3,1})$ contains to the group $G$ and studying the
$\operatorname{Aut}(\mathfrak{r}_{3,1})$-conjugacy classes of the matrices obtained
 at the beginning of \ref{classesN2}, we deduce that:

\begin{proposition}
 Any hom-Lie algebra $(\mathfrak{r}_{3,1} , A)$ with nilpotent twisting map is isomorphic
   to some hom-Lie algebra $\mathfrak{L}_{3}^{i}:= (\mathfrak{r}_{3,1} , A_i)$ where $A_0$ is the Zero map and
    the matrix of $A_i$ with respect to the ordered basis $\{e_1,e_2,e_3\}$ is
   \begin{center}
  \begin{tabular}{ c : c : c}
    $A_1$ & $A_2$ & $A_3$ \\
    $\left(
      \begin{array}{c | cc}
        0 & 0 & 0 \\
        \hline
        1 & 0 & 0 \\
        0 & 0 & 0 \\
      \end{array}
    \right)$
    &
    $\left(
      \begin{array}{c | cc}
        0 & 0 & 0 \\
        \hline
        0 & 0 & 1 \\
        0 & 0 & 0 \\
      \end{array}
    \right)$
    &
    $\left(
      \begin{array}{c | cc}
        0 & 0 & 0 \\
        \hline
        1 & 0 & 0 \\
        0 & 1 & 0 \\
      \end{array}
    \right)$
  \end{tabular}
   \end{center}
Two hom-Lie algebras $(\mathfrak{r}_{3,1},A_i)$ and $(\mathfrak{r}_{3,1},A_j)$ are isomorphic if and only if $A_i = A_j$.
\end{proposition}

\begin{remark}
An alternative proof can be given by noting that any matrix $A$ in $\mathcal{N}_{2}$
is $\widehat{G}$-similar to a matrix $\widetilde{A}$ in $\mathcal{N}_{2}$
such that $\widetilde{A} | \mathfrak{k}$ is a \textit{nilpotent Jordan block},
where $\widehat{G}$ is the subgroup of $\operatorname{Aut}(\mathfrak{r}_{3,1})$
consisting of linear maps of $\operatorname{Aut}(\mathfrak{r}_{3,1})$ that fix to the vector $e_1$.
\end{remark}

\subsubsection{$G$-conjugacy classes in $\mathcal{N}_{1} \setminus \mathcal{N}_{2}$ }\label{classesN21}

Let $A \in \mathcal{N}_{1} \setminus \mathcal{N}_{2}$. For simplicity, let us suppose that
$Ae_2 = e_1 + w$ with $w \in \mathfrak{k}$; since it is easy to see that any matrix in such
set is $G$-similar to a matrix like $A$.

If $A$ is a nilpotent matrix of degree $2$, we have $Ae_2$ spans $\operatorname{Im}(A)$.
It follows that $e_3 \in \operatorname{Ker}(A)$, since $Ae_3 \in \mathfrak{k} \cap \operatorname{Span}_{\mathbb{C}}\{Ae_2\}$.
Let $g$ be the linear map defined by $g e_1 = Ae_2$, $ge_2=e_2$ and $g e_3 = e_3$. So, $g \in G$ and $\widetilde{A} = g^{-1} A g$ is the matrix given by
\begin{itemize}
  \item $\widetilde{A} e_1 = 0$, $\widetilde{A} e_2 = g^{-1} Ae_2 = e_1$, $\widetilde{A} e_3 = 0$.
\end{itemize}

If the nilpotency degree of $A$ is $3$ and $e_3 \in \operatorname{Ker}(A)$, we have $e_2$ is
a cyclic vector for $A$, since $A e_2 \notin \operatorname{span}_{\mathbb{C}}\{e_3\}=\operatorname{Ker}(A)$.
As $A^2 e_2 \in \operatorname{Ker}(A)$ we have $A^2 e_2 = t e_3$ with $t\neq 0$. We consider the linear map $g$
defined by $ge_1 = A e_2$, $g e_2 = e_2$ and $g e_3 = t e_3 = A^2 e_2$. It is clear that $g \in G$ and
we have $\widetilde{A} = g^{-1} A g$ has the form:
\begin{itemize}
  \item $\widetilde{A} e_1 = g^{-1} A^2 e_2 =e_3$, $\widetilde{A} e_2 = g^{-1} Ae_2 = e_1$, $\widetilde{A} e_3 = 0$.
\end{itemize}

The case in which $A$ is nilpotent matrix of degree $3$ and $e_3 \notin \operatorname{Ker}(A)$,
we have $\{Ae_2 , Ae_3\}$ spans $\operatorname{Im} (A)$ and $e_3 \notin \operatorname{Ker}(A^2)$. Since $\operatorname{Im} (A)$ is
an invariant subspace of $A$, we $A |_{\operatorname{Im} (A)} $ is a nilpotent operator
of degree $2$, and so there are two cases: whether $A e_2$ is a vector in $\operatorname{Ker}A |_{\operatorname{Im} (A)}$ or not.
If $Ae_2 \in \operatorname{Ker}A |_{\operatorname{Im} (A)}$, equivalently $e_2 \in \operatorname{Ker}(A^2)$, since
$\operatorname{Im}(A) = \operatorname{Ker}(A^2)$, we have $Ae_3 = t e_2$ with $t \neq 0$ and so the cyclic basis of $\mathbb{C}^{3}$
is given by $e_3$ is $\{e_3, Ae_3 = t e_2, A^2e_3= t Ae_2\}$. Let $g$ be the linear map defined by
$ge_1 = Ae_2$, $g e_2 = e_2$ and $g e_3 = \tfrac{1}{t} e_3$. Then $g \in G$ and $\widehat{A} = g^{-1} A  g$ is define by:

\begin{itemize}
  \item $\widetilde{A} e_1 = g^{-1} A^2 e_2 =0$, $\widetilde{A} e_2 = g^{-1} Ae_2 = e_1$, $\widetilde{A} e_3 = g^{-1}e_2 =e_2$.
\end{itemize}

The remaining case, when $e_2, e_3 \notin \operatorname{Ker}(A^2)$, we require a bit more work to find a ``canonical form'' for $A$ in $\mathcal{N}_{1}\setminus \mathcal{N}_{2}$
with respect to the action of the group $G$. Let $v_2$ be a vector such that $v_2$ spans $\operatorname{Ker}(A)$. Then
\begin{eqnarray*}
v_2 &=& a Ae_2 - bAe_3
\end{eqnarray*}
with $a,b \neq 0$. We may assume that $a=1$, and so $v_2 = e_1 + w^{'}$ with $w^{'} \in \mathfrak{k}$.
Since $A^{2} e_3 \in \operatorname{Ker}(A)$, we have
\begin{eqnarray*}
A^{2} e_3 &=& \frac{\lambda}{b} v_2,\, {\lambda}\neq 0.
\end{eqnarray*}
If we write $A e_3 = x e_2 + y e_3$, we obtain
\begin{eqnarray*}
v_2 &=& \frac{b x}{\lambda}Ae_2 + \frac{by}{\lambda} Ae_3
\end{eqnarray*}
 by applying $A$ to $A e_3$, and so $bx = \lambda$, $y = -\lambda$.
 Let $g$ be the linear map defined by $g e_1 = v_2$, $g e_2 = e_2$ and $g e_3 = be_3$.
 We have $g \in G$ and  $\widetilde{A}=g^{-1} A g$ is of the form:
 \begin{itemize}
   \item $\widetilde{A} e_1 = 0$, $\widetilde{A} e_2 = e_1 + \lambda(e_2-e_3)$, $\widetilde{A} e_3 = \lambda (e_2-e_3)$
 \end{itemize}
since $\widetilde{A}e_2 - \widetilde{A}e_3 = e_1$.

It follows from subsections \ref{classesN2} and the above analysis that:

\begin{proposition}\label{L6}
   Any hom-Lie algebra $(\mathfrak{r}_{3,0}= \mathfrak{r}_2 \times \mathbb{C} , A)$ with $A$ nilpotent operator is isomorphic
   to some hom-Lie algebra $\mathfrak{L}_{6}^{i}:= (\mathfrak{r}_{2}\times \mathbb{C} , A_i)$ where $A_0$ is the zero map and
   the matrix of $A_i$ with respect to the ordered basis $\{e_1,e_2,e_3\}$ is
   \begin{center}
  \begin{tabular}{c : c : c}
    $A_1$ & $A_2$ & $A_3$ \\
    $\left(
      \begin{array}{c | cc}
        0 & 0 & 0 \\
        \hline
        1 & 0 & 0 \\
        0 & 0 & 0 \\
      \end{array}
    \right)$
    &
    $\left(
      \begin{array}{c | cc}
        0 & 0 & 0 \\
        \hline
        0 & 0 & 0 \\
        1 & 0 & 0 \\
      \end{array}
    \right)$
    &
    $\left(
      \begin{array}{c | cc}
        0 & 0 & 0 \\
        \hline
        1 & 0 & 0 \\
        1 & 0 & 0 \\
      \end{array}
    \right)$\\
 \hdashline
    $A_4$ & $A_5$ & $A_6(\lambda)$ \\
    $\left(
      \begin{array}{c | cc}
        0 & 0 & 0 \\
        \hline
        0 & 0 & 1 \\
        0 & 0 & 0 \\
      \end{array}
    \right)$
    &
    $\left(
      \begin{array}{c | cc}
        0 & 0 & 0 \\
        \hline
        0 & 0 & 0 \\
        0 & 1 & 0 \\
      \end{array}
    \right)$
    &
    $\left(
      \begin{array}{c | cc}
        0 & 0 & 0 \\
        \hline
        0 & \lambda &  \lambda \\
        0 & - \lambda & - \lambda \\
      \end{array}
    \right)$\\
  \hdashline
    $A_7$ & $A_8$ & $A_9(\lambda)$ \\
    $\left(
      \begin{array}{c | cc}
        0 & 0 & 0 \\
        \hline
        1 & 0 & 0 \\
        0 & 1 & 0 \\
      \end{array}
    \right)$
    &
    $\left(
      \begin{array}{c | cc}
        0 & 0 & 0 \\
        \hline
        0 & 0 & 1 \\
        1 & 0 & 0 \\
      \end{array}
    \right)$
    &
    $\left(
      \begin{array}{c | cc}
        0 & 0 & 0 \\
        \hline
        1 &  \lambda &  \lambda \\
        0 & - \lambda & - \lambda \\
      \end{array}
    \right)$\\
     \hdashline
     & $A_{10}$ &  \\
      &
    $\left(
      \begin{array}{ccc}
        0 & 1 & 0 \\
        0 & 0 & 0 \\
        0 & 0 & 0 \\
      \end{array}
    \right)$
    &\\
    \hdashline
    $A_{11}$ & $A_{12}$ & $A_{13}(\lambda)$ \\
    $\left(
      \begin{array}{ccc}
        0 & 1 & 0 \\
        0 & 0 & 0 \\
        1 & 0 & 0 \\
      \end{array}
    \right)$
    &
    $\left(
      \begin{array}{ccc}
        0 & 1 & 0 \\
        0 & 0 & 1 \\
        0 & 0 & 0 \\
      \end{array}
    \right)$
    &
    $\left(
      \begin{array}{c cc}
        0 & 1 & 0 \\
        0 &  \lambda &  \lambda \\
        0 & - \lambda & - \lambda \\
      \end{array}
    \right)$
  \end{tabular}
   \end{center}
Two hom-Lie algebras $(\mathfrak{r}_{2} \times \mathbb{C},A_i)$ and $(\mathfrak{r}_{2}\times \mathbb{C},A_j)$ are isomorphic if and only if $A_i = A_j$.
\end{proposition}

\subsection{hom-Lie structures on $\mathfrak{r}_{3}(\mathbb{C})$ }

It is easy to verify that $ \operatorname{hom-Lie}(\mathfrak{r}_{3}(\mathbb{C}))$ represented by matrices with
respect to the ordered basis $\{e_1,e_2,e_3\}$ is the set $\mathcal{N}_{2}$; as in \ref{classesN2}. So, we have to study conjugacy classes
on $\mathcal{N}_{2}$ with respect to the natural action of
\begin{eqnarray*}
\operatorname{Aut}(\mathfrak{r}_{3}(\mathbb{C})) &=& \left\{
\left[
\begin {array}{c | c c}
1&0&0\\
\hline
x&a&z\\
y&0&a\end {array}
\right]
\in M(3,\mathbb{C})
:
a \in \mathbb{C}^{*} \right\}.
\end{eqnarray*}

Recall all elements of $\mathcal{N}_{2}$ leave $\mathfrak{k} =  \operatorname{Span}_{\mathbb{C}}\{e_2,e_3\}$ invariant.
An analysis similar to the one in \ref{classesN2} shows that any $A \in \mathcal{N}_{2}$ is
$\operatorname{Aut}(\mathfrak{r}_{3}(\mathbb{C}))$-similar to a matrix $\widetilde{A} \in \mathcal{N}_{2}$
such that $\widetilde{A}|_{\mathfrak{k}}$ is the zero map or

\begin{itemize}
  \item[$\star$] $\widetilde{A} e_2 =0$, $\widetilde{A} e_3 = \lambda e_2$ with $\lambda \in \mathbb{C}^{\star}$
  \item[$\star$] $\widetilde{A} e_2 =\lambda e_3$, $\widetilde{A} e_3 = 0$ with $\lambda \in \mathbb{C}^{\star}$,
\end{itemize}

depending on whether $e_2$ spans $\operatorname{Ker} A |_\mathfrak{k}$ or not.
For instance, if $e_2 \notin \operatorname{Ker} A |_\mathfrak{k}$, then
there exists a vector $v_1 \in \mathfrak{k}$ of the form $z e_2 + a e_3$, $a\neq 0$,
that spans $\operatorname{Ker} A |_\mathfrak{k}$. Since $Ae_2 \in \operatorname{Ker} A |_\mathfrak{k}$,
we have $Ae_2 = \tfrac{\lambda}{a} v_1$ with $\lambda \neq 0$. Let $g$ be the linear map defined by
$g e_1 = e_1$, $g e_2 = a e_2$ and $g e_3 = v_1 = z e_2 + a e_3$. Clearly $g \in \operatorname{Aut}(\mathfrak{r}_{3}(\mathbb{C}))$
and $\widetilde{A}= g^{-1} A g$  satisfies $\widetilde{A} e_2 = g^{-1} (\lambda v_1) = \lambda e_3$ and $\widetilde{A} e_3 = 0$.

And so, if we repeat the same reasoning as in \ref{classesN2}, by considering the nilpotency degree of $A$, it is easy
to verify that $A$ is $\operatorname{Aut}(\mathfrak{r}_{3}(\mathbb{C}))$-similar to a matrix $\widehat{A}$ such that
$\widehat{A}e_1 \in \{0 , e_2,e_3\}$ and $\widehat{A}|_{\mathfrak{k}} = \widetilde{A}|_{\mathfrak{k}}$.
For example, if $A$ is a nilpotent matrix of degree $3$ and $\widetilde{A}$ is as given in the first type: $\widetilde{A}e_2=0$
and $\widetilde{A}e_3 = \lambda e_2$. Since $\operatorname{Im}\widetilde{A} = \mathfrak{k}$,  there must
exist a vector $v_2 = e_1 + w$ with $w\in \mathfrak{k}$ such that $\widetilde{A} v_2 = t e_3$ with $t \neq 0$. Let $g$ be the
linear map defined by $g e_1 = v_2$, $ge_2 = t e_2$, $ge_3 = te_3$, then $\widehat{A} = g^{1}\widetilde{A}g$ satisfies $\widehat{A}e_1 = e_3$ ,
$\widehat{A} e_2 =0$ and $\widehat{A} e_3 = \lambda e_2$.

\begin{proposition}
   Any hom-Lie algebra $(\mathfrak{r}_{3} , A)$ with $A$ nilpotent operator is isomorphic
   to some hom-Lie algebra $\mathfrak{L}_{2}^{i}:= (\mathfrak{r}_{3} , A_i)$ where $A_0$ is the Zero map and
   the matrix of $A_i$ with respect to the ordered basis $\{e_1,e_2,e_3\}$ is:
   \begin{center}
  \begin{tabular}{c : c}
    $A_1$ & $A_2$ \\
    $\left(
      \begin{array}{c | cc}
        0 & 0 & 0 \\
        \hline
        1 & 0 & 0 \\
        0 & 0 & 0 \\
      \end{array}
    \right)$
    &
    $\left(
      \begin{array}{c | cc}
        0 & 0 & 0 \\
        \hline
        0 & 0 & 0 \\
        1 & 0 & 0 \\
      \end{array}
    \right)$\\
 \hdashline
    $A_3(\lambda)$ & $A_4(\lambda)$ \\
    $\left(
      \begin{array}{c | cc}
        0 & 0 & 0 \\
        \hline
        0 & 0 & \lambda \\
        0 & 0 & 0 \\
      \end{array}
    \right)$
    &
    $\left(
      \begin{array}{c | cc}
        0 & 0 & 0 \\
        \hline
        0 & 0 & 0 \\
        0 & \lambda & 0 \\
      \end{array}
    \right)$\\
  \hdashline
    $A_5(\lambda)$ & $A_6(\lambda)$  \\
    $\left(
      \begin{array}{c | cc}
        0 & 0 & 0 \\
        \hline
        0 & 0 & \lambda \\
        1 & 0 & 0 \\
      \end{array}
    \right)$
    &
    $\left(
      \begin{array}{c | cc}
        0 & 0 & 0 \\
        \hline
        1 & 0 & 0 \\
        0 & \lambda & 0 \\
      \end{array}
    \right)$
  \end{tabular}
   \end{center}
Two hom-Lie algebras $(\mathfrak{r}_{3},A_i)$ and $(\mathfrak{r}_{3},A_j)$ are isomorphic if and only if $A_i = A_j$.
\end{proposition}

\subsection{hom-Lie structures on $\mathfrak{n}_{3}(\mathbb{C})$}

Note that any liner map on a $2$-step nilpotent Lie algebra gives a hom-Lie algebra structure.
In this case, we need to study $\operatorname{Aut}(\mathfrak{h}_{3}(\mathbb{C}))$-conjugation classes
of nilpotent matrices in $M(3,\mathbb{C})$.
Since $\operatorname{span}_{\mathbb{C}}\{e_3\}$ is an invariant subspace for any automorphism in $\operatorname{Aut}(\mathfrak{h}_{3}(\mathbb{C}))$,
we will consider several cases for a hom-Lie algebra structure $A$ on $\mathfrak{h}_{3}(\mathbb{C})$, with $A$ nilpotent matrix, depending on how
$e_3$ is related with $\operatorname{Ker}(A)$ or $\operatorname{Im}(A)$ and the nilpotency degree of $A$.

${}$

First, we suppose that $A^2=0$ and $A\neq0$.

${}$

\textbf{Case 1:}  $e_3 \in \operatorname{Im}(A) \subset \operatorname{Ker}(A)$. And so, let $v_1 \in \mathbb{C}^{3}$ such that $Av_1 = e_3$
and let $v_2 \in \operatorname{Ker}(A)$ be a linear complement of $e_3$ in $\operatorname{Ker}(A)$. Define $g$ the linear map
given by $g e_1 = \tfrac{1}{\lambda} v_2$, $g e_2 = v_1$ and $g e_3 = e_3$, where $\lambda = e^{1}\wedge e^{2}(v_1,v_2)$.
We have $g\in \operatorname{Aut}(\mathfrak{h}_{3}(\mathbb{C}))$ and $\widetilde{A}= g^{-1}Ag$ satisfies

\begin{itemize}
  \item $\widetilde{A} e_1 = 0$, $\widetilde{A} e_2 = e_3$  and $\widetilde{A}e_3 = 0$
\end{itemize}

\textbf{Case 2:} $e_3 \in \operatorname{Ker}(A)$ and $e_3 \notin \operatorname{Im}(A)$. Let $v_2 \in \operatorname{Im}(A)$
and $v_1 \in \mathbb{C}^{3}$ such that $A v_1 = v_2$. We have $\{v_1,v_2,e_3\}$ is a basis of $\mathbb{C}^{3}$
and the linear transformation $g$ given by $g e_1 = v_2$, $g e_2 = v_1$ and $g e_3 = \lambda e_3$
with $\lambda = e^{1}\wedge e^{2}(v_1,v_2)$ is in $g\in \operatorname{Aut}(\mathfrak{h}_{3}(\mathbb{C}))$.
Let $\widetilde{A}= g^{-1}Ag$, we have

\begin{itemize}
  \item $\widetilde{A} e_1 = 0$, $\widetilde{A} e_2 = e_1$  and $\widetilde{A}e_3 = 0$
\end{itemize}

\textbf{Case 3:} $e_3 \notin \operatorname{Ker}(A)$. Let $v_2 = A e_3$. Since $v_2 \in \operatorname{Ker}(A)$, let $v_1$ be a linear complement of $v_2$ in $\operatorname{Ker}(A)$.
Consider the linear map defined by $g e_1 = \tfrac{1}{\lambda}v_1$, $g e_2 = v_2 $ and $g e_3 = e_3$ with $\lambda = e^{1}\wedge e^{2}(v_1,v_2)$.
Then $g \in \operatorname{Aut}(\mathfrak{h}_{3}(\mathbb{C}))$ and $\widetilde{A} = g^{1} A g$ is such that:
\begin{itemize}
  \item $\widetilde{A} e_1 = 0$, $\widetilde{A} e_2 = 0$  and $\widetilde{A}e_3 = e_2$
\end{itemize}

Now, if the degree of nilpotency of $A$ is $3$, we have $\operatorname{dim}\operatorname{Ker}(A)=1$ and $\operatorname{Ker}(A) \subset \operatorname{Im}(A)$.

${}$

\textbf{Case 4:} $e_3 \in \operatorname{Ker}(A) \subseteq \operatorname{Im}(A)$. Let $v_2$ be a linear complement
of $e_3$ in $\operatorname{Im}(A)$ and let $v_1$ be a vector such that $A v_1 = v_2$. Since $Av_2 = A^{2}v_1 \in \operatorname{Ker}(A)$,
$Av_2 = \lambda e_3$ with $\lambda \neq 0$. We have $\{v_1, v_2, \lambda e_3\}$ is a cyclic basis for $\mathbb{C}^{3}$ from which it follows
that the linear map $g$ defined by $g e_1 = \alpha v_1$, $g e_2 = \alpha v_2$ and $g e_3 = \alpha^2 e^{1}\wedge e^{2}(v_1,v_2)$
with $\alpha = \tfrac{\lambda}{e^{1}\wedge e^{2}(v_1,v_2)}$ is an automorphism of $\mathfrak{h}_{3}(\mathbb{C})$ and $\widetilde{A} = g^{-1} A g$
satisfies
\begin{itemize}
  \item $\widetilde{A} e_1 = e_2$, $\widetilde{A} e_2 = e_3$  and $\widetilde{A}e_3 = 0$
\end{itemize}

\textbf{Case 5:} $e_3 \in \operatorname{Im}(A)$ and $e_3 \notin \operatorname{Ker}(A)$. Let $v_1$ a vector in $\mathbb{C}^{3}$
such that $Av_1 =e_3$ and let $v_2 = Ae_3$. We have $\{\lambda v_1, \lambda e_3, \lambda v_2\}$ is a cyclic basis for $\mathbb{C}^{3}$
and the linear map $g$ given by $g e_1 = \lambda v_1$, $g e_2 = \lambda v_2$ and $g e_3 = \lambda e_3$ with $\lambda = \tfrac{1}{e^{1}\wedge e^{2}(v_1,v_2)}$
is in $\operatorname{Aut}(\mathfrak{h}_{3}(\mathbb{C}))$. The linear map $\widetilde{A} = g^{-1} A g$ is such that:
\begin{itemize}
  \item $\widetilde{A} e_1 = e_3$, $\widetilde{A} e_2 = 0$  and $\widetilde{A}e_3 = e_2$
\end{itemize}

\textbf{Case 6:} $e_3 \notin \operatorname{Im}(A)$. And so, $e_3 \notin \operatorname{Ker}(A^2)$ and
$\{e_3, v_1:=A e_3, v_2:=A^{2} e_3\}$ is a cyclic basis for $\mathbb{C}^{3}$ and the linear map
$g$ defined by $g e_1 = \lambda v_2$, $g e_2 = \lambda v_1$ and $g e_3 = \lambda e_3$ with $\lambda = \tfrac{1}{e^{1}\wedge e^{2}(v_1,v_2)}$
is such that
\begin{itemize}
  \item $\widetilde{A} e_1 = 0$, $\widetilde{A} e_2 = e_1$  and $\widetilde{A}e_3 = e_2$
\end{itemize}

The following proposition summarizes the computations above.

\begin{proposition}
   Any hom-Lie algebra $(\mathfrak{h}_{3}(\mathbb{C}) , A)$ with $A$ nilpotent operator is isomorphic
   to some hom-Lie algebra $\mathfrak{L}_{1}^{i}:= (\mathfrak{h}_{3}(\mathbb{C}) , A_i)$ where $A_0$ is the Zero map and
   the matrix of $A_i$ with respect to the ordered basis $\{e_1,e_2,e_3\}$ is:
   \begin{center}
  \begin{tabular}{c : c : c}
    $A_1$ & $A_2$ & $A_3$ \\
    $\left(
      \begin{array}{c  cc}
        0 & 0 & 0 \\
        0 & 0 & 0 \\
        0 & 1 & 0 \\
      \end{array}
    \right)$
    &
    $\left(
      \begin{array}{ccc}
        0 & 1 & 0 \\
        0 & 0 & 0 \\
        0 & 0 & 0 \\
      \end{array}
    \right)$
     &
    $\left(
      \begin{array}{ccc}
        0 & 0 & 0 \\
        0 & 0 & 1 \\
        0 & 0 & 0 \\
      \end{array}
    \right)$\\
 \hdashline
    $A_4$ & $A_5$ & $A_6$ \\
    $\left(
      \begin{array}{ccc}
        0 & 1 & 0 \\
        0 & 0 & 1 \\
        0 & 0 & 0 \\
      \end{array}
    \right)$
    &
    $\left(
      \begin{array}{ccc}
        0 & 0 & 0 \\
        1 & 0 & 0 \\
        0 & 1 & 0 \\
      \end{array}
    \right)$
    &
    $\left(
      \begin{array}{ccc}
        0 & 0 & 0 \\
        0 & 0 & 1 \\
        1 & 0 & 0 \\
      \end{array}
    \right)$\\
  \end{tabular}
   \end{center}
Two hom-Lie algebras $(\mathfrak{h}_{3}(\mathbb{C}),A_i)$ and $(\mathfrak{h}_{3}(\mathbb{C}),A_j)$ are isomorphic if and only if $A_i = A_j$.
\end{proposition}

\section{Classification of Orbit Closures}

The aim of this section is to study the partial order given by the degeneration relation on the family of hom-Lie algebras obtained above.
We want to provide useful invariants of hom-Lie algebras which are preserved under the process of degeneration
and can be easily computed/verified. In addition to the above requirements, we also want to take advantage of the well-known
classification of orbit closures of $3$-dimensional complex Lie algebra (\cite{B2}).

We begin by studying degenerations between hom-Lie algebras given in section \ref{homLie3} with isomorphic underlying Lie algebra. By a case-by-case analysis, we have seen
that a hom-Lie algebra $\mathfrak{L}_{j}^{i}=(\mathbb{C}^{3},\mu,A_i)$ degenerates to a hom-Lie algebra $\mathfrak{L}_{j}^{k}=(\mathbb{C}^{3},\mu,A_k)$ if and only if $A_k$ in the Euclidean closure of $\operatorname{Aut}(\mathfrak{L}_{j}) \bullet A_{i}$.

\subsection{Invariants}

As we observed earlier in the Corollary \ref{niceDer}, the dimension of the algebra of Derivations of a hom-Lie algebra
is an important invariant to study degenerations. By elementary computations, such invariant for each
hom-Lie algebra $\mathfrak{L}_{j}^{i}$ is given in Table \ref{tab:table1}.

\subsubsection{Lie algebra realizations by hom-Lie algebras}

Let $\Omega$ be a finite subset of $\mathbb{Z}^{3}_{\geq 0}$, and let $(\alpha_{\lambda})_{\lambda \in \Omega}$ be a
family of complex constants, indexed by $\Omega$. We consider the (continuous) map $\varphi_{(\alpha_{\lambda})_{\lambda \in \Omega}} : C^{2}\times C^{1} \rightarrow L^{2}$, given by
\begin{eqnarray}\label{algebrarealization}
  \varphi_{(\alpha_{\lambda})_{\lambda \in \Omega}}(\mu, A) & = & \sum_{\lambda=(i,j,k) \in \Omega} \alpha_{\lambda} A^{i}\mu(A^{j}-, A^{k}-).
\end{eqnarray}

Here, $L^{2}$ is the vector space of bilinear maps from $\mathbb{C}^{3}\times \mathbb{C}^{3}$ to $\mathbb{C}^{3}$.
Note that for certain constants, we can restrict the codomain of $\varphi_{(\alpha_{\lambda})_{\lambda \in \Omega}}$ to $C^{2}$.
Let $\psi_{\alpha,\beta}$, $\phi_{\beta}$, $\rho$ be functions from $C^{2}\times C^{1}$ to $C^{2}$ defined by

\begin{eqnarray}
\nonumber  \psi_{\alpha,\beta}(\mu,A) & := &  \mu(\cdot,\cdot) + \alpha A \mu(\cdot,\cdot) + \beta\mu(A\cdot,\cdot) + \beta\mu(\cdot,A\cdot)\\
\label{invf}  \phi_{\beta}(\mu,A) & := &   A \mu(\cdot,\cdot) + \beta\mu(A\cdot,\cdot) + \beta\mu(\cdot,A\cdot)\\
\nonumber  \rho(\mu,A) & : = & \mu(A\cdot,\cdot) + \mu(\cdot,A\cdot).
\end{eqnarray}

Note that if $ (\mu,A) \xrightarrow{\text{\,\,deg\,\,}} (\lambda,B)$ then
$\varphi_{(\alpha_{\lambda})_{\lambda \in \Omega}} (\mu,A) \xrightarrow{\text{\,\,deg\,\,}}   \varphi_{(\alpha_{\lambda})_{\lambda \in \Omega}}(\lambda,B)$,
analogously for $\phi_{\beta}$ and $\rho$. If $(\mathbb{C}^{3},\mu)$ is an almost Abelian Lie algebra and $A$ leaves an ideal codimension $1$ of $(\mathbb{C}^{3},\mu)$ invariant,
then $\psi_{\alpha,\beta}$, $\phi_{\beta}$ and $\rho$ sends to $(\mu,A)$ to an almost Abelian Lie algebra structure. The purpose of this
part is to illustrate how evaluate the functions in (\ref{invf}) at each hom-Lie algebra given in section \ref{homLie3}.
The next example show how we obtained the results of the Tables \ref{tab:table2a}, \ref{tab:table2b}, \ref{tab:table3} and \ref{tab:table4}.

\begin{example}
We consider the hom-Lie algebras $\mathfrak{L}_{5}^{6}(z,\lambda)$. As we mentioned above, we have $\varphi_{(\alpha_{\lambda})_{\lambda \in \Omega}}(\mathfrak{L}_{5}^{9}(z,\lambda))$ is an Almost abelian Lie algebra. We want to determine the isomorphism class of
$\psi_{\alpha,\beta}(\mathfrak{L}_{5}^{9}(z,\lambda))$:
$$
\psi_{\alpha,\beta}(\mathfrak{L}_{5}^{9}(z,\lambda)) =
\left\{
\begin{array}{l}
{[{e_1},{e_2}]} = \left( 1+\alpha\,\lambda+ \beta\,\lambda \right) {e_2} -\lambda \left( \alpha+\beta z \right) {e_3},\\
{[{e_1},{e_3}]} =\lambda \left( \beta +\alpha z \right) {e_2}+ z\left( 1-\alpha\lambda-\beta\, \lambda \right) {e_3}.
\end{array}
\right.
$$
From Proposition \ref{almost3Lie}, we just have to study $\operatorname{Det}(B)$, $\operatorname{Tr}(B)$ and the discriminant
of $B$ where $B$ is
\begin{eqnarray*}
  B & = & \left[
            \begin{array}{cc}
               1+\alpha\,\lambda+ \beta\,\lambda & \lambda(\beta +\alpha z)  \\
               -\lambda( \alpha + \beta z)  &  z(1-\alpha \lambda - \beta\, \lambda)  \\
            \end{array}
          \right]
\end{eqnarray*}
Note that $\psi_{\alpha,\beta}(\mathfrak{L}_{5}^{9}(z,\lambda))$ is never the abelian Lie algebra, and so,
$\psi_{\alpha,\beta}(\mathfrak{L}_{5}^{9}(z,\lambda))$ is isomorphic to $\mathfrak{n_{3}(\mathbb{C})}$
if and only if $\operatorname{Det}(B) = 0 $ and $\operatorname{Tr}(B) = 0$; i.e.
\begin{eqnarray*}
\alpha \beta &=& -{\frac {z}{ \left( z-1 \right) ^{2}{\lambda}^{2}}}\\
\alpha+\beta  &=& {\frac {1+z}{ \left( z-1 \right) \lambda}}.
\end{eqnarray*}
or equivalently,
\begin{eqnarray*}
\alpha &=& {\frac {z+1 \pm \,\sqrt {{z}^{2}+6\,z+1}}{ 2 \left( z-1 \right) \lambda}}\\
\beta  &=& {\frac {z+1 \mp \,\sqrt {{z}^{2}+6\,z+1}}{ 2\left( z-1 \right) \lambda}}.
\end{eqnarray*}
 Similarly $B$ is never equal to a multiple of the identity matrix, therefore $\psi_{\alpha,\beta}(\mathfrak{L}_{5}^{9}(z,\lambda))$
 is isomorphic to $r_{3}(\mathbb{C})$ if and only if $\operatorname{Trace}(B)\neq 0$ and $\operatorname{Trace}(B)^2-4\operatorname{Det}(B)=0$, equivalently
$$
\left\{
\begin{array}{rcl}
( \left( 1-z \right)  \left( \alpha+\beta \right) \lambda+1+z)s-1 & = & 0 \\
\left( z-1 \right)  \left( \alpha-\beta \right) ^{2}{\lambda}^{2}-2\,\left( 1+z \right)  \left( \alpha+\beta \right) \lambda+(z-1) & = & 0
\end{array}
\right.
$$
for some $s \in \mathbb{C}^{\star}$; which is the same as
\begin{eqnarray*}
\alpha &= &  {\frac {s+sz-1 \pm \sqrt {s \left( s-2+6\,sz-2\,z+{ z}^{2}s \right) }}{2s\lambda\, \left( z-1 \right) }},\\
\beta &=&    {\frac {s+sz-1 \mp \sqrt {s \left( s-2+6\,sz-2\,z+{z}^{2}s \right) }}{2s \lambda\, \left( z-1 \right) }}.
\end{eqnarray*}

In other case, $\psi_{\alpha,\beta}(\mathfrak{L}_{5}^{9}(z,\lambda))$ is isomorphic to $r_{2}(\mathbb{C}) \times \mathbb{C}$ if and only if
$\operatorname{Trace}(B) s - 1 =0$, for some $s\in \mathbb{C}^{\star}$ and $\operatorname{Det}(B)=0$; this is

\begin{eqnarray*}
\alpha &= &  {\frac {sz+s-1 \pm \sqrt {{s}^{2}{z}^{2}+6\,{s}^{2}z+{s}^{2}-2\,sz-2\,s+1}}{2\lambda\,s \left( z-1 \right) }}, \\
\beta &=&    {\frac {sz+s-1 \mp \sqrt {{s}^{2}{z}^{2}+6\,{s}^{2}z+{s}^{2}-2\,sz-2\,s+1}}{2\lambda\,s \left( z-1 \right) }},
\end{eqnarray*}

and similarly, $\psi_{\alpha,\beta}(\mathfrak{L}_{5}^{9}(z,\lambda))$ is isomorphic to $\mathfrak{r}_{3,-1}$ if and only if $\operatorname{Det}(B) s - 1=0$, for some $s\in \mathbb{C}^{\star}$ and $\operatorname{Trace}(B) = 0$:
\begin{eqnarray*}
\alpha &=& {\frac {sz+s \pm \sqrt {s \left( s{z}^{2}+6\,sz+s+1 \right) }}{ 2\left( z-1 \right) \lambda\,s}}, \\
\beta  &=& {\frac {sz+s \mp \sqrt {s \left( s{z}^{2}+6\,sz+s+1 \right) }}{ 2\left( z-1 \right)\lambda\,s}}.
\end{eqnarray*}

In the remaining case, we have $\psi_{\alpha,\beta}(\mathfrak{L}_{5}^{9}(z,\lambda))$ is isomorphic to $\mathfrak{r}_{3,z}$ for some $z \in \mathbb{C}$, with $z(z^2-1)\neq0$,
if and only if $\operatorname{Trace}(B)s_{1} - 1 = 0$, $\operatorname{Det}(B)s_{2} - 1 = 0$ and $(\operatorname{Trace}(B)^2-4\operatorname{Det}(B))s_{3}-1 = 0$,
for some $s_1$, $s_2$ and $s_3$ in $\mathbb{C}^{\star}$. By solving the three equations for $\alpha$ and $\beta$ we have

\begin{eqnarray*}
\alpha &=& {\frac {s_{{2}} ( z s_{{1}}+s_{{1}}  - 1 ) \pm \sqrt {f(s_1,s_2)}}{2s_{{1}}s_{{2}} \left( z-1 \right) \lambda }}, \\
\beta&=& {\frac {s_{{2}} ( z s_{{1}}+s_{{1}}  - 1 ) \mp \sqrt {f(s_1,s_2)}}{2s_{{1}}s_{{2}} \left( z-1 \right) \lambda }}, \\
f(s_1,s_2) & = & \left( 1+ \left( {z}^{2}+1+6\,z \right) {s_{{1}}}^{2}-2\, \left( z+1 \right) s_{{1}} \right) {s_{{2}}}^{2}-4\,s_{{2}}{s_{{1}}}^{2}
\end{eqnarray*}
with $4 s_1^2 - s_2 \neq 0$.
\end{example}

In the case of the hom-Lie algebras $\mathfrak{L}_{7}^{i}$, a straightforward verification shows that
$\psi_{\alpha,\beta}(\mathfrak{L}_{7}^{i})$ is a Lie algebra with nondegenerate Killing form, and so
we have such Lie algebra is isomorphic to $\mathfrak{so}(3,\mathbb{C})$. With respect to $\mathfrak{L}_{6}^{12}$
and $\mathfrak{L}_{6}^{13}(\lambda)$, the function $\psi_{\alpha,\beta}$ does not necessarily send them to a Lie algebra.

\subsubsection{Others derivations}

Given a finite subset of $\mathbb{Z}_{\geq0}^{3}$, say $\Omega$, a family of complex constants $\{\alpha_{I}, \beta_{I}, \gamma_{I}: I=(i_1,i_2,i_3) \in \Omega\}$ and a hom-Lie algebra $(\mathbb{C}^{3},\mu,A)$, we can obtain a new algebra as in the equation (\ref{algebrarealization}). It is evident that the (dimension of) \textit{extended derivations} of
this new algebra is an invariant of $(\mathbb{C}^{3},\mu,A)$; we mean by such derivations to the kernel of $T_{(\mu,A)} : (C^1)^{|\Omega|} \times (C^1)^{|\Omega|} \times (C^1)^{|\Omega|} \rightarrow L^{2}$
defined by $T_{(\mu,A)}( D_{I},\ldots,F_{I},\ldots,G_{I} )$ equal to
\begin{eqnarray*}
&\sum_{I \in \Omega} \alpha_{I} D_{I} A^{i} \mu(A^{j} - , A^{k} - ) + \beta_{I}  A^{i} \mu(F_{I} A^{j} - , A^{k} - ) + \gamma_{I} A^{i} \mu(A^{j} - , G_{I} A^{k} - ).&
\end{eqnarray*}
In addition, note that we can impose extra conditions of linear dependence of the matrices $D_I$, $F_I$ and $G_I$ and commuting relations or others relations with the matrices $A^{k}$.

It follows easily by the upper semi-continuity of the \textit{nullity function} on linear operators
that the dimension of such subspace of derivations of a hom-Lie algebra $(\mathbb{C}^{n},\mu,A)$ is less or equal that in the orbit closure of $(\mu,A)$.

For our purposes, it suffices to consider the following subspaces of extended derivations: let $(\mathbb{C}^{3},\mu,A)$ be
a hom-Lie algebra and let $t$ be a complex constant
\begin{equation}
\left\{
\begin{array}{l}
  D_1\mu(\cdot,\cdot)+\mu(D_2\cdot,\cdot)+\mu(\cdot,D_3\cdot) = 0\\
  D_1 + tD_3 =0\\
  D_1A-AD_1=0,   D_2A-AD_2=0,   D_3A-AD_3=0.
\end{array}
\right.
\end{equation}
We denote by $\wedn{der}_{t}^{1}(\mathbb{C}^{3},\mu,A)$ the dimension of such subspace. Similarly, $\wedn{der}^{2}(\mathbb{C}^{3},\mu,A)$ stands for the dimension of the vector space
\begin{equation}\label{invf2}
\left\{
\begin{array}{l}
  D\mu(\cdot,\cdot) = 0\\
  DA-AD=0.
\end{array}
\right.
\end{equation}

\begin{example}
We consider the hom-Lie algebra $\mathfrak{L}_{5}^{5}(z)=(\mathfrak{r}_{3,z},A_5)$. If $D_1(w's)$, $D_2(x's)$  and $D_3(z's)$
are commuting matrices with $A_5$ and $D_1 = -tD_3$, then $\lambda(\cdot,\cdot) := D_1[\cdot,\cdot] + [D_2\cdot,\cdot] + [\cdot,D_3\cdot]$
is
$$
\begin{array}{l}
{e_1}\cdot{e_1}=\left( -x_{{3,1}} + y_{{3,1}}  \right)z {e_3},\,{e_1}\cdot{e_3}= \left( x_{{1,1}} -ty_{{3,3}} + y_{{3,3}}  \right)z {e_3},\,\\
{e_1}\cdot{e_2}=-ty_{{1,2}}{e_1}+ \left( x_{{1,1}}-ty_{{3,3}}+y_{{3,3}} \right) {e_2}+ \left( -ty_{{3,2}}+y_{{3,2}}z \right) {e_3},\,\\
{e_2}\cdot{e_1}=ty_{{1,2}}{e_1}+ \left( -y_{{1,1}}+ty_{{3,3}}-x_{{3,3}} \right) {e_2}+ \left( ty_{{3,2}}-x_{{3,2}}z \right) {e_3},\,\\
{e_2}\cdot{e_2}= \left( x_{{1,2}}-y_{{1,2}} \right) {e_2},\,
{e_2}\cdot{e_3}=x_{{1,2}}z{e_3},\,
{e_3}\cdot{e_2}=-y_{{1,2}}z{e_3},\,\\
{e_3}\cdot{e_1}= \left( -y_{{1,1}} + ty_{{3,3}} - x_{{3,3}}  \right)z {e_3}.
\end{array}
$$
\end{example}
Since $z\neq0$, $x_{1,2}=0$, $y_{1,2}=0$, $x_{3,1}=y_{3,1}$, $x_{3,2}=\tfrac{t}{z}y_{3,2}$, $x_{1,1}=y_{3,3}(t-1)$, $x_{3,3}=ty_{3,3}-y_{1,1}$ and $(t-z)y_{3,2}=0$.
It follows that if $t=z$ then $\wedn{der}^{1}_{z}=4$, and $\wedn{der}^{1}_{t}=3$ in other case.

\begin{table}[!!h]
\begin{center}
\begin{tabular}{| m{1.5cm} |  m{1.25cm} |  m{0.25cm} |  m{1.5cm} |  m{1.5cm} |  m{0.25cm} |}
\hline
hom-Lie & \multicolumn{2}{ c| }{$\wedn{der}^{1}_{t}$} & hom-Lie & \multicolumn{2}{ c| }{$\wedn{der}^{1}_{t}$}\\
\hline
\multirow{2}{*}{$\mathfrak{L}_{5}^{5}(z)$} & $t =z$ & 4 & \multirow{2}{*}{$\mathfrak{L}_{5}^{4}(z)$} & $t =1/z$ & 4 \\
 & $t\neq z$ & 3 & & $t\neq 1/z$ & 3 \\
\hline
\multirow{3}{*}{$\mathfrak{L}_{5}^{2}(z)$} & $t =1$ & 4 & \multirow{3}{*}{$\mathfrak{L}_{5}^{1}(z)$} & $t = 1$        & 4 \\
                                           & $t =z$ & 4 &                                            & $t = 1/z$      & 4 \\
                                           & $t \neq 1,z$ & 3 &                                      & $t \neq 1,1/z$ & 3 \\
\hline
\multirow{1}{*}{$\mathfrak{L}_{5}^{ 3}(z)$} & any $t$   & 3 &  \multicolumn{3}{c}{}   \\

\cline{1-3}
\end{tabular}
\caption{ The $\wedn{der}^{1}_{t}$-function on $\mathfrak{L}_{5}^{i}(z)$, $i=1,...,5$}
\label{tab:tableder}
\end{center}
\end{table}

\subsection{Degenerations in the family $\mathfrak{L}_{j}^{i}$ with $j$ fixed }\label{famfixed}

Before proceeding further with the degenerations of all hom-Lie algebras obtained before, we focus our attention for
a while on the special case of degenerations of ones with isomorphic underlying Lie algebra. Given $(\mathbb{C}^{3},\mu,A)$
and $(\mathbb{C}^{3},\mu,B)$ two hom-Lie algebras with $(\mathbb{C}^{3},\mu)$ a Lie algebra, our strategy to show that
$(\mathbb{C}^{3},\mu,A)$ degenerates in $(\mathbb{C}^{n},\mu,B)$ is to verify that $B$ is in the Euclidean closure
of the $\operatorname{Aut}(\mathbb{C}^{3},\mu)$-orbit of $A$. We recall that this type of problems are well-known
in the literature and go back at least as far as the contributions by Wim Hesselink and Murray Gerstenhaber to the
study the adjoint action of an algebraic group on the nilpotent elements of its Lie algebra (see the nice review of the problem given by
Joyce O'Halloran in \cite{OHalloran}).

\begin{theorem}[{\cite[Proposition 1.6]{Gerstenhaber2}, \cite[Theorem 3.10]{Hesselink}, \cite[Section 2]{OHalloran}}]
Let $A$ and $B$ be $m \times m$ complex nilpotent matrices and let $\operatorname{GL}(m,\mathbb{C})\bullet A$
denote the orbit of $A$ given by its similarity class. $B$ is in the (Zariski) closure of $\operatorname{GL}(m,\mathbb{C})\bullet A$ if and only if $\operatorname{rank}A^{k} \geq \operatorname{rank} B^{k}$, $k=1,\ldots,m-1$.
\end{theorem}

In our case, we are working with $3\times 3$ complex nilpotent matrices, and so the above theorem can be showed by
an easy verification.

Also, it is important mention the result about the classification of the degenerations of $3$-dimensional complex Lie algebras,
which plays an important role in our approach to obtain the analogous result in the family of hom-Lie algebras studied here.

\begin{proposition}[{\cite[Proposition 3]{B2}\cite[Section 4]{BB}}]\label{teorema}
All degenerations of the $3$-dimensional complex Lie algebras are:
\begin{enumerate}
\item $\mathbb{C}^{3} \xrightarrow{\text{\,\,deg\,\,}} \mathbb{C}^{3}$,
\item $\mathfrak{n}_{3}(\mathbb{C}) \xrightarrow{\text{\,\,deg\,\,}} \mathfrak{n}_{3}(\mathbb{C}), \mathbb{C}^{3} $,
\item $\mathfrak{r}_{3}(\mathbb{C}) \xrightarrow{\text{\,\,deg\,\,}} \mathfrak{r}_{3}(\mathbb{C}), \mathfrak{r}_{3,1}(\mathbb{C}), \mathfrak{n}_{3}(\mathbb{C}), \mathbb{C}^{3}$,
\item $\mathfrak{r}_{3,1}(\mathbb{C}) \xrightarrow{\text{\,\,deg\,\,}} \mathfrak{r}_{3,1}(\mathbb{C}), \mathbb{C}^{3}$,
\item $\mathfrak{r}_{3,-1}(\mathbb{C}) \xrightarrow{\text{\,\,deg\,\,}} \mathfrak{r}_{3,-1}(\mathbb{C}),  \mathfrak{n}_{3}(\mathbb{C}), \mathbb{C}^{3}$,
\item If $z\neq\pm1$, $\mathfrak{r}_{3,z}(\mathbb{C}) \xrightarrow{\text{\,\,deg\,\,}} \mathfrak{r}_{3,z}(\mathbb{C}) , \mathfrak{n}_{3}(\mathbb{C}), \mathbb{C}^{3}$,
\item $\mathfrak{r}_{2}\times \mathbb{C}  \xrightarrow{\text{\,\,deg\,\,}} \mathfrak{r}_{2}\times \mathbb{C}, \mathfrak{n}_{3}(\mathbb{C}), \mathbb{C}^{3}$,
\item $\mathfrak{sl}_{2}(\mathbb{C}) \xrightarrow{\text{\,\,deg\,\,}} \mathfrak{sl}_{2}(\mathbb{C}), \mathfrak{r}_{3,-1}(\mathbb{C}),\mathfrak{n}_{3}(\mathbb{C}), \mathbb{C}^{3}$.
\end{enumerate}

 The Hasse diagram is given by:

\begin{center}
\begin{tikzpicture}
  \node (sl2) at (-3,2) {$\mathfrak{sl}_{2}(\mathbb{C})$};
  \node (r-1) at (-3,1) {$\mathfrak{r}_{3,-1}(\mathbb{C})$};
  \node (rz) at (-1,1) {$\mathfrak{r}_{3,z}(\mathbb{C})$};
  \node (r2) at (1,1) {$\mathfrak{r}_{2}(\mathbb{C}) \times \mathbb{C}$};
  \node (r3) at (3,1) {$\mathfrak{r}_{3}(\mathbb{C})$};
  \node (hei) at (0,0) {$\mathfrak{n}_{3}(\mathbb{C})$};
  \node (r1) at (3,0) {$\mathfrak{r}_{3,1}(\mathbb{C})$};
  \node (ab) at (0,-1) {$\mathbb{C}^{3}$};
\draw [->] (sl2) edge (r-1) (r-1) edge (hei) (hei) edge (ab);
\draw [->] (rz) edge (hei) (r2) edge (hei);
\draw [->] (r3) edge (hei) (r3) edge (r1);
\draw [->] (r1) edge (ab);
\end{tikzpicture}
\end{center}
\end{proposition}

To illustrate, we study the degenerations in the family of hom-Lie algebras $\mathfrak{L}_{6}^{i}$;
we can now proceed analogously to analyze the degenerations in the remaining families
(the invariant $\wedn{der}_{t}^{1}$ given in Table \ref{tab:tableder} can be used to study degenerations in the family $\mathfrak{L}_{5}^{i}(z)$).

We use the Table \ref{tab:table1} like a starting point for our analysis because Corollary \ref{niceDer}.
The table \ref{tab:tableexe} shows the information about the degenerations in the family $\mathfrak{L}_{6}^{i}$.
To be more precise, the checkmark  in the intersection of row $\mathfrak{L}_{6}^{i}$ with column $\mathfrak{L}_{6}^{j}$
denotes that there is a degeneration of the hom-Lie algebra $\mathfrak{L}_{6}^{i}$ in $\mathfrak{L}_{6}^{j}$. For example,
$\mathfrak{L}_{6}^{13}(\lambda)$ degenerates to $\mathfrak{L}_{6}^{13}(\kappa)$ if and only if they are isomorphic,
equivalently $\lambda=\kappa$; because of the invariant $\psi$. Other example is that $\mathfrak{L}_{6}^{9}(\kappa)$
is a degeneration of $\mathfrak{L}_{6}^{13}(\lambda)$ if and only if $\lambda=\kappa$. To show that
$ \mathfrak{L}_{6}^{13}(\lambda) \xrightarrow{\text{\,\,deg\,\,}} \mathfrak{L}_{6}^{9}(\lambda)$, we explicitly find
a family $g_t$ in $\operatorname{Aut}(\mathfrak{r}_{2}(\mathbb{C})\times \mathbb{C})$ such that $A_9 = \lim_{t \to \infty } g_t \bullet A_{13}$ (see example \ref{exdeg}). Conversely,
if $ \mathfrak{L}_{6}^{13}(\lambda) \xrightarrow{\text{\,\,deg\,\,}} \mathfrak{L}_{6}^{9}(\kappa)$ then
$\psi_{\alpha,\beta} ( \mathfrak{L}_{6}^{13}(\lambda) )  \xrightarrow{\text{\,\,deg\,\,}}  \psi_{\alpha,\beta} ( \mathfrak{L}_{6}^{9}(\kappa) )$,
and so, by letting $\alpha=0$ and $\beta=-1/\lambda$ we get $\mathfrak{n}_{3} \xrightarrow{\text{\,\,deg\,\,}} \psi_{\alpha,\beta} ( \mathfrak{L}_{6}^{9}(\kappa) ) $,
which follows from Proposition \ref{teorema} and Table \ref{tab:table2b} that $\psi_{0,-1/\lambda} ( \mathfrak{L}_{6}^{9}(\kappa) )$ is isomorphic to $\mathfrak{n}_{3}$, and consequently $\beta = -1/\kappa = -1/\lambda$.

The symbol $\operatorname{Der} + \mbox{invariant}$ in the intersection of row $\mathfrak{L}_{6}^{i}$ with column $\mathfrak{L}_{6}^{j}$ stands
that the dimensions of the algebras of derivations of $\mathfrak{L}_{6}^{i}$ and $\mathfrak{L}_{6}^{j}$ are equal but $\mathfrak{L}_{6}^{i}$ and $\mathfrak{L}_{6}^{j}$ are not isomorphic because of such invariant; hence there is no degeneration from $\mathfrak{L}_{6}^{i}$ to $\mathfrak{L}_{6}^{j}$.
Analogously, the symbol $\operatorname{Der}$ denotes that  $\mathfrak{L}_{6}^{i}$ cannot degenerate to $\mathfrak{L}_{6}^{j}$
since the dimension of algebra of derivations of $\mathfrak{L}_{6}^{i}$ is greater than to the dimension of algebra de derivations of $\mathfrak{L}_{6}^{j}$,
and the symbol $\psi$ represents that there is no degeneration from $\mathfrak{L}_{6}^{i}$ to $\mathfrak{L}_{6}^{j}$ because
it is easy give values of $\alpha$ and $\beta$ such that $\psi_{\alpha,\beta}(\mathfrak{L}_{6}^{i})$ does not degenerate to $\psi_{\alpha,\beta}(\mathfrak{L}_{6}^{j})$ (here it is important to be aware of Proposition \ref{teorema}; as we showed above).  Similarly $\phi$ and
$\rho$ are used. Where we have used a clover symbol, the invariants mentioned above do not give reasons why such a degeneration is impossible,
and so we give other reasons; which are explained below.

\begin{enumerate}
  \item [$\clubsuit$-1] $\mathfrak{L}_{6}^{5}$ and $\mathfrak{L}_{6}^{3}$ are not isomorphic: $\mathfrak{L}_{6}^{3}$ is a multiplicative hom-Lie algebra
  and $\mathfrak{L}_{6}^{5}$ is not.
  \item [$\clubsuit$-2] $\mathfrak{L}_{6}^{5}$ satisfies the identity $[A_{5}\cdot,\cdot] \equiv 0$ but $\mathfrak{L}_{6}^{1}$ does not.
  \item [$\clubsuit$-3] An easy computation  shows that $\wedn{der}^{2} (\mathfrak{L}_{6}^{4}) = 4$ and $\wedn{der}^{2} (\mathfrak{L}_{6}^{2}) = 3$;
  therefore $\mathfrak{L}_{6}^{4}$ does not degenerate to $\mathfrak{L}_{6}^{2}$ (see the definition of the invariant $\wedn{der}^{2}$ in equation (\ref{invf2})).
  \item [$\clubsuit$-4] $\mathfrak{L}_{6}^{2}$ and $\mathfrak{L}_{6}^{1}$ are not isomorphic: $\mathfrak{L}_{6}^{2}$ satisfies the  identity $[A_{2}\cdot,\cdot] \equiv 0$ but $\mathfrak{L}_{6}^{1}$ does not.
\end{enumerate}

\begin{example}\label{exdeg}
To show that $(\mathfrak{L}_{6}, A_{13}(\lambda))$  degenerates to $(\mathfrak{L}_{6}, A_{9}(\lambda))$, given that the underlying algebra
is the same, we will find a family $g_t$ in $\operatorname{GL}(3,\mathbb{C})$ such that $g_t \bullet \mathfrak{L}_{6} = \mathfrak{L}_{6}$;
i.e. $g_t \in \operatorname{Aut}(\mathfrak{L}_{6})$, and $g_t A_{13}(\lambda) g_t^{-1}$ tends to $A_9(\lambda)$ as $t$ tends to infinity.
Here, we take advantage of the Proposition \ref{autoLie3} in the following way. Let $g$ be an arbitrary automorphism of $\mathfrak{L}_{6}$, then we know
that $g$ has the following expression:
$$
g = \left(\begin{array}{ccc} 1 & 0 & 0 \\  x & a & 0 \\ y & 0 & b \\ \end{array}\right),\, a,b \in \mathbb{C}^{*}.
$$
Let $E:=gA_{13}(\lambda)g^{-1} - A_{9}(\lambda)$
$$
E = \left[ \begin {array}{ccc} -{\frac {x}{a}}&\frac{1}{a}&0
\\ \noalign{\medskip}-{\frac {{x}^{2}b+xba\lambda+\lambda\,y{a}^{2}+ab
}{ab}}&{\frac {x}{a}}&{\frac {\lambda\, \left( a-b \right) }{b}}
\\ \noalign{\medskip}{\frac {-xy+xb\lambda+\lambda\,ya}{a}}&{\frac {y-
b\lambda+a\lambda}{a}}&0\end {array} \right]
$$

The basic idea is to try that $E$  is as close as possible to being the zero map.
To this end, we can start letting some entries of $E$ to be zero; for instance,
by solving the equations ${E[2,1],E[3,1],E[3,2]}$ for $a$, $b$ and $y$. Here, we can introduce
a new variable $z$ such that $z^2 = 1+4x\lambda$ and so
\begin{eqnarray*}
  a & = & \frac{(z+1)^2(z-1)}{8\lambda^2}\\
  b & = & a \\
  y & = & \frac{(z+1)(z-1)}{4\lambda}
\end{eqnarray*}
and
$$
E = \left[ \begin {array}{ccc}
{\frac {-2\lambda}{z+1}}&\frac {8{\lambda}^{2}}{(z-1)(z+1)^2}&0\\
\noalign{\medskip}0&{\frac {2\lambda}{z+1}}&{\frac {2\lambda}{z-1}}\\
\noalign{\medskip}0&0&0\end {array} \right].
$$
If we let $z = 1+\exp(t)$, then we obtain a parametric family of automorphisms of $\mathfrak{L}_{6}$, $g_t$, and it
satisfies $g_t \bullet A_{13}(\lambda) \xrightarrow[t \to \infty]{} {A_{9}(\lambda)}$

\end{example}

The respective Hasse diagram of the family $\mathfrak{L}_{6}^{j}$, $j=0,...,13$ with respect to
the partial order defined by the degeneration relation is:

\begin{center}
\begin{tikzpicture}

\node[label=above:{$\mathfrak{L}_{6}^{13}(\lambda)$}](L613) at (-4,2) {};
\node(L613a) at (-5,2) {};
\node(L613b) at (-3,2) {};
\draw[thick] (L613a)edge(L613b);

\node(L69) at (-4,1) {};
\node[label=left:{$\mathfrak{L}_{6}^{9}(\lambda)$}](L69a) at (-5,1) {};
\node(L69b) at (-3,1) {};
\draw[thick] (L69a)edge(L69b);
\node[label=below:{$\mathfrak{L}_{6}^{6}(\lambda)$}](L66) at (-4,0) {};
\node(L66a) at (-5,0) {};
\node(L66b) at (-3,0) {};
\draw[thick] (L66a)edge(L66b);
\draw [->] (L69a)edge(L66a) (L69)edge(L66) (L69b) edge (L66b) (L613a)edge(L69a) (L613)edge(L69) (L613b) edge (L69b);


\node[circle,fill,inner sep=0pt,minimum size=3pt,label=above:{$\mathfrak{L}_6^{11}$}]    (L611) at (0,2) {};
\node[circle,fill,inner sep=0pt,minimum size=3pt,label=above:{$\mathfrak{L}_6^{12}$}]    (L612) at (2,2) {};
\node[circle,fill,inner sep=0pt,minimum size=3pt,label=above:{$\mathfrak{L}_6^{10}$}]    (L610) at (1,1) {};
\node[circle,fill,inner sep=0pt,minimum size=3pt,label=above:{$\mathfrak{L}_6^7$}]       (L67)  at (-1,1) {};
\node[circle,fill,inner sep=0pt,minimum size=3pt,label=above:{$\mathfrak{L}_6^8$}]       (L68)  at (3,1) {};
\node[circle,fill,inner sep=0pt,minimum size=3pt,label=left: {$\mathfrak{L}_6^5$}]       (L65)  at (-1,0) {};
\node[circle,fill,inner sep=0pt,minimum size=3pt,label=below:{$\mathfrak{L}_6^3$}]       (L63)  at (1,0) {};
\node[circle,fill,inner sep=0pt,minimum size=3pt,label=right:{$\mathfrak{L}_6^4$}]       (L64)  at (3,0) {};
\node[circle,fill,inner sep=0pt,minimum size=3pt,label=below:{$\mathfrak{L}_6^2$}]       (L62)  at (0,-1) {};
\node[circle,fill,inner sep=0pt,minimum size=3pt,label=below:{$\mathfrak{L}_6^1$}]       (L61)  at (2,-1) {};
\node[circle,fill,inner sep=0pt,minimum size=3pt,label=below:{$\mathfrak{L}_6^0$}]       (L60)  at (1,-2) {};

\draw [->] (L612) edge (L68) (L612) edge (L610) (L612) edge (L67);
\draw [->] (L611) edge (L67) (L611) edge (L610);

\draw [->] (L610) edge (L65) (L610) edge (L63);

\draw [->] (L68) edge (L64) (L68) edge (L63);
\draw [->] (L67) edge (L63) (L67) edge (L65);

\draw [->] (L65) edge (L62);
\draw [->] (L64) edge (L61);
\draw [->] (L63) edge (L62) (L63) edge (L61);

\draw [->] (L62) edge (L60);
\draw [->] (L61) edge (L60);

\end{tikzpicture}
\end{center}

\begin{landscape}
\begin{table}[p!]
\begin{center}
\begin{tabular}{|c||c|c|c||c|c|c|c||c|c|c|c||c|c||c|}
  \hline
$\xrightarrow{\text{\,\,deg\,\,}}$
& $\mathfrak{L}_{6}^{13}(\kappa)$ & $\mathfrak{L}_{6}^{12}$ & $\mathfrak{L}_{6}^{11}$ & $\mathfrak{L}_{6}^{10}$ & $\mathfrak{L}_{6}^{9}(\kappa)$
& $\mathfrak{L}_{6}^{8}$ & $\mathfrak{L}_{6}^{7}$ & $\mathfrak{L}_{6}^{6}(\kappa)$ & $\mathfrak{L}_{6}^{5}$ & $\mathfrak{L}_{6}^{4}$ & $\mathfrak{L}_{6}^{3}$
& $\mathfrak{L}_{6}^{2}$ & $\mathfrak{L}_{6}^{1}$ & $\mathfrak{L}_{6}^{0}$ \\
\hline
\hline
$\mathfrak{L}_{6}^{13}(\lambda)$ & $\lambda\overset{\checkmark}{=}\kappa$, $\psi$ & $\operatorname{Der} + \rho$ & $\operatorname{Der} + \rho$ & $\psi$ & $\lambda\overset{\checkmark}{=}\kappa$, $\psi$ & $\psi$ & $\psi$ & $\lambda=\kappa$, $\psi$ & $\psi$ & $\psi$ & $\psi$ & $\psi$ &$\psi$ & $\psi$ \\
$\mathfrak{L}_{6}^{12}$ & $\operatorname{Der} + \rho$ & $\checkmark$ & $\operatorname{Der} + \rho$ & $\checkmark$ & $\rho$ & $\checkmark$ & $\checkmark$ & $\rho$ & $\checkmark$ & $\checkmark$ & $\checkmark$ & $\checkmark$ & $\checkmark$ & $\checkmark$ \\
$\mathfrak{L}_{6}^{11}$ & $\operatorname{Der} + \rho$ & $\operatorname{Der} +\rho$ & $\checkmark$ & $\checkmark$ & $\rho$ & $\rho$ & $\checkmark$ & $\rho$ & $\checkmark$ & $\rho$ & $\checkmark$ & $\checkmark$ & $\checkmark$ & $\checkmark$ \\
\hline
$\mathfrak{L}_{6}^{10}$ & $\operatorname{Der}$ & $\operatorname{Der}$ & $\operatorname{Der}$ & $\checkmark$ & $\operatorname{Der} + \rho$ & $\operatorname{Der} +\rho$ & $\operatorname{Der} + \phi$ & $\rho$ & $\checkmark$ & $\rho$ &  $\checkmark$ & $\checkmark$ & $\checkmark$ & \checkmark \\
$\mathfrak{L}_{6}^{ 9}(\lambda)$ & $\operatorname{Der}$ & $\operatorname{Der}$ & $\operatorname{Der}$ & $\operatorname{Der} + \rho$ & $\lambda\overset{\checkmark}{=}\kappa$, $\psi$ & $\operatorname{Der} + \rho$ & $\operatorname{Der} + \rho$ & $\lambda\overset{\checkmark}{=}\kappa$, $\psi$ & $\psi$ & $\psi$ & $\psi$ & $\psi$ & $\psi$ & $\psi$ \\
$\mathfrak{L}_{6}^{ 8}$ & $\operatorname{Der}$ & $\operatorname{Der}$ & $\operatorname{Der}$ & $\operatorname{Der} + \rho$ & $\operatorname{Der} + \rho$ & $\checkmark$ & $\operatorname{Der} + \rho$ & $\rho$ & $\rho$ & $\checkmark$ &  $\checkmark$ & $\checkmark$ & $\checkmark$ & $\checkmark$ \\
$\mathfrak{L}_{6}^{ 7}$ & $\operatorname{Der}$ & $\operatorname{Der}$ & $\operatorname{Der}$ & $\operatorname{Der} + \phi$ & $\operatorname{Der} + \rho$ & $\operatorname{Der} + \rho$ & $\checkmark$ & $\rho$ & $\checkmark$ & $\rho$ & $\checkmark$ & $\checkmark$ & $\checkmark$ & $\checkmark$ \\
\hline
$\mathfrak{L}_{6}^{ 6}(\lambda)$ & $\operatorname{Der}$ & $\operatorname{Der}$ & $\operatorname{Der}$ & $\operatorname{Der}$ & $\operatorname{Der}$ & $\operatorname{Der}$ & $\operatorname{Der}$ & $\lambda\overset{\checkmark}{=}\kappa$, $\psi$ & $\operatorname{Der} + \rho$ & $\operatorname{Der} + \rho$ &  $\operatorname{Der} + \rho$ & $\psi$ & $\psi$ & $\psi$ \\
$\mathfrak{L}_{6}^{ 5}$ & $\operatorname{Der}$ & $\operatorname{Der}$ & $\operatorname{Der}$ & $\operatorname{Der}$ &$\operatorname{Der}$ & $\operatorname{Der}$ &$\operatorname{Der}$& $\operatorname{Der} + \rho$& $\checkmark$ &$\operatorname{Der} + \rho$ & $\clubsuit$-1 & $\checkmark$ & $\clubsuit$-2 & $\checkmark$ \\
$\mathfrak{L}_{6}^{ 4}$ & $\operatorname{Der}$ & $\operatorname{Der}$&$\operatorname{Der}$ & $\operatorname{Der}$ &$\operatorname{Der}$ & $\operatorname{Der}$ &$\operatorname{Der}$ & $\operatorname{Der} + \rho$  & $\operatorname{Der} + \rho$   &  $\checkmark$  &  $\operatorname{Der} + \rho$ & $\clubsuit$-3 & $\checkmark$ & $\checkmark$ \\
$\mathfrak{L}_{6}^{ 3}$ & $\operatorname{Der}$& $\operatorname{Der}$ & $\operatorname{Der}$ & $\operatorname{Der}$&$\operatorname{Der}$& $\operatorname{Der}$ &$\operatorname{Der}$ & $\operatorname{Der} + \rho$ & $\clubsuit$-1 & $\operatorname{Der} + \rho$  &$\checkmark$ & $\checkmark$ &$\checkmark$& $\checkmark$ \\
\hline
$\mathfrak{L}_{6}^{ 2}$ & $\operatorname{Der}$ & $\operatorname{Der}$ &$\operatorname{Der}$& $\operatorname{Der}$ &$\operatorname{Der}$ & $\operatorname{Der}$ &$\operatorname{Der}$ &$\operatorname{Der}$ &$\operatorname{Der}$ & $\operatorname{Der}$& $\operatorname{Der}$& $\checkmark$ & $\clubsuit$-4 & $\checkmark$\\
$\mathfrak{L}_{6}^{ 1}$ & $\operatorname{Der}$ & $\operatorname{Der}$ &$\operatorname{Der}$& $\operatorname{Der}$ &$\operatorname{Der}$ & $\operatorname{Der}$ &$\operatorname{Der}$ &$\operatorname{Der}$ &$\operatorname{Der}$ & $\operatorname{Der}$& $\operatorname{Der}$ & $\clubsuit$-4 &$\checkmark$ & $\checkmark$\\
\hline
$\mathfrak{L}_{6}^{ 0}$ & $\operatorname{Der}$ & $\operatorname{Der}$ &$\operatorname{Der}$& $\operatorname{Der}$ &$\operatorname{Der}$ & $\operatorname{Der}$ &$\operatorname{Der}$ &$\operatorname{Der}$ &$\operatorname{Der}$ & $\operatorname{Der}$& $\operatorname{Der}$ & $\operatorname{Der}$ &$\operatorname{Der}$ & $\checkmark$\\
\hline
\end{tabular}
\caption{Degenerations in the family $\mathfrak{L}_{6}^{i}$}
\label{tab:tableexe}
\end{center}
\end{table}
\end{landscape}

By using similar arguments we obtain the Hasse diagrams for the closure ordering in the remaining families:

\begin{multicols}{3}

\begin{enumerate}

\item[0.] Degenerations in the family $\mathfrak{L}_{0}^{i}$\newline
\centering
\begin{tikzpicture}
  \node[circle,fill,inner sep=0pt,minimum size=3pt,label=right:{$\mathfrak{L}_0^2$}] (L02) at (0,1) {};
  \node[circle,fill,inner sep=0pt,minimum size=3pt,label=left:{$\mathfrak{L}_0^1$}] (L01) at (0,0) {};
  \node[circle,fill,inner sep=0pt,minimum size=3pt,label=right:{$\mathfrak{L}_0^0$}] (L00) at (0,-1) {};
  \node(empty) at (0,-5) {};
\draw [->] (L02) edge (L01) (L01) edge (L00);
\end{tikzpicture}

\item[1.] Degenerations in the family $\mathfrak{L}_{1}^{i}$\newline
\centering
\begin{tikzpicture}
  \node[circle,fill,inner sep=0pt,minimum size=3pt,label=above:{$\mathfrak{L}_1^4$}](L14) at (+0,2) {};
  \node[circle,fill,inner sep=0pt,minimum size=3pt,label=below:{$\mathfrak{L}_1^6$}](L16) at (+0,1) {};
  \node[circle,fill,inner sep=0pt,minimum size=3pt,label=below:{$\mathfrak{L}_1^3$}](L13) at (-1,0) {};
  \node[circle,fill,inner sep=0pt,minimum size=3pt,label=below:{$\mathfrak{L}_1^5$}](L15) at (+1,0) {};
  \node[circle,fill,inner sep=0pt,minimum size=3pt,label=above:{$\mathfrak{L}_1^2$}](L12) at (0,-1) {};
  \node[circle,fill,inner sep=0pt,minimum size=3pt,label=right:{$\mathfrak{L}_1^1$}](L11) at (0,-2) {};
  \node[circle,fill,inner sep=0pt,minimum size=3pt,label=below:{$\mathfrak{L}_1^0$}](L10) at (0,-3) {};

\draw [->] (L14) edge (L16) (L16) edge (L13) (L16) edge (L15) (L13) edge (L12) (L15) edge (L12) (L12) edge (L11) (L11) edge (L10);
\end{tikzpicture}

\item[3.] Degenerations in the family $\mathfrak{L}_{3}^{i}$
\centering
\begin{tikzpicture}
  \node[circle,fill,inner sep=0pt,minimum size=3pt,label=above:{$\mathfrak{L}_3^3$}] (L33) at (0,1) {};
  \node[circle,fill,inner sep=0pt,minimum size=3pt,label=left:{$\mathfrak{L}_3^2$}] (L32) at (0,0) {};
  \node[circle,fill,inner sep=0pt,minimum size=3pt,label=right:{$\mathfrak{L}_3^1$}] (L31) at (0,-1) {};
  \node[circle,fill,inner sep=0pt,minimum size=3pt,label=below:{$\mathfrak{L}_3^0$}] (L30) at (0,-2) {};
\draw [->] (L33) edge (L32) (L32) edge (L31) (L31) edge (L30);
\end{tikzpicture}

\end{enumerate}
\end{multicols}

\begin{enumerate}

\item[2.] Degenerations in the family $\mathfrak{L}_{2}^{i}$
\centering
\begin{tikzpicture}
\node[label=above:{$\mathfrak{L}_{2}^{6}$}](L26) at (-4,1) {};
\node(L26a) at (-5,1) {};
\node(L26b) at (-3,1) {};
\node[label=below:{$\mathfrak{L}_{2}^{4}$}](L24) at (-4,0) {};
\node(L24a) at (-5,0) {};
\node(L24b) at (-3,0) {};
\draw[thick] (L26a)edge(L26b);
\draw[thick] (L24a)edge(L24b);
\draw [->] (L26) edge (L24) (L26a)edge(L24a) (L26b) edge (L24b);

\node[label=above:{$\mathfrak{L}_{2}^{5}$}](L25) at (-1,0) {};
\node(L25a) at (-2,0) {};
\node(L25b) at (0,0) {};
\node[label=below:{$\mathfrak{L}_{2}^{3}$}](L23) at (-1,-1) {};
\node(L23a) at (-2,-1) {};
\node(L23b) at (+0,-1) {};
\draw[thick] (L25a)edge(L25b);
\draw[thick] (L23a)edge(L23b);
\draw [->] (L25) edge (L23) (L25a)edge(L23a) (L25b) edge (L23b);


\node[circle,fill,inner sep=0pt,minimum size=3pt,label=above:{$\mathfrak{L}_2^2$}] (L22) at (1,0) {};
\node[circle,fill,inner sep=0pt,minimum size=3pt,label=right:{$\mathfrak{L}_2^1$}]  (L21) at (1,-1) {};
\node[circle,fill,inner sep=0pt,minimum size=3pt,label=below:{$\mathfrak{L}_2^0$}] (L20) at (1,-2) {};
\draw [->] (L22) edge (L21) (L21) edge (L20);
\end{tikzpicture}

\item[4.] Degenerations in the family $\mathfrak{L}_{4}^{i}$

{%
\centering
\begin{tikzpicture}
\node[label=above:{$\mathfrak{L}_{4}^{6}(\lambda)$}](L46) at (-4,1) {};
\node(L46a) at (-5,1) {};
\node(L46b) at (-3,1) {};
\node[label=below:{$\mathfrak{L}_{4}^{4}(\lambda)$}](L44) at (-4,0) {};
\node(L44a) at (-5,0) {};
\node(L44b) at (-3,0) {};

\draw[thick] (L46a)edge(L46b);
\draw[thick] (L44a)edge(L44b);
\draw [->] (L46a)edge(L44a) (L46)edge(L44) (L46b) edge (L44b);


\node[circle,fill,inner sep=0pt,minimum size=3pt,label=above:{$\mathfrak{L}_4^5$}] (L45) at (-1,1) {};
\node[circle,fill,inner sep=0pt,minimum size=3pt,label=above:{$\mathfrak{L}_4^3$}] (L43) at (-2,0) {};
\node[circle,fill,inner sep=0pt,minimum size=3pt,label=above:{$\mathfrak{L}_4^2$}] (L42) at (0,0) {};
\node[circle,fill,inner sep=0pt,minimum size=3pt,label=above:{$\mathfrak{L}_4^1$}] (L41) at (-1,-1) {};
\node[circle,fill,inner sep=0pt,minimum size=3pt,label=below:{$\mathfrak{L}_4^0$}] (L40) at (-1,-2) {};
\draw [->] (L45) edge (L43) (L45) edge (L42) (L43) edge (L41) (L42) edge (L41) (L41) edge (L40);
\end{tikzpicture}
}

\item[5.] Degenerations in the family $\mathfrak{L}_{5}^{i}(z)$
\centering
\begin{tikzpicture}
\node[label=above:{$\mathfrak{L}_{5}^{9}(\lambda)$}](L59) at (-3,1) {};
\node(L59a) at (-4,1) {};
\node(L59b) at (-2,1) {};
\node[label=below:{$\mathfrak{L}_{5}^{6}(\lambda)$}](L56) at (-3,0) {};
\node(L56a) at (-4,0) {};
\node(L56b) at (-2,0) {};

\draw[thick] (L59a)edge(L59b);
\draw[thick] (L56a)edge(L56b);
\draw [->] (L59a)edge(L56a) (L59)edge(L56) (L59b) edge (L56b);


\node[circle,fill,inner sep=0pt,minimum size=3pt,label=above:{$\mathfrak{L}_5^7$}] (L57) at (0,1) {};
\node[circle,fill,inner sep=0pt,minimum size=3pt,label=above:{$\mathfrak{L}_5^8$}] (L58) at (2,1) {};
\node[circle,fill,inner sep=0pt,minimum size=3pt,label=above:{$\mathfrak{L}_5^5$}] (L55) at (-1,0) {};
\node[circle,fill,inner sep=0pt,minimum size=3pt,label=above:{$\mathfrak{L}_5^3$}] (L53) at (1,0) {};
\node[circle,fill,inner sep=0pt,minimum size=3pt,label=above:{$\mathfrak{L}_5^4$}] (L54) at (3,0) {};
\node[circle,fill,inner sep=0pt,minimum size=3pt,label=below:{$\mathfrak{L}_5^2$}] (L52) at (0,-1) {};
\node[circle,fill,inner sep=0pt,minimum size=3pt,label=below:{$\mathfrak{L}_5^1$}] (L51) at (2,-1) {};
\node[circle,fill,inner sep=0pt,minimum size=3pt,label=below:{$\mathfrak{L}_5^0$}] (L50) at (1,-2) {};

\draw [->] (L57) edge (L55) (L57) edge (L53) (L58) edge (L53) (L58) edge (L54);
\draw [->] (L55) edge (L52) (L54) edge (L51) (L53) edge (L51) (L53) edge (L52) (L52) edge (L50) (L51) edge (L50);

\end{tikzpicture}

\end{enumerate}

\subsection{Degenerations $3$-dimensional hom-Lie algebras structures with nilpotent twisting map on $3$-dimensional Lie algebras}

The remainder of this paper will be devoted to explain how to obtain the classification
of the orbits closure for the hom-Lie algebras obtained above. Again,
we use, as much as possible, Theorem \ref{teorema} and Proposition \ref{EucZar}.
In fact, if a hom-Lie algebra $(\mathbb{C}^{n},\mu,A)$ degenerates to $(\mathbb{C}^{n},\lambda,B)$,
then the algebra $(\mathbb{C}^{n},\mu)$ degenerates to $(\mathbb{C}^{n},\lambda)$ (because of Proposition \ref{EucZar}).
In a similar way as in subsection \ref{famfixed}, to show that $(\mathbb{C}^{n},\mu,A)$ degenerates to $(\mathbb{C}^{n},\lambda,A)$, we can try
to find a family $g_t$ in $\operatorname{GL}(3,\mathbb{C})_{A}$, the stabilizer of $A$ under conjugation, such that
$g_t\bullet \mu$ tends to $\lambda$ as $t$ tends to infinity.  Based on a case-by-case analysis, all the degenerations
in such situation can be obtained by executing such idea.

\begin{example}\label{exdeg2}
We will show that $ \mathfrak{L}_{6}^{9}(\lambda)  = ( \mathfrak{r}_{2}(\mathbb{C}) \times \mathbb{C} , A(\lambda))$  degenerates to $\mathfrak{L}_{1}^{5}=(\mathfrak{n}_{3}(C),B)$. Since the twisting maps of each such hom-Lie algebras are conjugate, we will find a one-parameter family of transformations in the left coset $g \operatorname{GL}(3,\mathbb{C})_{A(\lambda)}$ where $g A(\lambda) g^{-1} = B$, to realize the degeneration
of $\mathfrak{L}_{6} = \mathfrak{r}_{2}(\mathbb{C})\times \mathbb{C}$ in $\mathfrak{L}_{1}=\mathfrak{h}_{3}(\mathbb{C})$.

Note that any such $g$ has the form
$$
 g :=  \left[ \begin {array}{ccc} a&0&0\\ \noalign{\medskip}x&a&a
\\ \noalign{\medskip}y&x&{\frac {\lambda\,x - a}{\lambda}}\end {array}
 \right],\, a\in \mathbb{C}^{*}
$$

Now, by the change of basis given by $g$ we have the Lie algebra $\mathfrak{r}_{2}(\mathbb{C}) \times \mathbb{C}$ in this new basis
has the form
\begin{eqnarray}\label{changelie}
{[e_1,e_2]}  = {\frac { \left( a-\lambda\,x \right) }{{a }^{2}}}{ e_2}+{\frac {x \left( a-\lambda\,x \right) }{{a}^{3}}}{e_3},\,
{[e_1,e_3]}  = {\frac {\lambda}{a}}{e_2}+{\frac {x\lambda}{{a}^{2}}}{e_3}.
\end{eqnarray}

We want to try that the above Lie algebra is as close as possible to being the Heisenberg Lie algebra.
To this end, we can start letting the structure constant $e_3^{*}([e_1,e_2]) = 1$ by solving the equation
${\frac {x \left( a-\lambda\,x \right) }{{a}^{3}}}=1$ for $x$. Here, we can introduce a new variable $z$ such that $z^2 = 1 - 4a\lambda$ and so
$x = {\frac { \left( -1+{z}^{2} \right)  \left( -1\pm z \right) }{8{\lambda}^{2}}}$ and (\ref{changelie}) has the form
\begin{eqnarray*}
{[{e_1},{e_2}]}={\frac {2\lambda}{1-z}}{e_2}+{\it e_3},\, {[{e_1},{e_3}]}={\frac {4{\lambda}^{2}}{1-{z}^{2}}}{e_2}+{\frac {2\lambda}{z+1}}{e_3}.
\end{eqnarray*}

If we let $z = 1+\exp(t)$, then we obtain a parametric family of transformations, $g_t$,
such that $g_t A(\lambda) g_t^{-1} = B$ and $g_t \bullet { \mathfrak{r}_{2}(\mathbb{C}) \times \mathbb{C}} \xrightarrow[t \to \infty]{} {\mathfrak{n}_{3}(\mathbb{C})}$.

\end{example}

To obtain the classification of orbits closure in our family of hom-Lie algebras,
we proceed in the same way as in the examples of Subsection \ref{famfixed}; we
use the invariants $\operatorname{Der}$, $\psi_{\alpha,\beta}$, $\psi_{\beta}$
and $\rho$ to determine impossible degenerations and the ideas illustrated in
Examples \ref{exdeg} and \ref{exdeg2}. Here, we must mention there are five impossible degenerations
that they cannot be explained by the above-mentioned invariants; we need to give an different argument.

\begin{proposition}[$\clubsuit$-5]
The hom-Lie algebras  $\mathfrak{L}_{4}^{3}$, $\mathfrak{L}_{5}^{4}$, $\mathfrak{L}_{5}^{5}$, $\mathfrak{L}_{6}^{4}$ and
$\mathfrak{L}_{6}^{5}$ cannot degenerate to $\mathfrak{L}_{1}^{2}$
\end{proposition}

\begin{proof}
We consider the continuous map $\varpi:C^{2}\times C^{1} \rightarrow L^{2}\times C^{1}$ given by
\begin{eqnarray*}
  (\mu,A) & \mapsto& (\lambda:=\mu(A-,-), A).
\end{eqnarray*}
The map $\varpi$ sends $\mathfrak{L}_{4}^{3}$, $\mathfrak{L}_{5}^{4}$ and $\mathfrak{L}_{6}^{4}$ to
the same pair $(\lambda_1,B_1)$
\begin{eqnarray*}
  \lambda_1 &:=& \left\{ e_3 \cdot e_1 = -e_2\right.\\
   B_1 &:=&\left(
        \begin{array}{ccc}
          0 & 0 & 0 \\
          0 & 0 & 1 \\
          0 & 0 & 0 \\
        \end{array}
      \right),
\end{eqnarray*}
while $\varpi$ sends $\mathfrak{L}_{1}^{2}$ to $(\lambda_0,B_0)$
\begin{eqnarray*}
\lambda_0 &:=& \left\{ e_2 \cdot e_2 = e_3\right.\\
   B_0 &:=&\left(
        \begin{array}{ccc}
          0 & 1 & 0 \\
          0 & 0 & 0 \\
          0 & 0 & 0 \\
        \end{array}
      \right).
\end{eqnarray*}
If $\mathfrak{L}_{4}^{3}$, $\mathfrak{L}_{5}^{4}$, $\mathfrak{L}_{5}^{5}$, $\mathfrak{L}_{6}^{4}$ degenerate to $\mathfrak{L}_{1}^{2}$,
then $(\lambda_0,B_0)$ is in the closure of $\operatorname{GL}(3,\mathbb{C})$-orbit of $(\lambda_1,B_1)$, which cannot happen. In fact,
we consider the ($\operatorname{GL}(3,\mathbb{C})$-equivariant) continuous map
\begin{eqnarray*}
  T: L^{2}\times C^{1} \rightarrow \operatorname{End}(C^{1} , L^{2} \times C^{1})
\end{eqnarray*}
where $T(\lambda,B) := T_{(\lambda,B)}$ is the linear map from $C^{1}$ to $L^{2} \times C^{1}$ given by
\begin{eqnarray*}
  T_{(\lambda,B)}(X) &=& (X\lambda(-,-) , XB-BX).
\end{eqnarray*}
As in Equation (\ref{invf2}), we consider the invariant $\operatorname{Ker}T_{(-,-)}$. We have $\operatorname{Ker} T_{(\lambda_1,B_1)}$ is equal to $4$ and so $\operatorname{Ker} T_{g\cdot(\lambda_1,B_1)} $ is $4$ for any $g \in  \operatorname{GL}(3,\mathbb{C})$ and on the other hand $\operatorname{Ker} T_{(\lambda_0,B_0)} $ is equal to $3$. It follows that
$(\lambda_1,B_1)$ does not degenerate to $(\lambda_0,B_0)$ since $T_{g\cdot(\lambda_1,B_1)}$ cannot be near $T_{(\lambda_0,B_0)}$,
for any $g \in  \operatorname{GL}(3,\mathbb{C})$ (by the upper semi-continuity of the \textit{nullity function} on linear operators).

We can now proceed analogously to the previous case to prove that $\mathfrak{L}_{5}^{5}$ and $\mathfrak{L}_{6}^{5}$
do not degenerate to $\mathfrak{L}_{1}^{2}$ by considering the same invariant functions.
\end{proof}

\section*{Others classifications}

We must mention that there have been some attempts to classify complex hom-Lie algebras in dimension $3$. We have found several redundancies and mistakes in such works.

In \cite[Theorem 4.3]{ORS}, by using symbolic calculations in  {Wolfram} \textit{Mathematica}\textsuperscript{\textregistered}, it is given five non-isomorphic families of Hom-Lie algebras to parameterize $3$-dimensional hom-Lie algebras with nilpotent twisting map. As the authors already mentioned, the provided families in \cite[Theorem 4.3]{ORS} present redundancies. For example, we could take the subfamily of $\mathcal{H}_{7}^{3}$ given by

\begin{eqnarray*}
\mu_{(a,b,c)} &:=&
\left\{
\begin{array}{l}
[{e_1},{e_2}] = -{a}^{2}c{e_2}+ \left( {a}^{4}+ab \right) {e_3},
[{e_1},{e_3}] =  a{e_1}+b{e_3},\\
{[{e_2},{e_3}]} =   {e_1}+c{e_2}-{a}^{2}{e_3}.
\end{array}
\right.
\end{eqnarray*}
It is easy to see that $(\mu_{(a_0,b_0,c_0)},\alpha_3)$ is isomorphic to $(\mu_{(a_1,b_1,c_1)},\alpha_3)$ if and only if
$(a_1,b_1,c_1) = (a_0,b_0,c_0)$ or $(a_1,b_1,c_1) = (c_0, b_0 + (a_0 - c_0)(a_0^2+c_0^2) ,a_0)$. The algebra
$(\mathbb{C}^{3},\mu_{(a,b,c)})$ is a Lie algebra if and only if $a=0$ and $b=0$, or $c=0$ and $b=-a^3$. In theses
cases the hom-Lie algebra $(\mathbb{C}^{3},\mu_{(0,0,c)}, \alpha_3)$ is isomorphic to $(\mathbb{C}^{3},\mu_{(c,-c^3,0)}, \alpha_3)$
and also, it is isomorphic to our hom-Lie algebra $\mathfrak{L}_{6}^{13}(\frac{1}{c})$. We recall the $\alpha_3$ is the linear map defined by
$\alpha_3 e_1 = e_2$, $\alpha_3 e_2 = e_3$ and $\alpha_3 e_3 = 0$.

With respect to \cite{G-DSS-V}, there are also visible mistakes and redundancies in the classification obtained there.
By way of example, in \cite[Proposition 4.3]{G-DSS-V}, it is easily seen that the given twisting maps are far from being \textit{canonical forms} for the different hom-Lie structures on $\mathfrak{sl}(2,\mathbb{C})$. For instance, for any $x\in \mathbb{C}^{\star}$,
the hom-Lie algebra $(\mathfrak{sl}(2,\mathbb{C}),A(a,x))$ with $A(a,x)$ equal to
\begin{eqnarray*}
  \left(
    \begin{array}{c|cc}
      a & 0 & 0 \\
      \hline
      0 & a & x^2 \\
      0 & 0 & a \\
    \end{array}
  \right)
\end{eqnarray*}
is isomorphic to $(\mathfrak{sl}(2,\mathbb{C}),A(a,1))$ via the $g \in \operatorname{Aut}(\mathfrak{sl}(2,\mathbb{C}))$
defined by $gH = H$, $gE = \tfrac{1}{x}E$ and $gF = xF$. Similarly, in the second familiy, the hom-Lie algebra $(\mathfrak{sl}(2,\mathbb{C}),B(a,x))$
is isomorphic to $(\mathfrak{sl}(2,\mathbb{C}),B(a,1))$ via the same $g \in \operatorname{Aut}(\mathfrak{sl}(2,\mathbb{C}) )$ given above, where
$B(a,x)$ is
\begin{eqnarray*}
  \left(
    \begin{array}{c|cc}
      a  & 0  & x \\
      \hline
      2x & a  & 0 \\
      0  & 0  & a \\
    \end{array}
  \right).
\end{eqnarray*}
In the third case, a true canonical form could be
\begin{eqnarray*}
  \left(
    \begin{array}{c|cc}
      a   & x  & 1 \\
      \hline
      2   & a  & 0 \\
      2x  & 0  & a \\
    \end{array}
  \right), \mbox{ with } a,x\in\mathbb{C}, x\neq0,
\end{eqnarray*}
and in the last one, a canonical form could be $C(a,x,y)$ equal to
\begin{eqnarray*}
  \left(
    \begin{array}{c|cc}
       a   & x  & y \\
       \hline
       2y  & a  & 1 \\
       2x  & 0  & a \\
    \end{array}
  \right), \mbox{ with } a,x,y\in\mathbb{C}, x y\neq0,
\end{eqnarray*}
and $x \in \mathbb{R}_{>0}$ or $\mathfrak{Im}(x)>0$.

\begin{landscape}
\begin{table}[p!]
\begin{center}
\begin{tabular}{c : c : c : c : c : c : c : c : c}
  $\operatorname{Dim}\operatorname{Der}$ & $\mathfrak{L}_{0}$ & $\mathfrak{L}_{1}$ & $\mathfrak{L}_{2}$ & $\mathfrak{L}_{3}$ &
  $\mathfrak{L}_{4}$ & $\mathfrak{L}_{5}(z)$ & $\mathfrak{L}_{6}$ & $\mathfrak{L}_{7}$ \\
  \hdashline
  \hdashline
   0 &  & $\mathfrak{L}_{1}^{4}$ &  &  &  &  & $\mathfrak{L}_{6}^{13}(\lambda)$ , $\mathfrak{L}_{6}^{12}$, $\mathfrak{L}_{6}^{11}$ & $\mathfrak{L}_{7}^{2}$ \\
  \hdashline
   1 &  & $\mathfrak{L}_{1}^{6}$  & $\mathfrak{L}_{2}^{6}(\lambda)$  &  & $\mathfrak{L}_{4}^{6}(\lambda)$, $\mathfrak{L}_{4}^{5}$ & $\mathfrak{L}_{5}^{9}(z,\lambda)$,$\mathfrak{L}_{5}^{7}(z)$,$\mathfrak{L}_{5}^{8}(z)$ &
   $\mathfrak{L}_{6}^{9}(\lambda)$, $\mathfrak{L}_{6}^{10}$, $\mathfrak{L}_{6}^{8}$, $\mathfrak{L}_{6}^{7}$ & $\mathfrak{L}_{7}^{1}$ \\
  \hdashline
   2 &  & $\mathfrak{L}_{1}^{5}$, $\mathfrak{L}_{1}^{3}$   & $\mathfrak{L}_{2}^{4}(\lambda)$, $\mathfrak{L}_{2}^{5}(\lambda)$, $\mathfrak{L}_{2}^{2}$ &
           $\mathfrak{L}_{3}^{3}$ & $\mathfrak{L}_{4}^{4}(\lambda)$, $\mathfrak{L}_{4}^{3}$, $\mathfrak{L}_{4}^{2}$ &
           $\mathfrak{L}_{5}^{6}(z,\lambda)$, $\mathfrak{L}_{5}^{5}(z)$, $\mathfrak{L}_{5}^{4}(z)$, $\mathfrak{L}_{5}^{3}(z)$ &
            $\mathfrak{L}_{6}^{6}(\lambda)$, $\mathfrak{L}_{6}^{5}$, $\mathfrak{L}_{6}^{4}$, $\mathfrak{L}_{6}^{3}$ &  \\
  \hdashline
   3 & $\mathfrak{L}_{0}^{2}$ & $\mathfrak{L}_{1}^{2}$  & $\mathfrak{L}_{2}^{3}(\lambda)$,$\mathfrak{L}_{2}^{1}$ & $\mathfrak{L}_{3}^{2}$ &
       $\mathfrak{L}_{4}^{1}$  & $\mathfrak{L}_{5}^{2}(z)$, $\mathfrak{L}_{5}^{1}(z)$ &
       $\mathfrak{L}_{6}^{2}$, $\mathfrak{L}_{6}^{1}$ & $\mathfrak{L}_{7}^{0}$ \\
  \hdashline
   4 &  & $\mathfrak{L}_{1}^{1}$  & $\mathfrak{L}_{2}^{0}$ & $\mathfrak{L}_{3}^{1}$ & $\mathfrak{L}_{4}^{0}$ & $\mathfrak{L}_{5}^{0}(z)$ &
            $\mathfrak{L}_{6}^{0}$ &  \\
  \hdashline
   5 & $\mathfrak{L}_{0}^{1}$ &  &  &  &  &  &  &  \\
  \hdashline
   6 &  & $\mathfrak{L}_{1}^{0}$  &  & $\mathfrak{L}_{3}^{0}$ &  &  &  &  \\
  \hdashline
   9 & $\mathfrak{L}_{0}^{0}$ &  &  &  &  &  &  &  \\
  \end{tabular}
\caption{Dimension of the algebra of derivations of the hom-Lie algebras in Section \ref{homLie3}}
\label{tab:table1}
\end{center}
\end{table}
\end{landscape}

\begin{table}[h!]
\begin{center}
\begin{tabular}{| m{1cm} | m{6cm} | m{1cm} | }
\hline
                                       {} & \multicolumn{2}{ c| }{$\psi_{\alpha,\beta}$}  \\
\hline
\multirow{2}{*}{$
\begin{array}{l}
\mathfrak{L}_{1}^{0},\mathfrak{L}_{1}^{1},\\
\mathfrak{L}_{1}^{2},\mathfrak{L}_{1}^{5}
\end{array}$
}
& $\alpha$ & \multirow{2}{*}{$\mathfrak{n}_{3}$} \\

                                          & $\beta$  &                                                 \\
\hline
\multirow{10}{*}{$
\begin{array}{l}
\mathfrak{L}_{1}^{3},\mathfrak{L}_{1}^{4},\\
\mathfrak{L}_{1}^{6}
\end{array}$
}
                                          & $\alpha=\beta$ & \multirow{2}{*}{$\mathfrak{n}_{3}$} \\
                                          & $\alpha=0$     &                                                 \\
\cdashline{2-3}
                                          & ${\alpha=\beta}$ & \multirow{2}{*}{$\mathfrak{r}_{3} $} \\
                                          & ${\alpha\neq 0}$ &                                      \\
\cdashline{2-3}
                                          & ${\alpha\beta =0}$  & \multirow{2}{*}{$\mathfrak{r}_{2}\times\mathbb{C} $} \\
                                          & ${\alpha\neq\beta}$ & \\
\cdashline{2-3}
                                          & ${\alpha=-\beta}$ & \multirow{2}{*}{$ \mathfrak{r}_{3,-1}$} \\
                                          & ${\alpha\neq 0}$  & \\
\cdashline{2-3}
                                          & \multirow{2}{*}{{$\alpha \beta(\alpha^2-\beta^2)\neq 0$}} & \multirow{2}{*}{$\mathfrak{r}_{3,z} $} \\
                                          & ${}$ & \\
\hline

\multirow{2}{*}{$
\begin{array}{l}
\mathfrak{L}_{2}^{0},\mathfrak{L}_{2}^{1}\\
\mathfrak{L}_{2}^{2}
\end{array}$
}
& $\alpha$ & \multirow{2}{*}{$\mathfrak{r}_{3}$} \\
                                          & $\beta$  &                                                 \\
\hline
\multirow{2}{*}{$
\begin{array}{l}
\mathfrak{L}_{2}^{3}(\lambda),\\
\mathfrak{L}_{2}^{5}(\lambda)
\end{array}$
}
& $1 + \beta\lambda + \alpha\lambda \neq 0$  & \multirow{1}{*}{$\mathfrak{r}_{3}$} \\
\cdashline{2-3}
                                          & $1 + \beta\lambda + \alpha\lambda = 0$  &   \multirow{1}{*}{$\mathfrak{r}_{3,1}$}          \\
\hline
\multirow{11}{*}{$
\begin{array}{l}
\mathfrak{L}_{2}^{4}(\lambda),\\
\mathfrak{L}_{2}^{6}(\lambda)
\end{array}$
}
&$\alpha= { {(-1\mp\sqrt {2})}/{\lambda}}$   & \multirow{2}{*}{$\mathfrak{n}_3$} \\
&$\beta = { {(-1\pm\sqrt {2})}/{\lambda}}$   & \\
\cdashline{2-3}
&$\alpha=-{ {(s-1 \pm \sqrt {2\,{s}^{2}-2\,s})}/{s\lambda}}$   &  \multirow{2}{*}{$\mathfrak{r}_{3}$} \\
&$\beta= { {(-s+1 \pm \sqrt {2\,{s}^{2}-2\,s})}/{s\lambda}}$   & \\
\cdashline{2-3}
&$\alpha=-{ {(2\,s-1 \pm \sqrt {8\,{s}^{2}-4\,s+1})}/{2s\lambda}}$   &  \multirow{2}{*}{$\mathfrak{r}_{2}\times \mathbb{C}$} \\
&$ \beta= { {(-2\,s+1 \pm \sqrt {8\,{s}^{2}-4\,s+1})}/{2s\lambda}}$   & \\
\cdashline{2-3}
&$\alpha=-{{(2\,s \pm \sqrt {s \left( 8\,s+1 \right) })}/{2s\lambda}}$   &  \multirow{2}{*}{$\mathfrak{r}_{3,-1}$} \\
&$\beta=  {{(-2\,s \pm \sqrt {s \left( 8\,s+1 \right) })}/{2s\lambda}}$   & \\
\cdashline{2-3}
&$\alpha=-{\frac {2\,s_{{1}}s_{{2}}-s_{{1}}\pm\sqrt {s_{{1}}s_{{2}} \left( 8\,s_{{1}}s_{{2}}-4\,s_{{1}}+s_{{2}} \right) }}{2s_{{1}}s_{{2}}\lambda}}$   &  \multirow{2}{*}{$\mathfrak{r}_{3,z}$} \\
&$\beta={\frac {-2\,s_{{1}}s_{{2}}+s_{{1}}\pm\sqrt {s_{{1}}s_{{2}} \left( 8\,s_{{1}}s_{{2}}-4\,s_{{1}}+s_{{2}} \right) }}{2s_{{1}}s_{{2}}\lambda}}$   &  \\
&$s_{1}-s_{2}^2 \neq 0 $ & \\
\hline

\multirow{2}{*}{$
\begin{array}{l}
\mathfrak{L}_{3}^{0},\\
\mathfrak{L}_{3}^{1}
\end{array}$
}
& $\alpha$ & \multirow{2}{*}{$\mathfrak{r}_{3,1}$} \\
                                          & $\beta$  &                                                 \\
\hline
\multirow{2}{*}{$
\begin{array}{l}
\mathfrak{L}_{3}^{2},\\
\mathfrak{L}_{3}^{3}
\end{array}$
}
& $\alpha = -\beta$ & \multirow{1}{*}{$\mathfrak{r}_{3,1}$} \\
\cdashline{2-3}
& $\alpha \neq -\beta$  & \multirow{1}{*}{$\mathfrak{r}_{3}$}\\
\hline
\multirow{3}{*}{$
\begin{array}{l}
\mathfrak{L}_{4}^{0},\mathfrak{L}_{4}^{1},\\
\mathfrak{L}_{4}^{2},\mathfrak{L}_{4}^{3},\\
\mathfrak{L}_{4}^{5}
\end{array}$
}
& \multirow{3}{*}{$\alpha$, $\beta$} & \multirow{3}{*}{$\mathfrak{r}_{3,-1}$} \\
&                                    &                                       \\
&                                    &                                       \\
\hline
\multirow{11}{*}{$
\begin{array}{l}
\mathfrak{L}_{4}^{4}(\lambda),\\
\mathfrak{L}_{4}^{6}(\lambda)\\
\end{array}$
}
& ${\alpha = -\beta  }$                          & \multirow{2}{*}{$\mathfrak{n}_{3}$} \\
& ${\alpha = \pm \frac{\sqrt{-1}}{2 \lambda} }$  & \\
\cdashline{2-3}
& $\alpha={ {(1 \mp \sqrt{-1}s)}/{2s\lambda}},$ & \multirow{2}{*}{$\mathfrak{r}_{3}$} \\
& $\beta = { {(1 \pm \sqrt{-1}s)}/{2s\lambda}}$     & \\
\cdashline{2-3}
& $\alpha={ {(1 \mp \sqrt {1-4\,{s}^{2}})}/{4s\lambda}},$  & \multirow{2}{*}{$\mathfrak{r}_{2}\times \mathbb{C}$} \\
& $\beta=  { {( 1\pm\sqrt {1-4\,{s}^{2}})}/{4s\lambda}}$   & \\
\cdashline{2-3}
& $\alpha=\mp{ {(\sqrt {s-{s}^{2}})}/{2s\lambda}},$       & \multirow{2}{*}{$\mathfrak{r}_{3,-1}$} \\
& $\beta=\pm{ {(\sqrt {s-{s}^{2} })}/{2s\lambda}}$        & \\
\cdashline{2-3}
&  $\alpha={\frac {s_{{1}} \mp \sqrt {s_{{1}}{s_{{2}}}^{2} \left( 1-s_{{1}} \right) }}{2s_{{1}}s_{{2}}\lambda}}$ & \multirow{2}{*}{$\mathfrak{r}_{3,z}$} \\
&  $\beta={\frac {s_{{1}} \pm \sqrt {s_{{1}}{s_{{2}}}^{2} \left( 1-s_{{1}} \right) }}{2s_{{1}}s_{{2}}\lambda}}  $ & \\
&  $s_{1}-s_{2}^2 \neq 0$ & \\
\hline
\end{tabular}
\caption{$\psi_{\alpha,\beta}$-invariant of the hom-Lie algebras in Section \ref{homLie3}}
\label{tab:table2a}
\end{center}
\end{table}

\begin{table}[!h]
\begin{center}
\begin{tabular}{| m{2cm} | m{7cm} | m{1cm} | }
\hline
                                       {} & \multicolumn{2}{ c| }{$\psi_{\alpha,\beta}$}  \\
\hline
\multirow{4}{*}{$
\begin{array}{l}
\mathfrak{L}_{5}^{0}(z),\mathfrak{L}_{5}^{1}(z),\\
\mathfrak{L}_{5}^{2}(z),\mathfrak{L}_{5}^{3}(z),\\
\mathfrak{L}_{5}^{4}(z),\mathfrak{L}_{5}^{5}(z),\\
\mathfrak{L}_{5}^{7}(z),\mathfrak{L}_{5}^{8}(z)
\end{array}$
}
& \multirow{4}{*}{$\alpha$, $\beta$} & \multirow{4}{*}{$\mathfrak{r}_{3,z}$} \\
&                                    &                                       \\
&                                    &                                       \\
&                                    &                                       \\
\hline
\multirow{12}{*}
{$
\begin{array}{l}
\mathfrak{L}_{5}^{6}(z,\lambda),\\
\mathfrak{L}_{5}^{9}(z,\lambda)
\end{array}
$}
&$\alpha= { {(z+1 \mp \,\sqrt {{z}^{2}+6\,z+1})}/{ 2 \left( -1+z \right) \lambda}}$ & \multirow{2}{*}{$\mathfrak{n}_{3}$} \\
&$ \beta= { {(z+1 \pm \,\sqrt {{z}^{2}+6\,z+1})}/{ 2\left( -1+z \right) \lambda}}$                                    &                                       \\
\cdashline{2-3}
&$\alpha={\frac {sz+s-1 \pm \sqrt {s \left( s-2+6\,sz-2\,z+s{z}^{2} \right) }}{ 2\left( -1+z \right) \lambda\,s}}$ & \multirow{2}{*}{$\mathfrak{r}_{3}$} \\
&$ \beta={\frac {sz+s-1 \mp \sqrt {s \left( s-2+6\,sz-2\,z+s{z}^{2} \right) }}{ 2\left( -1+z \right) \lambda\,s}}$  &                                     \\
\cdashline{2-3}
&$\alpha={\frac {sz+s-1\pm\sqrt {{s}^{2}-2\,s+6\,z{s}^{2}+1-2\,sz+{z}^{2}{s}^{2}}}{2s\lambda\, \left( -1+z \right) }}$ & \multirow{2}{*}{$\mathfrak{r}_{2}\times \mathbb{C}$} \\
&$\beta={\frac {sz+s-1\mp\sqrt {{s}^{2}-2\,s+6\,z{s}^{2}+1-2\,sz+{z}^{2}{s}^{2}}}{2s\lambda\, \left( -1+z \right) }}$&                                       \\
\cdashline{2-3}
&$ \alpha={\frac {sz+s\pm\sqrt {s \left( {z}^{2}s+6\,sz+s-4 \right) }}{2s\lambda\, \left( -1+z \right) }}$& \multirow{2}{*}{$\mathfrak{r}_{3,-1}$} \\
&$\beta={\frac {sz+s\mp\sqrt {s \left( {z}^{2}s+6\,sz+s-4 \right) }}{2s\lambda\, \left( -1+z \right) }}$ &                                       \\
\cdashline{2-3}
&$\alpha={\frac {zs_{{1}}s_{{2}}-s_{{1}}+s_{{1}}s_{{2}}\pm\sqrt {f(s_1,s_2)}}
            {2s_{{1}}s_{{2}}\lambda\, \left( -1+z \right) }}$ & \multirow{4}{*}{$\mathfrak{r}_{3,\widetilde{z}}$} \\
&$\beta={\frac {zs_{{1}}s_{{2}}-s_{{1}}+s_{{1}}s_{{2}}\mp\sqrt {f(s_1,s_2)}}
            {2s_{{1}}s_{{2}}\lambda\, \left( -1+z \right) }}$  &                                       \\
& {\tiny $f(s_1,s_2)=s_{{1}}s_{{2}}( {z}^{2}s_{{1}}s_{{2}}-2zs_{{1}}+6zs_{{1}}s_{{2}}-2\,s_{{1}}+s_{{1}}s_{{2}}+s_{{2}})$}& \\
&$s_{{1}}-{s_{{2}}}^{2} \neq 0$&\\
\hline
\multirow{5}{*}{$
\begin{array}{l}
\mathfrak{L}_{6}^{0},\mathfrak{L}_{6}^{1},\\
\mathfrak{L}_{6}^{2},\mathfrak{L}_{6}^{3},\\
\mathfrak{L}_{6}^{4},\mathfrak{L}_{6}^{5},\\
\mathfrak{L}_{6}^{7},\mathfrak{L}_{6}^{8},\\
\mathfrak{L}_{6}^{10},\mathfrak{L}_{6}^{11}
\end{array}$
}
& \multirow{5}{*}{$\alpha$, $\beta$} & \multirow{5}{*}{$\mathfrak{r}_{2}\times \mathbb{C}$} \\
&                                    &                                       \\
&                                    &                                       \\
&                                    &                                       \\
&                                    &                                       \\
\hline
\multirow{11}{*}{$
\begin{array}{l}
\mathfrak{L}_{6}^{6}(\lambda)\\
\mathfrak{L}_{6}^{9}(\lambda)
\end{array}$
}
&$\alpha=-1/\lambda$, $\beta=0$ or & \multirow{2}{*}{$\mathfrak{n}_{3}$} \\
&$\alpha=0$, $\beta=-1/\lambda$.   &                                    \\
\cdashline{2-3}
&$\alpha={ {(-s+2 \mp \sqrt {{s}^{2}-4\,s})}/{2s\lambda}}$ & \multirow{2}{*}{$\mathfrak{r}_{3}$}    \\
&$ \beta={ {(-s+2 \pm \sqrt {{s}^{2}-4\,s})}/{2s\lambda}}$ &                                       \\
\cdashline{2-3}
&$\alpha=-{\frac {s-1}{s\lambda}},\beta=0$ or & \multirow{2}{*}{$\mathfrak{r}_{2}\times \mathbb{C}$}  \\
&$\alpha=0,\beta=-{\frac {s-1}{s\lambda}}$ &                                       \\
\cdashline{2-3}
&$\alpha={ {(-s\pm\sqrt {s \left( s-4 \right) })}/{2\lambda\,s}}$ & \multirow{2}{*}{$\mathfrak{r}_{3,-1}$}    \\
&$ \beta={ {(-s\mp\sqrt {s \left( s-4 \right) })}/{2\lambda\,s}}$  &                                       \\
\cdashline{2-3}
&$\alpha={\frac {-s_{{1}}s_{{2}}+s_{{1}}\pm\sqrt {s_{{1}}s_{{2}} \left( s_{{1}}s_{{2}}-2\,s_{{1}}+s_{{2}} \right) }}{2s_{{1}}s_{{2}}\lambda}}$& \multirow{2}{*}{$\mathfrak{r}_{3,z}$}    \\
&$\beta={\frac {-s_{{1}}s_{{2}}+s_{{1}}\mp\sqrt {s_{{1}}s_{{2}} \left( s_{{1}}s_{{2}}-2\,s_{{1}}+s_{{2}} \right) }}{2s_{{1}}s_{{2}}\lambda}}$ &                                       \\
&$s_{{1}}-{s_{{2}}}^{2} \neq 0$ &                                       \\
\hline
\multirow{2}{*}{$\mathfrak{L}_{6}^{12}$} & $\alpha\beta = 0$ & $\mathfrak{r}_{2}\times\mathbb{C}$ \\
\cdashline{2-3}
&$\alpha\beta \neq 0$& NoLie\\
\hline
\multirow{4}{*}{$\mathfrak{L}_{6}^{13}(\lambda)$} & $\alpha =0,\beta=-1/\lambda$&  $\mathfrak{n}_{3}$\\
\cdashline{2-3}
&$\alpha =0,\beta\neq-1/\lambda$ or &\multirow{2}{*}{$\mathfrak{r}_{2}\times\mathbb{C}$}\\
&$\beta=0$ &\\
\cdashline{2-3}
& $\alpha\beta \neq 0$& NoLie \\
\hline
\multirow{2}{*}{$
\begin{array}{l}
\mathfrak{L}_{7}^{0},\mathfrak{L}_{7}^{1},\\
\mathfrak{L}_{7}^{2}\\
\end{array}$
}
& \multirow{2}{*}{$\alpha$, $\beta$} & \multirow{2}{*}{$\mathfrak{so}(3,\mathbb{C})$} \\
&                                    &                                       \\
\hline
\end{tabular}
\caption{$\psi_{\alpha,\beta}$-invariant of the hom-Lie algebras in Section \ref{homLie3}}
\label{tab:table2b}
\end{center}
\end{table}

\begin{table}[!h]
\begin{center}
\begin{tabular}{| m{2.725cm} | m{3cm} | m{1cm} | }
\hline
                                       {} & \multicolumn{2}{ c| }{$\phi_{\beta}$}  \\
\hline
\multirow{8}{*}{$
\begin{array}{l}
\mathfrak{L}_{ 1}^{0 },\mathfrak{L}_{ 1}^{1 },\mathfrak{L}_{ 1}^{2 },\mathfrak{L}_{ 1}^{ 5},\\
\mathfrak{L}_{ 2}^{0 },\mathfrak{L}_{ 2}^{1 },\mathfrak{L}_{ 2}^{ 2},\\
\mathfrak{L}_{ 3}^{0 },\mathfrak{L}_{ 3}^{1 },\\
\mathfrak{L}_{ 4}^{0 },\mathfrak{L}_{ 4}^{1 },\mathfrak{L}_{4 }^{2 },\\
\mathfrak{L}_{ 5}^{0 },\mathfrak{L}_{ 5}^{1 },\mathfrak{L}_{5 }^{2 },\mathfrak{L}_{ 5}^{3 },\\
\mathfrak{L}_{ 6}^{0 },\mathfrak{L}_{ 6}^{1 },\mathfrak{L}_{ 6}^{2 },\mathfrak{L}_{ 6}^{3},\\
\mathfrak{L}_{ 6}^{5 },\mathfrak{L}_{ 6}^{7 },\mathfrak{L}_{ 6}^{10 },\mathfrak{L}_{ 6}^{11 },\\
\mathfrak{L}_{ 7}^{0 }.
\end{array}$
}
& \multirow{8}{*}{$\beta$} & \multirow{7}{*}{$\mathfrak{a}_{3}$} \\
&                                    &                                       \\
&                                    &                                       \\
&                                    &                                       \\
&                                    &                                       \\
&                                    &                                       \\
&                                    &                                       \\
&                                    &                                       \\
\hline
\multirow{4}{*}{$
\begin{array}{l}
\mathfrak{L}_{ 1 }^{ 3},\mathfrak{L}_{ 1 }^{4 },\mathfrak{L}_{ 1 }^{6 },\\
\mathfrak{L}_{ 4 }^{ 4}(\lambda),\mathfrak{L}_{ 4 }^{ 6 }(\lambda).
\end{array}$
}
& $\beta =  1$ & $\mathfrak{r}_{3,1}$ \\
\cdashline{2-3}
& $\beta =  0$ & $\mathfrak{r}_{2} \times \mathbb{C}$ \\
\cdashline{2-3}
& $\beta = -1$ & $\mathfrak{r}_{3,-1}$ \\
\cdashline{2-3}
& $\beta(\beta^2-1) \neq 0$ & $\mathfrak{r}_{3,z}$ \\
\hline
\multirow{2}{*}{$
\begin{array}{l}
\mathfrak{L}_{2}^{3}(\lambda),\mathfrak{L}_{2}^{ 5}(\lambda),\\
\mathfrak{L}_{3}^{2},\mathfrak{L}_{3}^{3}.
\end{array}$
}
&$\beta = -1$  & $\mathfrak{a}_{3}$\\
\cdashline{2-3}
&$\beta\neq -1$& $\mathfrak{n}_{3}$\\
\hline
\multirow{4}{*}{$
\begin{array}{l}
\mathfrak{L}_{2}^{4}(\lambda),\mathfrak{L}_{2}^{6}(\lambda),\\
\mathfrak{L}_{5}^{6}(z,\lambda),\mathfrak{L}_{5}^{9}(z,\lambda),\\
\mathfrak{L}_{6}^{6}(\lambda),\mathfrak{L}_{6}^{9}(\lambda).
\end{array}$
}
&$\beta = 1$ & $\mathfrak{r}_{3}$\\
\cdashline{2-3}
&$\beta = 0$ & $\mathfrak{r}_{2}\times \mathbb{C}$\\
\cdashline{2-3}
&$\beta = -1$& $\mathfrak{r}_{3,-1}$\\
\cdashline{2-3}
&$\beta(\beta^2-1)\neq0$& $\mathfrak{r}_{3,\widetilde{z}}$\\
\hline
\multirow{2}{*}{$
\begin{array}{l}
\mathfrak{L}_{4}^{3},\mathfrak{L}_{4}^{5},\\
\mathfrak{L}_{7}^{1}
\end{array}$
}
&$\beta = 1$& $\mathfrak{a}_{3}$ \\
\cdashline{2-3}
&$\beta\neq1$& $\mathfrak{n}_{3}$\\
\hline
\multirow{2}{*}
{
$\mathfrak{L}_{5}^{4},\mathfrak{L}_{5}^{8}.$
}
&$\beta=-z$& $\mathfrak{a}_{3}$\\
\cdashline{2-3}
&$\beta\neq-z$& $\mathfrak{n}_{3}$\\
\hline
\multirow{2}{*}
{
$\mathfrak{L}_{5}^{5},\mathfrak{L}_{5}^{7}$.
}
&$\beta=-1/z$& $\mathfrak{a}_{3}$\\
\cdashline{2-3}
&$\beta\neq-1/z$&$\mathfrak{n}_{3}$\\
\hline
\multirow{2}{*}
{
$\mathfrak{L}_{6}^{4},\mathfrak{L}_{6}^{8}.$
}
&$\beta =0$& $\mathfrak{a}_{3}$\\
\cdashline{2-3}
&$\beta\neq0$& $\mathfrak{n}_{3}$\\
\hline
\multirow{1}{*}
{
$\mathfrak{L}_{6}^{5},\mathfrak{L}_{6}^{7}.$
}
&$\beta$& $\mathfrak{n}_{3}$\\
\hline
\multirow{1}{*}
{
$\mathfrak{L}_{6}^{10},\mathfrak{L}_{6}^{11}.$
}
&$\beta$& $\mathfrak{r}_{2}\times\mathbb{C}$\\
\hline
\multirow{2}{*}
{
$\mathfrak{L}_{6}^{12},\mathfrak{L}_{6}^{13}.$
}
&$\beta =0$& $\mathfrak{r}_{2} \times \mathbb{C}$\\
\cdashline{2-3}
&$\beta\neq0$& NoLie\\
\hline
\multirow{2}{*}
{
$\mathfrak{L}_{7}^{2}.$
}
&$\beta =1$& $\mathfrak{a}_{3}$\\
\cdashline{2-3}
&$\beta\neq-1$& $\mathfrak{r}_{3,-1}$\\
\hline
\end{tabular}
\caption{$\phi_{\beta}$-invariant of the hom-Lie algebras in Section \ref{homLie3}}
\label{tab:table3}
\end{center}
\end{table}

\begin{table}[!!h]
\begin{center}
\begin{tabular}{| m{2.725cm} |  m{1cm} | }
\hline
                                        & $\rho   $  \\
\hline
\multirow{7}{*}{$
\begin{array}{l}
\mathfrak{L}_{1}^{0},\mathfrak{L}_{1}^{1}, \mathfrak{L}_{1}^{2}, \mathfrak{L}_{1}^{5},\\
\mathfrak{L}_{2}^{0},\mathfrak{L}_{2}^{1},\mathfrak{L}_{2}^{2},\\
\mathfrak{L}_{3}^{0},\mathfrak{L}_{3}^{1},\\
\mathfrak{L}_{4}^{0},\mathfrak{L}_{4}^{1},\mathfrak{L}_{4}^{2},\\
\mathfrak{L}_{5}^{0},\mathfrak{L}_{5}^{1},\mathfrak{L}_{5}^{2},\mathfrak{L}_{5}^{3},\\
\mathfrak{L}_{6}^{0},\mathfrak{L}_{6}^{1},\mathfrak{L}_{6}^{2},\mathfrak{L}_{6}^{3},\\
\mathfrak{L}_{6}^{5},\mathfrak{L}_{6}^{7},\mathfrak{L}_{6}^{10},\mathfrak{L}_{6}^{11},\\
\mathfrak{L}_{7}^{0}.
\end{array}$
}
& \multirow{7}{*}{$ \mathfrak{a}_{3}$} \\
& \\
& \\
& \\
& \\
& \\
& \\
& \\
\hline
\multirow{5}{*}{$
\begin{array}{l}
 \mathfrak{L}_{1}^{3},  \mathfrak{L}_{1}^{4},  \mathfrak{L}_{1}^{6}, \\
 \mathfrak{L}_{2}^{4},  \mathfrak{L}_{2}^{6},\\
 \mathfrak{L}_{4}^{4},   \mathfrak{L}_{4}^{6}, \\
 \mathfrak{L}_{5}^{6},  \mathfrak{L}_{5}^{9}, \\
  \mathfrak{L}_{6}^{6},  \mathfrak{L}_{6}^{9},  \mathfrak{L}_{6}^{13}.
\end{array}$
}
& \multirow{5}{*}{$ \mathfrak{r}_{2} \times \mathbb{C}$} \\
& \\
& \\
& \\
& \\
\hline
\multirow{5}{*}{$
\begin{array}{l}
 \mathfrak{L}_{2}^{3}, \mathfrak{L}_{2}^{5},  \\
  \mathfrak{L}_{3}^{2},  \mathfrak{L}_{3}^{3}, \\
   \mathfrak{L}_{4}^{3},  \mathfrak{L}_{4}^{5}, \\
    \mathfrak{L}_{5}^{4},  \mathfrak{L}_{5}^{5},  \mathfrak{L}_{5}^{7},  \mathfrak{L}_{5}^{8}, \\
     \mathfrak{L}_{6}^{4},  \mathfrak{L}_{6}^{8}, \mathfrak{L}_{6}^{12},\\
      \mathfrak{L}_{7}^{1},
\end{array}$
}
& \multirow{5}{*}{$ \mathfrak{n}_{3}$} \\
& \\
& \\
& \\
& \\
& \\
\hline\multirow{1}{*}{$
\begin{array}{l}
 \mathfrak{L}_{7}^{2}
\end{array}$
}
& \multirow{1}{*}{$ \mathfrak{r}_{3,-1}$} \\
\hline
\end{tabular}
\caption{$\rho$-invariant of the hom-Lie algebras in Section \ref{homLie3}}
\label{tab:table4}
\end{center}
\end{table}

\begin{landscape}
\begin{center}

\begin{figure}[h!]
\includegraphics[width=\paperwidth]{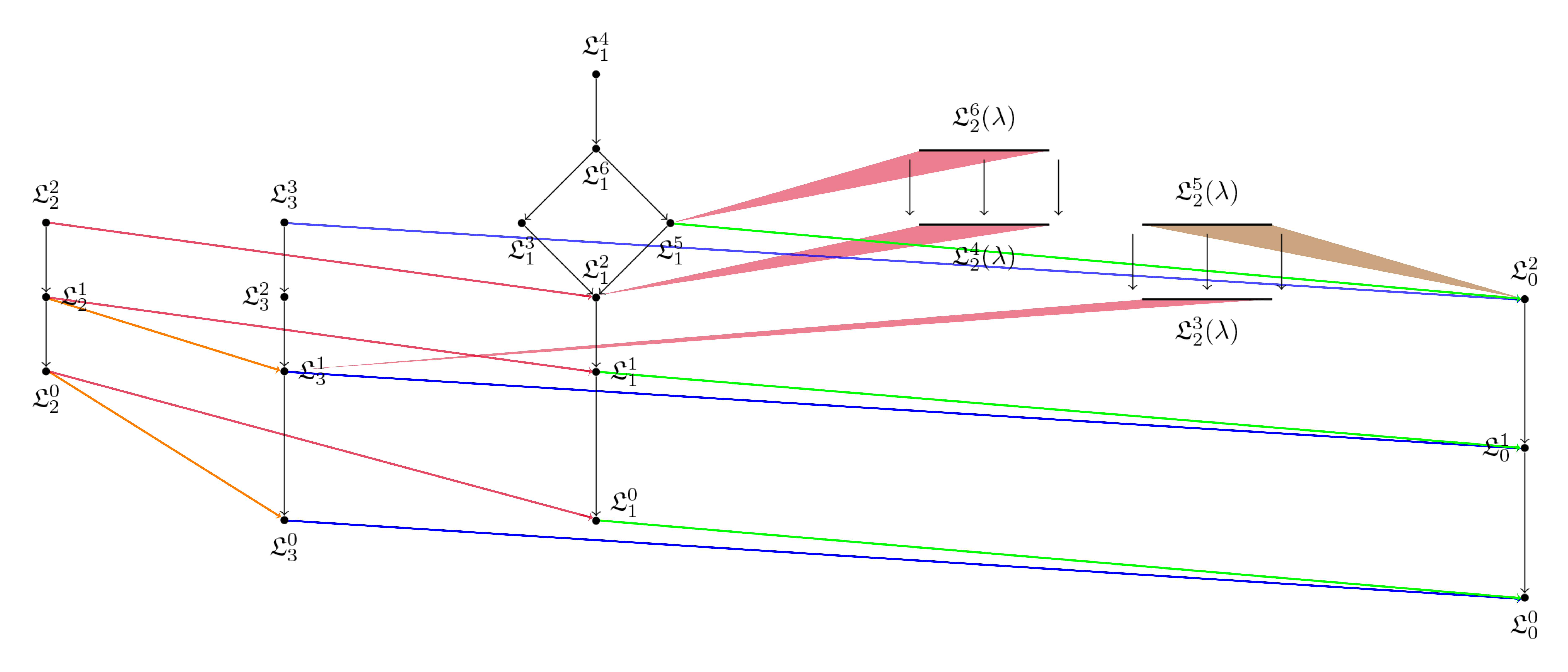}
\caption{Degenerations of hom-Lie algebras with underlying algebra isomorphic to $\mathfrak{r}_{3}(\mathbb{C})$, $\mathfrak{r}_{3,1}(\mathbb{C})$, $\mathfrak{h}_{3}(\mathbb{C})$
or $\mathfrak{a}_{3}$, and nilpotent twisting map.}
\end{figure}

\begin{figure}[h!]
\includegraphics[width=\paperwidth]{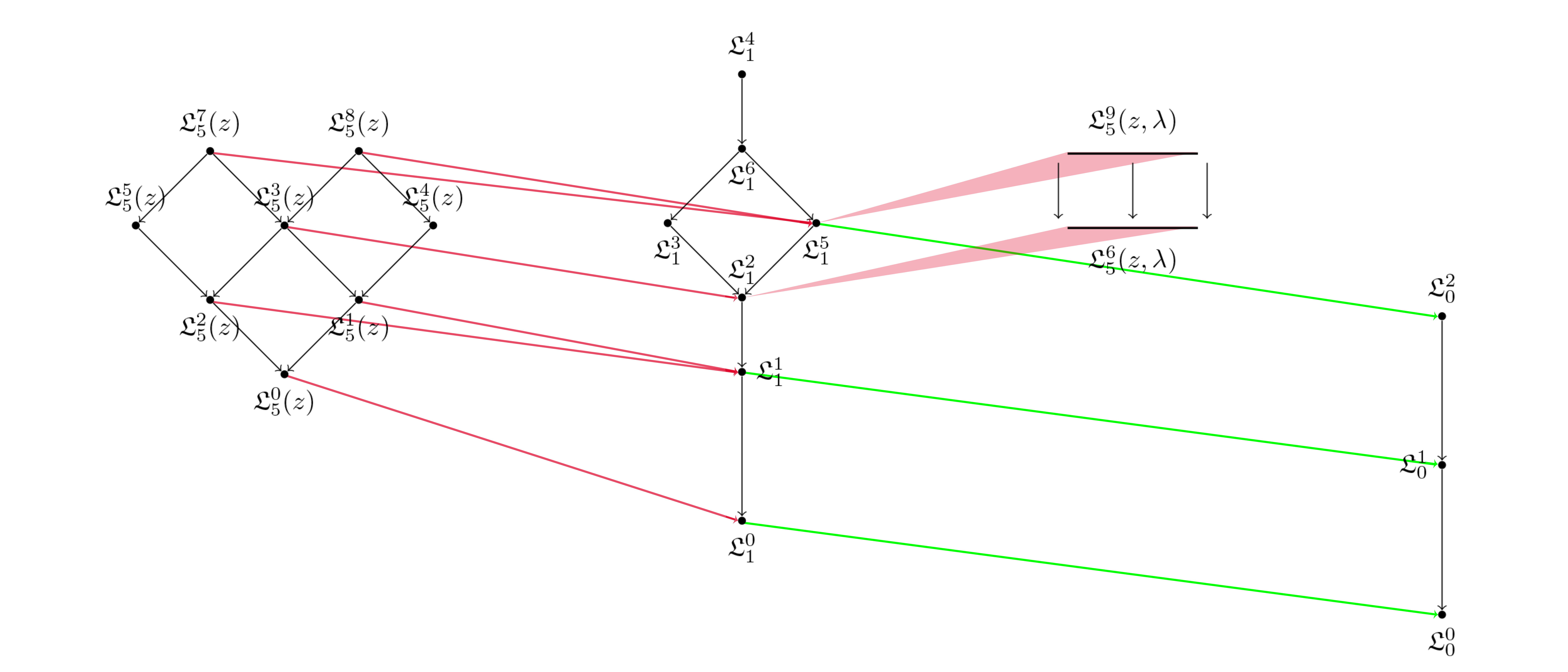}
\caption{Degenerations of hom-Lie algebras with underlying algebra isomorphic to $\mathfrak{r}_{3,z}(\mathbb{C})$, $\mathfrak{h}_{3}(\mathbb{C})$ or $\mathfrak{a}_{3}$, and nilpotent twisting map.}
\end{figure}

\begin{figure}[h!]
\includegraphics[width=\paperwidth]{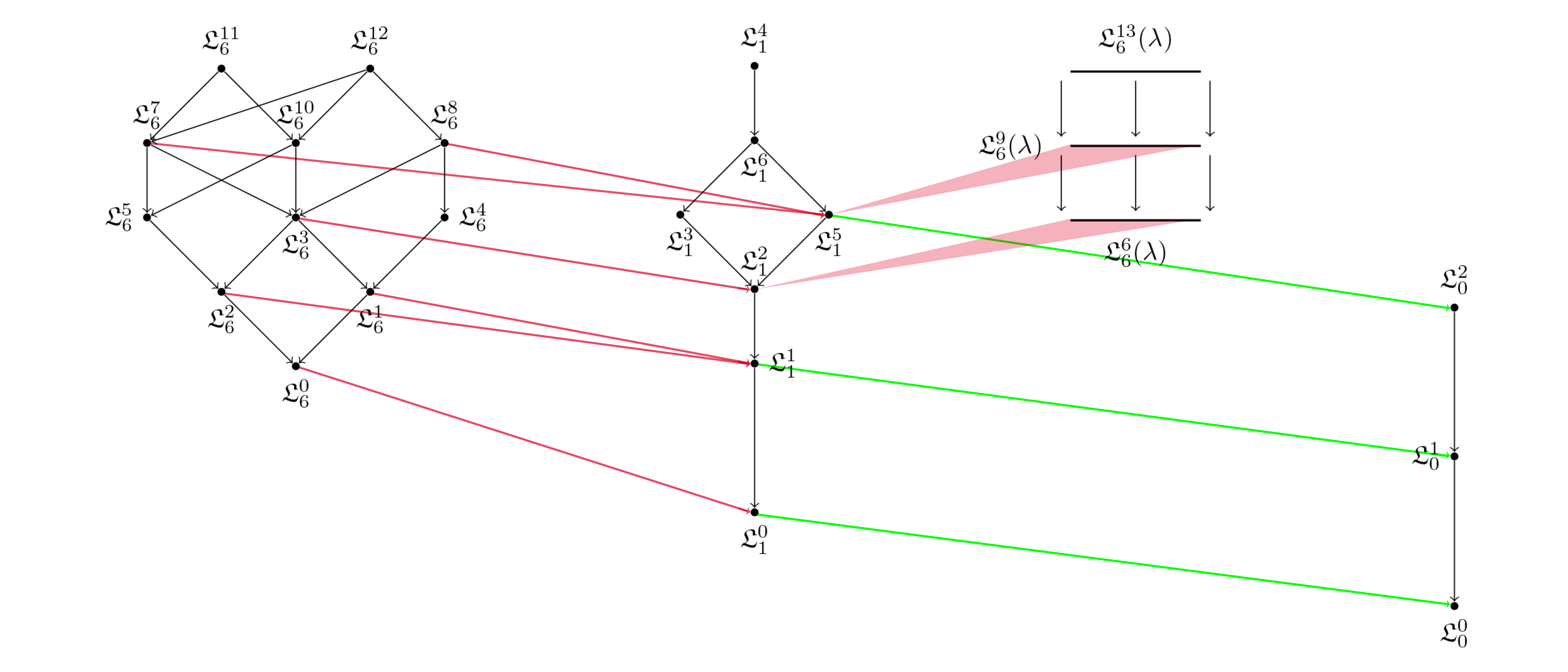}
\caption{Degenerations of hom-Lie algebras with underlying algebra isomorphic to $\mathfrak{r}_{2}(\mathbb{C}) \times \mathbb{C}$, $\mathfrak{h}_{3}(\mathbb{C})$ or $\mathfrak{a}_{3}$, and nilpotent twisting map.}
\end{figure}

\begin{figure}[h!]
\includegraphics[width=\paperwidth]{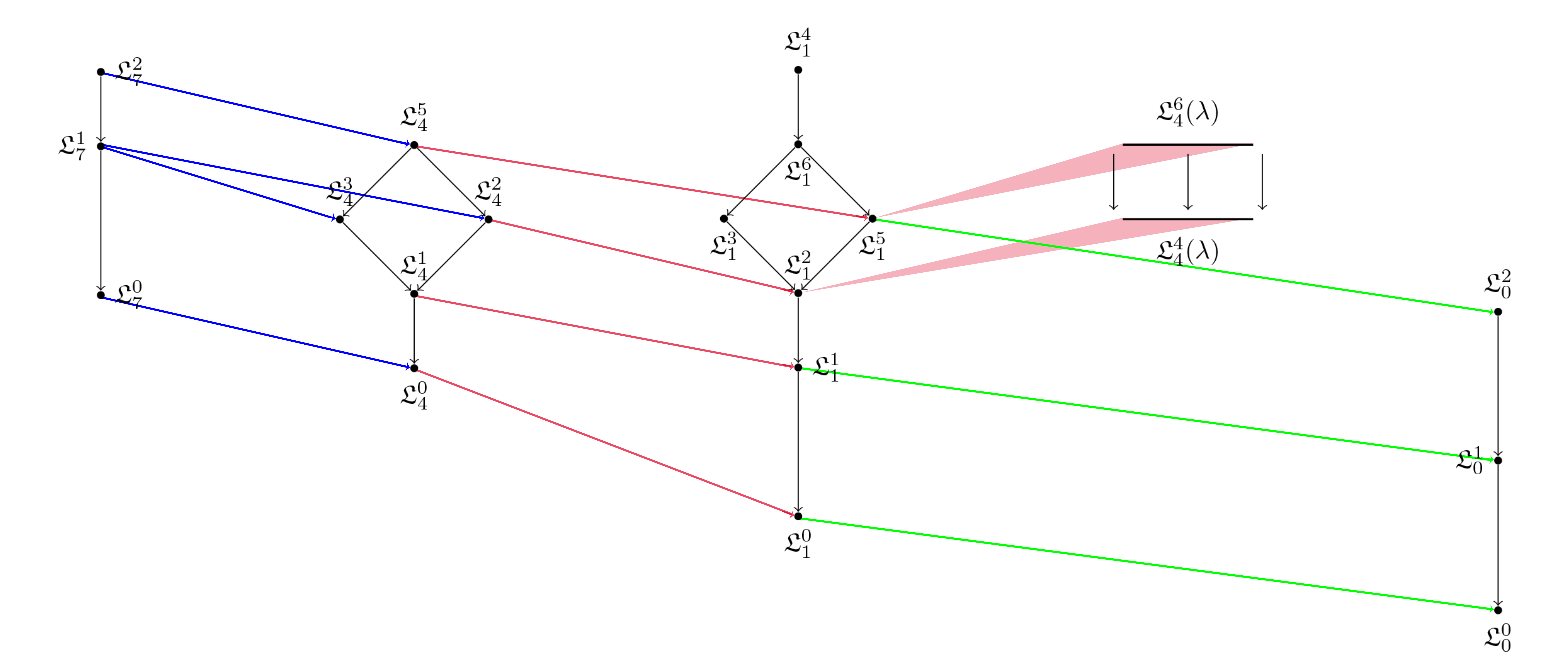}
\caption{Degenerations of hom-Lie algebras with underlying algebra isomorphic to a unimodular Lie algebra and nilpotent twisting map.}
\end{figure}

\end{center}
\end{landscape}

{

}



\section*{Supplementary material}

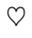

# Companion file

## 1 Invariants

### 1.1 Invariants

Let $V_2 := \{\mu : \mathbb{C}^n \times \mathbb{C}^n \to \mathbb{C}^n : \mu \text{ is a bilinear map}\}$. Given $\alpha_0, \ldots, \alpha_{26} \in \mathbb{C}$, let $\varphi_{\alpha_0,\ldots,\alpha_{26}} : V_2 \times M(n,\mathbb{C}) \to V_2$ to be defined by

$$\varphi_{\alpha_0,\ldots,\alpha_{26}}(\mu, A) = \begin{array}{l} \alpha_0 \mu(\cdot,\cdot)+ \\ \alpha_1 A\mu(\cdot,\cdot) + \alpha_2 \mu(A\cdot,\cdot) + \alpha_3 \mu(\cdot,A\cdot)+ \\ \alpha_4 A\mu(A\cdot,\cdot) + \alpha_5 A\mu(\cdot,A\cdot) + \alpha_6 \mu(A\cdot,A\cdot)+ \\ \alpha_7 A\mu(A\cdot,A\cdot)+ \\ \alpha_8 A^2\mu(\cdot,\cdot) + \alpha_9 \mu(A^2\cdot,\cdot) + \alpha_{10} \mu(\cdot,A^2\cdot)+ \\ \alpha_{11} A^2\mu(A\cdot,\cdot) + \alpha_{12} A^2\mu(\cdot,A\cdot) + \alpha_{13} \mu(A^2\cdot,A\cdot)+ \\ \alpha_{14} A\mu(A^2\cdot,\cdot) + \alpha_{15} A\mu(\cdot,A^2\cdot) + \alpha_{16} \mu(A\cdot,A^2\cdot)+ \\ \alpha_{17} A^2\mu(A^2\cdot,\cdot) + \alpha_{18} A^2\mu(\cdot,A^2\cdot) + \alpha_{19} \mu(A^2\cdot,A^2\cdot)+ \\ \alpha_{20} A^2\mu(A\cdot,A\cdot) + \alpha_{21} A\mu(A^2\cdot,A\cdot) + \alpha_{22} A\mu(A\cdot,A^2\cdot)+ \\ \alpha_{23} A\mu(A^2\cdot,A^2\cdot) + \alpha_{24} A^2\mu(A\cdot,A^2\cdot) + \alpha_{25} A^2\mu(A^2\cdot,A\cdot)+ \\ \alpha_{26} A^2\mu(A^2\cdot,A^2\cdot). \end{array}$$

and let $\vartheta_{\alpha_0,\ldots,\alpha_{26}}(\mu, A) = (\varphi_{\alpha_0,\ldots,\alpha_{26}}(\mu, A), A)$. Set

$$\begin{array}{rcl} \psi_{\alpha,\beta} & := & \varphi_{1,\alpha,\beta,\beta,0,\ldots,0} \\ \phi_\beta & := & \varphi_{0,1,\beta,\beta,0,\ldots,0} \\ \rho & := & \varphi_{0,0,1,1,0,\ldots,0} \end{array}$$

### 1.2 Some invariant vector spaces

Given $\theta_1, \ldots, \theta_6$ and $\beta_1, \ldots, \beta_9$ complex numbers, let $W_{\theta_1,\ldots,\theta_6}$ denote the vector space

$$W_{\alpha_1,\ldots,\alpha_6} := \left\{ (D_1, D_2, D_3) \in M(n,\mathbb{C}) \times M(n,\mathbb{C}) \times M(n,\mathbb{C}) : \begin{array}{l} \theta_1 D_1 + \theta_2 D_2 + \theta_3 D_3 = 0 \\ \theta_4 D_1 + \theta_5 D_2 + \theta_6 D_3 = 0 \end{array} \right\}$$

and we consider the function

$$T_{\theta_1,\ldots,\theta_6;\beta_1,\ldots,\beta_9} : V \times M(n,\mathbb{C}) \to \mathrm{End}(W_{\theta_1,\ldots,\theta_6}; M(n,\mathbb{C}) \times M(n,\mathbb{C}) \times M(n,\mathbb{C}) \times V),$$

which is given by

$$T_{\theta_1,\ldots,\theta_6;\beta_1,\ldots,\beta_9}(\mu, A) : \begin{array}{rcl} W_{\theta_1,\ldots,\theta_6} & \to & M(n,\mathbb{C}) \times M(n,\mathbb{C}) \times M(n,\mathbb{C}) \times V \\ (D_1, D_2, D_3) & \mapsto & (M_1, M_2, M_3, \nu) \end{array}$$

where

$$\begin{cases} M_1 & = & \beta_1 AD_1 + \beta_2 D_1 A + \beta_3 AD_1 A \\ M_2 & = & \beta_4 AD_2 + \beta_5 D_2 A + \beta_6 AD_2 A \\ M_3 & = & \beta_7 AD_3 + \beta_8 D_3 A + \beta_9 AD_3 A \\ \nu & = & D_1 \mu(\cdot,\cdot) + \mu(D_2 \cdot,\cdot) + \mu(\cdot, D_3 \cdot) \end{cases}$$

Note that if a hom-Lie algebra $(\mu, A)$ degenerates to a hom-Lie algebra $(\nu, B)$, then

$$\dim(\mathrm{Ker}(T_{\theta_1,\ldots,\theta_6;\beta_1,\ldots,\beta_9}(\mu, A))) \leq \dim(\mathrm{Ker}(T_{\theta_1,\ldots,\theta_6;\beta_1,\ldots,\beta_9}(\nu, B))). \tag{1}$$

Furthermore, when $\theta_1 = 1, \theta_2 = 1, \theta_3 = 0, \theta_4 = 1, \theta_5 = 0, \theta_6 = 1$ and $\beta_1 = 1, \beta_2 = -1, \beta_3 = 0, \beta_4 = 1, \beta_5 = -1, \beta_6 = 0, \beta_7 = 1, \beta_8 = -1, \beta_9 = 0$ we have $\mathrm{Ker}(T_{1,1,0,1,0,1;1,-1,0,1,-1,0,1,-1,0}(\mu, A))$ is the

Lie algebra of derivations of the hom-Lie algebra $(\mu, A)$ and if $(\mu, A)$ degenerates properly to a hom-Lie algebra $(\nu, B)$ then

$$\dim \operatorname{Ker} T_{1,1,0,1,0,1;1,-1,0,1,-1,0,1,-1,0}(\mu, A) < \dim \operatorname{Ker} T_{1,1,0,1,0,1;1,-1,0,1,-1,0,1,-1,0}(\nu, B).$$

**Theorem 1.1** *All the degenerations for $3$-dimensional complex Lie algebra are:*

1. $\mathbb{C}^3 \xrightarrow{deg} \mathbb{C}^3$.
2. $\mathfrak{n}_3(\mathbb{C}) \xrightarrow{deg} \mathfrak{n}_3(\mathbb{C}), \mathbb{C}^3$.
3. $\mathfrak{r}_3(\mathbb{C}) \xrightarrow{deg} \mathfrak{r}_3(\mathbb{C}), \mathfrak{r}_{3,1}(\mathbb{C}), \mathfrak{n}_3(\mathbb{C}), \mathbb{C}^3$.
4. $\mathfrak{r}_{3,1}(\mathbb{C}) \xrightarrow{deg} \mathfrak{r}_{3,1}(\mathbb{C}), \mathbb{C}^3$.
5. $\mathfrak{r}_{3,-1}(\mathbb{C}) \xrightarrow{deg} \mathfrak{r}_{3,-1}(\mathbb{C}), \mathfrak{n}_3(\mathbb{C}), \mathbb{C}^3$.
6. *If $z \neq \pm 1$,* $\mathfrak{r}_{3,z}(\mathbb{C}) \xrightarrow{deg} \mathfrak{r}_{3,z}(\mathbb{C}), \mathfrak{n}_3(\mathbb{C}), \mathbb{C}^3$,
7. $\mathfrak{r}_2 \times \mathbb{C} \xrightarrow{deg} \mathfrak{r}_2 \times \mathbb{C}, \mathfrak{n}_3(\mathbb{C}), \mathbb{C}^3$.
8. $\mathfrak{sl}_2(\mathbb{C}) \xrightarrow{deg} \mathfrak{sl}_2(\mathbb{C}), \mathfrak{r}_{3,-1}(\mathbb{C}), \mathfrak{n}_3(\mathbb{C}), \mathbb{C}^3$.

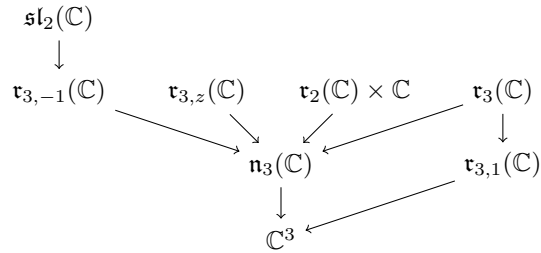

# 2 $\mathfrak{L}_0$: 3-dimensional Abelian Lie algebra

$\mathrm{Aut}(\mathfrak{L}_0) = GL(3,\mathbb{C})$.
$\mathrm{Der}(\mathfrak{L}_0) = M(3,\mathbb{C})$

## 2.1 $\mathfrak{L}_0^0$

Let $A$ be the Zero map and let us denote by $\mathfrak{L}_0^0$ the hom-Lie algebra $(\mathfrak{L}_0, A)$. We have:
$\mathrm{Aut}(\mathfrak{L}_0^0) = GL(3,\mathbb{C})$.
$\mathrm{Der}(\mathfrak{L}_0^0) = M(3,\mathbb{C})$

## 2.2 $\mathfrak{L}_0^1$

Let $A$ be the endomorphism of $\mathbb{C}^3$ whose matrix representation with respect to the ordered basis $\{e_1, e_2, e_3\}$ is

$$\begin{pmatrix} 0 & 0 & 0 \\ 0 & 0 & 1 \\ 0 & 0 & 0 \end{pmatrix}$$

and let us denote by $\mathfrak{L}_0^1$ the hom-Lie algebra $(\mathfrak{L}_0, A)$. We have:

$$\mathrm{Aut}(\mathfrak{L}_0^1) = \left\{ \begin{bmatrix} a & 0 & z \\ y & b & x \\ 0 & 0 & b \end{bmatrix} \in M(3,\mathbb{C}) : \begin{array}{l} a, b \in \mathbb{C}^*, \\ x, y, z \in \mathbb{C} \end{array} \right\}.$$

$$\mathrm{Der}(\mathfrak{L}_0^1) = \left\{ \begin{bmatrix} t_1 & 0 & t_5 \\ t_4 & t_2 & t_3 \\ 0 & 0 & t_2 \end{bmatrix} \in M(3,\mathbb{C}) : t_1, t_2, t_3, t_4, t_5 \in \mathbb{C} \right\} \cong (\mathbb{C} \ltimes \mathfrak{n}_3(\mathbb{C})) \times \mathbb{C}.$$

## 2.3 $\mathfrak{L}_0^2$

Let $A$ be the endomorphism of $\mathbb{C}^3$ whose matrix representation with respect to the ordered basis $\{e_1, e_2, e_3\}$ is

$$\begin{pmatrix} 0 & 1 & 0 \\ 0 & 0 & 1 \\ 0 & 0 & 0 \end{pmatrix}$$

and let us denote by $\mathfrak{L}_0^2$ the hom-Lie algebra $(\mathfrak{L}_0, A)$. We have:

$$\mathrm{Aut}(\mathfrak{L}_0^1) = \left\{ \begin{bmatrix} a & x & y \\ 0 & a & x \\ 0 & 0 & a \end{bmatrix} \in M(3,\mathbb{C}) : \begin{array}{l} a \in \mathbb{C}^*, \\ x, y \in \mathbb{C} \end{array} \right\}.$$

$$\mathrm{Der}(\mathfrak{L}_0^1) = \left\{ \begin{bmatrix} t_1 & t_2 & t_3 \\ 0 & t_1 & t_2 \\ 0 & 0 & t_1 \end{bmatrix} \in M(3,\mathbb{C}) : t_1, t_2, t_3 \in \mathbb{C} \right\} \cong \mathfrak{L}_0.$$

## 2.4 Degenerations between hom-Lie algebras ($\mathfrak{L}_0^i$)

If $\mathfrak{L}_i^j \xrightarrow{\deg} \mathfrak{L}_h^k$ then $\mathrm{Der}(\mathfrak{L}_i^j) \leq \mathrm{Der}(\mathfrak{L}_h^k)$. Therefore, we can organize the hom-Lie algebras in the following way:

| Dim(Der) | hom-Lie algebra |
|---|---|
| 3 | $\mathfrak{L}_0^2$ |
| 5 | $\mathfrak{L}_0^1$ |
| 9 | $\mathfrak{L}_0^0$ |

### 2.4.1 Degenerations

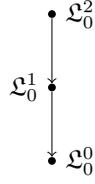

1. $\mathfrak{L}_0^2 \xrightarrow{\deg} \mathfrak{L}_0^1$
   
   In fact, set
   $$g(t) = \begin{bmatrix} e^{-t} & 0 & 0 \\ 0 & 1 & 0 \\ 0 & 0 & 1 \end{bmatrix}, \text{ with } t \in \mathbb{R}.$$
   
   We have $g(t) \in \operatorname{Aut}(\mathfrak{L}_0)$ and $g(t) \cdot \mathfrak{L}_0^2$ is the hom-Lie algebra $(\mathfrak{L}_0, A(t))$ with
   $$A(t) = \begin{bmatrix} 0 & e^{-t} & 0 \\ 0 & 0 & 1 \\ 0 & 0 & 0 \end{bmatrix}.$$
   
   It is easy to check that $(\mathfrak{L}_0, A(t)) \to \mathfrak{L}_0^1$ as $t$ tends to infinity.

2. $\mathfrak{L}_0^1 \xrightarrow{\deg} \mathfrak{L}_0^0$
   
   In fact, set
   $$g(t) = \begin{bmatrix} 1 & 0 & 0 \\ 0 & e^{-t} & 0 \\ 0 & 0 & 1 \end{bmatrix}, \text{ with } t \in \mathbb{R}.$$
   
   We have $g(t) \in \operatorname{Aut}(\mathfrak{L}_0)$ and $g(t) \cdot \mathfrak{L}_0^1$ is the hom-Lie algebra $(\mathfrak{L}_0, A(t))$ with
   $$A(t) = \begin{bmatrix} 0 & 0 & 0 \\ 0 & 0 & e^{-t} \\ 0 & 0 & 0 \end{bmatrix}.$$
   
   It is easy to check that $(\mathfrak{L}_0, A(t)) \to \mathfrak{L}_0^0$ as $t$ tends to infinity.

# 3    $\mathfrak{L}_1$: 3-dimensional Heisenberg Lie algebra, $\mathfrak{n}_3(\mathbb{C})$

$$\mathfrak{L}_1 := \{[e_1, e_2] = e_3$$

$$\operatorname{Aut}(\mathfrak{L}_1) = \left\{ \begin{bmatrix} a & b & 0 \\ c & d & 0 \\ x & y & ad-cb \end{bmatrix} \in M(3, \mathbb{C}) : \begin{array}{c} a,b,c,d,x,y \in \mathbb{C} \\ ad - bc \neq 0 \end{array} \right\}.$$

$$\operatorname{Der}(\mathfrak{L}_1) = \left\{ \begin{bmatrix} t_1 & t_3 & 0 \\ t_4 & t_2 & 0 \\ t_5 & t_6 & t_1 + t_2 \end{bmatrix} \in M(3, \mathbb{C}) : t_1, t_2, t_3, t_4, t_5, t_6 \in \mathbb{C} \right\} \cong \mathfrak{sl}(2, \mathbb{C}) \ltimes \mathfrak{r}_{3,1}(\mathbb{C}).$$

## 3.1    $\mathfrak{L}_1^0$

Let $A$ be the Zero map and let us denote by $\mathfrak{L}_1^0$ the hom-Lie algebra $(\mathfrak{L}_1, A)$. We have:
$\operatorname{Aut}(\mathfrak{L}_1^0) = \operatorname{Aut}(\mathfrak{L}_1)$.
$\operatorname{Der}(\mathfrak{L}_1^0) = \operatorname{Der}(\mathfrak{L}_1)$.

## 3.2    $\mathfrak{L}_1^1$

Let $A$ be the endomorphism of $\mathbb{C}^3$ whose matrix representation with respect to the ordered basis $\{e_1, e_2, e_3\}$ is

$$\begin{pmatrix} 0 & 0 & 0 \\ 0 & 0 & 0 \\ 0 & 1 & 0 \end{pmatrix}$$

and let us denote by $\mathfrak{L}_1^1$ the hom-Lie algebra $(\mathfrak{L}_1, A)$. We have:

$$\operatorname{Aut}(\mathfrak{L}_1^1) = \left\{ \begin{bmatrix} 1 & b & 0 \\ 0 & a & 0 \\ x & y & a \end{bmatrix} \in M(3, \mathbb{C}) : \begin{array}{c} a \in \mathbb{C}^* \\ b, x, y \in \mathbb{C} \end{array} \right\}.$$

$$\operatorname{Der}(\mathfrak{L}_1^1) = \left\{ \begin{bmatrix} 0 & t_2 & 0 \\ 0 & t_1 & 0 \\ t_3 & t_4 & t_1 \end{bmatrix} \in M(3, \mathbb{C}) : t_1, t_2, t_3, t_4 \in \mathbb{C} \right\} \cong \mathbb{C} \ltimes \mathfrak{n}_3(\mathbb{C}).$$

### 3.2.1    $\psi_{\alpha,\beta}(\mathfrak{L}_1^1)$

In this case, we have $\psi_{\alpha,\beta}(\mathfrak{L}_1^1) = \mathfrak{t}_{\alpha,\beta}$:

$$\mathfrak{t}_{\alpha,\beta} = \{[e_1, e_2] = e_3,$$

which is the Heisenberg Lie algebra.

### 3.2.2    $\phi_\beta(\mathfrak{L}_1^1)$

In this case, we have $\phi_\beta(\mathfrak{L}_1^1) = \widetilde{\mathfrak{t}}_\beta$ is the 3-dimensional abelian Lie algebra.

### 3.2.3    $\rho(\mathfrak{L}_1^1)$

In this case, we have $\rho(\mathfrak{L}_1^1) = \widehat{\mathfrak{t}}$ is the 3-dimensional abelian Lie algebra.

### 3.2.4    Identities

- $A[\cdot, \cdot] = 0$
- $[A\cdot, \cdot] = 0$

In particular, $\mathfrak{L}_1^1$ is a multiplicative hom-Lie algebra.

## 3.3 $\mathfrak{L}_1^2$

Let $A$ be the endomorphism of $\mathbb{C}^3$ whose matrix representation with respect to the ordered basis $\{e_1, e_2, e_3\}$ is

$$\begin{bmatrix} 0 & 1 & 0 \\ 0 & 0 & 0 \\ 0 & 0 & 0 \end{bmatrix}$$

and let us denote by $\mathfrak{L}_1^2$ the hom-Lie algebra $(\mathfrak{L}_1, A)$. We have:

$$\mathrm{Aut}(\mathfrak{L}_1^2) = \left\{ \begin{bmatrix} a & b & 0 \\ 0 & a & 0 \\ 0 & x & a^2 \end{bmatrix} \in M(3, \mathbb{C}) : \begin{array}{l} a \in \mathbb{C}^* \\ b, x \in \mathbb{C} \end{array} \right\}.$$

$$\mathrm{Der}(\mathfrak{L}_1^2) = \left\{ \begin{bmatrix} t_1 & t_2 & 0 \\ 0 & t_1 & 0 \\ 0 & t_3 & 2t_1 \end{bmatrix} \in M(3, \mathbb{C}) : t_1, t_2, t_3 \in \mathbb{C} \right\} \cong \mathfrak{r}_2 \times \mathbb{C}.$$

### 3.3.1 $\psi_{\alpha,\beta}(\mathfrak{L}_1^2)$

In this case, we have $\psi_{\alpha,\beta}(\mathfrak{L}_1^2) = \mathfrak{t}_{\alpha,\beta}$:

$$\mathfrak{t}_{\alpha,\beta} = \{[e1, e2] = e3.$$

which is the Heisenberg Lie algebra.

### 3.3.2 $\phi_\beta(\mathfrak{L}_1^2)$

In this case, we have $\phi_\beta(\mathfrak{L}_1^2) = \widetilde{\mathfrak{t}}_\beta$ is the 3-dimensional abelian Lie algebra.

### 3.3.3 $\rho(\mathfrak{L}_1^2)$

In this case, we have $\rho(\mathfrak{L}_1^2) = \widehat{\mathfrak{t}}$ is the 3-dimensional abelian Lie algebra.

### 3.3.4 Identities

- $A[\cdot, \cdot] = 0$
- $[A\cdot, \cdot] + [\cdot, A\cdot] = 0$
  - $[A\cdot, A\cdot] = 0$

In particular, $\mathfrak{L}_1^2$ is a multiplicative hom-Lie algebra.

## 3.4 $\mathfrak{L}_1^3$

Let $A$ be the endomorphism of $\mathbb{C}^3$ whose matrix representation with respect to the ordered basis $\{e_1, e_2, e_3\}$ is

$$\begin{bmatrix} 0 & 0 & 0 \\ 0 & 0 & 1 \\ 0 & 0 & 0 \end{bmatrix}$$

and let us denote by $\mathfrak{L}_1^3$ the hom-Lie algebra $(\mathfrak{L}_1, A)$. We have:

$$\mathrm{Aut}(\mathfrak{L}_1^3) = \left\{ \begin{bmatrix} 1 & 0 & 0 \\ b & a & 0 \\ 0 & 0 & a \end{bmatrix} \in M(3, \mathbb{C}) : \begin{array}{l} a \in \mathbb{C}^* \\ b \in \mathbb{C} \end{array} \right\}.$$

$$\mathrm{Der}(\mathfrak{L}_1^3) = \left\{ \begin{bmatrix} 0 & 0 & 0 \\ t_2 & t_1 & 0 \\ 0 & 0 & t_1 \end{bmatrix} \in M(3, \mathbb{C}) : t_1, t_2 \in \mathbb{C} \right\} \cong \mathfrak{r}_2(\mathbb{C}).$$

### 3.4.1 $\psi_{\alpha,\beta}(\mathfrak{L}_1^3)$

In this case, we have $\psi_{\alpha,\beta}(\mathfrak{L}_1^3) = \mathfrak{t}_{\alpha,\beta}$:

$$\mathfrak{t}_{\alpha,\beta} = \{[e1, e2] = \alpha\, e2 + e3, [e1, e3] = \beta\, e3.$$

We note that $\mathfrak{t}_{\alpha,\beta}$ is an almost abelian Lie algebra where $\mathrm{span}_\mathbb{C}\{e_2, e_3\}$ is a codimension 1 abelian ideal of $\mathfrak{t}_{\alpha,\beta}$; we can identify $\mathfrak{t}_{\alpha,\beta}$ the matrix

$$B = \begin{bmatrix} \alpha & 0 \\ 1 & \beta \end{bmatrix}.$$

The isomorphism class of $\mathfrak{t}_{\alpha,\beta}$ is determined by the eigenvalues of $B$:

1. $B$ has repeated eigenvalues if and only if $\alpha = \beta$
   (a) 0 is an repeated eigenvalue of $B$ if and only if $\alpha = \beta = 0$
      i. $\mathfrak{t}_{\alpha,\beta}$ is the 3-dimensional abelian Lie algebra if and only if $B = 0$ if and only if $1 = 0$
      ii. $\mathfrak{t}_{\alpha,\beta}$ is isomorphic to $\mathfrak{n}_3(\mathbb{C})$ if and only if $B \neq 0$ if and only if $\alpha = \beta = 0$.
   (b) $B$ has a non-zero repeated eigenvalue if and only if $\alpha = \beta \neq 0$
      i. $\mathfrak{t}_{\alpha,\beta}$ is $\mathfrak{r}_{3,1}$ if and only if $B = \lambda\,\mathrm{Id}$ for some $\lambda \neq 0$ if and only if $1 = 0$
      ii. $\mathfrak{t}_{\alpha,\beta}$ is isomorphic to $\mathfrak{r}_3$ if and only if $B \neq \lambda\,\mathrm{Id}$ for all $\lambda \in \mathbb{C}$ if and only if $\alpha = \beta \neq 0$

2. $B$ has distinct eigenvalues eigenvalues if and only if $\alpha \neq \beta$
   (a) $\mathfrak{t}_{\alpha,\beta}$ is isomorphic to $\mathfrak{r}_2(\mathbb{C}) \times \mathbb{C}$ if and only if 0 is an eigenvalue of $B$ with algebraic multiplicity 1 if and only if $\alpha = 0$ or $\beta = 0$, and $\alpha \neq \beta$.
   (b) $\mathfrak{t}_{\alpha,\beta}$ is isomorphic to $\mathfrak{r}_{3,-1}$ if and only if $B$ has distinct eigenvalues eigenvalues and $\mathrm{Trace}(B) = 0$ if and only if $\beta = -\alpha \neq 0$.
   (c) $\mathfrak{t}_{\alpha,\beta}$ is isomorphic to $\mathfrak{r}_{3,z}$ for some $z \in \mathbb{C}$ such that $z^2 \neq 1$ and $z \neq 0$ if and only if $\alpha \neq \beta$, $\alpha, \beta \neq 0$ and $\beta \neq -\alpha$.

### 3.4.2 $\phi_\beta(\mathfrak{L}_1^3)$

In this case, we have $\phi_\beta(\mathfrak{L}_1^3) = \widetilde{\mathfrak{t}}_\beta$:

$$\widetilde{\mathfrak{t}}_\beta = \{[e1, e2] = e2, [e1, e3] = \beta\, e3.$$

We note that $\widetilde{\mathfrak{t}}_\beta$ is an almost abelian Lie algebra where $\mathrm{span}_\mathbb{C}\{e_2, e_3\}$ is a codimension 1 abelian ideal of $\widetilde{\mathfrak{t}}_\beta$; we can identify $\widetilde{\mathfrak{t}}_\beta$ the matrix

$$\widetilde{B} = \begin{bmatrix} 1 & 0 \\ 0 & \beta \end{bmatrix}.$$

The isomorphism class of $\widetilde{\mathfrak{t}}_\beta$ is determined by the eigenvalues of $\widetilde{B}$:

1. $\widetilde{B}$ has repeated eigenvalues if and only if $\beta = 1$
   (a) 0 is a repeated eigenvalue of $\widetilde{B}$ if and only if $1 = 0$
   (b) $B$ has a non-zero repeated eigenvalue if and only if $\beta = 1$
      i. $\widetilde{\mathfrak{t}}_\beta$ is $\mathfrak{r}_{3,1}$ if and only if $\widetilde{B} = \lambda \operatorname{Id}$ for some $\lambda \neq 0$ if and only if $\beta = 1$
      ii. $\widetilde{\mathfrak{t}}_\beta$ is isomorphic to $\mathfrak{r}_3$ if and only if $\widetilde{B} \neq \lambda \operatorname{Id}$ for all $\lambda \in \mathbb{C}$ if and only if $1 = 0$

2. $\widetilde{B}$ has distinct eigenvalues eigenvalues if and only if $\beta \neq 1$
   (a) $\widetilde{\mathfrak{t}}_\beta$ is isomorphic to $\mathfrak{r}_2(\mathbb{C}) \times \mathbb{C}$ if and only if 0 is an eigenvalue of $B$ with algebraic multiplicity 1 if and only if $\beta = 0$
   (b) $\widetilde{\mathfrak{t}}_\beta$ is isomorphic to $\mathfrak{r}_{3,-1}$ if and only if $\operatorname{Trace}(\widetilde{B}) = 0$ if and only if $\beta = -1$.
   (c) $\widetilde{\mathfrak{t}}_\beta$ is isomorphic to $\mathfrak{r}_{3,z}$ for some $z$ such that $z^2 \neq 1$ and $z \neq 0$ if and only if $\beta \neq 1$, $\beta \neq 0$ and $\beta \neq -1$.

### 3.4.3 $\rho(\mathfrak{L}_1^3)$

In this case, we have $\rho(\mathfrak{L}_1^3) = \widehat{\mathfrak{t}}$:
$$\widehat{\mathfrak{t}} = \{[e1, e3] = e3.$$
which is isomorphic to $\mathfrak{r}_2 \times \mathbb{C}$

### 3.4.4 Identities

- $[A\cdot, A\cdot] = 0$

The hom-Lie algebra $\mathfrak{L}_1^3$ is a non-multiplicative hom-Lie algebra.

## 3.5 $\mathfrak{L}_1^4$

Let $A$ be the endomorphism of $\mathbb{C}^3$ whose matrix representation with respect to the ordered basis $\{e_1, e_2, e_3\}$ is
$$\begin{bmatrix} 0 & 1 & 0 \\ 0 & 0 & 1 \\ 0 & 0 & 0 \end{bmatrix}$$
and let us denote by $\mathfrak{L}_1^4$ the hom-Lie algebra $(\mathfrak{L}_1, A)$. We have:
$$\operatorname{Aut}(\mathfrak{L}_1^4) = \left\{ \begin{bmatrix} 1 & 0 & 0 \\ 0 & 1 & 0 \\ 0 & 0 & 1 \end{bmatrix} \right\}.$$
$$\operatorname{Der}(\mathfrak{L}_1^4) = \left\{ \begin{bmatrix} 0 & 0 & 0 \\ 0 & 0 & 0 \\ 0 & 0 & 0 \end{bmatrix} \right\}.$$

### 3.5.1 $\psi_{\alpha,\beta}(\mathfrak{L}_1^4)$

In this case, we have $\psi_{\alpha,\beta}(\mathfrak{L}_1^4) = \mathfrak{t}_{\alpha,\beta}$:
$$\mathfrak{t}_{\alpha,\beta} = \{[e1, e2] = \alpha\, e2 + e3, [e1, e3] = \beta\, e3.$$

We note that $\mathfrak{t}_{\alpha,\beta}$ is an almost abelian Lie algebra where $\operatorname{span}_\mathbb{C}\{e_2, e_3\}$ is a codimension 1 abelian ideal of $\mathfrak{t}_{\alpha,\beta}$; we can identify $\mathfrak{t}_{\alpha,\beta}$ the matrix
$$B = \begin{bmatrix} \alpha & 0 \\ 1 & \beta \end{bmatrix}.$$

The isomorphism class of $\mathfrak{t}_{\alpha,\beta}$ is determined by the eigenvalues of $B$:

1. $B$ has repeated eigenvalues if and only if $\alpha = \beta$

    (a) 0 is an repeated eigenvalue of $B$ if and only if $\alpha = \beta = 0$

       i. $\mathfrak{t}_{\alpha,\beta}$ is the 3-dimensional abelian Lie algebra if and only if $B = 0$ if and only if $1 = 0$
       ii. $\mathfrak{t}_{\alpha,\beta}$ is isomorphic to $\mathfrak{n}_3(\mathbb{C})$ if and only if $B \neq 0$ if and only if $\alpha = \beta = 0$

    (b) $B$ has a non-zero repeated eigenvalue if and only if $\alpha = \beta \neq 0$

       i. $\mathfrak{t}_{\alpha,\beta}$ is $\mathfrak{r}_{3,1}$ if and only if $B = \lambda \operatorname{Id}$ for some $\lambda \neq 0$ if and only if $1 = 0$
       ii. $\mathfrak{t}_{\alpha,\beta}$ is isomorphic to $\mathfrak{r}_3$ if and only if $B \neq \lambda \operatorname{Id}$ for all $\lambda \in \mathbb{C}$ if and only if $\alpha = \beta \neq 0$

2. $B$ has distinct eigenvalues eigenvalues if and only if $\alpha \neq \beta$

    (a) $\widetilde{\mathfrak{t}}_\beta$ is isomorphic to $\mathfrak{r}_2(\mathbb{C}) \times \mathbb{C}$ if and only if 0 is an eigenvalue of $B$ with algebraic multiplicity 1 if and only if $\alpha = 0$ or $\beta = 0$, and $\alpha \neq \beta$.

    (b) $\mathfrak{t}_{\alpha,\beta}$ is isomorphic to $\mathfrak{r}_{3,-1}$ if and only if $B$ has distinct eigenvalues eigenvalues and $\operatorname{Trace}(B) = 0$ if and only if $\beta = -\alpha \neq 0$.

    (c) $\mathfrak{t}_{\alpha,\beta}$ is isomorphic to $\mathfrak{r}_{3,z}$ for some $z \in \mathbb{C}$ such that $z^2 \neq 1$ and $z \neq 0$ if and only if $\alpha \neq \beta$, $\alpha, \beta \neq 0$ and $\beta \neq -\alpha$.

### 3.5.2  $\phi_\beta(\mathfrak{L}_1^4)$

In this case, we have $\phi_\beta(\mathfrak{L}_1^4) = \widetilde{\mathfrak{t}}_\beta$:

$$\widetilde{\mathfrak{t}}_\beta = \{[e1, e2] = e2, [e1, e3] = \beta\, e3.$$

We note that $\widetilde{\mathfrak{t}}_\beta$ is an almost abelian Lie algebra where $\operatorname{span}_{\mathbb{C}}\{e_2, e_3\}$ is a codimension 1 abelian ideal of $\widetilde{\mathfrak{t}}_\beta$; we can identify $\widetilde{\mathfrak{t}}_\beta$ the matrix

$$\widetilde{B} = \begin{bmatrix} 1 & 0 \\ 0 & \beta \end{bmatrix}.$$

The isomorphism class of $\widetilde{\mathfrak{t}}_\beta$ is determined by the eigenvalues of $\widetilde{B}$:

1. $\widetilde{B}$ has repeated eigenvalues if and only if $\beta = 1$

    (a) 0 is a repeated eigenvalue of $\widetilde{B}$ if and only if $1 = 0$

    (b) $B$ has a non-zero repeated eigenvalue if and only if $\beta = 1$

       i. $\widetilde{\mathfrak{t}}_\beta$ is $\mathfrak{r}_{3,1}$ if and only if $\widetilde{B} = \lambda \operatorname{Id}$ for some $\lambda \neq 0$ if and only if $\beta = 1$
       ii. $\widetilde{\mathfrak{t}}_\beta$ is isomorphic to $\mathfrak{r}_3$ if and only if $\widetilde{B} \neq \lambda \operatorname{Id}$ for all $\lambda \in \mathbb{C}$ if and only if $1 = 0$

2. $\widetilde{B}$ has distinct eigenvalues eigenvalues if and only if $\beta \neq 1$

    (a) $\widetilde{\mathfrak{t}}_\beta$ is isomorphic to $\mathfrak{r}_2(\mathbb{C}) \times \mathbb{C}$ if and only if 0 is an eigenvalue of $B$ with algebraic multiplicity 1 if and only if $\beta = 0$

    (b) $\widetilde{\mathfrak{t}}_\beta$ is isomorphic to $\mathfrak{r}_{3,-1}$ if and only if $\operatorname{Trace}(\widetilde{B}) = 0$ if and only if $\beta = -1$.

    (c) $\widetilde{\mathfrak{t}}_\beta$ is isomorphic to $\mathfrak{r}_{3,z}$ for some $z$ such that $z^2 \neq 1$ and $z \neq 0$ if and only if $\beta \neq 1$, $\beta \neq 0$ and $\beta \neq -1$.

### 3.5.3  $\rho(\mathfrak{L}_1^4)$

In this case, we have $\rho(\mathfrak{L}_1^4) = \widehat{\mathfrak{t}}$:

$$\widehat{\mathfrak{t}} = \{[e1, e3] = e3.$$

which is isomorphic to $\mathfrak{r}_2 \times \mathbb{C}$

### 3.5.4 Identities

- $[A^2\cdot,\cdot]+[A\cdot,A\cdot]+[\cdot,A^2\cdot]$
- $[A^2\cdot,A^2\cdot]$

The hom-Lie algebra $\mathfrak{L}_1^4$ is a non-multiplicative hom-Lie algebra.

## 3.6 $\mathfrak{L}_1^5$

Let $A$ be the endomorphism of $\mathbb{C}^3$ whose matrix representation with respect to the ordered basis $\{e_1, e_2, e_3\}$ is

$$\begin{bmatrix} 0 & 0 & 0 \\ 1 & 0 & 0 \\ 0 & 1 & 0 \end{bmatrix}$$

and let us denote by $\mathfrak{L}_1^5$ the hom-Lie algebra $(\mathfrak{L}_1, A)$. We have:

$$\mathrm{Aut}(\mathfrak{L}_1^5) = \left\{ \begin{bmatrix} 1 & 0 & 0 \\ y & 1 & 0 \\ x & y & 1 \end{bmatrix} \in M(3,\mathbb{C}) : x,y \in \mathbb{C} \right\}.$$

$$\mathrm{Der}(\mathfrak{L}_1^5) = \left\{ \begin{bmatrix} 0 & 0 & 0 \\ t_2 & 0 & 0 \\ t_1 & t_2 & 0 \end{bmatrix} \in M(3,\mathbb{C}) : t_1, t_2 \in \mathbb{C} \right\} \cong \mathbb{C}^2.$$

### 3.6.1 Identities

- $A[\cdot,\cdot]=0$
- $[A\cdot,\cdot]+[\cdot,A\cdot]=0$
- $[A\cdot,A\cdot]=0$

The hom-Lie algebra $\mathfrak{L}_1^5$ is a multiplicative hom-Lie algebra.

### 3.6.2 $\psi_{\alpha,\beta}(\mathfrak{L}_1^5)$

In this case, we have $\psi_{\alpha,\beta}(\mathfrak{L}_1^5) = \mathfrak{t}_{\alpha,\beta}$:

$$\mathfrak{t}_{\alpha,\beta} = \{[e1,e2] = e3.$$

which is the Heisenberg Lie algebra.

### 3.6.3 $\phi_\beta(\mathfrak{L}_1^5)$

In this case, we have $\phi_\beta(\mathfrak{L}_1^5) = \widetilde{\mathfrak{t}}_\beta$ is the 3-dimensional abelian Lie algebra.

### 3.6.4 $\rho(\mathfrak{L}_1^5)$

In this case, we have $\rho(\mathfrak{L}_1^5) = \widehat{\mathfrak{t}}$ is the 3-dimensional abelian Lie algebra.

## 3.7 $\mathfrak{L}_1^6$

Let $A$ be the endomorphism of $\mathbb{C}^3$ whose matrix representation with respect to the ordered basis $\{e_1, e_2, e_3\}$ is

$$\begin{bmatrix} 0 & 0 & 0 \\ 0 & 0 & 1 \\ 1 & 0 & 0 \end{bmatrix}$$

and let us denote by $\mathfrak{L}_1^6$ the hom-Lie algebra $(\mathfrak{L}_1, A)$. We have:

$$\operatorname{Aut}(\mathfrak{L}) = \left\{ \begin{bmatrix} 1 & 0 & 0 \\ x & 1 & 0 \\ 0 & 0 & 1 \end{bmatrix} \in M(3, \mathbb{C}) : x \in \mathbb{C} \right\}.$$

$$\operatorname{Der}(\mathfrak{L}) = \left\{ \begin{bmatrix} 0 & 0 & 0 \\ t_1 & 0 & 0 \\ 0 & 0 & 0 \end{bmatrix} \in M(3, \mathbb{C}) : t_1 \in \mathbb{C} \right\} \cong \mathbb{C}.$$

### 3.7.1 $\psi_{\alpha,\beta}(\mathfrak{L}_1^6)$

In this case, we have $\psi_{\alpha,\beta}(\mathfrak{L}_1^6) = \mathfrak{t}_{\alpha,\beta}$:

$$\mathfrak{t}_{\alpha,\beta} = \{[e1, e2] = \alpha\, e2 + e3, [e1, e3] = \beta\, e3.$$

We note that $\mathfrak{t}_{\alpha,\beta}$ is an almost abelian Lie algebra where $\operatorname{span}_{\mathbb{C}}\{e_2, e_3\}$ is a codimension 1 abelian ideal of $\mathfrak{t}_{\alpha,\beta}$; we can identify $\mathfrak{t}_{\alpha,\beta}$ the matrix

$$B = \begin{bmatrix} \alpha & 0 \\ 1 & \beta \end{bmatrix}.$$

The isomorphism class of $\mathfrak{t}_{\alpha,\beta}$ is determined by the eigenvalues of $B$:

1. $B$ has repeated eigenvalues if and only if $\alpha = \beta$

    (a) 0 is an repeated eigenvalue of $B$ if and only if $\alpha = \beta = 0$

       i. $\mathfrak{t}_{\alpha,\beta}$ is the 3-dimensional abelian Lie algebra if and only if $B = 0$ if and only if $1 = 0$
       ii. $\mathfrak{t}_{\alpha,\beta}$ is isomorphic to $\mathfrak{n}_3(\mathbb{C})$ if and only if $B \neq 0$ if and only if $\alpha = \beta = 0$

    (b) $B$ has a non-zero repeated eigenvalue if and only if $\alpha = \beta \neq 0$

       i. $\mathfrak{t}_{\alpha,\beta}$ is $\mathfrak{r}_{3,1}$ if and only if $B = \lambda\operatorname{Id}$ for some $\lambda \neq 0$ if and only if $1 = 0$
       ii. $\mathfrak{t}_{\alpha,\beta}$ is isomorphic to $\mathfrak{r}_3$ if and only if $B \neq \lambda\operatorname{Id}$ for all $\lambda \in \mathbb{C}$ if and only if $\alpha = \beta \neq 0$

2. $B$ has distinct eigenvalues eigenvalues if and only if $\alpha \neq \beta$

    (a) $\mathfrak{t}_{\alpha,\beta}$ is isomorphic to $\mathfrak{r}_2(\mathbb{C}) \times \mathbb{C}$ if and only if 0 is an eigenvalue of $B$ with algebraic multiplicity 1 if and only if $\alpha = 0$ or $\beta = 0$, and $\alpha \neq \beta$.

    (b) $\mathfrak{t}_{\alpha,\beta}$ is isomorphic to $\mathfrak{r}_{3,-1}$ if and only if $B$ has distinct eigenvalues eigenvalues and $\operatorname{Trace}(B) = 0$ if and only if $\beta = -\alpha \neq 0$.

    (c) $\mathfrak{t}_{\alpha,\beta}$ is isomorphic to $\mathfrak{r}_{3,z}$ for some $z \in \mathbb{C}$ such that $z^2 \neq 1$ and $z \neq 0$ if and only if $\alpha \neq \beta$, $\alpha, \beta \neq 0$ and $\beta \neq -\alpha$.

### 3.7.2   $\phi_\beta(\mathfrak{L}_1^6)$

In this case, we have $\phi_\beta(\mathfrak{L}_1^6) = \widetilde{\mathfrak{t}}_\beta$:

$$\widetilde{\mathfrak{t}}_\beta = \{[e1, e2] = e2, [e1, e3] = \beta\, e3.$$

We note that $\widetilde{\mathfrak{t}}_\beta$ is an almost abelian Lie algebra where $\text{span}_\mathbb{C}\{e_2, e_3\}$ is a codimension 1 abelian ideal of $\widetilde{\mathfrak{t}}_\beta$; we can identify $\widetilde{\mathfrak{t}}_\beta$ the matrix

$$\widetilde{B} = \begin{bmatrix} 1 & 0 \\ 0 & \beta \end{bmatrix}.$$

The isomorphism class of $\widetilde{\mathfrak{t}}_\beta$ is determined by the eigenvalues of $\widetilde{B}$:

1. $\widetilde{B}$ has repeated eigenvalues if and only if $\beta = 1$
   (a) 0 is a repeated eigenvalue of $\widetilde{B}$ if and only if $1 = 0$
   (b) $B$ has a non-zero repeated eigenvalue if and only if $\beta = 1$
      i. $\widetilde{\mathfrak{t}}_\beta$ is $\mathfrak{r}_{3,1}$ if and only if $\widetilde{B} = \lambda\,\text{Id}$ for some $\lambda \neq 0$ if and only if $\beta = 1$
      ii. $\widetilde{\mathfrak{t}}_\beta$ is isomorphic to $\mathfrak{r}_3$ if and only if $\widetilde{B} \neq \lambda\,\text{Id}$ for all $\lambda \in \mathbb{C}$ if and only if $1 = 0$

2. $\widetilde{B}$ has distinct eigenvalues eigenvalues if and only if $\beta \neq 1$
   (a) $\widetilde{\mathfrak{t}}_\beta$ is isomorphic to $\mathfrak{r}_2(\mathbb{C}) \times \mathbb{C}$ if and only if 0 is an eigenvalue of $B$ with algebraic multiplicity 1 if and only if $\beta = 0$
   (b) $\widetilde{\mathfrak{t}}_\beta$ is isomorphic to $\mathfrak{r}_{3,-1}$ if and only if $\text{Trace}(\widetilde{B}) = 0$ if and only if $\beta = -1$.
   (c) $\widetilde{\mathfrak{t}}_\beta$ is isomorphic to $\mathfrak{r}_{3,z}$ for some $z$ such that $z^2 \neq 1$ and $z \neq 0$ if and only if $\beta \neq 1$, $\beta \neq 0$ and $\beta \neq -1$.

### 3.7.3   $\rho(\mathfrak{L}_1^6)$

In this case, we have $\rho(\mathfrak{L}_1^6) = \widehat{\mathfrak{t}}$:

$$\widehat{\mathfrak{t}} = \{[e1, e3] = e3.$$

which is isomorphic to $\mathfrak{r}_2 \times \mathbb{C}$.

### 3.7.4   Identities

- $[A\cdot, A\cdot] = 0$
- $A^2[\cdot, \cdot] = 0$
- $[A^2\cdot, \cdot] + [\cdot, A^2\cdot] = 0$

The hom-Lie algebra $\mathfrak{L}_1^6$ is a non-multiplicative hom-Lie algebra.

## 3.8   Degenerations between hom-Lie algebras $(\mathfrak{L}_1^i)$

If $\mathfrak{L}_i^j \xrightarrow{\text{deg}} \mathfrak{L}_h^k$ then $\text{Der}(\mathfrak{L}_i^j) \leq \text{Der}(\mathfrak{L}_h^k)$. Therefore, we can organize the hom-Lie algebras in the following way:

| Dim(Der) | hom-Lie algebra |
|---|---|
| 0 | $\mathfrak{L}_1^4$ |
| 1 | $\mathfrak{L}_1^6$ |
| 2 | $\mathfrak{L}_1^5$, $\mathfrak{L}_1^3$ |
| 3 | $\mathfrak{L}_1^2$ |
| 4 | $\mathfrak{L}_1^1$ |
| 6 | $\mathfrak{L}_1^0$ |

### 3.8.1 Degenerations

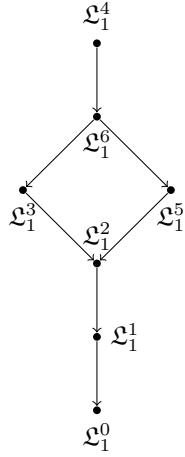

1. $\mathfrak{L}_1^4 \xrightarrow{\deg} \mathfrak{L}_1^6$
   In fact, set
   $$g(t) = \begin{bmatrix} 0 & e^t & 0 \\ e^{2t} & -e^{3t} & 0 \\ 0 & e^{2t} & -e^{3t} \end{bmatrix}, \text{ with } t \in \mathbb{R}.$$
   We have $g(t) \in \mathrm{Aut}(\mathfrak{L}_1)$ and $g(t) \cdot \mathfrak{L}_1^4$ is the hom-Lie algebra $(\mathfrak{L}_1, A(t))$ with
   $$A(t) = \begin{bmatrix} e^{-t} & 0 & -e^{-2t} \\ 0 & 0 & 1 \\ 1 & 0 & -e^{-t} \end{bmatrix}.$$
   It is easy to check that $(\mathfrak{L}_1, A(t)) \to \mathfrak{L}_1^6$ as $t$ tends to infinity.

2. $\mathfrak{L}_1^6 \xrightarrow{\deg} \mathfrak{L}_1^3$
   In fact, set
   $$g(t) = \begin{bmatrix} 1 & 0 & 0 \\ 0 & e^{-t} & 0 \\ 0 & 0 & e^{-t} \end{bmatrix}, \text{ with } t \in \mathbb{R}.$$
   We have $g(t) \in \mathrm{Aut}(\mathfrak{L}_1)$ and $g(t) \cdot \mathfrak{L}_1^6$ is the hom-Lie algebra $(\mathfrak{L}_1, A(t))$ with
   $$A(t) = \begin{bmatrix} 0 & 0 & 0 \\ 0 & 0 & 1 \\ e^{-t} & 0 & 0 \end{bmatrix}.$$
   It is easy to check that $(\mathfrak{L}_1, A(t)) \to \mathfrak{L}_1^3$ as $t$ tends to infinity.

3. $\mathfrak{L}_1^6 \xrightarrow{\deg} \mathfrak{L}_1^5$
   In fact, set
   $$g(t) = \begin{bmatrix} 0 & -1 & 0 \\ -e^{3t} & e^t & 0 \\ 0 & -e^{2t} & -e^{3t} \end{bmatrix}, \text{ with } t \in \mathbb{R}.$$

14

We have $g(t) \in \mathrm{Aut}(\mathfrak{L}_1)$ and $g(t) \cdot \mathfrak{L}_1^6$ is the hom-Lie algebra $(\mathfrak{L}_1, A(t))$ with

$$A(t) = \begin{bmatrix} -\mathrm{e}^{-t} & 0 & \mathrm{e}^{-3t} \\ 1 & 0 & -\mathrm{e}^{-2t} \\ 0 & 1 & \mathrm{e}^{-t} \end{bmatrix}.$$

It is easy to check that $(\mathfrak{L}_1, A(t)) \to \mathfrak{L}_1^5$ as $t$ tends to infinity.

4. $\mathfrak{L}_1^3 \xrightarrow{\deg} \mathfrak{L}_1^2$

   In fact, set

   $$g(t) = \begin{bmatrix} -\mathrm{e}^{3t} & \mathrm{e}^{2t} & 0 \\ 0 & 1 & 0 \\ 0 & \mathrm{e}^{t} & -\mathrm{e}^{3t} \end{bmatrix}, \text{ with } t \in \mathbb{R}.$$

   We have $g(t) \in \mathrm{Aut}(\mathfrak{L}_1)$ and $g(t) \cdot \mathfrak{L}_1^3$ is the hom-Lie algebra $(\mathfrak{L}_1, A(t))$ with

   $$A(t) = \begin{bmatrix} 0 & 1 & -\mathrm{e}^{-t} \\ 0 & \mathrm{e}^{-2t} & -\mathrm{e}^{-3t} \\ 0 & \mathrm{e}^{-t} & -\mathrm{e}^{-2t} \end{bmatrix}.$$

   It is easy to check that $(\mathfrak{L}_1, A(t)) \to \mathfrak{L}_1^2$ as $t$ tends to infinity.

5. $\mathfrak{L}_1^5 \xrightarrow{\deg} \mathfrak{L}_1^2$

   In fact, set

   $$g(t) = \begin{bmatrix} 0 & \mathrm{e}^{-t} & 0 \\ \mathrm{e}^{-t} & \mathrm{e}^{-2t} & 0 \\ 0 & 0 & -\mathrm{e}^{-2t} \end{bmatrix}, \text{ with } t \in \mathbb{R}.$$

   We have $g(t) \in \mathrm{Aut}(\mathfrak{L}_1)$ and $g(t) \cdot \mathfrak{L}_1^5$ is the hom-Lie algebra $(\mathfrak{L}_1, A(t))$ with

   $$A(t) = \begin{bmatrix} -\mathrm{e}^{-t} & 1 & 0 \\ -\mathrm{e}^{-2t} & \mathrm{e}^{-t} & 0 \\ -\mathrm{e}^{-t} & 0 & 0 \end{bmatrix}.$$

   It is easy to check that $(\mathfrak{L}_1, A(t)) \to \mathfrak{L}_1^2$ as $t$ tends to infinity.

6. $\mathfrak{L}_1^2 \xrightarrow{\deg} \mathfrak{L}_1^1$

   In fact, set

   $$g(t) = \begin{bmatrix} \mathrm{e}^{t} & -\mathrm{e}^{3t} & 0 \\ 1 & 0 & 0 \\ \mathrm{e}^{2t} & 0 & \mathrm{e}^{3t} \end{bmatrix}, \text{ with } t \in \mathbb{R}.$$

   We have $g(t) \in \mathrm{Aut}(\mathfrak{L}_1)$ and $g(t) \cdot \mathfrak{L}_1^2$ is the hom-Lie algebra $(\mathfrak{L}_1, A(t))$ with

   $$A(t) = \begin{bmatrix} -\mathrm{e}^{-2t} & \mathrm{e}^{-t} & 0 \\ -\mathrm{e}^{-3t} & \mathrm{e}^{-2t} & 0 \\ -\mathrm{e}^{-t} & 1 & 0 \end{bmatrix}.$$

   It is easy to check that $(\mathfrak{L}_1, A(t)) \to \mathfrak{L}_1^1$ as $t$ tends to infinity.

7. $\mathfrak{L}_1^1 \xrightarrow{\text{deg}} \mathfrak{L}_1^0$

    In fact, set
    $$g(t) = \begin{bmatrix} e^{-t} & 0 & 0 \\ 0 & 1 & 0 \\ 0 & 0 & e^{-t} \end{bmatrix}, \text{ with } t \in \mathbb{R}.$$

    We have $g(t) \in \text{Aut}(\mathfrak{L}_1)$ and $g(t) \cdot \mathfrak{L}_1^1$ is the hom-Lie algebra $(\mathfrak{L}_1, A(t))$ with
    $$A(t) = \begin{bmatrix} 0 & 0 & 0 \\ 0 & 0 & 0 \\ 0 & e^{-t} & 0 \end{bmatrix}.$$

    It is easy to check that $(\mathfrak{L}_1, A(t)) \to \mathfrak{L}_1^0$ as $t$ tends to infinity.

# 4 $\mathfrak{L}_2$: $\mathfrak{r}_3(\mathbb{C})$

$$\mathfrak{L}_1 := \{[e_1, e_2] = e_2, [e_1, e_3] = e_2 + e_3$$

$$\operatorname{Aut}(\mathfrak{L}_2) = \left\{ \begin{bmatrix} 1 & 0 & 0 \\ y & a & x \\ z & 0 & a \end{bmatrix} \in M(3,\mathbb{C}) : \begin{array}{c} a \in \mathbb{C}^* \\ x, y, z \in \mathbb{C} \end{array} \right\}.$$

$$\operatorname{Der}(\mathfrak{L}_2) = \left\{ \begin{bmatrix} 0 & 0 & 0 \\ t_3 & t_1 & t_2 \\ t_4 & 0 & t_1 \end{bmatrix} \in M(3,\mathbb{C}) : t_1, t_2, t_3, t_4 \in \mathbb{C} \right\} \cong \mathbb{C} \ltimes \mathfrak{n}_3(\mathbb{C}).$$

## 4.1 $\mathfrak{L}_2^0$

Let $A$ be the Zero map and let us denote by $\mathfrak{L}_2^0$ the hom-Lie algebra $(\mathfrak{L}_2, A)$. We have:
$\operatorname{Aut}(\mathfrak{L}_2^0) = \operatorname{Aut}(\mathfrak{L}_2)$.
$\operatorname{Der}(\mathfrak{L}_2^0) = \operatorname{Der}(\mathfrak{L}_2)$.

### 4.1.1 $\psi_{\alpha,\beta}(\mathfrak{L}_2^0)$

In this case, we have $\psi_{\alpha,\beta}(\mathfrak{L}_2^0) = \mathfrak{t}_{\alpha,\beta}$ is the Lie algebra $\mathfrak{r}_3$.

### 4.1.2 $\phi_\beta(\mathfrak{L}_2^0)$

In this case, we have $\phi_\beta(\mathfrak{L}_2^0) = \widetilde{\mathfrak{t}}_\beta$ is the 3-dimensional abelian Lie algebra.

### 4.1.3 $\rho(\mathfrak{L}_2^0)$

In this case, we have $\rho(\mathfrak{L}_2^0) = \widehat{\mathfrak{t}}$ is the 3-dimensional abelian Lie algebra.

## 4.2 $\mathfrak{L}_2^1$

Let $A$ be the endomorphism of $\mathbb{C}^3$ whose matrix representation with respect to the ordered basis $\{e_1, e_2, e_3\}$ is

$$\begin{bmatrix} 0 & 0 & 0 \\ 1 & 0 & 0 \\ 0 & 0 & 0 \end{bmatrix}$$

and let us denote by $\mathfrak{L}_2^1$ the hom-Lie algebra $(\mathfrak{L}_2, A)$. We have:

$$\operatorname{Aut}(\mathfrak{L}_2^1) = \left\{ \begin{bmatrix} 1 & 0 & 0 \\ y & 1 & x \\ z & 0 & 1 \end{bmatrix} \in M(3,\mathbb{C}) : x, y, z \in \mathbb{C} \right\}.$$

$$\operatorname{Der}(\mathfrak{L}_2^1) = \left\{ \begin{bmatrix} 0 & 0 & 0 \\ t_2 & 0 & t_3 \\ t_1 & 0 & 0 \end{bmatrix} \in M(3,\mathbb{C}) : t_1, t_2, t_3 \in \mathbb{C} \right\} \cong \mathfrak{n}_3(\mathbb{C}).$$

### 4.2.1 $\psi_{\alpha,\beta}(\mathfrak{L}_2^1)$

In this case, we have $\psi_{\alpha,\beta}(\mathfrak{L}_2^1) = \mathfrak{t}_{\alpha,\beta}$ is the Lie algebra $\mathfrak{r}_3$.

### 4.2.2 $\phi_\beta(\mathfrak{L}_2^1)$

In this case, we have $\phi_\beta(\mathfrak{L}_2^1) = \widetilde{\mathfrak{t}}_\beta$ is the 3-dimensional abelian Lie algebra.

### 4.2.3 $\rho(\mathfrak{L}_2^1)$

In this case, we have $\rho(\mathfrak{L}_2^1) = \widehat{\mathfrak{t}}$ is the 3-dimensional abelian Lie algebra.

### 4.2.4 Identities

- $A[\cdot, \cdot] = 0$
- $[A\cdot, \cdot] + [\cdot, A\cdot] = 0$
    - $[A\cdot, A\cdot] = 0$

In particular, the hom-Lie algebra $\mathfrak{L}_2^1$ is a multiplicative hom-Lie algebra.

## 4.3 $\mathfrak{L}_2^2$

Let $A$ be the endomorphism of $\mathbb{C}^3$ whose matrix representation with respect to the ordered basis $\{e_1, e_2, e_3\}$ is

$$\begin{bmatrix} 0 & 0 & 0 \\ 0 & 0 & 0 \\ 1 & 0 & 0 \end{bmatrix}$$

and let us denote by $\mathfrak{L}_2^2$ the hom-Lie algebra $(\mathfrak{L}_2, A)$. We have:

$$\mathrm{Aut}(\mathfrak{L}_2^2) = \left\{ \begin{bmatrix} 1 & 0 & 0 \\ x & 1 & 0 \\ y & 0 & 1 \end{bmatrix} \in M(3, \mathbb{C}) : x, y \in \mathbb{C} \right\}.$$

$$\mathrm{Der}(\mathfrak{L}_2^2) = \left\{ \begin{bmatrix} 0 & 0 & 0 \\ t_1 & 0 & 0 \\ t_2 & 0 & 0 \end{bmatrix} \in M(3, \mathbb{C}) : t_1, t_2 \in \mathbb{C} \right\} \cong \mathbb{C}^2.$$

### 4.3.1 $\psi_{\alpha,\beta}(\mathfrak{L}_2^2)$

In this case, we have $\psi_{\alpha,\beta}(\mathfrak{L}_2^2) = \mathfrak{t}_{\alpha,\beta}$ is the Lie algebra $\mathfrak{r}_3$.

### 4.3.2 $\phi_\beta(\mathfrak{L}_2^2)$

In this case, we have $\phi_\beta(\mathfrak{L}_2^2) = \widetilde{\mathfrak{t}}_\beta$ is the 3-dimensional abelian Lie algebra.

### 4.3.3 $\rho(\mathfrak{L}_2^2)$

In this case, we have $\rho(\mathfrak{L}_2^2) = \widehat{\mathfrak{t}}$ is the 3-dimensional abelian Lie algebra.

### 4.3.4 Identities

- $A[\cdot, \cdot] = 0$
- $[A\cdot, \cdot] + [\cdot, A\cdot] = 0$
    - $[A\cdot, A\cdot] = 0$

In particular, the hom-Lie algebra $\mathfrak{L}_2^2$ is a multiplicative hom-Lie algebra.

## 4.4 $\mathfrak{L}_2^3(\lambda)$

Let $A$ be the endomorphism of $\mathbb{C}^3$ whose matrix representation with respect to the ordered basis $\{e_1, e_2, e_3\}$ is

$$\begin{bmatrix} 0 & 0 & 0 \\ 0 & 0 & \lambda \\ 0 & 0 & 0 \end{bmatrix}$$

and let us denote by $\mathfrak{L}_2^3(\lambda)$ the hom-Lie algebra $(\mathfrak{L}_2, A)$, with $\lambda \neq 0$. We have:

$$\mathrm{Aut}(\mathfrak{L}_2^3(\lambda)) = \left\{ \begin{bmatrix} 1 & 0 & 0 \\ y & a & x \\ 0 & 0 & a \end{bmatrix} \in M(3, \mathbb{C}) : \begin{array}{l} a \in \mathbb{C}^* \\ x, y \in \mathbb{C} \end{array} \right\}.$$

$$\mathrm{Der}(\mathfrak{L}_2^3(\lambda)) = \left\{ \begin{bmatrix} 0 & 0 & 0 \\ t_3 & t_1 & t_2 \\ 0 & 0 & t_1 \end{bmatrix} \in M(3, \mathbb{C}) : t_1, t_2, t_3 \in \mathbb{C} \right\} \cong \mathfrak{r}_2(\mathbb{C}) \times \mathbb{C}.$$

### 4.4.1 $\psi_{\alpha,\beta}(\mathfrak{L}_2^3(\lambda))$

In this case, we have $\psi_{\alpha,\beta}(\mathfrak{L}_2^3(\lambda)) = \mathfrak{t}_{\alpha,\beta}$:

$$\mathfrak{t}_{\alpha,\beta} = \{[e1, e2] = e2, [e1, e3], (1 + \beta\lambda + \alpha\lambda) e2 + e3.$$

We note that $\mathfrak{t}_{\alpha,\beta}$ is an almost abelian Lie algebra where $\mathrm{span}_\mathbb{C}\{e_2, e_3\}$ is a codimension 1 abelian ideal of $\mathfrak{t}_{\alpha,\beta}$; we can identify $\mathfrak{t}_{\alpha,\beta}$ the matrix

$$B = \begin{bmatrix} 1 & 1 + \beta\lambda + \alpha\lambda \\ 0 & 1 \end{bmatrix}.$$

The isomorphism class of $\mathfrak{t}_{\alpha,\beta}$ is determined by the eigenvalues of $B$:

1. $B$ has repeated eigenvalues if and only if $1 = 1$
    (a) 0 is a repeated eigenvalue of $B$ if and only if $1 = 0$
    (b) $B$ has a non-zero repeated eigenvalue if and only if $1 = 1$
        i. $\mathfrak{t}_{\alpha,\beta}$ is $\mathfrak{r}_{3,1}$ if and only if $B = \lambda\,\mathrm{Id}$ for some $\lambda \neq 0$ if and only if $1 + \beta\lambda + \alpha\lambda = 0$
        ii. $\mathfrak{t}_{\alpha,\beta}$ is isomorphic to $\mathfrak{r}_3$ if and only if $B \neq t\,\mathrm{Id}$ for all $t \in \mathbb{C}$ if and only if $1 + \beta\lambda + \alpha\lambda \neq 0$
2. $B$ has distinct eigenvalues eigenvalues if and only if $1 = 0$

### 4.4.2 $\phi_\beta(\mathfrak{L}_2^3)$

In this case, we have $\phi_\beta(\mathfrak{L}_2^3) = \widetilde{\mathfrak{t}}_\beta$:

$$\widetilde{\mathfrak{t}}_\beta = \{[e1, e3] = (\beta\lambda + \lambda) e2.$$

We note that $\widetilde{\mathfrak{t}}_\beta$ is an almost abelian Lie algebra where $\mathrm{span}_\mathbb{C}\{e_2, e_3\}$ is a codimension 1 abelian ideal of $\widetilde{\mathfrak{t}}_\beta$; we can identify $\widetilde{\mathfrak{t}}_\beta$ the matrix

$$\widetilde{B} = \begin{bmatrix} 0 & \beta\lambda + \lambda \\ 0 & 0 \end{bmatrix}.$$

The isomorphism class of $\widetilde{\mathfrak{t}}_\beta$ is determined by the eigenvalues of $\widetilde{B}$:

1. $\widetilde{B}$ has repeated eigenvalues if and only if $0 = 0$
    (a) 0 is a repeated eigenvalue of $\widetilde{B}$ if and only if $0 = 0$
        i. $\widetilde{\mathfrak{t}}_\beta$ is the 3-dimensional abelian Lie algebra if and only if $\widetilde{B} = 0$ if and only if $\beta\lambda + \lambda = 0$
        ii. $\widetilde{\mathfrak{t}}_\beta$ is isomorphic to $\mathfrak{n}_3(\mathbb{C})$ if and only if $\widetilde{B} \neq 0$ if and only if $\beta\lambda + \lambda \neq 0$
    (b) $B$ has a non-zero repeated eigenvalue if and only if $1 = 0$
2. $\widetilde{B}$ has distinct eigenvalues eigenvalues if and only if $1 = 0$

### 4.4.3 $\rho(\mathfrak{L}_2^3)$

In this case, we have $\rho(\mathfrak{L}_2^3) = \widehat{\mathfrak{t}}$:
$$\widehat{\mathfrak{t}} = \{[e1, e3] = \lambda\, e2.$$

which is isomorphic to $\mathfrak{n}_3$

### 4.4.4 Identities

- $A[\cdot, \cdot] = [A\cdot, \cdot] + [\cdot, A\cdot]$ (Derivation)
- $[A\cdot, A\cdot] = 0$
    - $A[A\cdot, \cdot] = 0$

The hom-Lie algebra $\mathfrak{L}_2^3(\lambda)$ is a non-multiplicative hom-Lie algebra.

## 4.5 $\mathfrak{L}_2^4(\lambda)$

Let $A$ be the endomorphism of $\mathbb{C}^3$ whose matrix representation with respect to the ordered basis $\{e_1, e_2, e_3\}$ is
$$\begin{bmatrix} 0 & 0 & 0 \\ 0 & 0 & 0 \\ 0 & \lambda & 0 \end{bmatrix}$$

with $\lambda \in \mathbb{C}$, $\lambda \neq 0$ and let us denote by $\mathfrak{L}_2^4(\lambda)$ the hom-Lie algebra $(\mathfrak{L}_2, A)$. We have:

$$\mathrm{Aut}(\mathfrak{L}_2^4(\lambda)) = \left\{ \begin{bmatrix} 1 & 0 & 0 \\ 0 & a & 0 \\ x & 0 & a \end{bmatrix} \in M(3, \mathbb{C}) : \begin{array}{l} a \in \mathbb{C}^* \\ x \in \mathbb{C} \end{array} \right\}.$$

$$\mathrm{Der}(\mathfrak{L}_2^4(\lambda)) = \left\{ \begin{bmatrix} 0 & 0 & 0 \\ 0 & t_1 & 0 \\ t_2 & 0 & t_1 \end{bmatrix} \in M(3, \mathbb{C}) : t_1, t_2 \in \mathbb{C} \right\} \cong \mathfrak{r}_2(\mathbb{C}).$$

### 4.5.1 $\psi_{\alpha,\beta}(\mathfrak{L}_2^4(\lambda))$

In this case, we have $\psi_{\alpha,\beta}(\mathfrak{L}_2^4(\lambda)) = \mathfrak{t}_{\alpha,\beta}$:

$$\mathfrak{t}_{\alpha,\beta} = \{[e1, e2] = (1 + \beta\lambda)\, e2 + (\alpha\lambda + \beta\lambda)\, e3, [e1, e3] = e2 + (1 + \alpha\lambda)\, e3.$$

We note that $\mathfrak{t}_{\alpha,\beta}$ is an almost abelian Lie algebra where $\mathrm{span}_{\mathbb{C}}\{e_2, e_3\}$ is a codimension 1 abelian ideal of $\mathfrak{t}_{\alpha,\beta}$; we can identify $\mathfrak{t}_{\alpha,\beta}$ the matrix

$$B = \begin{bmatrix} 1 + \beta\lambda & 1 \\ \alpha\lambda + \beta\lambda & 1 + \alpha\lambda \end{bmatrix}.$$

The isomorphism class of $\mathfrak{t}_{\alpha,\beta}$ is determined by the eigenvalues of $B$:

1. $B$ has repeated eigenvalues if and only if $(\alpha - \beta)^2 \lambda + 4(\alpha + \beta) = 0$

    (a) $0$ is a repeated eigenvalue of $B$ if and only if $\alpha = -\frac{1 \pm \sqrt{2}}{\lambda}, \beta = \frac{-1 \pm \sqrt{2}}{\lambda}$.
    
    i. $\mathfrak{t}_{\alpha,\beta}$ is the 3-dimensional abelian Lie algebra if and only if $B = 0$ if and only if $1 = 0$
    
    ii. $\mathfrak{t}_{\alpha,\beta}$ is isomorphic to $\mathfrak{n}_3(\mathbb{C})$ if and only if $B \neq 0$ if and only if $\alpha = -\frac{1 \pm \sqrt{2}}{\lambda}, \beta = \frac{-1 \pm \sqrt{2}}{\lambda}$.
    
    (b) $B$ has a non-zero repeated eigenvalue if and only if
    $\left\{\alpha = -\frac{s-1 \pm \sqrt{2s^2-2s}}{s\lambda}, \beta = \frac{-s+1 \pm \sqrt{2s^2-2s}}{s\lambda}\right\}$ for any $s \in \mathbb{C} \setminus \{0\}$.
    
    i. $\mathfrak{t}_{\alpha,\beta}$ is $\mathfrak{r}_{3,1}$ if and only if $B = t \operatorname{Id}$ for some $t \neq 0$ if and only if $1 = 0$
    
    ii. $\mathfrak{t}_{\alpha,\beta}$ is isomorphic to $\mathfrak{r}_3$ if and only if $B \neq t \operatorname{Id}$ for all $t \in \mathbb{C}$ if and only if
    $\left\{\alpha = -\frac{s-1 \pm \sqrt{2s^2-2s}}{s\lambda}, \beta = \frac{-s+1 \pm \sqrt{2s^2-2s}}{s\lambda}\right\}$ for any $s \in \mathbb{C} \setminus \{0\}$.

2. $B$ has distinct eigenvalues if and only if $(\alpha - \beta)^2 \lambda + 4(\alpha + \beta) \neq 0$

    (a) $\mathfrak{t}_{\alpha,\beta}$ is isomorphic to $\mathfrak{r}_2(\mathbb{C}) \times \mathbb{C}$ if and only if $0$ is an eigenvalue of $B$ with algebraic multiplicity 1 if and only if
    $\left\{\alpha = -\frac{2s-1 \pm \sqrt{8s^2-4s+1}}{2s\lambda}, \beta = \frac{-2s+1 \pm \sqrt{8s^2-4s+1}}{2s\lambda}\right\}$ for any $s \in \mathbb{C} \setminus \{0\}$.
    
    (b) $\mathfrak{t}_{\alpha,\beta}$ is isomorphic to $\mathfrak{r}_{3,-1}$ if and only if $\operatorname{Trace}(B) = 0$ if and only if
    $\left\{\alpha = -\frac{2s \pm \sqrt{s(8s+1)}}{2s\lambda}, \beta = \frac{-2s \pm \sqrt{s(8s+1)}}{2s\lambda}\right\}$ for any $s \in \mathbb{C} \setminus \{0\}$.
    
    (c) $\mathfrak{t}_{\alpha,\beta}$ is isomorphic to $\mathfrak{r}_{3,z}$ for some $z$ such that $z(z^2 - 1) \neq 0$ if and only if
    $\left\{\alpha = -\frac{2s_1 s_2 - s_1 \pm \sqrt{s_1 s_2 (8 s_1 s_2 - 4 s_1 + s_2)}}{2s_1 s_2 \lambda}, \beta = \frac{-2s_1 s_2 + s_1 \pm \sqrt{s_1 s_2 (8 s_1 s_2 - 4 s_1 + s_2)}}{2s_1 s_2 \lambda}\right\}$ for any $s_1, s_2 \in \mathbb{C} \setminus \{0\}$ such that $s_1 - s^2 \neq 0$.

### 4.5.2 $\phi_\beta(\mathfrak{L}_2^4(\lambda))$

In this case, we have $\phi_\beta(\mathfrak{L}_2^4(\lambda)) = \widetilde{\mathfrak{t}}_\beta$:

$$\widetilde{\mathfrak{t}}_\beta = \{[e1, e2] = \beta \lambda\, e2 + (\lambda + \beta \lambda)\, e3, [e1, e3], \lambda\, e3.$$

We note that $\widetilde{\mathfrak{t}}_\beta$ is an almost abelian Lie algebra where $\operatorname{span}_{\mathbb{C}}\{e_2, e_3\}$ is a codimension 1 abelian ideal of $\widetilde{\mathfrak{t}}_\beta$; we can identify $\widetilde{\mathfrak{t}}_\beta$ the matrix

$$\widetilde{B} = \begin{bmatrix} \beta \lambda & 0 \\ \lambda + \beta \lambda & \lambda \end{bmatrix}.$$

The isomorphism class of $\widetilde{\mathfrak{t}}_\beta$ is determined by the eigenvalues of $\widetilde{B}$:

1. $\widetilde{B}$ has repeated eigenvalues if and only if $\beta = 1$

    (a) $0$ is a repeated eigenvalue of $\widetilde{B}$ if and only if $\lambda = 0$ (false)
    
    (b) $B$ has a non-zero repeated eigenvalue if and only if $\beta = 1$
    
    i. $\widetilde{\mathfrak{t}}_\beta$ is $\mathfrak{r}_{3,1}$ if and only if $\widetilde{B} = t \operatorname{Id}$ for some $t \neq 0$ if and only if $\lambda = 0$ (false)
    
    ii. $\widetilde{\mathfrak{t}}_\beta$ is isomorphic to $\mathfrak{r}_3$ if and only if $\widetilde{B} \neq t \operatorname{Id}$ for all $t \in \mathbb{C}$ if and only if $\beta = 1$

2. $\widetilde{B}$ has distinct eigenvalues eigenvalues if and only if $\beta \neq 1$

    (a) $\widetilde{\mathfrak{t}}_\beta$ is isomorphic to $\mathfrak{r}_2(\mathbb{C}) \times \mathbb{C}$ if and only if $0$ is an eigenvalue of $B$ with algebraic multiplicity 1 if and only if $\beta = 0$
    
    (b) $\widetilde{\mathfrak{t}}_\beta$ is isomorphic to $\mathfrak{r}_{3,-1}$ if and only if $\operatorname{Trace}(\widetilde{B}) = 0$ if and only if $\beta = -1$.
    
    (c) $\widetilde{\mathfrak{t}}_\beta$ is isomorphic to $\mathfrak{r}_{3,z}$ for some $z$ such that $z^2 \neq 1$ and $z \neq 0$ if and only if $\beta \neq 1$ and $\beta \neq 0$ and $\beta \neq -1$

### 4.5.3 $\rho(\mathfrak{L}_2^4)$

In this case, we have $\rho(\mathfrak{L}_2^4) = \widehat{\mathfrak{t}}$:

$$\widehat{\mathfrak{t}} = \{[e1, e2] = \lambda\, e2 + \lambda\, e3.$$

which is isomorphic to $\mathfrak{r}_2 \times \mathbb{C}$.

### 4.5.4 Identities

- $[A\cdot, A\cdot] = 0$

The hom-Lie algebra $\mathfrak{L}_2^4(\lambda)$ is a non-multiplicative hom-Lie algebra.

## 4.6 $\mathfrak{L}_2^5(\lambda)$

Let $A$ be the endomorphism of $\mathbb{C}^3$ whose matrix representation with respect to the ordered basis $\{e_1, e_2, e_3\}$ is

$$\begin{bmatrix} 0 & 0 & 0 \\ 0 & 0 & \lambda \\ 1 & 0 & 0 \end{bmatrix}$$

and let us denote by $\mathfrak{L}_2^5(\lambda)$ the hom-Lie algebra $(\mathfrak{L}_2, A)$. We have:

$$\mathrm{Aut}(\mathfrak{L}_2^5(\lambda)) = \left\{ \begin{bmatrix} 1 & 0 & 0 \\ y & 1 & x\lambda \\ x & 0 & 1 \end{bmatrix} \in M(3, \mathbb{C}) : x, y \in \mathbb{C} \right\}.$$

$$\mathrm{Der}(\mathfrak{L}_2^5(\lambda)) = \left\{ \begin{bmatrix} 0 & 0 & 0 \\ t_2 & 0 & \lambda t_1 \\ t_1 & 0 & 0 \end{bmatrix} \in M(3, \mathbb{C}) : t_1, t_2 \in \mathbb{C} \right\} \cong \mathbb{C}^2.$$

### 4.6.1 $\psi_{\alpha,\beta}(\mathfrak{L}_2^5(\lambda))$

In this case, we have $\psi_{\alpha,\beta}(\mathfrak{L}_2^5(\lambda)) = \mathfrak{t}_{\alpha,\beta}$:

$$\mathfrak{t}_{\alpha,\beta} = \{[e1, e2] = e2, [e1, e3], (1 + \beta\,\lambda + \alpha\,\lambda)\, e2 + e3.$$

We note that $\mathfrak{t}_{\alpha,\beta}$ is an almost abelian Lie algebra where $\mathrm{span}_\mathbb{C}\{e_2, e_3\}$ is a codimension 1 abelian ideal of $\mathfrak{t}_{\alpha,\beta}$; we can identify $\mathfrak{t}_{\alpha,\beta}$ the matrix

$$B = \begin{bmatrix} 1 & 1 + \beta\,\lambda + \alpha\,\lambda \\ 0 & 1 \end{bmatrix}.$$

The isomorphism class of $\mathfrak{t}_{\alpha,\beta}$ is determined by the eigenvalues of $B$:

1. $B$ has repeated eigenvalues if and only if $1 = 1$
   (a) 0 is a repeated eigenvalue of $B$ if and only if $1 = 0$
   (b) $B$ has a non-zero repeated eigenvalue if and only if $1 = 1$
      i. $\mathfrak{t}_{\alpha,\beta}$ is $\mathfrak{r}_{3,1}$ if and only if $B = \lambda\,\mathrm{Id}$ for some $\lambda \neq 0$ if and only if $1 + \beta\,\lambda + \alpha\,\lambda = 0$
      ii. $\mathfrak{t}_{\alpha,\beta}$ is isomorphic to $\mathfrak{r}_3$ if and only if $B \neq t\,\mathrm{Id}$ for all $t \in \mathbb{C}$ if and only if $1 + \beta\,\lambda + \alpha\,\lambda \neq 0$

2. $B$ has distinct eigenvalues eigenvalues if and only if $1 = 0$

### 4.6.2 $\phi_\beta(\mathfrak{L}_2^5)$

In this case, we have $\phi_\beta(\mathfrak{L}_2^5) = \widetilde{\mathfrak{t}}_\beta$:

$$\widetilde{\mathfrak{t}}_\beta = \{[e1, e3] = (\beta\lambda + \lambda)\, e2.$$

We note that $\widetilde{\mathfrak{t}}_\beta$ is an almost abelian Lie algebra where $\text{span}_\mathbb{C}\{e_2, e_3\}$ is a codimension 1 abelian ideal of $\widetilde{\mathfrak{t}}_\beta$; we can identify $\widetilde{\mathfrak{t}}_\beta$ the matrix

$$\widetilde{B} = \begin{bmatrix} 0 & \beta\lambda + \lambda \\ 0 & 0 \end{bmatrix}.$$

The isomorphism class of $\widetilde{\mathfrak{t}}_\beta$ is determined by the eigenvalues of $\widetilde{B}$:

1. $\widetilde{B}$ has repeated eigenvalues if and only if $0 = 0$
   
   (a) 0 is a repeated eigenvalue of $\widetilde{B}$ if and only if $0 = 0$
   
   i. $\widetilde{\mathfrak{t}}_\beta$ is the 3-dimensional abelian Lie algebra if and only if $\widetilde{B} = 0$ if and only if $\beta\lambda + \lambda = 0$
   
   ii. $\widetilde{\mathfrak{t}}_\beta$ is isomorphic to $\mathfrak{n}_3(\mathbb{C})$ if and only if $\widetilde{B} \neq 0$ if and only if $\beta\lambda + \lambda \neq 0$
   
   (b) $B$ has a non-zero repeated eigenvalue if and only if $1 = 0$

2. $\widetilde{B}$ has distinct eigenvalues eigenvalues if and only if $1 = 0$

### 4.6.3 $\rho(\mathfrak{L}_2^5)$

In this case, we have $\rho(\mathfrak{L}_2^5) = \widehat{\mathfrak{t}}$:

$$\widehat{\mathfrak{t}} = \{[e1, e3] = \lambda\, e2.$$

which is isomorphic to $\mathfrak{n}_3$

### 4.6.4 Identities

- $A[\cdot, \cdot] = [A\cdot, \cdot] + [\cdot, A\cdot]$ (Derivation)
- $[A\cdot, A\cdot] = 0$
- $A^2[\cdot, \cdot] = 0$

The hom-Lie algebra $\mathfrak{L}_2^5(\lambda)$ is a non-multiplicative hom-Lie algebra.

## 4.7 $\mathfrak{L}_2^6(\lambda)$

Let $A$ be the endomorphism of $\mathbb{C}^3$ whose matrix representation with respect to the ordered basis $\{e_1, e_2, e_3\}$ is

$$\begin{bmatrix} 0 & 0 & 0 \\ 1 & 0 & 0 \\ 0 & \lambda & 0 \end{bmatrix}$$

and let us denote by $\mathfrak{L}_2^6(\lambda)$ the hom-Lie algebra $(\mathfrak{L}_2, A)$. We have:

$$\text{Aut}(\mathfrak{L}_2^6(\lambda)) = \left\{ \begin{bmatrix} 1 & 0 & 0 \\ 0 & 1 & 0 \\ x & 0 & 1 \end{bmatrix} \in M(3, \mathbb{C}) : x \in \mathbb{C} \right\}.$$

$$\text{Der}(\mathfrak{L}_2^6(\lambda)) = \left\{ \begin{bmatrix} 0 & 0 & 0 \\ 0 & 0 & 0 \\ t_1 & 0 & 0 \end{bmatrix} \in M(3, \mathbb{C}) : t_1 \in \mathbb{C} \right\} \cong \mathbb{C}.$$

### 4.7.1 $\psi_{\alpha,\beta}(\mathfrak{L}_2^6(\lambda))$

In this case, we have $\psi_{\alpha,\beta}(\mathfrak{L}_2^6(\lambda)) = \mathfrak{t}_{\alpha,\beta}$:

$$\mathfrak{t}_{\alpha,\beta} = \{[e1, e2] = (1 + \beta\lambda)\, e2 + (\alpha\lambda + \beta\lambda)\, e3, [e1, e3] = e2 + (1 + \alpha\lambda)\, e3.$$

We note that $\mathfrak{t}_{\alpha,\beta}$ is an almost abelian Lie algebra where $\mathrm{span}_{\mathbb{C}}\{e_2, e_3\}$ is a codimension 1 abelian ideal of $\mathfrak{t}_{\alpha,\beta}$; we can identify $\mathfrak{t}_{\alpha,\beta}$ the matrix

$$B = \begin{bmatrix} 1 + \beta\lambda & 1 \\ \alpha\lambda + \beta\lambda & 1 + \alpha\lambda \end{bmatrix}.$$

The isomorphism class of $\mathfrak{t}_{\alpha,\beta}$ is determined by the eigenvalues of $B$:

1. $B$ has repeated eigenvalues if and only if $(\alpha - \beta)^2 \lambda + 4(\alpha + \beta) = 0$.

    (a) 0 is a repeated eigenvalue of $B$ if and only if $\alpha = -\frac{1+\sqrt{2}}{\lambda}, \beta = \frac{\sqrt{2}-1}{\lambda}$ or $\alpha = -\frac{1-\sqrt{2}}{\lambda}, \beta = \frac{-1-\sqrt{2}}{\lambda}$.
    
    i. $\mathfrak{t}_{\alpha,\beta}$ is the 3-dimensional abelian Lie algebra if and only if $B = 0$ if and only if $1 = 0$
    
    ii. $\mathfrak{t}_{\alpha,\beta}$ is isomorphic to $\mathfrak{n}_3(\mathbb{C})$ if and only if $B \neq 0$ if and only if $\alpha = -\frac{1\pm\sqrt{2}}{\lambda}, \beta = \frac{-1\pm\sqrt{2}}{\lambda}$.
    
    (b) $B$ has a non-zero repeated eigenvalue if and only if
    $\left\{\alpha = -\frac{s-1\pm\sqrt{2s^2-2s}}{s\lambda}, \beta = \frac{-s+1\pm\sqrt{2s^2-2s}}{s\lambda}\right\}$ for any $s \in \mathbb{C} \setminus \{0\}$.
    
    i. $\mathfrak{t}_{\alpha,\beta}$ is $\mathfrak{r}_{3,1}$ if and only if $B = t\,\mathrm{Id}$ for some $t \neq 0$ if and only if $1 = 0$
    
    ii. $\mathfrak{t}_{\alpha,\beta}$ is isomorphic to $\mathfrak{r}_3$ if and only if $B \neq t\,\mathrm{Id}$ for all $t \in \mathbb{C}$ if and only if
    $\left\{\alpha = -\frac{s-1\pm\sqrt{2s^2-2s}}{s\lambda}, \beta = \frac{-s+1\pm\sqrt{2s^2-2s}}{s\lambda}\right\}$ for any $s \in \mathbb{C} \setminus \{0\}$.

2. $B$ has distinct eigenvalues if and only if $(\alpha - \beta)^2 \lambda + 4(\alpha + \beta) \neq 0$

    (a) $\mathfrak{t}_{\alpha,\beta}$ is isomorphic to $\mathfrak{r}_2(\mathbb{C}) \times \mathbb{C}$ if and only if 0 is an eigenvalue of $B$ with algebraic multiplicity 1 if and only if
    $\left\{\alpha = -\frac{2s-1\pm\sqrt{8s^2-4s+1}}{2s\lambda}, \beta = \frac{-2s+1\pm\sqrt{8s^2-4s+1}}{2s\lambda}\right\}$ for any $s \in \mathbb{C} \setminus \{0\}$.
    
    (b) $\mathfrak{t}_{\alpha,\beta}$ is isomorphic to $\mathfrak{r}_{3,-1}$ if and only if $\mathrm{Trace}(B) = 0$ if and only if
    $\left\{\alpha = -\frac{2s\pm\sqrt{s(8s+1)}}{2s\lambda}, \beta = \frac{-2s\pm\sqrt{s(8s+1)}}{2s\lambda}\right\}$ for any $s \in \mathbb{C} \setminus \{0\}$.
    
    (c) $\mathfrak{t}_{\alpha,\beta}$ is isomorphic to $\mathfrak{r}_{3,z}$ for some $z$ such that $z(z^2 - 1) \neq 0$ if and only if
    $\left\{\alpha = -\frac{2s_1s_2 - s_1 \pm \sqrt{s_1s_2(8s_1s_2 - 4s_1 + s_2)}}{2s_1s_2\lambda}, \beta = \frac{-2s_1s_2 + s_1 \pm \sqrt{s_1s_2(8s_1s_2 - 4s_1 + s_2)}}{2s_1s_2\lambda}\right\}$ for any $s_1, s_2 \in \mathbb{C} \setminus \{0\}$ such that $s_1 - s^2 \neq 0$.

### 4.7.2 $\phi_\beta(\mathfrak{L}_2^6(\lambda))$

In this case, we have $\phi_\beta(\mathfrak{L}_2^6(\lambda)) = \widetilde{\mathfrak{t}}_\beta$:

$$\widetilde{\mathfrak{t}}_\beta = \{[e1, e2] = \beta\lambda\, e2 + (\lambda + \beta\lambda)\, e3, [e1, e3], \lambda\, e3.$$

We note that $\widetilde{\mathfrak{t}}_\beta$ is an almost abelian Lie algebra where $\mathrm{span}_{\mathbb{C}}\{e_2, e_3\}$ is a codimension 1 abelian ideal of $\widetilde{\mathfrak{t}}_\beta$; we can identify $\widetilde{\mathfrak{t}}_\beta$ the matrix

$$\widetilde{B} = \begin{bmatrix} \beta\lambda & 0 \\ \lambda + \beta\lambda & \lambda \end{bmatrix}.$$

The isomorphism class of $\widetilde{\mathfrak{t}}_\beta$ is determined by the eigenvalues of $\widetilde{B}$:

1. $\widetilde{B}$ has repeated eigenvalues if and only if $\beta = 1$
   (a) 0 is a repeated eigenvalue of $\widetilde{B}$ if and only if $\lambda = 0$ (false)
   (b) $B$ has a non-zero repeated eigenvalue if and only if $\beta = 1$
      i. $\widetilde{\mathfrak{t}}_\beta$ is $\mathfrak{r}_{3,1}$ if and only if $\widetilde{B} = t\,\mathrm{Id}$ for some $t \neq 0$ if and only if $\lambda = 0$ (false)
      ii. $\widetilde{\mathfrak{t}}_\beta$ is isomorphic to $\mathfrak{r}_3$ if and only if $\widetilde{B} \neq t\,\mathrm{Id}$ for all $t \in \mathbb{C}$ if and only if $\beta = 1$

2. $\widetilde{B}$ has distinct eigenvalues eigenvalues if and only if $\beta \neq 1$
   (a) $\widetilde{\mathfrak{t}}_\beta$ is isomorphic to $\mathfrak{r}_2(\mathbb{C}) \times \mathbb{C}$ if and only if 0 is an eigenvalue of $B$ with algebraic multiplicity 1 if and only if $\beta = 0$
   (b) $\widetilde{\mathfrak{t}}_\beta$ is isomorphic to $\mathfrak{r}_{3,-1}$ if and only if $\mathrm{Trace}(\widetilde{B}) = 0$ if and only if $\beta = -1$.
   (c) $\widetilde{\mathfrak{t}}_\beta$ is isomorphic to $\mathfrak{r}_{3,z}$ for some $z$ such that $z^2 \neq 1$ and $z \neq 0$ if and only if $\beta \neq 1$ and $\beta \neq 0$ and $\beta \neq -1$

### 4.7.3 $\rho(\mathfrak{L}_2^6)$

In this case, we have $\rho(\mathfrak{L}_2^6) = \widehat{\mathfrak{t}}$:

$$\widehat{\mathfrak{t}}_\beta = \{[e1, e2] = \lambda\,e2 + \lambda\,e3.$$

which is isomorphic to $\mathfrak{r}_2 \times \mathbb{C}$

### 4.7.4 Identities

- $[A\cdot, A\cdot] = 0$
- $A^2[\cdot, \cdot] = 0$
- $[A^2\cdot, \cdot] + [\cdot, A^2\cdot] = 0$

The hom-Lie algebra $\mathfrak{L}_2^6(\lambda)$ is a non-multiplicative hom-Lie algebra.

## 4.8 Degenerations between hom-Lie algebras ($\mathfrak{L}_2^i$)

If $\mathfrak{L}_i^j \xrightarrow{\deg} \mathfrak{L}_h^k$ then $\mathrm{Der}(\mathfrak{L}_i^j) \leq \mathrm{Der}(\mathfrak{L}_h^k)$. Therefore, we can organize the hom-Lie algebras in the following way:

| Dim(Der) | hom-Lie algebra |
|---|---|
| 1 | $\mathfrak{L}_2^6(\lambda)$ |
| 2 | $\mathfrak{L}_2^4(\lambda)$, $\mathfrak{L}_2^5(\lambda)$, $\mathfrak{L}_2^2$ |
| 3 | $\mathfrak{L}_2^3(\lambda)$, $\mathfrak{L}_2^1$ |
| 4 | $\mathfrak{L}_2^0$ |

### 4.8.1 Degenerations

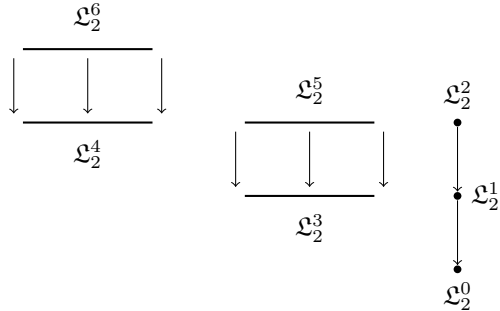

1. $\mathfrak{L}_2^6(\lambda) \xrightarrow{\deg} \mathfrak{L}_2^4(\lambda)$
   In fact, set
   $$g(t) = \begin{bmatrix} 1 & 0 & 0 \\ 0 & e^{-t} & 0 \\ 0 & 0 & e^{-t} \end{bmatrix}, \text{ with } t \in \mathbb{R}.$$
   We have $g(t) \in \mathrm{Aut}(\mathfrak{L}_2)$ and $g(t) \cdot \mathfrak{L}_2^6(\lambda)$ is the hom-Lie algebra $(\mathfrak{L}_2, A(\lambda, t))$ with
   $$A(t) = \begin{bmatrix} 0 & 0 & 0 \\ e^{-t} & 0 & 0 \\ 0 & \lambda & 0 \end{bmatrix}.$$
   It is easy to check that $(\mathfrak{L}_2, A(\lambda, t)) \to \mathfrak{L}_2^4(\lambda)$ as $t$ tends to infinity.

2. $\mathfrak{L}_2^5(\lambda) \xrightarrow{\deg} \mathfrak{L}_2^3(\lambda)$
   In fact, set
   $$g(t) = \begin{bmatrix} 1 & 0 & 0 \\ 0 & e^{-t} & 0 \\ 0 & 0 & e^{-t} \end{bmatrix}, \text{ with } t \in \mathbb{R}.$$
   We have $g(t) \in \mathrm{Aut}(\mathfrak{L}_2)$ and $g(t) \cdot \mathfrak{L}_2^5$ is the hom-Lie algebra $(\mathfrak{L}_2, A(t))$ with
   $$A(t) = \begin{bmatrix} 0 & 0 & 0 \\ 0 & 0 & \lambda \\ e^{-t} & 0 & 0 \end{bmatrix}.$$
   It is easy to check that $(\mathfrak{L}_2, A(t)) \to \mathfrak{L}_2^3(\lambda)$ as $t$ tends to infinity.

3. $\mathfrak{L}_2^2 \xrightarrow{\deg} \mathfrak{L}_2^1$
   In fact, set
   $$g(t) = \begin{bmatrix} 1 & 0 & 0 \\ 0 & e^{-t} & 1 \\ 0 & 0 & e^{-t} \end{bmatrix}, \text{ with } t \in \mathbb{R}.$$
   We have $g(t) \in \mathrm{Aut}(\mathfrak{L}_2)$ and $g(t) \cdot \mathfrak{L}_2^2$ is the hom-Lie algebra $(\mathfrak{L}_2, A(t))$ with
   $$A(t) = \begin{bmatrix} 0 & 0 & 0 \\ 1 & 0 & 0 \\ e^{-t} & 0 & 0 \end{bmatrix}.$$
   It is easy to check that $(\mathfrak{L}_2, A(t)) \to \mathfrak{L}_2^1$ as $t$ tends to infinity.

4. $\mathfrak{L}_2^1 \xrightarrow{\deg} \mathfrak{L}_2^0$
   In fact, set
   $$g(t) = \begin{bmatrix} 1 & 0 & 0 \\ 0 & e^{-t} & 0 \\ 0 & 0 & e^{-t} \end{bmatrix}, \text{ with } t \in \mathbb{R}.$$
   We have $g(t) \in \mathrm{Aut}(\mathfrak{L}_2)$ and $g(t) \cdot \mathfrak{L}_2^1$ is the hom-Lie algebra $(\mathfrak{L}_2, A(t))$ with
   $$A(t) = \begin{bmatrix} 0 & 0 & 0 \\ e^{-t} & 0 & 0 \\ 0 & 0 & 0 \end{bmatrix}.$$

It is easy to check that $(\mathfrak{L}_2, A(t)) \to \mathfrak{L}_2^0$ as $t$ tends to infinity.

5. $\mathfrak{L}_2^6(\lambda) \xrightarrow{\deg} \mathfrak{L}_2^4(\kappa)$ with $(\lambda \neq \kappa)$.

   Suppose, contrary to our claim, that $\mathfrak{L}_2^6(\lambda) \xrightarrow{\deg} \mathfrak{L}_2^4(\kappa)$ with $(\lambda \neq \kappa)$. In such case, we have

   $$\psi_{\alpha,\beta}(\mathfrak{L}_2^6(\lambda)) \xrightarrow{\deg} \psi_{\alpha,\beta}(\mathfrak{L}_2^4(\kappa)).$$

   By taking $\alpha = -\frac{1+\sqrt{2}}{\lambda}, \beta = \frac{\sqrt{2}-1}{\lambda}$ we have from 4.7.1 $\psi_{\alpha,\beta}(\mathfrak{L}_2^6(\lambda))$ is a Lie algebra isomorphic to $\mathfrak{n}_3(\mathbb{C})$, therefore, $\psi_{\alpha,\beta}(\mathfrak{L}_2^4(\kappa))$ is a Lie algebra which is isomorphic to one of the following: $\mathfrak{n}_3(\mathbb{C})$ or $\mathfrak{a}_3(\mathbb{C})$ (by Theorem 1.1). We must have $\psi_{\alpha,\beta}(\mathfrak{L}_2^4(\kappa))$ is isomorphic to $\mathfrak{n}_3(\mathbb{C})$ (from 4.5.1). But, we now apply again 4.5.1, to obtain $\{\alpha = -\frac{1-\sqrt{2}}{\kappa}, \beta = \frac{-1-\sqrt{2}}{\kappa}\}$ or $\{\alpha = -\frac{1+\sqrt{2}}{\kappa}, \beta = \frac{\sqrt{2}-1}{\lambda}\}$; this is a contradiction.

6. $\mathfrak{L}_2^6(\lambda) \xrightarrow{\deg} \mathfrak{L}_2^3(\kappa)$ with $(\kappa \in \mathbb{C}^\star)$.

   Suppose, contrary to our claim, that $\mathfrak{L}_2^6(\lambda) \xrightarrow{\deg} \mathfrak{L}_2^3(\kappa)$ with $(\kappa \in \mathbb{C}^\star)$. In such case, we have

   $$\psi_{\alpha,\beta}(\mathfrak{L}_2^6(\lambda)) \xrightarrow{\deg} \psi_{\alpha,\beta}(\mathfrak{L}_2^3(\kappa)).$$

   By taking $\alpha = (-1 \mp \sqrt{2})/\lambda$ and $\beta = (-1 \pm \sqrt{2})/\lambda$, we have from 4.7.1 that $\psi_{\alpha,\beta}(\mathfrak{L}_2^6(\lambda))$ is a Lie algebra isomorphic to $\mathfrak{n}_3$, therefore, $\psi_{\alpha,\beta}(\mathfrak{L}_2^3(\kappa))$ is a Lie algebra which is isomorphic to one of the following: $\mathfrak{n}_3$ or $\mathfrak{a}_3$ (by Theorem 1.1). But, from 4.4.1, $\psi_{\alpha,\beta}(\mathfrak{L}_2^3(\kappa))$ is isomorphic to $\mathfrak{r}_{3,1}$ or $\mathfrak{r}_3$; this is a contradiction.

7. $\mathfrak{L}_2^6(\lambda) \xrightarrow{\deg} \mathfrak{L}_2^0$ with $(\lambda \in \mathbb{C}^\star)$.

   Suppose, contrary to our claim, that $\mathfrak{L}_2^6(\lambda) \xrightarrow{\deg} \mathfrak{L}_2^0$ with $(\kappa \in \mathbb{C}^\star)$. In such case, we have

   $$\psi_{\alpha,\beta}(\mathfrak{L}_2^6(\lambda)) \xrightarrow{\deg} \psi_{\alpha,\beta}(\mathfrak{L}_2^0).$$

   By taking $\alpha = (-1 \mp \sqrt{2})/\lambda$ and $\beta = (-1 \pm \sqrt{2})/\lambda$, we have from 4.7.1 that $\psi_{\alpha,\beta}(\mathfrak{L}_2^6(\lambda))$ is a Lie algebra isomorphic to $\mathfrak{n}_3$, therefore, $\psi_{\alpha,\beta}(\mathfrak{L}_2^0)$ is a Lie algebra which is isomorphic to one of the following: $\mathfrak{n}_3$ or $\mathfrak{a}_3$ (by Theorem 1.1). But, from 4.1.1, $\psi_{\alpha,\beta}(\mathfrak{L}_2^0)$ is isomorphic to $\mathfrak{r}_3$; this is a contradiction.

8. $\mathfrak{L}_2^5(\lambda) \xrightarrow{\deg} \mathfrak{L}_2^3(\kappa)$ with $(\lambda \neq \kappa)$.

   Suppose, contrary to our claim, that $\mathfrak{L}_2^5(\lambda) \xrightarrow{\deg} \mathfrak{L}_2^3(\kappa)$ with $(\lambda \neq \kappa)$. In such case, we have

   $$\psi_{\alpha,\beta}(\mathfrak{L}_2^5(\lambda)) \xrightarrow{\deg} \psi_{\alpha,\beta}(\mathfrak{L}_2^3(\kappa)).$$

   By taking $\alpha = -\frac{1}{\lambda}$ and $\beta = 0$, we have from 4.6.1 that $\psi_{\alpha,\beta}(\mathfrak{L}_2^5(\lambda))$ is a Lie algebra isomorphic to $\mathfrak{r}_{3,1}$, therefore, $\psi_{\alpha,\beta}(\mathfrak{L}_2^3(\kappa))$ is a Lie algebra which is isomorphic to one of the following: $\mathfrak{r}_{3,1}$ or $\mathfrak{a}_3$ (by Theorem 1.1). We must have $\psi_{\alpha,\beta}(\mathfrak{L}_2^3(\kappa))$ is isomorphic to $\mathfrak{r}_{3,1}$ (from 4.4.1). But, we now apply again 4.4.1, to obtain $1 + \beta\kappa + \alpha\kappa = 0$ and therefore $\lambda = \kappa$; this is a contradiction.

9. $\mathfrak{L}_2^5(\lambda) \xrightarrow{\deg} \mathfrak{L}_2^0$.

   Suppose, contrary to our claim, that $\mathfrak{L}_2^5(\lambda) \xrightarrow{\deg} \mathfrak{L}_2^0$. In such case, we have

   $$\psi_{\alpha,\beta}(\mathfrak{L}_2^5(\lambda)) \xrightarrow{\deg} \psi_{\alpha,\beta}(\mathfrak{L}_2^0).$$

   By taking $\alpha = -\frac{1}{\lambda}$ and $\beta = 0$, we have from 4.6.1 that $\psi_{\alpha,\beta}(\mathfrak{L}_2^5(\lambda))$ is a Lie algebra isomorphic to $\mathfrak{r}_{3,1}$, therefore, $\psi_{\alpha,\beta}(\mathfrak{L}_2^0)$ is a Lie algebra which is isomorphic to one of the following: $\mathfrak{r}_{3,1}$ or $\mathfrak{a}_3$ (by Theorem 1.1). But, we now apply again 4.4.1 and we have $\psi_{\alpha,\beta}(\mathfrak{L}_2^0)$ is isomorphic to $\mathfrak{r}_3$; this is a contradiction.

10. $\mathfrak{L}_2^2 \xrightarrow{\text{deg}}\!\!\!\!\!\!\!/\ \  \mathfrak{L}_2^3(\lambda)$ with $\lambda \neq 0$.

    Suppose, contrary to our claim, that $\mathfrak{L}_2^2 \xrightarrow{\text{deg}} \mathfrak{L}_2^3(\lambda)$ with $\lambda \neq 0$. In such case, we have

    $$\phi_\beta(\mathfrak{L}_2^2) \xrightarrow{\text{deg}} \psi_\beta(\mathfrak{L}_2^3(\lambda)).$$

    By taking a $\beta \neq -1$, we have from 4.3.2 that $\phi_\beta(\mathfrak{L}_2^2)$ is a Lie algebra isomorphic to $\mathfrak{a}_3$, therefore, $\phi_\beta(\mathfrak{L}_2^3(\lambda))$ is a Lie algebra which is isomorphic to $\mathfrak{a}_3$. But, from 4.4.2, we have $\phi_\beta(\mathfrak{L}_2^3(\lambda))$ is isomorphic to $\mathfrak{n}_3$; this is a contradiction.

# 5 $\mathfrak{L}_3$: Bianchi type V Lie algebra, $\mathfrak{r}_{3,1}$

$$\mathfrak{L}_3 := \{[e_1, e_2] = e_2, [e_1, e_3] = e_3.$$

$$\text{Aut}(\mathfrak{L}_3) = \left\{ \begin{bmatrix} 1 & 0 & 0 \\ \hline x & & \\ y & & h \end{bmatrix} \in M(3, \mathbb{C}) : \begin{array}{l} h \in GL(2, \mathbb{C}), \\ x, y \in \mathbb{C}. \end{array} \right\}.$$

$$\text{Der}(\mathfrak{L}_3) = \left\{ \begin{bmatrix} 0 & 0 & 0 \\ \hline t_1 & & \\ t_2 & & B \end{bmatrix} \in M(3, \mathbb{C}) : \begin{array}{l} B \in M(2, \mathbb{C}), \\ t_1, t_2 \in \mathbb{C}. \end{array} \right\} \cong \mathfrak{sl}(2, \mathbb{C}) \ltimes \mathfrak{r}_{3,1}.$$

## 5.1 $\mathfrak{L}_3^0$

Let $A$ be the Zero map and let us denote by $\mathfrak{L}_3^0$ the hom-Lie algebra $(\mathfrak{L}_3, A)$. We have:
$\text{Aut}(\mathfrak{L}_3^0) = \text{Aut}(\mathfrak{L}_3)$.
$\text{Der}(\mathfrak{L}_3^0) = \text{Der}(\mathfrak{L}_3)$.

### 5.1.1 $\psi_{\alpha,\beta}(\mathfrak{L}_3^0)$

In this case, we have $\psi_{\alpha,\beta}(\mathfrak{L}_3^0) = \mathfrak{t}_{\alpha,\beta}$ is the Lie algebra $\mathfrak{r}_{3,1}$.

### 5.1.2 $\phi_\beta(\mathfrak{L}_3^0)$

In this case, we have $\phi_\beta(\mathfrak{L}_3^0) = \widetilde{\mathfrak{t}}_\beta$ is the 3-dimensional abelian Lie algebra.

### 5.1.3 $\rho(\mathfrak{L}_3^0)$

In this case, we have $\rho(\mathfrak{L}_3^0) = \widehat{\mathfrak{t}}$ is the 3-dimensional abelian Lie algebra.

## 5.2 $\mathfrak{L}_3^1$

Let $A$ be the endomorphism of $\mathbb{C}^3$ whose matrix representation with respect to the ordered basis $\{e_1, e_2, e_3\}$ is

$$\begin{bmatrix} 0 & 0 & 0 \\ 1 & 0 & 0 \\ 0 & 0 & 0 \end{bmatrix}$$

and let us denote by $\mathfrak{L}_3^1$ the hom-Lie algebra $(\mathfrak{L}_3, A)$. We have:

$$\text{Aut}(\mathfrak{L}_3^1) = \left\{ \begin{bmatrix} 1 & 0 & 0 \\ y & 1 & x \\ z & 0 & a \end{bmatrix} \in M(3, \mathbb{C}) : \begin{array}{l} a \in \mathbb{C}^*, \\ x, y, z \in \mathbb{C} \end{array} \right\}.$$

$$\text{Der}(\mathfrak{L}_3^1) = \left\{ \begin{bmatrix} 0 & 0 & 0 \\ t_1 & 0 & t_3 \\ t_2 & 0 & t_4 \end{bmatrix} \in M(3, \mathbb{C}) : t_1, t_2, t_3, t_4 \in \mathbb{C} \right\} \cong \mathbb{C} \ltimes \mathfrak{n}_3(\mathbb{C}).$$

### 5.2.1 $\psi_{\alpha,\beta}(\mathfrak{L}_3^1)$

In this case, we have $\psi_{\alpha,\beta}(\mathfrak{L}_3^1) = \mathfrak{t}_{\alpha,\beta}$ is the Lie algebra $\mathfrak{r}_{3,1}$.

### 5.2.2 $\phi_\beta(\mathfrak{L}_3^1)$

In this case, we have $\phi_\beta(\mathfrak{L}_3^1) = \widetilde{\mathfrak{t}}_\beta$ is the 3-dimensional abelian Lie algebra.

### 5.2.3 $\rho(\mathfrak{L}_3^1)$

In this case, we have $\rho(\mathfrak{L}_3^1) = \widehat{\mathfrak{t}}$ is the 3-dimensional abelian Lie algebra.

### 5.2.4 Identities

- $A[\cdot, \cdot] = 0$
- $[A\cdot, \cdot] + [\cdot, A\cdot] = 0$
  - $[A\cdot, A\cdot] = 0$

In particular, the hom-Lie algebra $\mathfrak{L}_3^1$ is a multiplicative hom-Lie algebra.

## 5.3 $\mathfrak{L}_3^2$

Let $A$ be the endomorphism of $\mathbb{C}^3$ whose matrix representation with respect to the ordered basis $\{e_1, e_2, e_3\}$ is

$$\begin{bmatrix} 0 & 0 & 0 \\ 0 & 0 & 1 \\ 0 & 0 & 0 \end{bmatrix}$$

and let us denote by $\mathfrak{L}_3^2$ the hom-Lie algebra $(\mathfrak{L}_3, A)$. We have:

$$\mathrm{Aut}(\mathfrak{L}_3^2) = \left\{ \begin{bmatrix} 1 & 0 & 0 \\ y & a & x \\ 0 & 0 & a \end{bmatrix} \in M(3, \mathbb{C}) : \begin{array}{c} a \in \mathbb{C}^* \\ x, y \in \mathbb{C} \end{array} \right\}.$$

$$\mathrm{Der}(\mathfrak{L}_3^2) = \left\{ \begin{bmatrix} 0 & 0 & 0 \\ t_3 & t_1 & t_2 \\ 0 & 0 & t_1 \end{bmatrix} \in M(3, \mathbb{C}) : t_1, t_2, t_3 \in \mathbb{C} \right\} \cong \mathfrak{r}_2(\mathbb{C}) \times \mathbb{C}.$$

### 5.3.1 $\psi_{\alpha,\beta}(\mathfrak{L}_3^2)$

In this case, we have $\psi_{\alpha,\beta}(\mathfrak{L}_3^2) = \mathfrak{t}_{\alpha,\beta}$:

$$\mathfrak{t}_{\alpha,\beta} = \{[e1, e2] = e2, [e1, e3] = (\beta + \alpha)\, e2 + e3.$$

We note that $\mathfrak{t}_{\alpha,\beta}$ is an almost abelian Lie algebra where $\mathrm{span}_{\mathbb{C}}\{e_2, e_3\}$ is a codimension 1 abelian ideal of $\mathfrak{t}_{\alpha,\beta}$; we can identify $\mathfrak{t}_{\alpha,\beta}$ the matrix

$$B = \begin{bmatrix} 1 & \beta + \alpha \\ 0 & 1 \end{bmatrix}.$$

The isomorphism class of $\mathfrak{t}_{\alpha,\beta}$ is determined by the eigenvalues of $B$:

1. $B$ has repeated eigenvalues if and only if $1 = 1$

   (a) 0 is a repeated eigenvalue of $B$ if and only if $1 = 0$
   
   (b) $B$ has a non-zero repeated eigenvalue if and only if $1 = 1$
   
      i. $\mathfrak{t}_{\alpha,\beta}$ is $\mathfrak{r}_{3,1}$ if and only if $B = \lambda \operatorname{Id}$ for some $\lambda \neq 0$ if and only if $\beta = -\alpha$
      
      ii. $\mathfrak{t}_{\alpha,\beta}$ is isomorphic to $\mathfrak{r}_3$ if and only if $B \neq \lambda \operatorname{Id}$ for all $\lambda \in \mathbb{C}$ if and only if $\beta \neq -\alpha$

2. $B$ has distinct eigenvalues eigenvalues if and only if $1 = 0$

### 5.3.2 $\phi_\beta(\mathfrak{L}_3^2)$

In this case, we have $\phi_\beta(\mathfrak{L}_3^2) = \widetilde{\mathfrak{t}}_\beta$:

$$\widetilde{\mathfrak{t}}_\beta = \{[e1, e3] = (\beta + 1)\, e2.$$

We note that $\widetilde{\mathfrak{t}}_\beta$ is an almost abelian Lie algebra where $\mathrm{span}_\mathbb{C}\{e_2, e_3\}$ is a codimension 1 abelian ideal of $\widetilde{\mathfrak{t}}_\beta$; we can identify $\widetilde{\mathfrak{t}}_\beta$ the matrix

$$\widetilde{B} = \begin{bmatrix} 0 & \beta+1 \\ 0 & 0 \end{bmatrix}.$$

The isomorphism class of $\widetilde{\mathfrak{t}}_\beta$ is determined by the eigenvalues of $\widetilde{B}$:

1. $\widetilde{B}$ has repeated eigenvalues if and only if $0 = 0$
    (a) $0$ is a repeated eigenvalue of $\widetilde{B}$ if and only if $0 = 0$
        i. $\widetilde{\mathfrak{t}}_\beta$ is the 3-dimensional abelian Lie algebra if and only if $\widetilde{B} = 0$ if and only if $\beta = -1$
        ii. $\widetilde{\mathfrak{t}}_\beta$ is isomorphic to $\mathfrak{n}_3(\mathbb{C})$ if and only if $\widetilde{B} \neq 0$ if and only if $\beta \neq -1$
    (b) $\widetilde{B}$ has a non-zero repeated eigenvalue if and only if $1 = 0$
2. $\widetilde{B}$ has distinct eigenvalues eigenvalues if and only if $1 = 0$

### 5.3.3 $\rho(\mathfrak{L}_3^2)$

In this case, we have $\rho(\mathfrak{L}_3^2) = \widehat{\mathfrak{t}}$:

$$\widehat{\mathfrak{t}} = \{[e1, e3] = e2.$$

which is isomorphic to $\mathfrak{n}_3$

### 5.3.4 Identities

- $A[\cdot, \cdot] = [A\cdot, \cdot] + [\cdot, A\cdot]$
- $A[A\cdot, \cdot]$
    - $[A\cdot, A\cdot] = 0$

The hom-Lie algebra $\mathfrak{L}_3^2$ is a non-multiplicative hom-Lie algebra.

## 5.4 $\mathfrak{L}_3^3$

Let $A$ be the endomorphism of $\mathbb{C}^3$ whose matrix representation with respect to the ordered basis $\{e_1, e_2, e_3\}$ is

$$\begin{bmatrix} 0 & 0 & 0 \\ 1 & 0 & 0 \\ 0 & 1 & 0 \end{bmatrix}$$

and let us denote by $\mathfrak{L}_3^3$ the hom-Lie algebra $(\mathfrak{L}_3, A)$. We have:

$$\mathrm{Aut}(\mathfrak{L}_3^3) = \left\{ \begin{bmatrix} 1 & 0 & 0 \\ x & 1 & 0 \\ y & x & 1 \end{bmatrix} \in M(3, \mathbb{C}) : x, y \in \mathbb{C} \right\}.$$

$$\mathrm{Der}(\mathfrak{L}_3^3) = \left\{ \begin{bmatrix} 0 & 0 & 0 \\ t_1 & 0 & 0 \\ t_2 & t_1 & 0 \end{bmatrix} \in M(3, \mathbb{C}) : t_1, t_2 \in \mathbb{C} \right\} \cong \mathbb{C}^2.$$

### 5.4.1   $\psi_{\alpha,\beta}(\mathfrak{L}_3^3)$

In this case, we have $\psi_{\alpha,\beta}(\mathfrak{L}_3^3) = \mathfrak{t}_{\alpha,\beta}$:

$$\mathfrak{t}_{\alpha,\beta} = \{[e1, e2] = e2 + (\alpha + \beta)\, e3, [e1, e3] = e3$$

We note that $\mathfrak{t}_{\alpha,\beta}$ is an almost abelian Lie algebra where $\mathrm{span}_{\mathbb{C}}\{e_2, e_3\}$ is a codimension 1 abelian ideal of $\mathfrak{t}_{\alpha,\beta}$; we can identify $\mathfrak{t}_{\alpha,\beta}$ the matrix

$$B = \begin{bmatrix} 1 & 0 \\ \alpha + \beta & 1 \end{bmatrix}.$$

The isomorphism class of $\mathfrak{t}_{\alpha,\beta}$ is determined by the eigenvalues of $B$:

1. $B$ has repeated eigenvalues if and only if $1 = 1$

    (a) 0 is a repeated eigenvalue of $B$ if and only if $1 = 0$
    
    (b) $B$ has a non-zero repeated eigenvalue if and only if $1 = 1$
    
      i. $\mathfrak{t}_{\alpha,\beta}$ is $\mathfrak{r}_{3,1}$ if and only if $B = \lambda\,\mathrm{Id}$ for some $\lambda \neq 0$ if and only if $\beta = -\alpha$
      
      ii. $\mathfrak{t}_{\alpha,\beta}$ is isomorphic to $\mathfrak{r}_3$ if and only if $B \neq \lambda\,\mathrm{Id}$ for all $\lambda \in \mathbb{C}$ if and only if $\beta \neq -\alpha$

2. $B$ has distinct eigenvalues eigenvalues if and only if $1 = 0$

### 5.4.2   $\phi_\beta(\mathfrak{L}_3^3)$

In this case, we have $\phi_\beta(\mathfrak{L}_3^3) = \widetilde{\mathfrak{t}}_\beta$:

$$\widetilde{\mathfrak{t}}_\beta = \{[e1, e2] = (1 + \beta)\, e3.$$

We note that $\widetilde{\mathfrak{t}}_\beta$ is an almost abelian Lie algebra where $\mathrm{span}_{\mathbb{C}}\{e_2, e_3\}$ is a codimension 1 abelian ideal of $\widetilde{\mathfrak{t}}_\beta$; we can identify $\widetilde{\mathfrak{t}}_\beta$ the matrix

$$\widetilde{B} = \begin{bmatrix} 0 & 0 \\ 1 + \beta & 0 \end{bmatrix}.$$

The isomorphism class of $\widetilde{\mathfrak{t}}_\beta$ is determined by the eigenvalues of $\widetilde{B}$:

1. $\widetilde{B}$ has repeated eigenvalues if and only if $0 = 0$

    (a) 0 is a repeated eigenvalue of $\widetilde{B}$ if and only if $0 = 0$
    
      i. $\widetilde{\mathfrak{t}}_\beta$ is the 3-dimensional abelian Lie algebra if and only if $\widetilde{B} = 0$ if and only if $\beta = -1$
      
      ii. $\widetilde{\mathfrak{t}}_\beta$ is isomorphic to $\mathfrak{n}_3(\mathbb{C})$ if and only if $\widetilde{B} \neq 0$ if and only if $\beta \neq -1$
    
    (b) $\widetilde{B}$ has a non-zero repeated eigenvalue if and only if $1 = 0$

2. $\widetilde{B}$ has distinct eigenvalues eigenvalues if and only if $1 = 0$

### 5.4.3   $\rho(\mathfrak{L}_3^3)$

In this case, we have $\rho(\mathfrak{L}_3^3) = \widehat{\mathfrak{t}}$:

$$\widehat{\mathfrak{t}}_\beta = \{[e1, e2] = e3.$$

which is isomorphic to $\mathfrak{n}_3$

### 5.4.4   Identities

- $A[\cdot,\cdot] = [A\cdot,\cdot] + [\cdot, A\cdot] = 0$
- $A^2[\cdot,\cdot] = 0$
- $[A\cdot, A\cdot] = 0$

The hom-Lie algebra $\mathfrak{L}_3^3$ is a non-multiplicative hom-Lie algebra.

## 5.5 Degenerations between hom-Lie algebras ($\mathfrak{L}_3^i$)

If $\mathfrak{L}_i^j \xrightarrow{\deg} \mathfrak{L}_h^k$ then $\mathrm{Der}(\mathfrak{L}_i^j) \leq \mathrm{Der}(\mathfrak{L}_h^k)$. Therefore, we can organize the hom-Lie algebras in the following way:

| Dim(Der) | hom-Lie algebra |
|---|---|
| 2 | $\mathfrak{L}_3^3$ |
| 3 | $\mathfrak{L}_3^2$ |
| 4 | $\mathfrak{L}_3^1$ |
| 6 | $\mathfrak{L}_3^0$ |

### 5.5.1 Degenerations

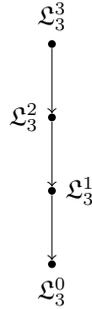

1. $\mathfrak{L}_3^3 \xrightarrow{\deg} \mathfrak{L}_3^2$
   In fact, set
   $$g(t) = \begin{bmatrix} 1 & 0 & 0 \\ 0 & 0 & \mathrm{e}^{-t} \\ 0 & \mathrm{e}^{-t} & \mathrm{e}^{-2t} \end{bmatrix}, \text{ with } t \in \mathbb{R}.$$
   We have $g(t) \in \mathrm{Aut}(\mathfrak{L}_3)$ and $g(t) \cdot \mathfrak{L}_3^3$ is the hom-Lie algebra $(\mathfrak{L}_3, A(t))$ with
   $$A(t) = \begin{bmatrix} 0 & 0 & 0 \\ 0 & -\mathrm{e}^{-t} & 1 \\ \mathrm{e}^{-t} & -\mathrm{e}^{-2t} & \mathrm{e}^{-t} \end{bmatrix}.$$
   It is easy to check that $(\mathfrak{L}_3, A(t)) \to \mathfrak{L}_3^2$ as $t$ tends to infinity.

2. $\mathfrak{L}_3^2 \xrightarrow{\deg} \mathfrak{L}_3^1$
   In fact, set
   $$g(t) = \begin{bmatrix} 1 & 0 & 0 \\ 0 & 1 & 0 \\ -\mathrm{e}^{t} & 0 & \mathrm{e}^{t} \end{bmatrix}, \text{ with } t \in \mathbb{R}.$$
   We have $g(t) \in \mathrm{Aut}(\mathfrak{L}_3)$ and $g(t) \cdot \mathfrak{L}_3^2$ is the hom-Lie algebra $(\mathfrak{L}_3, A(t))$ with
   $$A(t) = \begin{bmatrix} 0 & 0 & 0 \\ 1 & 0 & \mathrm{e}^{-t} \\ 0 & 0 & 0 \end{bmatrix}.$$
   It is easy to check that $(\mathfrak{L}_3, A(t)) \to \mathfrak{L}_3^1$ as $t$ tends to infinity.

3. $\mathfrak{L}_3^1 \xrightarrow{\text{deg}} \mathfrak{L}_3^0$

In fact, set
$$g(t) = \begin{bmatrix} 1 & 0 & 0 \\ 0 & e^{-t} & 0 \\ 0 & 0 & 1 \end{bmatrix}, \text{ with } t \in \mathbb{R}.$$

We have $g(t) \in \text{Aut}(\mathfrak{L}_3)$ and $g(t) \cdot \mathfrak{L}_3^1$ is the hom-Lie algebra $(\mathfrak{L}_3, A(t))$ with
$$A(t) = \begin{bmatrix} 0 & 0 & 0 \\ e^{-t} & 0 & 0 \\ 0 & 0 & 0 \end{bmatrix}$$

It is easy to check that $(\mathfrak{L}_3, A(t)) \to \mathfrak{L}_3^0$ as $t$ tends to infinity.

# 6 $\mathfrak{L}_4$: The Lie algebra of the isometry group of 2-dimensional Minkowski spacetime, $\mathfrak{r}_{3,-1}$

$$\mathfrak{L}_1 := \{[e_1, e_2] = e_2, [e_1, e_3] = -e_3$$

$$\mathrm{Aut}(\mathfrak{L}_4) = \left\{ \begin{bmatrix} 1 & 0 & 0 \\ x & a & 0 \\ y & 0 & b \end{bmatrix} \in M(3, \mathbb{C}) : \begin{matrix} a, b \in \mathbb{C}^* \\ x, y \in \mathbb{C} \end{matrix} \right\} \dot\cup \left\{ \begin{bmatrix} -1 & 0 & 0 \\ x & 0 & a \\ y & b & 0 \end{bmatrix} \in M(3, \mathbb{C}) : \begin{matrix} a, b \in \mathbb{C}^* \\ x, y \in \mathbb{C} \end{matrix} \right\}.$$

$$\mathrm{Der}(\mathfrak{L}_4) = \left\{ \begin{bmatrix} 0 & 0 & 0 \\ t_3 & t_1 & 0 \\ t_4 & 0 & t_2 \end{bmatrix} \in M(3, \mathbb{C}) : t_1, t_2, t_3, t_4 \in \mathbb{C} \right\} \cong \mathfrak{r}_2(\mathbb{C}) \times \mathfrak{r}_2(\mathbb{C}).$$

## 6.1 $\mathfrak{L}_4^0$

Let $A$ be the Zero map and let us denote by $\mathfrak{L}_4^0$ the hom-Lie algebra $(\mathfrak{L}_4, A)$. We have:
$\mathrm{Aut}(\mathfrak{L}_4^0) = \mathrm{Aut}(\mathfrak{L}_4)$.
$\mathrm{Der}(\mathfrak{L}_4^0) = \mathrm{Der}(\mathfrak{L}_4)$.

### 6.1.1 $\psi_{\alpha,\beta}(\mathfrak{L}_4^0)$

In this case, we have $\psi_{\alpha,\beta}(\mathfrak{L}_4^0) = \mathfrak{t}_{\alpha,\beta}$ is the Lie algebra $\mathfrak{r}_{3,-1}$

### 6.1.2 $\phi_\beta(\mathfrak{L}_4^0)$

In this case, we have $\phi_\beta(\mathfrak{L}_4^0) = \widetilde{\mathfrak{t}}_\beta$ is the 3-dimensional abelian Lie algebra.

### 6.1.3 $\rho(\mathfrak{L}_4^0)$

In this case, we have $\rho(\mathfrak{L}_4^0) = \widehat{\mathfrak{t}}$ is the 3-dimensional abelian Lie algebra.

## 6.2 $\mathfrak{L}_4^1$

Let $A$ be the endomorphism of $\mathbb{C}^3$ whose matrix representation with respect to the ordered basis $\{e_1, e_2, e_3\}$ is

$$\begin{bmatrix} 0 & 0 & 0 \\ 1 & 0 & 0 \\ 0 & 0 & 0 \end{bmatrix}$$

and let us denote by $\mathfrak{L}_4^1$ the hom-Lie algebra $(\mathfrak{L}_4, A)$. We have:

$$\mathrm{Aut}(\mathfrak{L}_4^1) = \left\{ \begin{bmatrix} 1 & 0 & 0 \\ x & 1 & 0 \\ y & 0 & a \end{bmatrix} \in M(3, \mathbb{C}) : \begin{matrix} a \in \mathbb{C}^* \\ x, y \in \mathbb{C} \end{matrix} \right\}.$$

$$\mathrm{Der}(\mathfrak{L}_4^1) = \left\{ \begin{bmatrix} 0 & 0 & 0 \\ t_2 & 0 & 0 \\ t_3 & 0 & t_1 \end{bmatrix} \in M(3, \mathbb{C}) : t_1, t_2, t_3 \in \mathbb{C} \right\} \cong \mathfrak{r}_2(\mathbb{C}) \times \mathbb{C}..$$

### 6.2.1 $\psi_{\alpha,\beta}(\mathfrak{L}_4^1)$

In this case, we have $\psi_{\alpha,\beta}(\mathfrak{L}_4^1) = \mathfrak{t}_{\alpha,\beta}$ is the Lie algebra $\mathfrak{r}_{3,-1}$

### 6.2.2 $\phi_\beta(\mathfrak{L}_4^1)$

In this case, we have $\phi_\beta(\mathfrak{L}_4^1) = \widetilde{\mathfrak{t}}_\beta$ is the 3-dimensional abelian Lie algebra.

### 6.2.3 $\rho(\mathfrak{L}_4^1)$

In this case, we have $\rho(\mathfrak{L}_4^1) = \widehat{\mathfrak{t}}$ is the 3-dimensional abelian Lie algebra.

### 6.2.4 Identities

- $A[\cdot,\cdot] = 0$
- $[A\cdot,\cdot] + [\cdot, A\cdot] = 0$
    - $[A\cdot, A\cdot] = 0$

In particular, the hom-Lie algebra $\mathfrak{L}_4^1$ is a multiplicative hom-Lie algebra.

## 6.3 $\mathfrak{L}_4^2$

Let $A$ be the endomorphism of $\mathbb{C}^3$ whose matrix representation with respect to the ordered basis $\{e_1, e_2, e_3\}$ is

$$\begin{bmatrix} 0 & 0 & 0 \\ 1 & 0 & 0 \\ 1 & 0 & 0 \end{bmatrix}$$

and let us denote by $\mathfrak{L}_4^2$ the hom-Lie algebra $(\mathfrak{L}_4, A)$. We have:

$$\mathrm{Aut}(\mathfrak{L}_4^2) = \left\{ \begin{bmatrix} 1 & 0 & 0 \\ x & 1 & 0 \\ y & 0 & 1 \end{bmatrix} \in M(3,\mathbb{C}) : x, y \in \mathbb{C} \right\} \dot\cup \left\{ \begin{bmatrix} -1 & 0 & 0 \\ x & 0 & -1 \\ y & -1 & 0 \end{bmatrix} \in M(3,\mathbb{C}) : x, y \in \mathbb{C} \right\}.$$

$$\mathrm{Der}(\mathfrak{L}_4^2) = \left\{ \begin{bmatrix} 0 & 0 & 0 \\ t_1 & 0 & 0 \\ t_2 & 0 & 0 \end{bmatrix} \in M(3,\mathbb{C}) : t_1, t_2 \in \mathbb{C} \right\} \cong \mathbb{C}^2.$$

### 6.3.1 $\psi_{\alpha,\beta}(\mathfrak{L}_4^2)$

In this case, we have $\psi_{\alpha,\beta}(\mathfrak{L}_4^2) = \mathfrak{t}_{\alpha,\beta}$ is the Lie algebra $\mathfrak{r}_{3,-1}$

### 6.3.2 $\phi_\beta(\mathfrak{L}_4^2)$

In this case, we have $\phi_\beta(\mathfrak{L}_4^2) = \widetilde{\mathfrak{t}}_\beta$ is the 3-dimensional abelian Lie algebra.

### 6.3.3 $\rho(\mathfrak{L}_4^2)$

In this case, we have $\rho(\mathfrak{L}_4^2) = \widehat{\mathfrak{t}}$ is the 3-dimensional abelian Lie algebra.

### 6.3.4 Identities

- $A[\cdot,\cdot] = 0$
- $[A\cdot,\cdot] + [\cdot, A\cdot] = 0$
    - $[A\cdot, A\cdot] = 0$

In particular, the hom-Lie algebra $\mathfrak{L}_4^2$ is a multiplicative hom-Lie algebra.

## 6.4 $\mathfrak{L}_4^3$

Let $A$ be the endomorphism of $\mathbb{C}^3$ whose matrix representation with respect to the ordered basis $\{e_1, e_2, e_3\}$ is

$$\begin{bmatrix} 0 & 0 & 0 \\ 0 & 0 & 1 \\ 0 & 0 & 0 \end{bmatrix}$$

and let us denote by $\mathfrak{L}_4^3$ the hom-Lie algebra $(\mathfrak{L}_4, A)$. We have:

$$\mathrm{Aut}(\mathfrak{L}_4^3) = \left\{ \begin{bmatrix} 1 & 0 & 0 \\ x & a & 0 \\ 0 & 0 & a \end{bmatrix} \in M(3, \mathbb{C}) : \begin{matrix} a \in \mathbb{C}^* \\ x \in \mathbb{C} \end{matrix} \right\}.$$

$$\mathrm{Der}(\mathfrak{L}_4^3) = \left\{ \begin{bmatrix} 0 & 0 & 0 \\ t_2 & t_1 & 0 \\ 0 & 0 & t_1 \end{bmatrix} \in M(3, \mathbb{C}) : t_1, t_2 \in \mathbb{C} \right\} \cong \mathbb{C}^2.$$

### 6.4.1 $\psi_{\alpha,\beta}(\mathfrak{L}_4^3)$

In this case, we have $\psi_{\alpha,\beta}(\mathfrak{L}_4^3) = \mathfrak{t}_{\alpha,\beta}$:

$$\mathfrak{t}_{\alpha,\beta} = \{[e1, e2] = e2, [e1, e3] = (\beta - \alpha)\, e2 - e3.$$

We note that $\mathfrak{t}_{\alpha,\beta}$ is an almost abelian Lie algebra where $\mathrm{span}_{\mathbb{C}}\{e_2, e_3\}$ is a codimension 1 abelian ideal of $\mathfrak{t}_{\alpha,\beta}$; we can identify $\mathfrak{t}_{\alpha,\beta}$ the matrix

$$B = \begin{bmatrix} 1 & \beta - \alpha \\ 0 & -1 \end{bmatrix}.$$

The isomorphism class of $\mathfrak{t}_{\alpha,\beta}$ is determined by the eigenvalues of $B$:

1. $B$ has repeated eigenvalues if and only if $1 = 0$
2. $B$ has distinct eigenvalues eigenvalues if and only if $1 = 1$
   (a) $\mathfrak{t}_{\alpha,\beta}$ is isomorphic to $\mathfrak{r}_2(\mathbb{C}) \times \mathbb{C}$ if and only if $0$ is an eigenvalue of $B$ with algebraic multiplicity 1 if and only if $0 = 1$
   (b) $\mathfrak{t}_{\alpha,\beta}$ is isomorphic to $\mathfrak{r}_{3,-1}$ if and only if $\mathrm{Trace}(B) = 0$ if and only if $1 = 1$.
   (c) $\mathfrak{t}_{\alpha,\beta}$ is isomorphic to $\mathfrak{r}_{3,z}$ for some $z$ such that $z^2 \neq 1$ and $z \neq 0$ if and only if $1 = 0$.

### 6.4.2 $\phi_\beta(\mathfrak{L}_4^3)$

In this case, we have $\phi_\beta(\mathfrak{L}_4^3) = \widetilde{\mathfrak{t}}_\beta$:

$$\widetilde{\mathfrak{t}}_\beta = \{[e1, e3] = (\beta - 1)\, e2.$$

We note that $\widetilde{\mathfrak{t}}_\beta$ is an almost abelian Lie algebra where $\mathrm{span}_{\mathbb{C}}\{e_2, e_3\}$ is a codimension 1 abelian ideal of $\widetilde{\mathfrak{t}}_\beta$; we can identify $\widetilde{\mathfrak{t}}_\beta$ the matrix

$$\widetilde{B} = \begin{bmatrix} 0 & \beta - 1 \\ 0 & 0 \end{bmatrix}.$$

The isomorphism class of $\widetilde{\mathfrak{t}}_\beta$ is determined by the eigenvalues of $\widetilde{B}$:

1. $\widetilde{B}$ has repeated eigenvalues if and only if $0 = 0$

(a) 0 is a repeated eigenvalue of $\widetilde{B}$ if and only if $0 = 0$
    i. $\widetilde{\mathfrak{t}}_\beta$ is the 3-dimensional abelian Lie algebra if and only if $\widetilde{B} = 0$ if and only if $\beta = 1$
    ii. $\widetilde{\mathfrak{t}}_\beta$ is isomorphic to $\mathfrak{n}_3(\mathbb{C})$ if and only if $\widetilde{B} \neq 0$ if and only if $\beta \neq 1$
  (b) $\widetilde{B}$ has a non-zero repeated eigenvalue if and only if $1 = 0$
2. $\widetilde{B}$ has distinct eigenvalues eigenvalues if and only if $1 = 0$

### 6.4.3 $\rho(\mathfrak{L}_4^3)$

In this case, we have $\rho(\mathfrak{L}_4^3) = \widehat{\mathfrak{t}}$:
$$\widehat{\mathfrak{t}} = \{[e1, e3] = e2.$$
which is isomorphic to $\mathfrak{n}_3$

### 6.4.4 Identities

- $A[\cdot, \cdot] + [A\cdot, \cdot] + [\cdot, A\cdot] = 0$
- $A[A\cdot, \cdot] = 0$
  - $[A\cdot, A\cdot] = 0$

The hom-Lie algebra $\mathfrak{L}_4^3$ is a non-multiplicative hom-Lie algebra.

## 6.5 $\mathfrak{L}_4^4(\lambda)$

Let $A$ be the endomorphism of $\mathbb{C}^3$ whose matrix representation with respect to the ordered basis $\{e_1, e_2, e_3\}$ is
$$\begin{bmatrix} 0 & 0 & 0 \\ 0 & \lambda & \lambda \\ 0 & -\lambda & -\lambda \end{bmatrix}$$
and let us denote by $\mathfrak{L}_4^4(\lambda)$ the hom-Lie algebra $(\mathfrak{L}_4, A)$. We have:
$$\mathrm{Aut}(\mathfrak{L}_4^4(\lambda)) = \left\{ \begin{bmatrix} 1 & 0 & 0 \\ -x & a & 0 \\ x & 0 & a \end{bmatrix} \in M(3, \mathbb{C}) : \begin{matrix} a \in \mathbb{C}^* \\ x \in \mathbb{C} \end{matrix} \right\}.$$
$$\mathrm{Der}(\mathfrak{L}_4^4(\lambda)) = \left\{ \begin{bmatrix} 0 & 0 & 0 \\ -t_2 & t_1 & 0 \\ t_2 & 0 & t_1 \end{bmatrix} \in M(3, \mathbb{C}) : t_1, t_2 \in \mathbb{C} \right\} \cong \mathfrak{r}_2(\mathbb{C}).$$

We note that the hom-Lie algebra $\mathfrak{L}_4^4(\lambda)$ is isomorphic to $\mathfrak{L}_4^4(-\lambda)$. In fact,
$$\begin{bmatrix} 0 & 0 & 0 \\ 0 & \lambda & \lambda \\ 0 & -\lambda & -\lambda \end{bmatrix} \text{ is similar to } \begin{bmatrix} 0 & 0 & 0 \\ 0 & -\lambda & -\lambda \\ 0 & \lambda & \lambda \end{bmatrix}$$
via the invertible matrix $g \in \mathrm{Aut}(\mathfrak{L}_4)$:
$$g = \begin{bmatrix} -1 & 0 & 0 \\ 0 & 0 & -1 \\ 0 & -1 & 0 \end{bmatrix}$$

### 6.5.1 $\psi_{\alpha,\beta}(\mathfrak{L}_4^4(\lambda))$

In this case, we have $\psi_{\alpha,\beta}(\mathfrak{L}_4^4(\lambda)) = \mathfrak{t}_{\alpha,\beta}$:

$$\mathfrak{t}_{\alpha,\beta} = \{[e1, e2] = (1 + \alpha\lambda + \beta\lambda)\, e2 + (-\alpha\lambda + \beta\lambda)\, e3, [e1, e3] = (-\alpha\lambda + \beta\lambda)\, e2 + (-1 + \alpha\lambda + \beta\lambda)\, e3.$$

We note that $\mathfrak{t}_{\alpha,\beta}$ is an almost abelian Lie algebra where $\mathrm{span}_{\mathbb{C}}\{e_2, e_3\}$ is a codimension 1 abelian ideal of $\mathfrak{t}_{\alpha,\beta}$; we can identify $\mathfrak{t}_{\alpha,\beta}$ the matrix

$$B = \begin{bmatrix} 1 + \alpha\lambda + \beta\lambda & -\alpha\lambda + \beta\lambda \\ -\alpha\lambda + \beta\lambda & -1 + \alpha\lambda + \beta\lambda \end{bmatrix}.$$

The isomorphism class of $\mathfrak{t}_{\alpha,\beta}$ is determined by the eigenvalues of $B$:

1. $B$ has repeated eigenvalues if and only if $1 + (\alpha - \beta)^2 \lambda^2 = 0$

    (a) 0 is a repeated eigenvalue of $B$ if and only if $\alpha = -\beta = \pm\frac{\sqrt{-1}}{2\lambda}$
    
    i. $\mathfrak{t}_{\alpha,\beta}$ is the 3-dimensional abelian Lie algebra if and only if $B = 0$ if and only if $1 = 0$
    
    ii. $\mathfrak{t}_{\alpha,\beta}$ is isomorphic to $\mathfrak{n}_3(\mathbb{C})$ if and only if $B \neq 0$ if and only if $\alpha = -\beta = \pm\frac{\sqrt{-1}}{2\lambda}$
    
    (b) $B$ has a non-zero repeated eigenvalue if and only if
    $\left\{\alpha = -\frac{-1 \pm \sqrt{-1}s}{2s\lambda}, \beta = \frac{1 \pm \sqrt{-1}s}{2s\lambda}\right\}$ for any $s \in \mathbb{C} \setminus \{0\}$.
    
    i. $\mathfrak{t}_{\alpha,\beta}$ is $\mathfrak{r}_{3,1}$ if and only if $B = t\,\mathrm{Id}$ for some $t \neq 0$ if and only if $1 = 0$
    
    ii. $\mathfrak{t}_{\alpha,\beta}$ is isomorphic to $\mathfrak{r}_3$ if and only if $B \neq t\,\mathrm{Id}$ for all $t \in \mathbb{C}$ if and only if $\left\{\alpha = -\frac{-1 \pm \sqrt{-1}s}{2s\lambda}, \beta = \frac{1 \pm \sqrt{-1}s}{2s\lambda}\right\}$ for any $s \in \mathbb{C} \setminus \{0\}$.

2. $B$ has distinct eigenvalues if and only if $1 + (\alpha - \beta)^2 \lambda^2 \neq 0$

    (a) $\mathfrak{t}_{\alpha,\beta}$ is isomorphic to $\mathfrak{r}_2(\mathbb{C}) \times \mathbb{C}$ if and only if 0 is an eigenvalue of $B$ with algebraic multiplicity 1 if and only if
    $\left\{\alpha = -\frac{-1 \pm \sqrt{1-4s^2}}{4s\lambda}, \beta = \frac{1 \pm \sqrt{1-4s^2}}{4s\lambda}\right\}$ for any $s \in \mathbb{C} \setminus \{0\}$.
    
    (b) $\mathfrak{t}_{\alpha,\beta}$ is isomorphic to $\mathfrak{r}_{3,-1}$ if and only if $\mathrm{Trace}(B) = 0$ if and only if
    $\left\{\alpha = \mp\frac{\sqrt{s-s^2}}{2s\lambda}, \beta = \pm\frac{\sqrt{s-s^2}}{2s\lambda}\right\}$ for any $s \in \mathbb{C} \setminus \{0\}$.
    
    (c) $\mathfrak{t}_{\alpha,\beta}$ is isomorphic to $\mathfrak{r}_{3,z}$ for some $z$ such that $z(z^2 - 1) \neq 0$ if and only if
    $\left\{\alpha = \frac{s_1 \mp \sqrt{s_1 s_2^2(1-s_1)}}{2s_1 s_2 \lambda}, \beta = \frac{s_1 \pm \sqrt{s_1 s_2^2(1-s_1)}}{2s_1 s_2 \lambda}\right\}$ for any $s_1, s_2 \in \mathbb{C} \setminus \{0\}$ such that $s_1 - s_2^2 \neq 0$.

### 6.5.2 $\phi_\beta(\mathfrak{L}_4^4(\lambda))$

In this case, we have $\phi_\beta(\mathfrak{L}_4^4(\lambda)) = \widetilde{\mathfrak{t}_\beta}$:

$$\widetilde{\mathfrak{t}_\beta} = \{[e1, e2] = (\lambda + \beta\lambda)\, e2 + (-\lambda + \beta\lambda)\, e3, [e1, e3] = (-\lambda + \beta\lambda)\, e2 + (\lambda + \beta\lambda)\, e3.$$

We note that $\widetilde{\mathfrak{t}_\beta}$ is an almost abelian Lie algebra where $\mathrm{span}_{\mathbb{C}}\{e_2, e_3\}$ is a codimension 1 abelian ideal of $\widetilde{\mathfrak{t}_\beta}$; we can identify $\widetilde{\mathfrak{t}_\beta}$ the matrix

$$\widetilde{B} = \begin{bmatrix} \lambda + \beta\lambda & -\lambda + \beta\lambda \\ -\lambda + \beta\lambda & \lambda + \beta\lambda \end{bmatrix}.$$

The isomorphism class of $\widetilde{\mathfrak{t}_\beta}$ is determined by the eigenvalues of $\widetilde{B}$:

1. $\widetilde{B}$ has repeated eigenvalues if and only if $\beta = 1$

(a) 0 is a repeated eigenvalue of $\widetilde{B}$ if and only if $\lambda = 0$ (False)

(b) $\widetilde{B}$ has a non-zero repeated eigenvalue if and only if $\beta = 1$

  i. $\widetilde{\mathfrak{t}}_\beta$ is $\mathfrak{r}_{3,1}$ if and only if $\widetilde{B} = t\,\mathrm{Id}$ for some $\lambda \neq 0$ if and only if $\lambda = \lambda$
  
  ii. $\widetilde{\mathfrak{t}}_\beta$ is isomorphic to $\mathfrak{r}_3$ if and only if $\widetilde{B} \neq t\,\mathrm{Id}$ for all $\lambda \in \mathbb{C}$ if and only if $1 = 0$

2. $\widetilde{B}$ has distinct eigenvalues eigenvalues if and only if $\beta \neq 1$

    (a) $\widetilde{\mathfrak{t}}_\beta$ is isomorphic to $\mathfrak{r}_2(\mathbb{C}) \times \mathbb{C}$ if and only if 0 is an eigenvalue of $B$ with algebraic multiplicity 1 if and only if $\beta = 0$
    
    (b) $\widetilde{\mathfrak{t}}_\beta$ is isomorphic to $\mathfrak{r}_{3,-1}$ if and only if $\mathrm{Trace}(\widetilde{B}) = 0$ if and only if $\beta = -1$.
    
    (c) $\widetilde{\mathfrak{t}}_\beta$ is isomorphic to $\mathfrak{r}_{3,z}$ for some $z$ such that $z(z^2 - 1) \neq 0$ if and only if $\beta \neq 1$ and $\beta \neq -1$ and $\beta \neq 0$.

### 6.5.3   $\rho(\mathfrak{L}_4^4(\lambda))$

In this case, we have $\rho(\mathfrak{L}_4^4(\lambda)) = \widehat{\mathfrak{t}}$:

$$\widehat{\mathfrak{t}}_\beta = \{[e1, e2] = \lambda\,e2 + \lambda\,e3, [e1, e3] = \lambda\,e2 + \lambda\,e3.$$

which is isomorphic to $\mathfrak{r}_2 \times \mathbb{C}$

### 6.5.4   Identities

- $[A\cdot, A\cdot] = 0$

The hom-Lie algebra $\mathfrak{L}_4^4(\lambda)$ is a non-multiplicative hom-Lie algebra.

## 6.6   $\mathfrak{L}_4^5$

Let $A$ be the endomorphism of $\mathbb{C}^3$ whose matrix representation with respect to the ordered basis $\{e_1, e_2, e_3\}$ is

$$\begin{bmatrix} 0 & 0 & 0 \\ 1 & 0 & 0 \\ 0 & 1 & 0 \end{bmatrix}$$

and let us denote by $\mathfrak{L}_4^5$ the hom-Lie algebra $(\mathfrak{L}_4, A)$. We have:

$$\mathrm{Aut}(\mathfrak{L}_4^5) = \left\{ \begin{bmatrix} 1 & 0 & 0 \\ 0 & 1 & 0 \\ x & 0 & 1 \end{bmatrix} \in M(3, \mathbb{C}) : x \in \mathbb{C} \right\}.$$

$$\mathrm{Der}(\mathfrak{L}_4^5) = \left\{ \begin{bmatrix} 0 & 0 & 0 \\ 0 & 0 & 0 \\ t_1 & 0 & 0 \end{bmatrix} \in M(3, \mathbb{C}) : t_1 \in \mathbb{C} \right\} \cong \mathbb{C}.$$

### 6.6.1   $\psi_{\alpha,\beta}(\mathfrak{L}_4^5)$

In this case, we have $\psi_{\alpha,\beta}(\mathfrak{L}_4^5) = \mathfrak{t}_{\alpha,\beta}$:

$$\mathfrak{t}_{\alpha,\beta} = [e1, e2] = e2 + (\alpha - \beta)\,e3, [e1, e3] = -e3$$

We note that $\mathfrak{t}_{\alpha,\beta}$ is an almost abelian Lie algebra where $\mathrm{span}_\mathbb{C}\{e_2, e_3\}$ is a codimension 1 abelian ideal of $\mathfrak{t}_{\alpha,\beta}$; we can identify $\mathfrak{t}_{\alpha,\beta}$ the matrix

$$B = \begin{bmatrix} 1 & 0 \\ \alpha - \beta & -1 \end{bmatrix}.$$

The isomorphism class of $\mathfrak{t}_{\alpha,\beta}$ is determined by the eigenvalues of $B$:

1. $B$ has repeated eigenvalues if and only if $1 = 0$

2. $B$ has distinct eigenvalues eigenvalues if and only if $1 = 1$

    (a) $\mathfrak{t}_{\alpha,\beta}$ is isomorphic to $\mathfrak{r}_2(\mathbb{C}) \times \mathbb{C}$ if and only if $0$ is an eigenvalue of $B$ with algebraic multiplicity 1 if and only if $0 = 1$
    
    (b) $\mathfrak{t}_{\alpha,\beta}$ is isomorphic to $\mathfrak{r}_{3,-1}$ if and only if $\mathrm{Trace}(B) = 0$ if and only if $1 = 1$.
    
    (c) $\mathfrak{t}_{\alpha,\beta}$ is isomorphic to $\mathfrak{r}_{3,z}$ for some $z$ such that $z^2 \neq 1$ and $z \neq 0$ if and only if $1 = 0$.

### 6.6.2 $\phi_\beta(\mathfrak{L}_4^5)$

In this case, we have $\phi_\beta(\mathfrak{L}_4^5) = \widetilde{\mathfrak{t}}_\beta$:

$$\widetilde{\mathfrak{t}}_\beta = \{[e1, e2] = (1 - \beta)\, e3.$$

We note that $\widetilde{\mathfrak{t}}_\beta$ is an almost abelian Lie algebra where $\mathrm{span}_\mathbb{C}\{e_2, e_3\}$ is a codimension 1 abelian ideal of $\widetilde{\mathfrak{t}}_\beta$; we can identify $\widetilde{\mathfrak{t}}_\beta$ the matrix

$$\widetilde{B} = \begin{bmatrix} 0 & 0 \\ 1 - \beta & 0 \end{bmatrix}.$$

The isomorphism class of $\widetilde{\mathfrak{t}}_\beta$ is determined by the eigenvalues of $\widetilde{B}$:

1. $\widetilde{B}$ has repeated eigenvalues if and only if $0 = 0$

    (a) $0$ is a repeated eigenvalue of $\widetilde{B}$ if and only if $0 = 0$
    
    i. $\widetilde{\mathfrak{t}}_\beta$ is the 3-dimensional abelian Lie algebra if and only if $\widetilde{B} = 0$ if and only if $\beta = 1$
    
    ii. $\widetilde{\mathfrak{t}}_\beta$ is isomorphic to $\mathfrak{n}_3(\mathbb{C})$ if and only if $\widetilde{B} \neq 0$ if and only if $\beta \neq 1$
    
    (b) $\widetilde{B}$ has a non-zero repeated eigenvalue if and only if $1 = 0$

2. $\widetilde{B}$ has distinct eigenvalues eigenvalues if and only if $1 = 0$

### 6.6.3 $\rho(\mathfrak{L}_4^5)$

In this case, we have $\rho(\mathfrak{L}_4^5) = \widehat{\mathfrak{t}}$:

$$\widehat{\mathfrak{t}} = \{[e1, e2] = -e3.$$

which is isomorphic to $\mathfrak{n}_3$

### 6.6.4 Identities

- $A[\cdot, \cdot] + [A\cdot, \cdot] + [\cdot, A\cdot] = 0$
- $A^2[\cdot, \cdot] = 0$
- $[A\cdot, A\cdot] = 0$

The hom-Lie algebra $\mathfrak{L}_4^5$ is a non-multiplicative hom-Lie algebra.

## 6.7 $\mathfrak{L}_4^6(\lambda)$

Let $A$ be the endomorphism of $\mathbb{C}^3$ whose matrix representation with respect to the ordered basis $\{e_1, e_2, e_3\}$ is

$$\begin{bmatrix} 0 & 0 & 0 \\ 1 & \lambda & \lambda \\ 0 & -\lambda & -\lambda \end{bmatrix}$$

and let us denote by $\mathfrak{L}_4^6(\lambda)$ the hom-Lie algebra $(\mathfrak{L}_4, A)$. We have:

$$\text{Aut}(\mathfrak{L}_4^6(\lambda)) = \left\{ \begin{bmatrix} 1 & 0 & 0 \\ x & 1 & 0 \\ -x & 0 & 1 \end{bmatrix} \in M(3, \mathbb{C}) : x \in \mathbb{C} \right\}.$$

$$\text{Der}(\mathfrak{L}_4^6(\lambda)) = \left\{ \begin{bmatrix} 0 & 0 & 0 \\ t_1 & 0 & 0 \\ -t_1 & 0 & 0 \end{bmatrix} \in M(3, \mathbb{C}) : t_1 \in \mathbb{C} \right\} \cong \mathbb{C}.$$

We note that the hom-Lie algebra $\mathfrak{L}_4^6(\lambda)$ is isomorphic to $\mathfrak{L}_4^6(-\lambda)$. In fact,

$$\begin{bmatrix} 0 & 0 & 0 \\ 1 & \lambda & \lambda \\ 0 & -\lambda & -\lambda \end{bmatrix} \text{ is similar to } \begin{bmatrix} 0 & 0 & 0 \\ 1 & -\lambda & -\lambda \\ 0 & \lambda & \lambda \end{bmatrix}$$

via the invertible matrix $g \in \text{Aut}(\mathfrak{L}_4)$:

$$g = \begin{bmatrix} -1 & 0 & 0 \\ -\frac{1}{2a} & 0 & -1 \\ -\frac{1}{2a} & -1 & 0 \end{bmatrix}.$$

### 6.7.1 $\psi_{\alpha,\beta}(\mathfrak{L}_4^6(\lambda))$

In this case, we have $\psi_{\alpha,\beta}(\mathfrak{L}_4^6(\lambda)) = \mathfrak{t}_{\alpha,\beta}$:

$$\mathfrak{t}_{\alpha,\beta} = \{[e1, e2] = (1 + \alpha\,\lambda + \beta\,\lambda)\,e2 + (-\alpha\,\lambda + \beta\,\lambda)\,e3, [e1, e3] = (-\alpha\,\lambda + \beta\,\lambda)\,e2 + (-1 + \alpha\,\lambda + \beta\,\lambda)\,e3.$$

We note that $\mathfrak{t}_{\alpha,\beta}$ is an almost abelian Lie algebra where $\text{span}_{\mathbb{C}}\{e_2, e_3\}$ is a codimension 1 abelian ideal of $\mathfrak{t}_{\alpha,\beta}$; we can identify $\mathfrak{t}_{\alpha,\beta}$ the matrix

$$B = \begin{bmatrix} 1 + \alpha\,\lambda + \beta\,\lambda & -\alpha\,\lambda + \beta\,\lambda \\ -\alpha\,\lambda + \beta\,\lambda & -1 + \alpha\,\lambda + \beta\,\lambda \end{bmatrix}.$$

The isomorphism class of $\mathfrak{t}_{\alpha,\beta}$ is determined by the eigenvalues of $B$:

1. $B$ has repeated eigenvalues if and only if $1 + (\alpha - \beta)^2 \lambda^2 = 0$

    (a) $0$ is a repeated eigenvalue of $B$ if and only if $\alpha = -\beta = \pm\frac{\sqrt{-1}}{2\lambda}$

        i. $\mathfrak{t}_{\alpha,\beta}$ is the 3-dimensional abelian Lie algebra if and only if $B = 0$ if and only if $1 = 0$

        ii. $\mathfrak{t}_{\alpha,\beta}$ is isomorphic to $\mathfrak{n}_3(\mathbb{C})$ if and only if $B \neq 0$ if and only if $\alpha = -\beta = \pm\frac{\sqrt{-1}}{2\lambda}$

    (b) $B$ has a non-zero repeated eigenvalue if and only if
    $\left\{ \alpha = -\frac{-1 \pm \sqrt{-1}s}{2s\lambda}, \beta = \frac{1 \pm \sqrt{-1}s}{2s\lambda} \right\}$ for any $s \in \mathbb{C} \setminus \{0\}$.

        i. $\mathfrak{t}_{\alpha,\beta}$ is $\mathfrak{r}_{3,1}$ if and only if $B = t\,\text{Id}$ for some $t \neq 0$ if and only if $1 = 0$

ii. $\mathfrak{t}_{\alpha,\beta}$ is isomorphic to $\mathfrak{r}_3$ if and only if $B \neq t\,\mathrm{Id}$ for all $t \in \mathbb{C}$ if and only if $\left\{\alpha = -\frac{-1\pm\sqrt{-1}s}{2s\lambda}, \beta = \frac{1\pm\sqrt{-1}s}{2s\lambda}\right\}$ for any $s \in \mathbb{C}\setminus\{0\}$.

2. $B$ has distinct eigenvalues if and only if $1 + (\alpha - \beta)^2 \lambda^2 \neq 0$

   (a) $\mathfrak{t}_{\alpha,\beta}$ is isomorphic to $\mathfrak{r}_2(\mathbb{C}) \times \mathbb{C}$ if and only if $0$ is an eigenvalue of $B$ with algebraic multiplicity 1 if and only if
   $\left\{\alpha = -\frac{-1\pm\sqrt{1-4s^2}}{4s\lambda}, \beta = \frac{1\pm\sqrt{1-4s^2}}{4s\lambda}\right\}$ for any $s \in \mathbb{C}\setminus\{0\}$.

   (b) $\mathfrak{t}_{\alpha,\beta}$ is isomorphic to $\mathfrak{r}_{3,-1}$ if and only if $\mathrm{Trace}(B) = 0$ if and only if
   $\left\{\alpha = \mp\frac{\sqrt{s-s^2}}{2s\lambda}, \beta = \pm\frac{\sqrt{s-s^2}}{2s\lambda}\right\}$ for any $s \in \mathbb{C}\setminus\{0\}$.

   (c) $\mathfrak{t}_{\alpha,\beta}$ is isomorphic to $\mathfrak{r}_{3,z}$ for some $z$ such that $z(z^2 - 1) \neq 0$ if and only if
   $\left\{\alpha = \frac{s_1 \mp \sqrt{s_1 s_2{}^2(1-s_1)}}{2s_1 s_2 \lambda}, \beta = \frac{s_1 \pm \sqrt{s_1 s_2{}^2(1-s_1)}}{2s_1 s_2 \lambda}\right\}$ for any $s_1, s_2 \in \mathbb{C}\setminus\{0\}$ such that $s_1 - s_2^2 \neq 0$.

### 6.7.2 $\phi_\beta(\mathfrak{L}_4^6(\lambda))$

In this case, we have $\phi_\beta(\mathfrak{L}_4^6(\lambda)) = \widetilde{\mathfrak{t}}_\beta$:

$$\widetilde{\mathfrak{t}}_\beta = \{[e1, e2] = (\lambda + \beta\lambda)\,e2 + (-\lambda + \beta\lambda)\,e3, [e1, e3] = (-\lambda + \beta\lambda)\,e2 + (\lambda + \beta\lambda)\,e3.$$

We note that $\widetilde{\mathfrak{t}}_\beta$ is an almost abelian Lie algebra where $\mathrm{span}_\mathbb{C}\{e_2, e_3\}$ is a codimension 1 abelian ideal of $\widetilde{\mathfrak{t}}_\beta$; we can identify $\widetilde{\mathfrak{t}}_\beta$ the matrix

$$\widetilde{B} = \begin{bmatrix} \lambda + \beta\lambda & -\lambda + \beta\lambda \\ -\lambda + \beta\lambda & \lambda + \beta\lambda \end{bmatrix}.$$

The isomorphism class of $\widetilde{\mathfrak{t}}_\beta$ is determined by the eigenvalues of $\widetilde{B}$:

1. $\widetilde{B}$ has repeated eigenvalues if and only if $\beta = 1$

   (a) $0$ is a repeated eigenvalue of $\widetilde{B}$ if and only if $\lambda = 0$ (False)

   (b) $\widetilde{B}$ has a non-zero repeated eigenvalue if and only if $\beta = 1$

      i. $\widetilde{\mathfrak{t}}_\beta$ is $\mathfrak{r}_{3,1}$ if and only if $\widetilde{B} = t\,\mathrm{Id}$ for some $t \neq 0$ if and only if $\lambda = \lambda$

      ii. $\widetilde{\mathfrak{t}}_\beta$ is isomorphic to $\mathfrak{r}_3$ if and only if $\widetilde{B} \neq t\,\mathrm{Id}$ for all $\lambda \in \mathbb{C}$ if and only if $1 = 0$

2. $\widetilde{B}$ has distinct eigenvalues eigenvalues if and only if $\beta \neq 1$

   (a) $\widetilde{\mathfrak{t}}_\beta$ is isomorphic to $\mathfrak{r}_2(\mathbb{C}) \times \mathbb{C}$ if and only if $0$ is an eigenvalue of $B$ with algebraic multiplicity 1 if and only if $\beta = 0$

   (b) $\widetilde{\mathfrak{t}}_\beta$ is isomorphic to $\mathfrak{r}_{3,-1}$ if and only if $\mathrm{Trace}(\widetilde{B}) = 0$ if and only if $\beta = -1$.

   (c) $\widetilde{\mathfrak{t}}_\beta$ is isomorphic to $\mathfrak{r}_{3,z}$ for some $z$ such that $z(z^2 - 1) \neq 0$ if and only if $\beta \neq 1$ and $\beta \neq -1$ and $\beta \neq 0$.

### 6.7.3 $\rho(\mathfrak{L}_4^6)$

In this case, we have $\rho(\mathfrak{L}_4^6) = \widehat{\mathfrak{t}}$:

$$\widehat{\mathfrak{t}} = \{[e1, e2] = \lambda\,e2 + \lambda\,e3, [e1, e3] = \lambda\,e2 + \lambda\,e3.$$

which is isomorphic to $\mathfrak{r}_2 \times \mathbb{C}$.

### 6.7.4 Identities

- $[A\cdot, A\cdot] = 0$
- $A^2[\cdot,\cdot] = 0$
- $[A^2\cdot,\cdot] + [\cdot, A^2\cdot] = 0$

The hom-Lie algebra $\mathfrak{L}_4^6(\lambda)$ is a non-multiplicative hom-Lie algebra.

## 6.8 Degenerations between hom-Lie algebras ($\mathfrak{L}_4^i$)

If $\mathfrak{L}_i^j \xrightarrow{\deg} \mathfrak{L}_h^k$ then $\mathrm{Der}(\mathfrak{L}_i^j) \leq \mathrm{Der}(\mathfrak{L}_h^k)$. Therefore, we can organize the hom-Lie algebras in the following way:

| Dim(Der) | hom-Lie algebra |
|:---:|:---:|
| 1 | $\mathfrak{L}_4^6(\lambda), \mathfrak{L}_4^5$ |
| 2 | $\mathfrak{L}_4^4(\lambda), \mathfrak{L}_4^3, \mathfrak{L}_4^2$ |
| 3 | $\mathfrak{L}_4^1$ |
| 4 | $\mathfrak{L}_4^0$ |

### 6.8.1 Degenerations

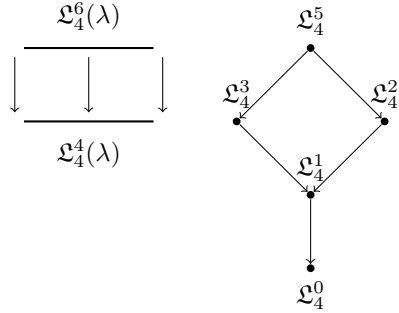

1. $\mathfrak{L}_4^6(\lambda) \xrightarrow{\deg} \mathfrak{L}_4^4(\pm\lambda)$
   In fact, set
   $$g(t) = \begin{bmatrix} 1 & 0 & 0 \\ 0 & \mathrm{e}^{-t} & 0 \\ 0 & 0 & \mathrm{e}^{-t} \end{bmatrix}, \text{ with } t \in \mathbb{R}.$$
   
   We have $g(t) \in \mathrm{Aut}(\mathfrak{L}_4)$ and $g(t) \cdot \mathfrak{L}_4^6(\lambda)$ is the hom-Lie algebra $(\mathfrak{L}_4, A(t))$ with
   $$A(t) = \begin{bmatrix} 0 & 0 & 0 \\ \mathrm{e}^{-t} & \lambda & \lambda \\ 0 & -\lambda & -\lambda \end{bmatrix}$$
   
   It is easy to check that $(\mathfrak{L}_4, A(t)) \to \mathfrak{L}_4^4(\lambda)$ as $t$ tends to infinity.
   
   Since, $\mathfrak{L}_2^4(\lambda) \cong \mathfrak{L}_2^4(-\lambda)$, then $\mathfrak{L}_2^6(\lambda) \xrightarrow{\deg} \mathfrak{L}_2^4(-\lambda)$.

2. $\mathfrak{L}_4^5 \xrightarrow{\deg} \mathfrak{L}_4^3$
   In fact, set
   $$g(t) = \begin{bmatrix} 0 & 0 & 0 \\ 0 & 0 & 1 \\ \mathrm{e}^{-t} & 0 & 0 \end{bmatrix}, \text{ with } t \in \mathbb{R}.$$

We have $g(t) \in \mathrm{Aut}(\mathfrak{L}_4)$ and $g(t) \cdot \mathfrak{L}_4^5$ is the hom-Lie algebra $(\mathfrak{L}_4, A(t))$ with

$$A(t) = \begin{bmatrix} 0 & 0 & 0 \\ 0 & 0 & 1 \\ \mathrm{e}^{-t} & 0 & 0 \end{bmatrix}$$

It is easy to check that $(\mathfrak{L}_4, A(t)) \to \mathfrak{L}_4^3$ as $t$ tends to infinity.

3. $\mathfrak{L}_4^5 \xrightarrow{\mathrm{deg}} \mathfrak{L}_4^2$

In fact, set

$$g(t) = \begin{bmatrix} 1 & 0 & 0 \\ -\mathrm{e}^t & 1 & 0 \\ 0 & 0 & \mathrm{e}^{-t} \end{bmatrix}, \text{ with } t \in \mathbb{R}.$$

We have $g(t) \in \mathrm{Aut}(\mathfrak{L}_4)$ and $g(t) \cdot \mathfrak{L}_4^5$ is the hom-Lie algebra $(\mathfrak{L}_4, A(t))$ with

$$A(t) = \begin{bmatrix} 0 & 0 & 0 \\ 1 & 0 & 0 \\ 1 & \mathrm{e}^{-t} & 0 \end{bmatrix}$$

It is easy to check that $(\mathfrak{L}_4, A(t)) \to \mathfrak{L}_4^2$ as $t$ tends to infinity.

4. $\mathfrak{L}_4^3 \xrightarrow{\mathrm{deg}} \mathfrak{L}_4^1$

In fact, set

$$g(t) = \begin{bmatrix} 1 & 0 & 0 \\ 0 & 1 & 0 \\ -\mathrm{e}^t & 0 & \mathrm{e}^t \end{bmatrix}, \text{ with } t \in \mathbb{R}.$$

We have $g(t) \in \mathrm{Aut}(\mathfrak{L}_4)$ and $g(t) \cdot \mathfrak{L}_4^3$ is the hom-Lie algebra $(\mathfrak{L}_4, A(t))$ with

$$A(t) = \begin{bmatrix} 0 & 0 & 0 \\ 1 & 0 & \mathrm{e}^{-t} \\ 0 & 0 & 0 \end{bmatrix}$$

It is easy to check that $(\mathfrak{L}_4, A(t)) \to \mathfrak{L}_4^1$ as $t$ tends to infinity.

5. $\mathfrak{L}_4^2 \xrightarrow{\mathrm{deg}} \mathfrak{L}_4^1$

In fact, set

$$g(t) = \begin{bmatrix} 1 & 0 & 0 \\ 0 & 1 & 0 \\ 0 & 0 & \mathrm{e}^{-t} \end{bmatrix}, \text{ with } t \in \mathbb{R}.$$

We have $g(t) \in \mathrm{Aut}(\mathfrak{L}_4)$ and $g(t) \cdot \mathfrak{L}_4^2$ is the hom-Lie algebra $(\mathfrak{L}_4, A(t))$ with

$$A(t) = \begin{bmatrix} 0 & 0 & 0 \\ 1 & 0 & 0 \\ \mathrm{e}^{-t} & 0 & 0 \end{bmatrix}$$

It is easy to check that $(\mathfrak{L}_4, A(t)) \to \mathfrak{L}_4^1$ as $t$ tends to infinity.

6. $\mathfrak{L}_4^1 \xrightarrow{\deg} \mathfrak{L}_4^0$

   In fact, set
   $$g(t) = \begin{bmatrix} 1 & 0 & 0 \\ 0 & e^{-t} & 0 \\ 0 & 0 & 1 \end{bmatrix}, \text{ with } t \in \mathbb{R}.$$

   We have $g(t) \in \operatorname{Aut}(\mathfrak{L}_4)$ and $g(t) \cdot \mathfrak{L}_4^1$ is the hom-Lie algebra $(\mathfrak{L}_4, A(t))$ with
   $$A(t) = \begin{bmatrix} 0 & 0 & 0 \\ e^{-t} & 0 & 0 \\ 0 & 0 & 0 \end{bmatrix}$$

   It is easy to check that $(\mathfrak{L}_4, A(t)) \to \mathfrak{L}_4^0$ as $t$ tends to infinity.

7. $\mathfrak{L}_4^6(\lambda) \xrightarrow{\deg} \mathfrak{L}_4^4(\kappa)$ with $(\kappa \neq \pm\lambda)$.

   Suppose, contrary to our claim, that $\mathfrak{L}_4^6(\lambda) \xrightarrow{\deg} \mathfrak{L}_4^4(\kappa)$ with $(\kappa \neq \pm\lambda)$. In such case, we have
   $$\psi_{\alpha,\beta}(\mathfrak{L}_4^6(\lambda)) \xrightarrow{\deg} \psi_{\alpha,\beta}(\mathfrak{L}_4^4(\kappa)).$$

   By taking $\alpha = -\beta$ and $\alpha = \pm\frac{\sqrt{-1}}{2\lambda}$, we have from 6.7.1 that $\psi_{\alpha,\beta}(\mathfrak{L}_4^6(\lambda))$ is a Lie algebra isomorphic to $\mathfrak{n}_3$, therefore, $\psi_{\alpha,\beta}(\mathfrak{L}_4^4(\kappa))$ is a Lie algebra which is isomorphic to one of the following: $\mathfrak{n}_3$ or $\mathfrak{a}_3$ (by Theorem 1.1). We must have $\psi_{\alpha,\beta}(\mathfrak{L}_4^4(\kappa))$ is isomorphic to $\mathfrak{n}_3$ (from 6.5.1). But, we now apply again 6.5.1, to obtain $\alpha = -\beta$ and $\alpha = \pm\frac{\sqrt{-1}}{2\kappa}$ and therefore $\lambda = \pm\kappa$; this is a contradiction.

8. $\mathfrak{L}_4^6(\lambda) \xrightarrow{\deg} \mathfrak{L}_4^0$ with $(\lambda \neq 0)$.

   Suppose, contrary to our claim, that $\mathfrak{L}_4^6(\lambda) \xrightarrow{\deg} \mathfrak{L}_4^0$ with $(\lambda \neq 0)$. In such case, we have
   $$\psi_{\alpha,\beta}(\mathfrak{L}_4^6(\lambda)) \xrightarrow{\deg} \psi_{\alpha,\beta}(\mathfrak{L}_4^0).$$

   By taking $\alpha = -\beta$ and $\alpha = \pm\frac{\sqrt{-1}}{2\lambda}$, we have from 6.7.1 that $\psi_{\alpha,\beta}(\mathfrak{L}_4^6(\lambda))$ is a Lie algebra isomorphic to $\mathfrak{n}_3$, therefore, $\psi_{\alpha,\beta}(\mathfrak{L}_4^0)$ is a Lie algebra which is isomorphic to one of the following: $\mathfrak{n}_3$ or $\mathfrak{a}_3$ (by Theorem 1.1). But, from 6.1.1, we have $\psi_{\alpha,\beta}(\mathfrak{L}_4^0)$ is isomorphic to $\mathfrak{r}_{3,-1}$; this is a contradiction.

9. $\mathfrak{L}_4^5 \xrightarrow{\deg} \mathfrak{L}_4^4(\lambda)$ with $(\lambda \neq 0)$.

   Suppose, contrary to our claim, that $\mathfrak{L}_4^5 \xrightarrow{\deg} \mathfrak{L}_4^4(\lambda)$ with $(\lambda \neq 0)$. In such case, we have
   $$\phi_\beta(\mathfrak{L}_4^5) \xrightarrow{\deg} \phi_\beta(\mathfrak{L}_4^4(\lambda)).$$

   By taking $\beta = 1$, we have from 6.6.2 that $\phi_\beta(\mathfrak{L}_4^5)$ is the 3-dimensional abelian Lie algebra, therefore, $\phi_\beta(\mathfrak{L}_4^4(\lambda))$ is the 3-dimensional abelian Lie algebra. But, from 6.5.2, we have $\phi_\beta(\mathfrak{L}_4^4(\lambda))$ is isomorphic to $\mathfrak{r}_{3,1}$; this is a contradiction.

# 7 $\mathfrak{L}_5(z)$ with $z(z^2-1) \neq 0$: Bianchi type VI Lie algebra, $\mathfrak{r}_{3,z}$ with $z(z^2-1) \neq 0$

$$\mathfrak{L}_1 := \{[e_1, e_2] = e_2, [e_1, e_3] = ze_3$$

$$\operatorname{Aut}(\mathfrak{L}_5(z)) = \left\{ \begin{bmatrix} 1 & 0 & 0 \\ x & a & 0 \\ y & 0 & b \end{bmatrix} \in M(3, \mathbb{C}) : \begin{matrix} a, b \in \mathbb{C}^* \\ x, y \in \mathbb{C} \end{matrix} \right\}.$$

$$\operatorname{Der}(\mathfrak{L}_5(z)) = \left\{ \begin{bmatrix} 0 & 0 & 0 \\ t_3 & t_1 & 0 \\ t_4 & 0 & t_2 \end{bmatrix} \in M(3, \mathbb{C}) : t_1, t_2, t_3, t_4 \in \mathbb{C} \right\} \cong \mathfrak{r}_2(\mathbb{C}) \times \mathfrak{r}_2(\mathbb{C}).$$

## 7.1 $\mathfrak{L}_5^0(z)$

Let $A$ be the Zero map and let us denote by $\mathfrak{L}_5^0(z)$ the hom-Lie algebra $(\mathfrak{L}_5(z), A)$. We have:
$\operatorname{Aut}(\mathfrak{L}_5^0(z)) = \operatorname{Aut}(\mathfrak{L}_5(z))$.
$\operatorname{Der}(\mathfrak{L}_5^0(z)) = \operatorname{Der}(\mathfrak{L}_5(z))$.

### 7.1.1 $\psi_{\alpha,\beta}(\mathfrak{L}_5^0(z))$

In this case, we have $\psi_{\alpha,\beta}(\mathfrak{L}_5^0(z)) = \mathfrak{t}_{\alpha,\beta}$ is the Lie algebra $\mathfrak{r}_{3,z}$.

### 7.1.2 $\phi_\beta(\mathfrak{L}_5^0(z))$

In this case, we have $\phi_\beta(\mathfrak{L}_5^0(z)) = \widetilde{\mathfrak{t}}_\beta$ is the 3-dimensional abelian Lie algebra.

### 7.1.3 $\rho(\mathfrak{L}_5^0)$

In this case, we have $\rho(\mathfrak{L}_5^0) = \widehat{\mathfrak{t}}$ is the 3-dimensional abelian Lie algebra.

## 7.2 $\mathfrak{L}_5^1(z)$

Let $A$ be the endomorphism of $\mathbb{C}^3$ whose matrix representation with respect to the ordered basis $\{e_1, e_2, e_3\}$ is

$$\begin{bmatrix} 0 & 0 & 0 \\ 1 & 0 & 0 \\ 0 & 0 & 0 \end{bmatrix}$$

and let us denote by $\mathfrak{L}_5^1(z)$ the hom-Lie algebra $(\mathfrak{L}_5(z), A)$. We have:

$$\mathfrak{gl}(n, \mathbb{C})_A = \left\{ \begin{bmatrix} t_1 & 0 & 0 \\ t_3 & t_1 & t_5 \\ t_4 & 0 & t_2 \end{bmatrix} \in M(3, \mathbb{C}) : t_1, \ldots, t_5 \in \mathbb{C} \right\}.$$

$$\operatorname{Aut}(\mathfrak{L}_5^1(z)) = \left\{ \begin{bmatrix} 1 & 0 & 0 \\ x & 1 & 0 \\ y & 0 & a \end{bmatrix} \in M(3, \mathbb{C}) : \begin{matrix} a \in \mathbb{C}^* \\ x, y \in \mathbb{C} \end{matrix} \right\}.$$

$$\operatorname{Der}(\mathfrak{L}_5^1(z)) = \left\{ \begin{bmatrix} 0 & 0 & 0 \\ t_2 & 0 & 0 \\ t_3 & 0 & t_1 \end{bmatrix} \in M(3, \mathbb{C}) : t_1, t_2, t_3 \in \mathbb{C} \right\} \cong \mathfrak{r}_2(\mathbb{C}) \times \mathbb{C}.$$

### 7.2.1 $T_{\theta_1,\ldots,\theta_6;\beta_1,\ldots\beta_9}(\mathfrak{L}_5^1)$

Given $D_1(w's)$, $D_2(x's)$ and $D_3(y's)$ in $\mathfrak{gl}(n,\mathbb{C})_A$, we have $\lambda_1 := D_1[\cdot,\cdot] + [D_2\cdot,\cdot] + [\cdot,D_3\cdot]$

$$\lambda := \begin{cases} e1 \cdot e1 = (-x_{2,1} + y_{2,1})\, e2 + z\,(-x_{3,1} + y_{3,1})\, e3, \\ e1 \cdot e2 = (x_{2,2} + w_{2,2} + y_{2,2})\, e2, \\ e1 \cdot e3 = (y_{2,3} + w_{2,3}z)\, e2 + z\,(x_{2,2} + w_{3,3} + y_{3,3})\, e3, \\ e2 \cdot e1 = (-y_{2,2} - w_{2,2} - x_{2,2})\, e2, \\ e3 \cdot e1 = (-x_{2,3} - w_{2,3}z)\, e2 + z\,(-y_{2,2} - w_{3,3} - x_{3,3})\, e3, \end{cases}$$

By taking $\{\theta_1 = 1, \theta_2 = 0, \theta_3 = t, \theta_4 = 0, \theta_5 = 0, \theta_6 = 0\}$ and $\{\beta_1 = 1, \beta_2 = -1, \beta_3 = 0, \beta_4 = 1, \beta_5 = -1, \beta_6 = 0, \beta_7 = 1, \beta_8 = -1, \beta_9 = 0,\}$ and then, by replacing $D_1$ by $-tD_3$, we have $\mathrm{Ker}(T_{\theta_1,\ldots,\theta_6;\beta_1,\ldots,\beta_9}(\mathfrak{L}_5^1))$ is determined by the equations

$$\begin{cases} -x_{2,1} + y_{2,1} = 0, \\ -x_{3,1}z + y_{3,1}z = 0, \\ x_{2,2} - ty_{2,2} + y_{2,2} = 0, \\ y_{2,3} - ty_{2,3}z = 0, \\ x_{2,2}z - ty_{3,3}z + y_{3,3}z = 0, \\ -y_{2,2}z + ty_{3,3}z - x_{3,3}z = 0 \\ -x_{2,3} + ty_{2,3}z = 0, \\ -y_{2,2} + ty_{2,2} - x_{2,2} = 0. \end{cases}$$

Therefore $x_{2,1} = y_{2,1}$, $x_{2,2} = -y_{2,2} + ty_{2,2}$, $x_{2,3} = ty_{2,3}z$, $x_{3,1} = y_{3,1}$ and $x_{3,3} = -y_{2,2} + ty_{3,3}$, and so:

$$\begin{cases} y_{2,3}\,(-1 + tz) = 0 \\ (y_{2,2} - y_{3,3})\,(t - 1) = 0 \end{cases}$$

Since $z \neq 0, \pm 1$, we have

$$\mathrm{Dim}\, T_{\theta_1,\ldots,\theta_6;\beta_1,\ldots\beta_9}(\mathfrak{L}_5^1) = \begin{cases} 4, & \text{if } t = 1 \text{ or } t = 1/z; \\ 3, & \text{in other case.} \end{cases}$$

### 7.2.2 $\psi_{\alpha,\beta}(\mathfrak{L}_5^1(z))$

In this case, we have $\psi_{\alpha,\beta}(\mathfrak{L}_5^1(z)) = \mathfrak{t}_{\alpha,\beta}$ is the Lie algebra $\mathfrak{r}_{3,z}$.

### 7.2.3 $\phi_\beta(\mathfrak{L}_5^1(z))$

In this case, we have $\phi_\beta(\mathfrak{L}_5^1(z)) = \widetilde{\mathfrak{t}}_\beta$ is the 3-dimensional abelian Lie algebra.

### 7.2.4 $\rho(\mathfrak{L}_5^1(z))$

In this case, we have $\rho(\mathfrak{L}_5^1(z)) = \widehat{\mathfrak{t}}$ is the 3-dimensional abelian Lie algebra.

### 7.2.5 Identities

- $A[\cdot,\cdot]$
- $[A\cdot,\cdot] + [\cdot,A\cdot]$
    - $[A\cdot,A\cdot] = 0$

In particular, the hom-Lie algebra $\mathfrak{L}_5^1(z)$ is a multiplicative hom-Lie algebra.

## 7.3 $\mathfrak{L}_5^2(z)$

Let $A$ be the endomorphism of $\mathbb{C}^3$ whose matrix representation with respect to the ordered basis $\{e_1, e_2, e_3\}$ is

$$\begin{bmatrix} 0 & 0 & 0 \\ 0 & 0 & 0 \\ 1 & 0 & 0 \end{bmatrix}$$

and let us denote by $\mathfrak{L}_5^2(z)$ the hom-Lie algebra $(\mathfrak{L}_5(z), A)$. We have:

$$\mathfrak{gl}(n,\mathbb{C})_A = \left\{ \begin{bmatrix} t_1 & 0 & 0 \\ t_3 & t_2 & 0 \\ t_5 & t_4 & t_1 \end{bmatrix} \in M(3,\mathbb{C}) : t_1, \ldots, t_5 \in \mathbb{C} \right\}.$$

$$\mathrm{Aut}(\mathfrak{L}_5^2(z)) = \left\{ \begin{bmatrix} 1 & 0 & 0 \\ x & a & 0 \\ y & 0 & 1 \end{bmatrix} \in M(3,\mathbb{C}) : \begin{array}{l} a \in \mathbb{C}^* \\ x, y \in \mathbb{C} \end{array} \right\}.$$

$$\mathrm{Der}(\mathfrak{L}_5^2(z)) = \left\{ \begin{bmatrix} 0 & 0 & 0 \\ t_2 & t_1 & 0 \\ t_3 & 0 & 0 \end{bmatrix} \in M(3,\mathbb{C}) : t_1, t_2, t_3 \in \mathbb{C} \right\} \cong \mathfrak{r}_2(\mathbb{C}) \times \mathbb{C}.$$

### 7.3.1 $T_{\theta_1,\ldots,\theta_6;\beta_1,\ldots,\beta_9}(\mathfrak{L}_5^2)$

Given $D_1(w's)$, $D_2(x's)$ and $D_3(y's)$ in $\mathfrak{gl}(n,\mathbb{C})_A$, we have $\lambda := D_1[\cdot,\cdot] + [D_2\cdot,\cdot] + [\cdot, D_3\cdot]$

$$\lambda := \begin{cases} e1 \cdot e1 = (-x_{2,1} + y_{2,1})\, e2 + (-x_{3,1}z + y_{3,1}z)\, e3, \\ e1 \cdot e2 = (x_{3,3} + w_{2,2} + y_{2,2})\, e2 + (w_{3,2} + y_{3,2}z)\, e3, \\ e1 \cdot e3 = (x_{3,3}z + zw_{3,3} + y_{3,3}z)\, e3, \\ e2 \cdot e1 = (-y_{3,3} - w_{2,2} - x_{2,2})\, e2 + (-w_{3,2} - x_{3,2}z)\, e3, \\ e3 \cdot e1 = (-y_{3,3}z - zw_{3,3} - x_{3,3}z)\, e3 \end{cases}$$

By taking $\{\theta_1 = 1, \theta_2 = 0, \theta_3 = t, \theta_4 = 0, \theta_5 = 0, \theta_6 = 0\}$ and $\{\beta_1 = 1, \beta_2 = -1, \beta_3 = 0, \beta_4 = 1, \beta_5 = -1, \beta_6 = 0, \beta_7 = 1, \beta_8 = -1, \beta_9 = 0,\}$ and then, by replacing $D_1$ by $-tD_3$, we have $\mathrm{Ker}(T_{\theta_1,\ldots,\theta_6;\beta_1,\ldots,\beta_9}(\mathfrak{L}_5^2))$ is determined by the equations

$$\begin{cases} -x_{2,1} + y_{2,1} = 0, \\ -ty_{3,2} + y_{3,2}z = 0 \\ , ty_{3,2} - x_{3,2}z = 0, \\ -x_{3,1}z + y_{3,1}z = 0, \\ x_{3,3} - ty_{2,2} + y_{2,2} = 0, \\ -y_{3,3} + ty_{2,2} - x_{2,2} = 0, \\ x_{3,3}z - ty_{3,3}z + y_{3,3}z = 0, \\ -y_{3,3}z + ty_{3,3}z - x_{3,3}z = 0. \end{cases}$$

Therefore, since $z \neq 0$, $x_{2,1} = y_{2,1}$, $x_{2,2} = -y_{3,3} + ty_{2,2}$, $x_{3,1} = y_{3,1}$, $x_{3,2} = t\frac{y_{3,2}}{z}$ and $x_{3,3} = ty_{2,2} - y_{2,2}$, and so:

$$\begin{cases} -y_{3,2}(t-z) = 0 \\ (-y_{3,3} + y_{2,2})(t-1) = 0 \end{cases}$$

Since

$$\mathrm{Dim}\, T_{\theta_1,\ldots,\theta_6;\beta_1,\ldots,\beta_9}(\mathfrak{L}_5^2) = \begin{cases} 4, & \text{if } t = 1 \text{ or } t = z; \\ 3, & \text{in other case.} \end{cases}$$

### 7.3.2 $\psi_{\alpha,\beta}(\mathfrak{L}_5^2(z))$

In this case, we have $\psi_{\alpha,\beta}(\mathfrak{L}_5^2(z)) = \mathfrak{t}_{\alpha,\beta}$ is the Lie algebra $\mathfrak{r}_{3,z}$.

### 7.3.3 $\phi_\beta(\mathfrak{L}_5^2(z))$

In this case, we have $\phi_\beta(\mathfrak{L}_5^2(z)) = \widetilde{\mathfrak{t}}_\beta$ is the 3-dimensional abelian Lie algebra.

### 7.3.4 $\rho(\mathfrak{L}_5^2(z))$

In this case, we have $\rho(\mathfrak{L}_5^2(z)) = \widehat{\mathfrak{t}}$ is the 3-dimensional abelian Lie algebra.

### 7.3.5 Identities

- $A[\cdot, \cdot] = 0$
- $[A\cdot, \cdot] + [\cdot, A\cdot] = 0$
  - $[A\cdot, A\cdot] = 0$

In particular, the hom-Lie algebra $\mathfrak{L}_5^2(z)$ is a multiplicative hom-Lie algebra.

## 7.4 $\mathfrak{L}_5^3(z)$

Let $A$ be the endomorphism of $\mathbb{C}^3$ whose matrix representation with respect to the ordered basis $\{e_1, e_2, e_3\}$ is

$$\begin{bmatrix} 0 & 0 & 0 \\ 1 & 0 & 0 \\ 1 & 0 & 0 \end{bmatrix}$$

and let us denote by $\mathfrak{L}_5^3(z)$ the hom-Lie algebra $(\mathfrak{L}_5(z), A)$. We have:

$$\text{Aut}(\mathfrak{L}_5^3(z)) = \left\{ \begin{bmatrix} 1 & 0 & 0 \\ x & 1 & 0 \\ y & 0 & 1 \end{bmatrix} \in M(3, \mathbb{C}) : x, y \right\}.$$

$$\text{Der}(\mathfrak{L}_5^3(z)) = \left\{ \begin{bmatrix} 0 & 0 & 0 \\ t_1 & 0 & 0 \\ t_2 & 0 & 0 \end{bmatrix} \in M(3, \mathbb{C}) : t_1, t_2 \in \mathbb{C} \right\} \cong \mathbb{C}^2.$$

### 7.4.1 $T_{\theta_1, \ldots, \theta_6; \beta_1, \ldots \beta_9}(\mathfrak{L}_5^3)$

Given $D_1(w's)$, $D_2(x's)$ and $D_3(y's)$ in $\mathfrak{gl}(n, \mathbb{C})_A$, we have $\lambda := D_1[\cdot, \cdot] + [D_2\cdot, \cdot] + [\cdot, D_3\cdot]$

$$\lambda := \begin{cases} e1 \cdot e1 = (-x_{2,1} + y_{2,1})\, e2 + (-x_{3,1}z + y_{3,1}z)\, e3, \\ e1 \cdot e2 = (x_{3,2} + x_{3,3} + w_{3,2} + w_{3,3} - w_{2,3} + y_{3,2} + y_{3,3} - y_{2,3})\, e2 + (w_{3,2} + y_{3,2}z)\, e3, \\ e1 \cdot e3 = (y_{2,3} + w_{2,3}z)\, e2 + ((x_{3,2} + x_{3,3})\, z + w_{3,3}z + y_{3,3}z)\, e3, \\ e2 \cdot e1 = (-y_{3,2} - y_{3,3} - w_{3,2} - w_{3,3} + w_{2,3} - x_{3,2} - x_{3,3} + x_{2,3})\, e2 + (-w_{3,2} - x_{3,2}z)\, e3, \\ e3 \cdot e1 = (-x_{2,3} - w_{2,3}z)\, e2 + (-(y_{3,2} + y_{3,3})\, z - w_{3,3}z - x_{3,3}z)\, e3 \end{cases}$$

By taking $\{\theta_1 = 1, \theta_2 = 0, \theta_3 = t, \theta_4 = 0, \theta_5 = 0, \theta_6 = 0\}$ and $\{\beta_1 = 1, \beta_2 = -1, \beta_3 = 0, \beta_4 = 1, \beta_5 = -1, \beta_6 = 0, \beta_7 = 1, \beta_8 = -1, \beta_9 = 0,\}$ and then, by replacing $D_1$ by $-tD_3$, we have

$\mathrm{Ker}(T_{\theta_1,\dots,\theta_6;\beta_1,\dots,\beta_9}(\mathfrak{L}_5^3))$ is determined by the equations

$$\begin{cases} -x_{2,1} + y_{2,1} = 0, \\ -x_{2,3} + ty_{2,3}z = 0, \\ y_{2,3} - ty_{2,3}z = 0, \\ -ty_{3,2} + y_{3,2}z = 0, \\ ty_{3,2} - x_{3,2}z = 0, \\ -x_{3,1}z + y_{3,1}z = 0, \\ x_{3,2}z + x_{3,3}z - ty_{3,3}z + y_{3,3}z = 0, \\ -y_{3,2}z - y_{3,3}z + ty_{3,3}z - x_{3,3}z = 0, \\ x_{3,2} + x_{3,3} - ty_{3,2} - ty_{3,3} + ty_{2,3} + y_{3,2} + y_{3,3} - y_{2,3} = 0, \\ -y_{3,2} - y_{3,3} + ty_{3,2} + ty_{3,3} - ty_{2,3} - x_{3,2} - x_{3,3} + x_{2,3} = 0. \end{cases}$$

Therefore $x_{2,1} = y_{2,1}$, $x_{2,3} = ty_{2,3}z$, $x_{3,1} = y_{3,1}$, $x_{3,2} = \frac{ty_{3,2}}{z}$ and $x_{3,3} = -y_{3,2} - y_{3,3} + ty_{3,3}$, and so:

$$\begin{cases} t(z-1)(y_{2,3}z + y_{3,2}) = 0, \\ -y_{2,3}(-1 + tz) = 0, \\ -y_{3,2}(t - z) = 0`. \end{cases}$$

Since $z \neq \pm 1, 0$, we have $y_{2,3} = 0$ and $y_{3,2} = 0$, and

$$\mathrm{Dim}\, T_{\theta_1,\dots,\theta_6;\beta_1,\dots\beta_9}(\mathfrak{L}_5^3) = 3,\ \forall t \in \mathbb{C}.$$

### 7.4.2 $\psi_{\alpha,\beta}(\mathfrak{L}_5^3(z))$

In this case, we have $\psi_{\alpha,\beta}(\mathfrak{L}_5^3(z)) = \mathfrak{t}_{\alpha,\beta}$ is the Lie algebra $\mathfrak{r}_{3,z}$.

### 7.4.3 $\phi_\beta(\mathfrak{L}_5^3(z))$

In this case, we have $\phi_\beta(\mathfrak{L}_5^3(z)) = \widetilde{\mathfrak{t}}_\beta$ is the 3-dimensional abelian Lie algebra.

### 7.4.4 $\rho(\mathfrak{L}_5^3(z))$

In this case, we have $\rho(\mathfrak{L}_5^3(z)) = \widehat{\mathfrak{t}}$ is the 3-dimensional abelian Lie algebra.

### 7.4.5 Identities

- $A[\cdot,\cdot] = 0$
- $[A\cdot,\cdot] + [\cdot, A\cdot]$
  - $[A\cdot, A\cdot] = 0$

In particular, the hom-Lie algebra $\mathfrak{L}_5^3(z)$ is a multiplicative hom-Lie algebra.

## 7.5 $\mathfrak{L}_5^4(z)$

Let $A$ be the endomorphism of $\mathbb{C}^3$ whose matrix representation with respect to the ordered basis $\{e_1, e_2, e_3\}$ is

$$\begin{bmatrix} 0 & 0 & 0 \\ 0 & 0 & 1 \\ 0 & 0 & 0 \end{bmatrix}$$

and let us denote by $\mathfrak{L}_5^4(z)$ the hom-Lie algebra $(\mathfrak{L}_5(z), A)$. We have:

$$\mathfrak{gl}(n,\mathbb{C})_A = \left\{ \begin{bmatrix} t_1 & 0 & t_3 \\ t_5 & t_2 & t_4 \\ 0 & 0 & t_2 \end{bmatrix} \in M(3,\mathbb{C}) : t_1, \dots, t_5 \in \mathbb{C} \right\}.$$

$$\mathrm{Aut}(\mathfrak{L}_5^4(z)) = \left\{ \begin{bmatrix} 1 & 0 & 0 \\ x & a & 0 \\ 0 & 0 & a \end{bmatrix} \in M(3,\mathbb{C}) : \begin{array}{l} a \in \mathbb{C}^* \\ x \in \mathbb{C} \end{array} \right\}.$$

$$\mathrm{Der}(\mathfrak{L}_5^4(z)) = \left\{ \begin{bmatrix} 0 & 0 & 0 \\ t_2 & t_1 & 0 \\ 0 & 0 & t_1 \end{bmatrix} \in M(3,\mathbb{C}) : t_1, t_2 \in \mathbb{C} \right\} \cong \mathfrak{r}_2(\mathbb{C}).$$

### 7.5.1 $T_{\theta_1,\dots,\theta_6;\beta_1,\dots\beta_9}(\mathfrak{L}_5^4)$

Given $D_1(w's)$, $D_2(x's)$ and $D_3(y's)$ in $\mathfrak{gl}(n,\mathbb{C})_A$, we have $\lambda := D_1[\cdot,\cdot] + [D_2\cdot,\cdot] + [\cdot, D_3\cdot]$

$$\lambda := \begin{cases} e1, e1 = (-x_{2,1} + y_{2,1})\,e2, \\ e1, e2 = (x_{1,1} + w_{3,3} + y_{3,3})\,e2, \\ e1, e3 = zw_{1,3}e1 + (y_{2,3} + zw_{2,3})\,e2 + (x_{1,1}z + zw_{3,3} + y_{3,3}z)\,e3, \\ e2, e1 = (-y_{1,1} - w_{3,3} - x_{3,3})\,e2, \\ e2, e3 = -y_{1,3}e2, \\ e3, e1 = -zw_{1,3}e1 + (-x_{2,3} - zw_{2,3})\,e2 + (-y_{1,1}z - zw_{3,3} - x_{3,3}z)\,e3, \\ e3, e2 = x_{1,3}e2, \\ e3, e3 = (x_{1,3}z - y_{1,3}z)\,e3 \end{cases}$$

By taking $\{\theta_1 = 1, \theta_2 = 0, \theta_3 = t, \theta_4 = 0, \theta_5 = 0, \theta_6 = 0\}$ and $\{\beta_1 = 1, \beta_2 = -1, \beta_3 = 0, \beta_4 = 1, \beta_5 = -1, \beta_6 = 0, \beta_7 = 1, \beta_8 = -1, \beta_9 = 0,\}$ and then, by replacing $D_1$ by $-tD_3$, we have $\mathrm{Ker}(T_{\theta_1,\dots,\theta_6;\beta_1,\dots,\beta_9}(\mathfrak{L}_5^4))$ is determined by the equations

$$\begin{cases} x_{1,3} = 0, \\ -y_{1,3} = 0, \\ -x_{2,1} + y_{2,1} = 0, \\ -x_{2,3} + ty_{2,3}z = 0, \\ y_{2,3} - ty_{2,3}z = 0, \\ x_{1,3}z - y_{1,3}z = 0, \\ x_{1,1} - ty_{3,3} + y_{3,3} = 0, \\ -y_{1,1} + ty_{3,3} - x_{3,3} = 0. \end{cases}$$

Therefore, since $z \neq 0$, $x_{1,1} = ty_{3,3} - y_{3,3}$, $x_{1,3} = 0$, $x_{2,1} = y_{2,1}$, $x_{2,3} = ty_{2,3}z$, $x_{3,3} = -y_{1,1} + ty_{3,3}$ and $y_{1,3} = 0$, and so:

$$\{\, y_{2,3}\,(-1 + tz) = 0$$

and consequently so

$$\mathrm{Dim}\, T_{\theta_1,\dots,\theta_6;\beta_1,\dots\beta_9}(\mathfrak{L}_5^4) = \begin{cases} 4, & \text{if } t = 1/z; \\ 3, & \text{in other case.} \end{cases}$$

### 7.5.2 $\psi_{\alpha,\beta}(\mathfrak{L}_5^4(z))$

In this case, we have $\psi_{\alpha,\beta}(\mathfrak{L}_5^4(z)) = \mathfrak{t}_{\alpha,\beta}$:

$$\mathfrak{t}_{\alpha,\beta} = \{[e1, e2] = e2, [e1, e3] = (\beta + \alpha\,z)\,e2 + z\,e3.$$

We note that $\mathfrak{t}_{\alpha,\beta}$ is isomorphic to $\mathfrak{r}_{3,z}$.

### 7.5.3 $\phi_\beta(\mathfrak{L}_5^4(z))$

In this case, we have $\phi_\beta(\mathfrak{L}_5^4(z)) = \widetilde{\mathfrak{t}}_\beta$:

$$\widetilde{\mathfrak{t}}_\beta = \{[e1, e3] = (\beta + z)\,e2.$$

We note that $\widetilde{\mathfrak{t}}_\beta$ is an almost abelian Lie algebra where $\mathrm{span}_{\mathbb{C}}\{e_2, e_3\}$ is a codimension 1 abelian ideal of $\widetilde{\mathfrak{t}}_\beta$; we can identify $\widetilde{\mathfrak{t}}_\beta$ the matrix

$$\widetilde{B} = \begin{bmatrix} 0 & \beta + z \\ 0 & 0 \end{bmatrix}.$$

The isomorphism class of $\widetilde{\mathfrak{t}}_\beta$ is determined by the eigenvalues of $\widetilde{B}$:

1. $\widetilde{B}$ has repeated eigenvalues if and only if $0 = 0$
    (a) $0$ is a repeated eigenvalue of $\widetilde{B}$ if and only if $0 = 0$
        i. $\widetilde{\mathfrak{t}}_\beta$ is the 3-dimensional abelian Lie algebra if and only if $\widetilde{B} = 0$ if and only if $\beta = -z$
        ii. $\widetilde{\mathfrak{t}}_\beta$ is isomorphic to $\mathfrak{n}_3(\mathbb{C})$ if and only if $\widetilde{B} \neq 0$ if and only if $\beta \neq -z$
    (b) $\widetilde{B}$ has a non-zero repeated eigenvalue if and only if $1 = 0$
2. $\widetilde{B}$ has distinct eigenvalues if and only if $0 = 1$

### 7.5.4 $\rho(\mathfrak{L}_5^4(z))$

In this case, we have $\rho(\mathfrak{L}_5^4(z)) = \widehat{\mathfrak{t}}$:

$$\widehat{\mathfrak{t}} = \{[e1, e3] = e2.$$

which is isomorphic to $\mathfrak{n}_3$.

### 7.5.5 Identities

- $A[\cdot, \cdot] - z[A\cdot, \cdot] - z[\cdot, A\cdot] = 0$ (Generalized derivation)
- $[A\cdot, A\cdot] = 0$

The hom-Lie algebra $\mathfrak{L}_5^4(z)$ is a non-multiplicative hom-Lie algebra.

## 7.6 $\mathfrak{L}_5^5(z)$

Let $A$ be the endomorphism of $\mathbb{C}^3$ whose matrix representation with respect to the ordered basis $\{e_1, e_2, e_3\}$ is

$$\begin{bmatrix} 0 & 0 & 0 \\ 0 & 0 & 0 \\ 0 & 1 & 0 \end{bmatrix}$$

and let us denote by $\mathfrak{L}_5^5(z)$ the hom-Lie algebra $(\mathfrak{L}_5(z), A)$. We have:

$$\mathfrak{gl}(n, \mathbb{C})_A = \left\{ \begin{bmatrix} t_1 & t_5 & 0 \\ 0 & t_2 & 0 \\ t_3 & t_4 & t_2 \end{bmatrix} \in M(3, \mathbb{C}) : t_1, \ldots, t_5 \in \mathbb{C} \right\}.$$

$$\mathrm{Aut}(\mathfrak{L}_5^5(z)) = \left\{ \begin{bmatrix} 1 & 0 & 0 \\ 0 & a & 0 \\ x & 0 & a \end{bmatrix} \in M(3, \mathbb{C}) : \begin{array}{l} a \in \mathbb{C}^* \\ x \in \mathbb{C} \end{array} \right\}.$$

$$\mathrm{Der}(\mathfrak{L}_5^5(z)) = \left\{ \begin{bmatrix} 0 & 0 & 0 \\ 0 & t_1 & 0 \\ t_2 & 0 & t_1 \end{bmatrix} \in M(3, \mathbb{C}) : t_1, t_2 \in \mathbb{C} \right\} \cong \mathfrak{r}_2(\mathbb{C}).$$

### 7.6.1 $T_{\theta_1,\ldots,\theta_6;\beta_1,\ldots\beta_9}(\mathfrak{L}_5^5)$

Given $D_1(w's)$, $D_2(x's)$ and $D_3(y's)$ in $\mathfrak{gl}(n,\mathbb{C})_A$, we have $\lambda := D_1[\cdot,\cdot] + [D_2\cdot,\cdot] + [\cdot,D_3\cdot]$

$$\lambda := \begin{cases} e1 \cdot e1 = (-x_{3,1}z + y_{3,1}z)\,e3, \\ e1 \cdot e2 = w_{1,2}e1 + (x_{1,1} + w_{3,3} + y_{3,3})\,e2 + (w_{3,2} + y_{3,2}z)\,e3, \\ e1 \cdot e3 = (x_{1,1}z + w_{3,3}z + y_{3,3}z)\,e3, \\ e2 \cdot e1 = -w_{1,2}e1 + (-y_{1,1} - w_{3,3} - x_{3,3})\,e2 + (-w_{3,2} - x_{3,2}z)\,e3, \\ e2 \cdot e2 = (x_{1,2} - y_{1,2})\,e2, [e2,e3], x_{1,2}z e3, \\ e3 \cdot e1 = (-y_{1,1}z - w_{3,3}z - x_{3,3}z)\,e3, \\ e3 \cdot e2 = -y_{1,2}z e3 \end{cases}$$

By taking $\{\theta_1 = 1, \theta_2 = 0, \theta_3 = t, \theta_4 = 0, \theta_5 = 0, \theta_6 = 0\}$ and $\{\beta_1 = 1, \beta_2 = -1, \beta_3 = 0, \beta_4 = 1, \beta_5 = -1, \beta_6 = 0, \beta_7 = 1, \beta_8 = -1, \beta_9 = 0,\}$ and then, by replacing $D_1$ by $-tD_3$, we have $\text{Ker}(T_{\theta_1,\ldots,\theta_6;\beta_1,\ldots,\beta_9}(\mathfrak{L}_5^5))$ is determined by the equations

$$\begin{cases} x_{1,2} = 0, \\ y_{1,2} = 0, \\ x_{1,2} - y_{1,2} = 0, \\ -ty_{3,2} + y_{3,2}z = 0, \\ ty_{3,2} - x_{3,2}z = 0, \\ -x_{3,1} + y_{3,1} = 0, \\ x_{1,1} - ty_{3,3} + y_{3,3} = 0, \\ -y_{1,1} + ty_{3,3} - x_{3,3} = 0. \end{cases}$$

Therefore $x_{1,1} = ty_{3,3} - y_{3,3}$, $x_{1,2} = 0$, $x_{3,1} = y_{3,1}$, $x_{3,2} = t\frac{y_{3,2}}{z}$, $x_{3,3} = -y_{1,1} + ty_{3,3}$ and $y_{1,2} = 0$, and so:

$$\{\, y_{3,2}(t - z) = 0$$

and consequently so

$$\text{Dim}\, T_{\theta_1,\ldots,\theta_6;\beta_1,\ldots\beta_9}(\mathfrak{L}_5^5) = \begin{cases} 4, & \text{if } t = z; \\ 3, & \text{in other case.} \end{cases}$$

### 7.6.2 $\psi_{\alpha,\beta}(\mathfrak{L}_5^5(z))$

In this case, we have $\psi_{\alpha,\beta}(\mathfrak{L}_5^5(z)) = \mathfrak{t}_{\alpha,\beta}$:

$$\mathfrak{t}_{\alpha,\beta} = \{[e1, e2] = e2 + (\alpha + \beta z)\,e3, [e1, e3] = z e3.$$

We note that $\mathfrak{t}_{\alpha,\beta}$ is isomorphic to $\mathfrak{r}_{3,z}$.

### 7.6.3 $\phi_\beta(\mathfrak{L}_5^5(z))$

In this case, we have $\phi_\beta(\mathfrak{L}_5^5(z)) = \widetilde{\mathfrak{t}}_\beta$:

$$\widetilde{\mathfrak{t}}_\beta = \{[e1, e2] = (1 + \beta z)\,e3.$$

We note that $\widetilde{\mathfrak{t}}_\beta$ is an almost abelian Lie algebra where $\text{span}_\mathbb{C}\{e_2, e_3\}$ is a codimension 1 abelian ideal of $\widetilde{\mathfrak{t}}_\beta$; we can identify $\widetilde{\mathfrak{t}}_\beta$ the matrix

$$\widetilde{B} = \begin{bmatrix} 0 & 0 \\ 1 + \beta z & 0 \end{bmatrix}.$$

The isomorphism class of $\widetilde{\mathfrak{t}}_\beta$ is determined by the eigenvalues of $\widetilde{B}$:

1. $\widetilde{B}$ has repeated eigenvalues if and only if $0 = 0$
    (a) 0 is a repeated eigenvalue of $\widetilde{B}$ if and only if $0 = 0$

i. $\widetilde{\mathfrak{t}}_\beta$ is the 3-dimensional abelian Lie algebra if and only if $\widetilde{B} = 0$ if and only if $\{\beta = -\frac{1}{z}\}$

ii. $\widetilde{\mathfrak{t}}_\beta$ is isomorphic to $\mathfrak{n}_3(\mathbb{C})$ if and only if $\widetilde{B} \neq 0$ if and only if $\{\beta \neq -\frac{1}{z}\}$

(b) $\widetilde{B}$ has a non-zero repeated eigenvalue if and only if $1 = 0$

2. $\widetilde{B}$ has distinct eigenvalues if and only if $0 = 1$

### 7.6.4 $\rho(\mathfrak{L}_5^5(z))$

In this case, we have $\rho(\mathfrak{L}_5^5(z)) = \widehat{\mathfrak{t}}$:

$$\widehat{\mathfrak{t}} = \{[e1, e2] = ze3.$$

which is isomorphic to $\mathfrak{n}_3$.

### 7.6.5 Identities

- $A[\cdot, \cdot] - \frac{1}{z}[A\cdot, \cdot] - \frac{1}{z}[\cdot, A\cdot] = 0$ (Generalized derivation)
- $[A\cdot, A\cdot]$

The hom-Lie algebra $\mathfrak{L}_5^5(z)$ is a non-multiplicative hom-Lie algebra.

## 7.7 $\mathfrak{L}_5^6(z, \lambda)$

Let $A$ be the endomorphism of $\mathbb{C}^3$ whose matrix representation with respect to the ordered basis $\{e_1, e_2, e_3\}$ is

$$\begin{bmatrix} 0 & 0 & 0 \\ 0 & \lambda & \lambda \\ 0 & -\lambda & -\lambda \end{bmatrix}$$

and let us denote by $\mathfrak{L}_5^6(z, \lambda)$ the hom-Lie algebra $(\mathfrak{L}_6(z), A)$. We have:

$$\mathrm{Aut}(\mathfrak{L}_5^6(z, \lambda)) = \left\{ \begin{bmatrix} 1 & 0 & 0 \\ x & a & 0 \\ -x & 0 & a \end{bmatrix} \in M(3, \mathbb{C}) : \begin{matrix} a \in \mathbb{C}^* \\ x \in \mathbb{C} \end{matrix} \right\}.$$

$$\mathrm{Der}(\mathfrak{L}_5^6(z, \lambda)) = \left\{ \begin{bmatrix} 0 & 0 & 0 \\ t_2 & t_1 & 0 \\ -t_2 & 0 & t_1 \end{bmatrix} \in M(3, \mathbb{C}) : t_1, t_2 \in \mathbb{C} \right\} \cong \mathfrak{r}_2(\mathbb{C}) \times \mathbb{C}.$$

### 7.7.1 $\psi_{\alpha,\beta}(\mathfrak{L}_5^6(z))$

In this case, we have $\psi_{\alpha,\beta}(\mathfrak{L}_5^6(z)) = \mathfrak{t}_{\alpha,\beta}$:

$\mathfrak{t}_{\alpha,\beta} = \{[e1, e2] = (1 + \alpha\lambda + \beta\lambda)\,e2 + (-\alpha\lambda - \beta\lambda z)\,e3, [e1, e3] = (\beta\lambda + \alpha z\lambda)\,e2 + (z - \alpha z\lambda - \beta\lambda z)\,e3.$

We note that $\mathfrak{t}_{\alpha,\beta}$ is an almost abelian Lie algebra where $\mathrm{span}_\mathbb{C}\{e_2, e_3\}$ is a codimension 1 abelian ideal of $\mathfrak{t}_{\alpha,\beta}$; we can identify $\mathfrak{t}_{\alpha,\beta}$ the matrix

$$B = \begin{bmatrix} 1 + \alpha\lambda + \beta\lambda & \beta\lambda + \alpha z\lambda \\ -\alpha\lambda - \beta\lambda z & z - \alpha z\lambda - \beta\lambda z \end{bmatrix}.$$

The isomorphism class of $\mathfrak{t}_{\alpha,\beta}$ is determined by the eigenvalues of $B$:

1. $B$ has repeated eigenvalues if and only if $(-1 + z)(\alpha - \beta)^2 \lambda^2 - 2(z + 1)(\beta + \alpha)\lambda - 1 + z = 0$

(a) 0 is a repeated eigenvalue of $B$ if and only if $\left\{\alpha = -\frac{-z-1\pm\sqrt{z^2+6z+1}}{2(-1+z)\lambda}, \beta = \frac{z+1\pm\sqrt{z^2+6z+1}}{2(-1+z)\lambda}\right\}$.

   i. $\mathfrak{t}_{\alpha,\beta}$ is the 3-dimensional abelian Lie algebra if and only if $B = 0$ if and only if $z = 0$ (False)
   
   ii. $\mathfrak{t}_{\alpha,\beta}$ is isomorphic to $\mathfrak{n}_3(\mathbb{C})$ if and only if $B \neq 0$ if and only if
   $$\left\{\alpha = -\frac{-z-1\pm\sqrt{z^2+6z+1}}{2(-1+z)\lambda}, \beta = \frac{z+1\pm\sqrt{z^2+6z+1}}{2(-1+z)\lambda}\right\}.$$

(b) $B$ has a non-zero repeated eigenvalue if and only if
$$\left\{\alpha = -\frac{-tz-t+1\pm\sqrt{t^2-2t+6t^2z-2tz+t^2z^2}}{2t\lambda(-1+z)}, \beta = \frac{t-1+tz\pm\sqrt{t^2-2t+6t^2z-2tz+t^2z^2}}{2t\lambda(-1+z)}\right\}$$ for some $t \in \mathbb{C}\setminus\{0\}$.

   i. $\mathfrak{t}_{\alpha,\beta}$ is $\mathfrak{r}_{3,1}$ if and only if $B = s\,\text{Id}$ for some $s \neq 0$ if and only if $1 = 0$
   
   ii. $\mathfrak{t}_{\alpha,\beta}$ is isomorphic to $\mathfrak{r}_3$ if and only if $B \neq s\,\text{Id}$ for all $s \in \mathbb{C}$ if and only if
   $$\left\{\alpha = -\frac{-tz-t+1\pm\sqrt{t^2-2t+6t^2z-2tz+t^2z^2}}{2t\lambda(-1+z)}, \beta = \frac{t-1+tz\pm\sqrt{t^2-2t+6t^2z-2tz+t^2z^2}}{2t\lambda(-1+z)}\right\}$$

2. $B$ has distinct eigenvalues eigenvalues if and only if $(-1+z)(\alpha-\beta)^2\lambda^2 - 2(z+1)(\beta+\alpha)\lambda - 1 + z \neq 0$

   (a) $\mathfrak{t}_{\alpha,\beta}$ is isomorphic to $\mathfrak{r}_2(\mathbb{C})\times\mathbb{C}$ if and only if 0 is an eigenvalue of $B$ with algebraic multiplicity 1 if and only if $\alpha = -\frac{z}{(-1+z)^2\lambda^2\beta}$ and $\beta \neq \frac{z+1\pm\sqrt{z^2+6z+1}}{2(-1+z)\lambda} \neq 0$
   
   (b) $\mathfrak{t}_{\alpha,\beta}$ is isomorphic to $\mathfrak{r}_{3,-1}$ if and only if $\text{Trace}(B) = 0$ if and only if $\alpha = -\frac{-1-\beta\lambda-z+\beta\lambda z}{\lambda(-1+z)}$ and $\beta \neq \frac{z+1\pm\sqrt{z^2+6z+1}}{2\lambda(-1+z)}$
   
   (c) $\mathfrak{t}_{\alpha,\beta}$ is isomorphic to $\mathfrak{r}_{3,z}$ for some $z$ such that $z^2 \neq 1$ and $z \neq 0$ if and only if $(-1+z)(\alpha-\beta)^2\lambda^2 - 2(z+1)(\beta+\alpha)\lambda - 1 + z \neq 0$ and $1 + \alpha\lambda + \beta\lambda + z - \alpha z\lambda - \beta\lambda z \neq 0$ and $z + \beta\lambda^2\alpha - 2\alpha\lambda^2\beta z + \alpha z^2\lambda^2\beta \neq 0$

### 7.7.2  $\phi_\beta(\mathfrak{L}_5^6(z))$

In this case, we have $\phi_\beta(\mathfrak{L}_5^6(z)) = \widetilde{\mathfrak{t}}_\beta$:

$$\widetilde{\mathfrak{t}}_\beta = \{[e1, e2] = (\lambda+\beta\lambda)e2 + (-\lambda-\beta\lambda z)e3, [e1, e3] = (\beta\lambda+z\lambda)e2 + (-z\lambda-\beta\lambda z)e3.$$

We note that $\widetilde{\mathfrak{t}}_\beta$ is an almost abelian Lie algebra where $\text{span}_\mathbb{C}\{e_2, e_3\}$ is a codimension 1 abelian ideal of $\widetilde{\mathfrak{t}}_\beta$; we can identify $\widetilde{\mathfrak{t}}_\beta$ the matrix

$$\widetilde{B} = \begin{bmatrix} \lambda+\beta\lambda & \beta\lambda+z\lambda \\ -\lambda-\beta\lambda z & -z\lambda-\beta\lambda z \end{bmatrix}.$$

The isomorphism class of $\widetilde{\mathfrak{t}}_\beta$ is determined by the eigenvalues of $\widetilde{B}$:

1. $\widetilde{B}$ has repeated eigenvalues if and only if $\beta = 1$

   (a) 0 is a repeated eigenvalue of $\widetilde{B}$ if and only if $\lambda = 0$ (False)
   
   (b) $\widetilde{B}$ has a non-zero repeated eigenvalue if and only if $\beta = 1$
   
   i. $\widetilde{\mathfrak{t}}_\beta$ is $\mathfrak{r}_{3,1}$ if and only if $\widetilde{B} = t\,\text{Id}$ for some $t \neq 0$ if and only if $z = -1$ (False)
   
   ii. $\widetilde{\mathfrak{t}}_\beta$ is isomorphic to $\mathfrak{r}_3$ if and only if $\widetilde{B} \neq t\,\text{Id}$ for all $t \in \mathbb{C}$ if and only if $\beta = 1$

2. $\widetilde{B}$ has distinct eigenvalues if and only if $\beta \neq 1$

   (a) $\widetilde{\mathfrak{t}}_\beta$ is isomorphic to $\mathfrak{r}_2(\mathbb{C})\times\mathbb{C}$ if and only if 0 is an eigenvalue of $B$ with algebraic multiplicity 1 if and only if $\beta = 0$
   
   (b) $\widetilde{\mathfrak{t}}_\beta$ is isomorphic to $\mathfrak{r}_{3,-1}$ if and only if $\text{Trace}(\widetilde{B}) = 0$ if and only if $\beta = -1$.
   
   (c) $\widetilde{\mathfrak{t}}_\beta$ is isomorphic to $\mathfrak{r}_{3,z}$ for some $z$ such that $z^2 \neq 1$ and $z \neq 0$ if and only if $\beta \neq 1$ and $\beta \neq 0$ and $\beta \neq -1$.

### 7.7.3 $\rho(\mathfrak{L}_5^6(z,\lambda))$

In this case, we have $\rho(\mathfrak{L}_5^6(z,\lambda)) = \widehat{\mathfrak{t}}$:

$$\widehat{\mathfrak{t}} = \{[e1, e2] = \lambda\, e2 - \lambda\, z e3, [e1, e3] = \lambda\, e2 - \lambda\, z e3.$$

which is isomorphic to $\mathfrak{r}_2 \times \mathbb{C}$.

### 7.7.4 Identities

- $[A\cdot, A\cdot] = 0$

The hom-Lie algebra $\mathfrak{L}_5^6(z,\lambda)$ is a non-multiplicative hom-Lie algebra.

## 7.8 $\mathfrak{L}_5^7(z)$

Let $A$ be the endomorphism of $\mathbb{C}^3$ whose matrix representation with respect to the ordered basis $\{e_1, e_2, e_3\}$ is

$$\begin{bmatrix} 0 & 0 & 0 \\ 1 & 0 & 0 \\ 0 & 1 & 0 \end{bmatrix}$$

and let us denote by $\mathfrak{L}_5^7(z)$ the hom-Lie algebra $(\mathfrak{L}_5(z), A)$. We have:

$$\mathrm{Aut}(\mathfrak{L}_5^7(z)) = \left\{ \begin{bmatrix} 1 & 0 & 0 \\ 0 & 1 & 0 \\ x & 0 & 1 \end{bmatrix} \in M(3, \mathbb{C}) : x \in \mathbb{C} \right\}.$$

$$\mathrm{Der}(\mathfrak{L}_5^7(z)) = \left\{ \begin{bmatrix} 0 & 0 & 0 \\ 0 & 0 & 0 \\ t_1 & 0 & 0 \end{bmatrix} \in M(3, \mathbb{C}) : t_1 \in \mathbb{C} \right\}.$$

### 7.8.1 $\psi_{\alpha,\beta}(\mathfrak{L}_5^7(z))$

In this case, we have $\psi_{\alpha,\beta}(\mathfrak{L}_5^7(z)) = \mathfrak{t}_{\alpha,\beta}$:

$$\mathfrak{t}_{\alpha,\beta} = \{[e1, e2] = e2 + (\alpha + \beta\, z)\, e3, [e1, e3] = z e3.$$

We note that $\mathfrak{t}_{\alpha,\beta}$ is isomorphic to $\mathfrak{r}_{3,z}$.

### 7.8.2 $\phi_\beta(\mathfrak{L}_5^7(z))$

In this case, we have $\phi_\beta(\mathfrak{L}_5^7(z)) = \widetilde{\mathfrak{t}}_\beta$:

$$\widetilde{\mathfrak{t}}_\beta = \{[e1, e2] = (1 + \beta\, z)\, e3.$$

We note that $\widetilde{\mathfrak{t}}_\beta$ is an almost abelian Lie algebra where $\mathrm{span}_\mathbb{C}\{e_2, e_3\}$ is a codimension 1 abelian ideal of $\widetilde{\mathfrak{t}}_\beta$; we can identify $\widetilde{\mathfrak{t}}_\beta$ the matrix

$$\widetilde{B} = \begin{bmatrix} 0 & 0 \\ 1 + \beta\, z & 0 \end{bmatrix}.$$

The isomorphism class of $\widetilde{\mathfrak{t}}_\beta$ is determined by the eigenvalues of $\widetilde{B}$:

1. $\widetilde{B}$ has repeated eigenvalues if and only if $0 = 0$
   (a) 0 is a repeated eigenvalue of $\widetilde{B}$ if and only if $0 = 0$

i. $\widetilde{\mathfrak{t}}_\beta$ is the 3-dimensional abelian Lie algebra if and only if $\widetilde{B} = 0$ if and only if $\{\beta = -\frac{1}{z}\}$

ii. $\widetilde{\mathfrak{t}}_\beta$ is isomorphic to $\mathfrak{n}_3(\mathbb{C})$ if and only if $\widetilde{B} \neq 0$ if and only if $\{\beta \neq -\frac{1}{z}\}$

(b) $\widetilde{B}$ has a non-zero repeated eigenvalue if and only if $1 = 0$

2. $\widetilde{B}$ has distinct eigenvalues if and only if $0 = 1$

### 7.8.3 $\rho(\mathfrak{L}_5^7(z))$

In this case, we have $\rho(\mathfrak{L}_5^7(z)) = \widehat{\mathfrak{t}}$:

$$\widehat{\mathfrak{t}} = \{[e1, e2] = ze3.$$

which is isomorphic to $\mathfrak{n}_3$

### 7.8.4 Identities

- $A[\cdot, \cdot] - \frac{1}{z}[A\cdot, \cdot] - \frac{1}{z}[\cdot, A\cdot] = 0$ (Generalized derivation)
- $A^2[\cdot, \cdot] = 0$
- $[A\cdot, A\cdot] = 0$

The hom-Lie algebra $\mathfrak{L}_5^7(z)$ is a non-multiplicative hom-Lie algebra.

## 7.9 $\mathfrak{L}_5^8(z)$

Let $A$ be the endomorphism of $\mathbb{C}^3$ whose matrix representation with respect to the ordered basis $\{e_1, e_2, e_3\}$ is

$$\begin{bmatrix} 0 & 0 & 0 \\ 0 & 0 & 1 \\ 1 & 0 & 0 \end{bmatrix}$$

and let us denote by $\mathfrak{L}_5^8(z)$ the hom-Lie algebra $(\mathfrak{L}(z), A)$. We have:

$$\operatorname{Aut}(\mathfrak{L}_5^8(z)) = \left\{ \begin{bmatrix} 1 & 0 & 0 \\ x & 1 & 0 \\ 0 & 0 & 1 \end{bmatrix} \in M(3, \mathbb{C}) : x \in \mathbb{C} \right\}.$$

$$\operatorname{Der}(\mathfrak{L}_5^8(z)) = \left\{ \begin{bmatrix} 0 & 0 & 0 \\ t_1 & 0 & 0 \\ 0 & 0 & 0 \end{bmatrix} \in M(3, \mathbb{C}) : t_1 \in \mathbb{C} \right\} \cong \mathbb{C}.$$

### 7.9.1 $\psi_{\alpha,\beta}(\mathfrak{L}_5^8(z))$

In this case, we have $\psi_{\alpha,\beta}(\mathfrak{L}_5^8(z)) = \mathfrak{t}_{\alpha,\beta}$:

$$\mathfrak{t}_{\alpha,\beta} = \{[e1, e2] = e2, [e1, e3] = (\beta + \alpha z)\, e2 + ze3.$$

We note that $\mathfrak{t}_{\alpha,\beta}$ is isomorphic to $\mathfrak{r}_{3,z}$.

### 7.9.2 $\phi_\beta(\mathfrak{L}_5^8(z))$

In this case, we have $\phi_\beta(\mathfrak{L}_5^8(z)) = \widetilde{\mathfrak{t}}_\beta$:

$$\widetilde{\mathfrak{t}}_\beta = \{[e1, e3] = (\beta + z)\, e2.$$

We note that $\widetilde{\mathfrak{t}}_\beta$ is an almost abelian Lie algebra where $\mathrm{span}_\mathbb{C}\{e_2, e_3\}$ is a codimension 1 abelian ideal of $\widetilde{\mathfrak{t}}_\beta$; we can identify $\widetilde{\mathfrak{t}}_\beta$ the matrix

$$\widetilde{B} = \begin{bmatrix} 0 & \beta + z \\ 0 & 0 \end{bmatrix}.$$

The isomorphism class of $\widetilde{\mathfrak{t}}_\beta$ is determined by the eigenvalues of $\widetilde{B}$:

1. $\widetilde{B}$ has repeated eigenvalues if and only if $0 = 0$
    (a) $0$ is a repeated eigenvalue of $\widetilde{B}$ if and only if $0 = 0$
        i. $\widetilde{\mathfrak{t}}_\beta$ is the 3-dimensional abelian Lie algebra if and only if $\widetilde{B} = 0$ if and only if $\beta = -z$
        ii. $\widetilde{\mathfrak{t}}_\beta$ is isomorphic to $\mathfrak{n}_3(\mathbb{C})$ if and only if $\widetilde{B} \neq 0$ if and only if $\beta \neq -z$
    (b) $\widetilde{B}$ has a non-zero repeated eigenvalue if and only if $1 = 0$
2. $\widetilde{B}$ has distinct eigenvalues if and only if $0 = 1$

### 7.9.3 $\rho(\mathfrak{L}_5^8(z))$

In this case, we have $\rho(\mathfrak{L}_5^8(z)) = \widehat{\mathfrak{t}}$:

$$\widehat{\mathfrak{t}} = \{[e1, e3] = e2.$$

which is isomorphic to $\mathfrak{n}_3$.

### 7.9.4 Identities

- $A[\cdot, \cdot] - z[A\cdot, \cdot] - z[\cdot, A\cdot] = 0$
- $[A\cdot, A\cdot] = 0$
- $A^2[\cdot, \cdot] = 0$
    → $[A^2\cdot, \cdot] + [\cdot, A^2\cdot]$

The hom-Lie algebra $\mathfrak{L}_5^8(z)$ is a non-multiplicative hom-Lie algebra.

## 7.10 $\mathfrak{L}_5^9(z, \lambda)$

Let $A$ be the endomorphism of $\mathbb{C}^3$ whose matrix representation with respect to the ordered basis $\{e_1, e_2, e_3\}$ is

$$\begin{bmatrix} 0 & 0 & 0 \\ 1 & \lambda & \lambda \\ 0 & -\lambda & -\lambda \end{bmatrix}$$

and let us denote by $\mathfrak{L}_5^9(z, \lambda)$ the hom-Lie algebra $(\mathfrak{L}_5(z), A)$. We have:

$$\mathrm{Aut}(\mathfrak{L}_5^9(z, \lambda)) = \left\{ \begin{bmatrix} 1 & 0 & 0 \\ x & 1 & 0 \\ -x & 0 & 1 \end{bmatrix} \in M(3, \mathbb{C}) : x \in \mathbb{C} \right\}.$$

$$\mathrm{Der}(\mathfrak{L}_5^9(z, \lambda)) = \left\{ \begin{bmatrix} 0 & 0 & 0 \\ t_1 & 0 & 0 \\ -t_1 & 0 & 0 \end{bmatrix} \in M(3, \mathbb{C}) : t_1 \in \mathbb{C} \right\} \cong \mathbb{C}.$$

### 7.10.1 $\psi_{\alpha,\beta}(\mathfrak{L}_5^9(z))$

In this case, we have $\psi_{\alpha,\beta}(\mathfrak{L}_5^9(z)) = \mathfrak{t}_{\alpha,\beta}$:

$$\mathfrak{t}_{\alpha,\beta} = \{[e1,e2] = (1+\alpha\lambda+\beta\lambda)\,e2 + (-\alpha\lambda-\beta\lambda z)\,e3, [e1,e3] = (\beta\lambda+\alpha z\lambda)\,e2 + (z-\alpha z\lambda-\beta\lambda z)\,e3.$$

We note that $\mathfrak{t}_{\alpha,\beta}$ is an almost abelian Lie algebra where $\mathrm{span}_{\mathbb{C}}\{e_2,e_3\}$ is a codimension 1 abelian ideal of $\mathfrak{t}_{\alpha,\beta}$; we can identify $\mathfrak{t}_{\alpha,\beta}$ the matrix

$$B = \begin{bmatrix} 1+\alpha\lambda+\beta\lambda & \lambda(\beta+\alpha z) \\ -\lambda(\alpha+\beta z) & z(1-\alpha\lambda-\beta\lambda) \end{bmatrix}.$$

The isomorphism class of $\mathfrak{t}_{\alpha,\beta}$ is determined by the eigenvalues of $B$:

1. $B$ has repeated eigenvalues if and only if $(-1+z)(\alpha-\beta)^2\lambda^2 - 2(z+1)(\beta+\alpha)\lambda - 1 + z = 0$

   (a) 0 is a repeated eigenvalue of $B$ if and only if $\left\{\alpha = -\frac{-z-1\pm\sqrt{z^2+6z+1}}{2(-1+z)\lambda}, \beta = \frac{z+1\pm\sqrt{z^2+6z+1}}{2(-1+z)\lambda}\right\}$.
   
   i. $\mathfrak{t}_{\alpha,\beta}$ is the 3-dimensional abelian Lie algebra if and only if $B = 0$ if and only if $z = 0$ (False)
   
   ii. $\mathfrak{t}_{\alpha,\beta}$ is isomorphic to $\mathfrak{n}_3(\mathbb{C})$ if and only if $B \ne 0$ if and only if
   $\left\{\alpha = -\frac{-z-1\pm\sqrt{z^2+6z+1}}{2(-1+z)\lambda}, \beta = \frac{z+1\pm\sqrt{z^2+6z+1}}{2(-1+z)\lambda}\right\}$.

   (b) $B$ has a non-zero repeated eigenvalue if and only if
   $\left\{\alpha = \frac{sz+s-1\pm\sqrt{s(s-2+6sz-2z+sz^2)}}{2(-1+z)\lambda s}, \beta = \frac{s-1+sz\mp\sqrt{s(s-2+6sz-2z+sz^2)}}{2(-1+z)\lambda s}\right\}$ for any $s \in \mathbb{C}^\star$.
   
   i. $\mathfrak{t}_{\alpha,\beta}$ is $\mathfrak{r}_{3,1}$ if and only if $B = s\,\mathrm{Id}$ for some $s \ne 0$ if and only if $1 = 0$
   
   ii. $\mathfrak{t}_{\alpha,\beta}$ is isomorphic to $\mathfrak{r}_3$ if and only if $B \ne t\,\mathrm{Id}$ for all $t \in \mathbb{C}$ if and only if
   $\left\{\alpha = \frac{sz+s-1\pm\sqrt{s(s-2+6sz-2z+sz^2)}}{2s\lambda(-1+z)}, \beta = \frac{sz+s-1\mp\sqrt{s(s-2+6sz-2z+sz^2)}}{2s\lambda(-1+z)}\right\}$ for any $s \in \mathbb{C}^\star$.

2. $B$ has distinct eigenvalues eigenvalues if and only if $(-1+z)(\alpha-\beta)^2\lambda^2 - 2(z+1)(\beta+\alpha)\lambda - 1 + z \ne 0$

   (a) $\mathfrak{t}_{\alpha,\beta}$ is isomorphic to $\mathfrak{r}_2(\mathbb{C}) \times \mathbb{C}$ if and only if 0 is an eigenvalue of $B$ with algebraic multiplicity 1 if and only if
   $\left\{\alpha = \frac{sz+s-1\pm\sqrt{s^2-2s+6zs^2+1-2sz+z^2s^2}}{2s\lambda(-1+z)}, \beta = \frac{sz+s-1\mp\sqrt{s^2-2s+6zs^2+1-2sz+z^2s^2}}{2s\lambda(-1+z)}\right\}$ for any $s \in \mathbb{C}^\star$.

   (b) $\mathfrak{t}_{\alpha,\beta}$ is isomorphic to $\mathfrak{r}_{3,-1}$ if and only if $\mathrm{Trace}(B) = 0$ if and only if
   $\left\{\alpha = \frac{sz+s\pm\sqrt{s(z^2s+6sz+s-4)}}{2s\lambda(-1+z)}, \beta = \frac{sz+s\mp\sqrt{s(z^2s+6sz+s-4)}}{2s\lambda(-1+z)}\right\}$

   (c) $\mathfrak{t}_{\alpha,\beta}$ is isomorphic to $\mathfrak{r}_{3,\widetilde{z}}$ for some $\widetilde{z}$ such that $\widetilde{z}(\widetilde{z}^2-1) \ne 0$ if and only if
   $\alpha = \frac{zs_1s_2-s_1+s_1s_2\pm\sqrt{s_1s_2(z^2s_1s_2-2zs_1+6zs_1s_2-2s_1+s_1s_2+s_2)}}{2s_1s_2\lambda(-1+z)}$ and
   $\beta = \frac{zs_1s_2-s_1+s_1s_2\mp\sqrt{s_1s_2(z^2s_1s_2-2zs_1+6zs_1s_2-2s_1+s_1s_2+s_2)}}{2s_1s_2\lambda(-1+z)}$ for any $s_1,s_2 \in \mathbb{C}^\star$ with $s_1 - s_2^2 \ne 0$.

### 7.10.2 $\phi_\beta(\mathfrak{L}_5^9(z))$

In this case, we have $\phi_\beta(\mathfrak{L}_5^9(z)) = \widetilde{\mathfrak{t}}_\beta$:

$$\widetilde{\mathfrak{t}}_\beta = \{[e1,e2] = (\lambda+\beta\lambda)\,e2 + (-\lambda-\beta\lambda z)\,e3, [e1,e3] = (\beta\lambda+z\lambda)\,e2 + (-z\lambda-\beta\lambda z)\,e3.$$

We note that $\widetilde{\mathfrak{t}}_\beta$ is an almost abelian Lie algebra where $\operatorname{span}_\mathbb{C}\{e_2, e_3\}$ is a codimension 1 abelian ideal of $\widetilde{\mathfrak{t}}_\beta$; we can identify $\widetilde{\mathfrak{t}}_\beta$ the matrix

$$\widetilde{B} = \begin{bmatrix} \lambda + \beta\lambda & \beta\lambda + z\lambda \\ -\lambda - \beta\lambda z & -z\lambda - \beta\lambda z \end{bmatrix}.$$

The isomorphism class of $\widetilde{\mathfrak{t}}_\beta$ is determined by the eigenvalues of $\widetilde{B}$:

1. $\widetilde{B}$ has repeated eigenvalues if and only if $\beta = 1$
    (a) 0 is a repeated eigenvalue of $\widetilde{B}$ if and only if $\lambda = 0$ (False)
    (b) $\widetilde{B}$ has a non-zero repeated eigenvalue if and only if $\beta = 1$
        i. $\widetilde{\mathfrak{t}}_\beta$ is $\mathfrak{r}_{3,1}$ if and only if $\widetilde{B} = t\operatorname{Id}$ for some $t \neq 0$ if and only if $z = -1$ (False)
        ii. $\widetilde{\mathfrak{t}}_\beta$ is isomorphic to $\mathfrak{r}_3$ if and only if $\widetilde{B} \neq t\operatorname{Id}$ for all $t \in \mathbb{C}$ if and only if $\beta = 1$

2. $\widetilde{B}$ has distinct eigenvalues if and only if $\beta \neq 1$
    (a) $\widetilde{\mathfrak{t}}_\beta$ is isomorphic to $\mathfrak{r}_2(\mathbb{C}) \times \mathbb{C}$ if and only if 0 is an eigenvalue of $B$ with algebraic multiplicity 1 if and only if $\beta = 0$
    (b) $\widetilde{\mathfrak{t}}_\beta$ is isomorphic to $\mathfrak{r}_{3,-1}$ if and only if $\operatorname{Trace}(\widetilde{B}) = 0$ if and only if $\beta = -1$.
    (c) $\widetilde{\mathfrak{t}}_\beta$ is isomorphic to $\mathfrak{r}_{3,z}$ for some $z$ such that $z^2 \neq 1$ and $z \neq 0$ if and only if $\beta \neq 1$ and $\beta \neq 0$ and $\beta \neq -1$.

### 7.10.3 $\rho(\mathfrak{L}_5^9(z, \lambda))$

In this case, we have $\rho(\mathfrak{L}_5^9(z, \lambda)) = \widehat{\mathfrak{t}}$:

$$\widehat{\mathfrak{t}} = \{[e1, e2] = \lambda\, e2 - \lambda\, z\, e3, [e1, e3] = \lambda\, e2 - \lambda\, z\, e3.$$

which is isomorphic to $\mathfrak{r}_2 \times \mathbb{C}$.

### 7.10.4 Identities

- $[A\cdot, A\cdot] = 0$
- $A^2[\cdot, \cdot] = 0$
- $[A^2\cdot, \cdot] + [\cdot, A^2\cdot] = 0$

The hom-Lie algebra $\mathfrak{L}_5^9(z, \lambda)$ is a non-multiplicative hom-Lie algebra.

## 7.11 Degenerations between hom-Lie algebras ($\mathfrak{L}_5^i$)

If $\mathfrak{L}_i^j \xrightarrow{\deg} \mathfrak{L}_h^k$ then $\operatorname{Der}(\mathfrak{L}_i^j) \leq \operatorname{Der}(\mathfrak{L}_h^k)$. Therefore, we can organize the hom-Lie algebras in the following way:

| Dim(Der) | hom-Lie algebra |
|---|---|
| 1 | $\mathfrak{L}_5^9(\lambda), \mathfrak{L}_5^7, \mathfrak{L}_5^8$ |
| 2 | $\mathfrak{L}_5^6(\lambda), \mathfrak{L}_5^5, \mathfrak{L}_5^4, \mathfrak{L}_5^3$ |
| 3 | $\mathfrak{L}_5^2, \mathfrak{L}_5^1$ |
| 4 | $\mathfrak{L}_5^0$ |

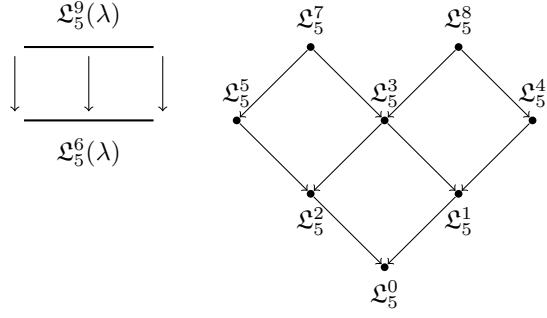

### 7.11.1 Degenerations

1. $\mathfrak{L}_5^9(z,\lambda) \xrightarrow{\deg} \mathfrak{L}_5^6(z,\lambda)$
   In fact, set
   $$g(t) = \begin{bmatrix} 1 & 0 & 0 \\ 0 & e^{-t} & 0 \\ e^{-t} & 0 & e^{-t} \end{bmatrix}, \text{ with } t \in \mathbb{R}.$$
   We have $g(t) \in \operatorname{Aut}(\mathfrak{L}_5(z))$ and $g(t) \cdot \mathfrak{L}_5^9(z)$ is the hom-Lie algebra $(\mathfrak{L}_5(z), A(t))$ with
   $$A(t) = \begin{bmatrix} 0 & 0 & 0 \\ -e^{-t}(-1+\lambda) & \lambda & \lambda \\ e^{-t}\lambda & -\lambda & -\lambda \end{bmatrix}$$
   It is easy to check that $(\mathfrak{L}_5, A(t)) \to \mathfrak{L}_5^6(z,\lambda)$ as $t$ tends to infinity.

2. $\mathfrak{L}_5^8(z) \xrightarrow{\deg} \mathfrak{L}_5^4(z)$
   In fact, set
   $$g(t) = \begin{bmatrix} 1 & 0 & 0 \\ 0 & e^{-t} & 0 \\ e^{-t} & 0 & e^{-t} \end{bmatrix}, \text{ with } t \in \mathbb{R}.$$
   We have $g(t) \in \operatorname{Aut}(\mathfrak{L}_5(z))$ and $g(t) \cdot \mathfrak{L}_5^8(z)$ is the hom-Lie algebra $(\mathfrak{L}_5, A(t))$ with
   $$A(t) = \begin{bmatrix} 0 & 0 & 0 \\ -e^{-t} & 0 & 1 \\ e^{-t} & 0 & 0 \end{bmatrix}$$
   It is easy to check that $(\mathfrak{L}_5(z), A(t)) \to \mathfrak{L}_5^4(z)$ as $t$ tends to infinity.

3. $\mathfrak{L}_5^8(z) \xrightarrow{\deg} \mathfrak{L}_5^3(z)$
   In fact, set
   $$g(t) = \begin{bmatrix} 1 & 0 & 0 \\ 0 & -e^{-t} & 0 \\ e^t & 0 & 1 \end{bmatrix}, \text{ with } t \in \mathbb{R}.$$
   We have $g(t) \in \operatorname{Aut}(\mathfrak{L}_5(z))$ and $g(t) \cdot \mathfrak{L}_5^8(z)$ is the hom-Lie algebra $(\mathfrak{L}_5, A(t))$ with
   $$A(t) = \begin{bmatrix} 0 & 0 & 0 \\ 1 & 0 & -e^{-t} \\ 1 & 0 & 0 \end{bmatrix}$$

It is easy to check that $(\mathfrak{L}_5(z), A(t)) \to \mathfrak{L}_5^3(z)$ as $t$ tends to infinity.

4. $\mathfrak{L}_5^7(z) \xrightarrow{\text{deg}} \mathfrak{L}_5^5(z)$
   In fact, set
   $$g(t) = \begin{bmatrix} 1 & 0 & 0 \\ 0 & e^{-t} & 0 \\ 0 & 0 & e^{-t} \end{bmatrix}, \text{ with } t \in \mathbb{R}.$$
   
   We have $g(t) \in \text{Aut}(\mathfrak{L}_5(z))$ and $g(t) \cdot \mathfrak{L}_5^7(z)$ is the hom-Lie algebra $(\mathfrak{L}_5, A(t))$ with
   $$A(t) = \begin{bmatrix} 0 & 0 & 0 \\ e^{-t} & 0 & 0 \\ 0 & 1 & 0 \end{bmatrix}$$
   
   It is easy to check that $(\mathfrak{L}_5(z), A(t)) \to \mathfrak{L}_5^5(z)$ as $t$ tends to infinity.

5. $\mathfrak{L}_5^7(z) \xrightarrow{\text{deg}} \mathfrak{L}_5^3(z)$
   In fact, set
   $$g(t) = \begin{bmatrix} 1 & 0 & 0 \\ -e^t & 1 & 0 \\ 0 & 0 & e^{-t} \end{bmatrix}, \text{ with } t \in \mathbb{R}.$$
   
   We have $g(t) \in \text{Aut}(\mathfrak{L}_5(z))$ and $g(t) \cdot \mathfrak{L}_5^7(z)$ is the hom-Lie algebra $(\mathfrak{L}_5, A(t))$ with
   $$A(t) = \begin{bmatrix} 0 & 0 & 0 \\ 1 & 0 & 0 \\ 1 & e^{-t} & 0 \end{bmatrix}$$
   
   It is easy to check that $(\mathfrak{L}_5(z), A(t)) \to \mathfrak{L}_5^3(z)$ as $t$ tends to infinity.

6. $\mathfrak{L}_5^5(z) \xrightarrow{\text{deg}} \mathfrak{L}_5^2(z)$
   In fact, set
   $$g(t) = \begin{bmatrix} 1 & 0 & 0 \\ e^t & -e^t & 0 \\ 0 & 0 & 1 \end{bmatrix}, \text{ with } t \in \mathbb{R}.$$
   
   We have $g(t) \in \text{Aut}(\mathfrak{L}_5(z))$ and $g(t) \cdot \mathfrak{L}_5^5(z)$ is the hom-Lie algebra $(\mathfrak{L}_5, A(t))$ with
   $$A(t) = \begin{bmatrix} 0 & 0 & 0 \\ 0 & 0 & 0 \\ 1 & -e^{-t} & 0 \end{bmatrix}$$
   
   It is easy to check that $(\mathfrak{L}_5, A(t)) \to \mathfrak{L}_5^2(z)$ as $t$ tends to infinity.

7. $\mathfrak{L}_5^4(z) \xrightarrow{\text{deg}} \mathfrak{L}_5^1(z)$
   In fact, set
   $$g(t) = \begin{bmatrix} 1 & 0 & 0 \\ 0 & 1 & 0 \\ e^t & 0 & -e^t \end{bmatrix}, \text{ with } t \in \mathbb{R}.$$

63

We have $g(t) \in \text{Aut}(\mathfrak{L}_5)$ and $g(t) \cdot \mathfrak{L}_5^4(z)$ is the hom-Lie algebra $(\mathfrak{L}_5, A(t))$ with

$$A(t) = \begin{bmatrix} 0 & 0 & 0 \\ 1 & 0 & -e^{-t} \\ 0 & 0 & 0 \end{bmatrix}$$

It is easy to check that $(\mathfrak{L}_5, A(t)) \to \mathfrak{L}_5^1(z)$ as $t$ tends to infinity.

8. $\mathfrak{L}_5^3(z) \xrightarrow{\deg} \mathfrak{L}_5^2(z)$
   In fact, set

$$g(t) = \begin{bmatrix} 1 & 0 & 0 \\ 0 & e^{-t} & 0 \\ 0 & 0 & 1 \end{bmatrix}, \text{ with } t \in \mathbb{R}.$$

We have $g(t) \in \text{Aut}(\mathfrak{L}_5)$ and $g(t) \cdot \mathfrak{L}_5^3(z)$ is the hom-Lie algebra $(\mathfrak{L}_5, A(t))$ with

$$A(t) = \begin{bmatrix} 0 & 0 & 0 \\ e^{-t} & 0 & 0 \\ 1 & 0 & 0 \end{bmatrix}$$

It is easy to check that $(\mathfrak{L}_5, A(t)) \to \mathfrak{L}_5^2(z)$ as $t$ tends to infinity.

9. $\mathfrak{L}_5^3(z) \xrightarrow{\deg} \mathfrak{L}_5^1(z)$
   In fact, set

$$g(t) = \begin{bmatrix} 1 & 0 & 0 \\ 0 & 1 & 0 \\ 0 & 0 & e^{-t} \end{bmatrix}, \text{ with } t \in \mathbb{R}.$$

We have $g(t) \in \text{Aut}(\mathfrak{L}_5)$ and $g(t) \cdot \mathfrak{L}_5^3(z)$ is the hom-Lie algebra $(\mathfrak{L}_5, A(t))$ with

$$A(t) = \begin{bmatrix} 0 & 0 & 0 \\ 1 & 0 & 0 \\ e^{-t} & 0 & 0 \end{bmatrix}$$

It is easy to check that $(\mathfrak{L}_5, A(t)) \to \mathfrak{L}_5^1(z)$ as $t$ tends to infinity.

10. $\mathfrak{L}_5^2(z) \xrightarrow{\deg} \mathfrak{L}_5^0(z)$
    In fact, set

$$g(t) = \begin{bmatrix} 1 & 0 & 0 \\ 0 & 1 & 0 \\ 0 & 0 & e^{-t} \end{bmatrix}, \text{ with } t \in \mathbb{R}.$$

We have $g(t) \in \text{Aut}(\mathfrak{L}_5)$ and $g(t) \cdot \mathfrak{L}_5^2(z)$ is the hom-Lie algebra $(\mathfrak{L}_5, A(t))$ with

$$A(t) = \begin{bmatrix} 0 & 0 & 0 \\ 0 & 0 & 0 \\ e^{-t} & 0 & 0 \end{bmatrix}$$

It is easy to check that $(\mathfrak{L}_5, A(t)) \to \mathfrak{L}_5^0(z)$ as $t$ tends to infinity.

11. $\mathfrak{L}_5^1(z) \xrightarrow{\deg} \mathfrak{L}_5^0(z)$

    In fact, set
    $$g(t) = \begin{bmatrix} 1 & 0 & 0 \\ 0 & e^{-t} & 0 \\ 0 & 0 & 1 \end{bmatrix}, \text{ with } t \in \mathbb{R}.$$

    We have $g(t) \in \mathrm{Aut}(\mathfrak{L}_5)$ and $g(t) \cdot \mathfrak{L}_5^1(z)$ is the hom-Lie algebra $(\mathfrak{L}_5, A(t))$ with
    $$A(t) = \begin{bmatrix} 0 & 0 & 0 \\ e^{-t} & 0 & 0 \\ 0 & 0 & 0 \end{bmatrix}$$

    It is easy to check that $(\mathfrak{L}_5, A(t)) \to \mathfrak{L}_5^0(z)$ as $t$ tends to infinity.

12. $\mathfrak{L}_5^9(z, \lambda) \xrightarrow{\deg} \mathfrak{L}_5^6(z, \kappa)$ with $(\lambda \neq \kappa)$.

    Suppose, contrary to our claim, that $\mathfrak{L}_5^9(z, \lambda) \xrightarrow{\deg} \mathfrak{L}_5^6(z, \kappa)$ with $(\lambda \neq \kappa)$. In such case, we have
    $$\psi_{\alpha,\beta}(\mathfrak{L}_5^9(z, \lambda)) \xrightarrow{\deg} \psi_{\alpha,\beta}(\mathfrak{L}_5^6(z, \kappa)).$$

    By taking $\alpha = (z+1 \mp \sqrt{z^2+6z+1})/2(-1+z)\lambda$ and $\beta = (z+1 \pm \sqrt{z^2+6z+1})/2(-1+z)\lambda$, we have from 7.10.1 that $\psi_{\alpha,\beta}(\mathfrak{L}_5^9(z, \lambda))$ is a Lie algebra isomorphic to $\mathfrak{n}_3$, therefore, $\psi_{\alpha,\beta}(\mathfrak{L}_5^6(z, \kappa))$ is a Lie algebra which is isomorphic to one of the following: $\mathfrak{n}_3$ or $\mathfrak{a}_3$ (by Theorem 1.1). We must have $\psi_{\alpha,\beta}(\mathfrak{L}_5^6(z, \kappa))$ is isomorphic to $\mathfrak{n}_3$ (from 7.7.1). But, we now apply again 7.7.1, to obtain $\alpha = (z+1 \mp \sqrt{z^2+6z+1})/2(-1+z)\kappa$ and $\beta = (z+1 \pm \sqrt{z^2+6z+1})/2(-1+z)\kappa$, and therefore $\lambda = \kappa$; this is contradiction.

13. $\mathfrak{L}_5^9(z, \lambda) \xrightarrow{\deg} \mathfrak{L}_5^0$.

    Suppose, contrary to our claim, that $\mathfrak{L}_5^9(z, \lambda) \xrightarrow{\deg} \mathfrak{L}_5^0$. In such case, we have
    $$\psi_{\alpha,\beta}(\mathfrak{L}_5^9(z, \lambda)) \xrightarrow{\deg} \psi_{\alpha,\beta}(\mathfrak{L}_5^0).$$

    By taking $\alpha = (z+1 \mp \sqrt{z^2+6z+1})/2(-1+z)\lambda$ and $\beta = (z+1 \pm \sqrt{z^2+6z+1})/2(-1+z)\lambda$, we have from 7.10.1 that $\psi_{\alpha,\beta}(\mathfrak{L}_5^9(z, \lambda))$ is a Lie algebra isomorphic to $\mathfrak{n}_3$, therefore, $\psi_{\alpha,\beta}(\mathfrak{L}_5^0(z))$ is a Lie algebra which is isomorphic to one of the following: $\mathfrak{n}_3$ or $\mathfrak{a}_3$ (by Theorem 1.1). But, from 7.1.1, we have $\psi_{\alpha,\beta}(\mathfrak{L}_5^0(z))$ is isomorphic to $\mathfrak{r}_{3,z}$; this is contradiction.

14. $\mathfrak{L}_5^8(z) \xrightarrow{\deg} \mathfrak{L}_5^6(z, \lambda)$ with $(\lambda \neq 0)$.

    Suppose, contrary to our claim, that $\mathfrak{L}_5^8(z) \xrightarrow{\deg} \mathfrak{L}_5^6(z, \lambda)$ with $(\lambda \neq 0)$. In such case, we have
    $$\phi_\beta(\mathfrak{L}_5^8(z)) \xrightarrow{\deg} \phi_\beta(\mathfrak{L}_5^6(z, \lambda)).$$

    By taking $\beta = -z$, we have from 7.9.2 that $\phi_\beta(\mathfrak{L}_5^8(z))$ is the 3-dimensional abelian Lie algebra, therefore, $\phi_\beta(\mathfrak{L}_5^6(z, \lambda))$ is a Lie algebra which is isomorphic to $\mathfrak{a}_3$ (by Theorem 1.1), but this contradicts to 7.7.2.

15. $\mathfrak{L}_5^8(z) \xrightarrow{\deg} \mathfrak{L}_5^5(z)$.

    Suppose, contrary to our claim, that $\mathfrak{L}_5^8(z) \xrightarrow{\deg} \mathfrak{L}_5^5(z)$. In such case, we have
    $$\phi_\beta(\mathfrak{L}_5^8(z)) \xrightarrow{\deg} \phi_\beta(\mathfrak{L}_5^6(z)).$$

    By taking $\beta = -z$, we have from 7.9.2 that $\phi_\beta(\mathfrak{L}_5^8(z))$ is the 3-dimensional abelian Lie algebra, therefore, $\phi_\beta(\mathfrak{L}_5^5(z))$ is a Lie algebra which is isomorphic to $\mathfrak{a}_3$ (by Theorem 1.1). From 7.6.3, we have $\beta = -1/z$ and so $z^2 = 1$; this is a contradiction.

16. $\mathfrak{L}_5^7(z) \xrightarrow{\deg} \mathfrak{L}_5^6(z, \lambda)$ with $(\lambda \neq 0)$.

   Suppose, contrary to our claim, that $\mathfrak{L}_5^7(z) \xrightarrow{\deg} \mathfrak{L}_5^6(z, \lambda)$ with $(\lambda \neq 0)$. In such case, we have

   $$\phi_\beta(\mathfrak{L}_5^7(z)) \xrightarrow{\deg} \phi_\beta(\mathfrak{L}_5^6(z, \lambda)).$$

   By taking $\beta = -1/z$, we have from 7.8.2 that $\phi_\beta(\mathfrak{L}_5^7(z))$ is the 3-dimensional abelian Lie algebra, therefore, $\phi_\beta(\mathfrak{L}_5^6(z, \lambda))$ is a Lie algebra which is isomorphic to $\mathfrak{a}_3$ (by Theorem 1.1), but this contradicts to 7.7.2.

17. $\mathfrak{L}_5^7(z) \xrightarrow{\deg} \mathfrak{L}_5^4(z)$.

   Suppose, contrary to our claim, that $\mathfrak{L}_5^7(z) \xrightarrow{\deg} \mathfrak{L}_5^4(z)$. In such case, we have

   $$\phi_\beta(\mathfrak{L}_5^7(z)) \xrightarrow{\deg} \phi_\beta(\mathfrak{L}_5^4(z)).$$

   By taking $\beta = -1/z$, we have from 7.9.2 that $\phi_\beta(\mathfrak{L}_5^7(z))$ is the 3-dimensional abelian Lie algebra, therefore, $\phi_\beta(\mathfrak{L}_5^4(z))$ is a Lie algebra which is isomorphic to $\mathfrak{a}_3$ (by Theorem 1.1). From 7.5.3, we have $\beta = -z$ and so $z^2 = 1$; this is a contradiction.

18. $\mathfrak{L}_5^5(z) \xrightarrow{\deg} \mathfrak{L}_5^1(z)$.

   Suppose, contrary to our claim, that $\mathfrak{L}_5^5(z) \xrightarrow{\deg} \mathfrak{L}_5^1(z)$. In such case, for $t = z$, we have

   $$\dim \operatorname{Ker}(T_{1,0,z,0,0,0;1,-1,0,1,-1,0,1,-1,0}(\mathfrak{L}_5^5(z))) \stackrel{7.6.1}{=} 4 \leq \dim \operatorname{Ker}(T_{1,0,z,0,0,0;1,-1,0,1,-1,0,1,-1,0}(\mathfrak{L}_5^1(z)))$$

   From 7.2.1, we have $\dim \operatorname{Ker}(T_{1,0,z,0,0,0;1,-1,0,1,-1,0,1,-1,0}(\mathfrak{L}_5^1(z))) = 4$, and therefore, $z = 1$ or $z = 1/z$. Since $z \neq \pm 1$; we have a contradiction.

19. $\mathfrak{L}_5^4(z) \xrightarrow{\deg} \mathfrak{L}_5^2(z)$.

   Suppose, contrary to our claim, that $\mathfrak{L}_5^4(z) \xrightarrow{\deg} \mathfrak{L}_5^2(z)$. In such case, for $t = 1/z$, we have

   $$\dim(\operatorname{Ker} T_{1,0,1/z,0,0,0;1,-1,0,1,-1,0,1,-1,0}(\mathfrak{L}_5^4(z))) \stackrel{7.5.1}{=} 4 \leq \dim(\operatorname{Ker} T_{1,0,1/z,0,0,0;1,-1,0,1,-1,0,1,-1,0}(\mathfrak{L}_5^2(z)))$$

   From 7.3.1, we have $\dim \operatorname{Ker}(T_{1,0,1/z,0,0,0;1,-1,0,1,-1,0,1,-1,0}(\mathfrak{L}_5^2(z))) = 4$, and therefore, $z = 1$ or $z = 1/z$. Since $z \neq \pm 1$; we have a contradiction.

# 8   $\mathfrak{L}_6$: $\mathfrak{r}_2(\mathbb{C}) \times \mathbb{C}$

$$\mathfrak{L}_6 := \{[e_1, e_2] = e_2$$

$$\mathrm{Aut}(\mathfrak{L}_6) = \left\{ \begin{bmatrix} 1 & 0 & 0 \\ x & a & 0 \\ y & 0 & b \end{bmatrix} \in M(3, \mathbb{C}) : \begin{matrix} a, b \in \mathbb{C}^* \\ x, y \in \mathbb{C} \end{matrix} \right\}.$$

$$\mathrm{Der}(\mathfrak{L}_6) = \left\{ \begin{bmatrix} 0 & 0 & 0 \\ t_3 & t_1 & 0 \\ t_4 & 0 & t_2 \end{bmatrix} \in M(3, \mathbb{C}) : t_1, t_2, t_3, t_4 \in \mathbb{C} \right\} \cong \mathfrak{r}_2(\mathbb{C}) \times \mathfrak{r}_2(\mathbb{C}).$$

## 8.1   $\mathfrak{L}_6^0$

Let $A$ be the Zero map and let us denote by $\mathfrak{L}_6^0$ the hom-Lie algebra $(\mathfrak{L}_6, A)$. We have:
$\mathrm{Aut}(\mathfrak{L}_6^0) = \mathrm{Aut}(\mathfrak{L}_6)$.
$\mathrm{Der}(\mathfrak{L}_6^0) = \mathrm{Der}(\mathfrak{L}_6)$.

### 8.1.1   $\psi_{\alpha,\beta}(\mathfrak{L}_6^0)$

In this case, we have $\psi_{\alpha,\beta}(\mathfrak{L}_6^0) = \mathfrak{t}_{\alpha,\beta}$ is the Lie algebra $\mathfrak{r}_2(\mathbb{C}) \times \mathbb{C}$.

### 8.1.2   $\phi_\beta(\mathfrak{L}_6^0)$

In this case, we have $\phi_\beta(\mathfrak{L}_6^0) = \widetilde{\mathfrak{t}}_\beta$ is the 3-dimensional abelian Lie algebra.

### 8.1.3   $\rho(\mathfrak{L}_6^0)$

In this case, we have $\rho(\mathfrak{L}_6^0) = \widehat{\mathfrak{t}}$ is the 3-dimensional abelian Lie algebra.

## 8.2   $\mathfrak{L}_6^1$

Let $A$ be the endomorphism of $\mathbb{C}^3$ whose matrix representation with respect to the ordered basis $\{e_1, e_2, e_3\}$ is

$$\begin{bmatrix} 0 & 0 & 0 \\ 1 & 0 & 0 \\ 0 & 0 & 0 \end{bmatrix}$$

and let us denote by $\mathfrak{L}_6^1$ the hom-Lie algebra $(\mathfrak{L}_6, A)$. We have:

$$\mathfrak{gl}(n, \mathbb{C})_A = \left\{ \begin{bmatrix} t_1 & 0 & 0 \\ t_3 & t_1 & t_5 \\ t_4 & 0 & t_2 \end{bmatrix} \in M(3, \mathbb{C}) : t_1, \ldots, t_5 \in \mathbb{C} \right\}.$$

$$\mathrm{Aut}(\mathfrak{L}_6^1) = \left\{ \begin{bmatrix} 1 & 0 & 0 \\ x & 1 & 0 \\ y & 0 & a \end{bmatrix} \in M(3, \mathbb{C}) : \begin{matrix} a \in \mathbb{C}^* \\ x, y \in \mathbb{C} \end{matrix} \right\}.$$

$$\mathrm{Der}(\mathfrak{L}_6^1) = \left\{ \begin{bmatrix} 0 & 0 & 0 \\ t_2 & 0 & 0 \\ t_3 & 0 & t_1 \end{bmatrix} \in M(3, \mathbb{C}) : t_1, t_2, t_3 \in \mathbb{C} \right\} \cong \mathfrak{r}_2(\mathbb{C}) \times \mathbb{C}.$$

### 8.2.1   $\psi_{\alpha,\beta}(\mathfrak{L}_6^1)$

In this case, we have $\psi_{\alpha,\beta}(\mathfrak{L}_6^1) = \mathfrak{t}_{\alpha,\beta}$ is the Lie algebra $\mathfrak{r}_2(\mathbb{C}) \times \mathbb{C}$.

### 8.2.2 $\phi_\beta(\mathfrak{L}_6^1)$

In this case, we have $\phi_\beta(\mathfrak{L}_6^1) = \widetilde{\mathfrak{t}}_\beta$ is the 3-dimensional abelian Lie algebra.

### 8.2.3 $\rho(\mathfrak{L}_6^1)$

In this case, we have $\rho(\mathfrak{L}_6^1) = \widehat{\mathfrak{t}}$ is the 3-dimensional abelian Lie algebra.

### 8.2.4 Identities

- $A[\cdot,\cdot] = 0$
- $[A\cdot,\cdot] + [\cdot,A\cdot] = 0$
  - $[A\cdot,A\cdot] = 0$

In particular, the hom-Lie algebra $\mathfrak{L}_6^1$ is a multiplicative hom-Lie algebra.

## 8.3 $\mathfrak{L}_6^2$

Let $A$ be the endomorphism of $\mathbb{C}^3$ whose matrix representation with respect to the ordered basis $\{e_1, e_2, e_3\}$ is

$$\begin{bmatrix} 0 & 0 & 0 \\ 0 & 0 & 0 \\ 1 & 0 & 0 \end{bmatrix}$$

and let us denote by $\mathfrak{L}_6^2$ the hom-Lie algebra $(\mathfrak{L}_6, A)$. We have:

$$\mathfrak{gl}(n,\mathbb{C})_A = \left\{ \begin{bmatrix} t_1 & 0 & 0 \\ t_3 & t_2 & 0 \\ t_5 & t_4 & t_1 \end{bmatrix} \in M(3,\mathbb{C}) : t_1, \ldots, t_5 \in \mathbb{C} \right\}.$$

$$\mathrm{Aut}(\mathfrak{L}_6^2) = \left\{ \begin{bmatrix} 1 & 0 & 0 \\ x & a & 0 \\ y & 0 & 1 \end{bmatrix} \in M(3,\mathbb{C}) : \begin{array}{l} a \in \mathbb{C}^* \\ x, y \in \mathbb{C} \end{array} \right\}.$$

$$\mathrm{Der}(\mathfrak{L}_6^2) = \left\{ \begin{bmatrix} 0 & 0 & 0 \\ t_2 & t_1 & 0 \\ t_3 & 0 & 0 \end{bmatrix} \in M(3,\mathbb{C}) : t_1, t_2, t_3 \in \mathbb{C} \right\} \cong \mathfrak{r}_2(\mathbb{C}) \times \mathbb{C}.$$

### 8.3.1 $T_{\theta_1,\ldots,\theta_6;\beta_1,\ldots\beta_9}(\mathfrak{L}_6^2)$

Given $D_1(w's)$, $D_2(x's)$ and $D_3(y's)$ in $\mathfrak{gl}(n,\mathbb{C})_A$, we have $\lambda := D_1[\cdot,\cdot] + [D_2\cdot,\cdot] + [\cdot,D_3\cdot]$

$$\lambda := \begin{cases} e1 \cdot e1 = (-x_{2,1} + y_{2,1}) \, e2, \\ e1 \cdot e2 = (x_{3,3} + w_{2,2} + y_{2,2}) \, e2 + w_{3,2} e3, \\ e2 \cdot e1 = (-y_{3,3} - w_{2,2} - x_{2,2}) \, e2 - w_{3,2} e3, \end{cases}$$

By taking $\{\alpha_1 = 0, \alpha_2 = 1, \alpha_3 = 0, \alpha_4 = 0, \alpha_5 = 0, \alpha_6 = 1\}$ and $\{\beta_1 = 1, \beta_2 = -1, \beta_3 = 0, \beta_4 = 1, \beta_5 = -1, \beta_6 = 0, \beta_7 = 1, \beta_8 = -1, \beta_9 = 0,\}$ and then, by replacing $D_2$ and $D_3$ by $0$, we have $\mathrm{Ker}(T_{\theta_1,\ldots,\theta_6;\beta_1,\ldots,\beta_9}(\mathfrak{L}_6^2))$ is determined by the equations

$$\{ w_{2,2} = 0, \, w_{3,2} = 0$$

Therefore

$$\mathrm{Dim}\, T_{\theta_1,\ldots,\theta_6;\beta_1,\ldots\beta_9}(\mathfrak{L}_6^2) = 3.$$

### 8.3.2 $\psi_{\alpha,\beta}(\mathfrak{L}_6^2)$

In this case, we have $\psi_{\alpha,\beta}(\mathfrak{L}_6^2) = \mathfrak{t}_{\alpha,\beta}$ is the Lie algebra $\mathfrak{r}_2(\mathbb{C}) \times \mathbb{C}$.

### 8.3.3 $\phi_\beta(\mathfrak{L}_6^2)$

In this case, we have $\phi_\beta(\mathfrak{L}_6^2) = \widetilde{\mathfrak{t}}_\beta$ is the 3-dimensional abelian Lie algebra.

### 8.3.4 $\rho(\mathfrak{L}_6^2)$

In this case, we have $\rho(\mathfrak{L}_6^2) = \widehat{\mathfrak{t}}$ is the 3-dimensional abelian Lie algebra.

### 8.3.5 Identities

- $A[\cdot,\cdot] = 0$
- $[A\cdot,\cdot] = 0$

In particular, the hom-Lie algebra $\mathfrak{L}_6^2$ is a multiplicative hom-Lie algebra.

## 8.4 $\mathfrak{L}_6^3$

Let $A$ be the endomorphism of $\mathbb{C}^3$ whose matrix representation with respect to the ordered basis $\{e_1, e_2, e_3\}$ is

$$\begin{bmatrix} 0 & 0 & 0 \\ 1 & 0 & 0 \\ 1 & 0 & 0 \end{bmatrix}$$

and let us denote by $\mathfrak{L}_6^3$ the hom-Lie algebra $(\mathfrak{L}_6, A)$. We have:

$$\mathrm{Aut}(\mathfrak{L}_6^3) = \left\{ \begin{bmatrix} 1 & 0 & 0 \\ x & 1 & 0 \\ y & 0 & 1 \end{bmatrix} \in M(3,\mathbb{C}) : x, y \in \mathbb{C} \right\}.$$

$$\mathrm{Der}(\mathfrak{L}_6^3) = \left\{ \begin{bmatrix} 0 & 0 & 0 \\ t_1 & 0 & 0 \\ t_2 & 0 & 0 \end{bmatrix} \in M(3,\mathbb{C}) : t_1, t_2 \in \mathbb{C} \right\} \cong \mathbb{C}^2.$$

### 8.4.1 $\psi_{\alpha,\beta}(\mathfrak{L}_6^3)$

In this case, we have $\psi_{\alpha,\beta}(\mathfrak{L}_6^3) = \mathfrak{t}_{\alpha,\beta}$ is the Lie algebra $\mathfrak{r}_2(\mathbb{C}) \times \mathbb{C}$.

### 8.4.2 $\phi_\beta(\mathfrak{L}_6^3)$

In this case, we have $\phi_\beta(\mathfrak{L}_6^3) = \widetilde{\mathfrak{t}}_\beta$ is the 3-dimensional abelian Lie algebra.

### 8.4.3 $\rho(\mathfrak{L}_6^3)$

In this case, we have $\rho(\mathfrak{L}_6^3) = \widehat{\mathfrak{t}}$ is the 3-dimensional abelian Lie algebra.

### 8.4.4 Identities

- $A[\cdot,\cdot] = 0$
- $[A\cdot,\cdot] + [\cdot,A\cdot] = 0$
    - $[A\cdot,A\cdot] = 0$

In particular, the hom-Lie algebra $\mathfrak{L}_6^3$ is a multiplicative hom-Lie algebra.

## 8.5 $\mathfrak{L}_6^4$

Let $A$ be the endomorphism of $\mathbb{C}^3$ whose matrix representation with respect to the ordered basis $\{e_1, e_2, e_3\}$ is

$$\begin{bmatrix} 0 & 0 & 0 \\ 0 & 0 & 1 \\ 0 & 0 & 0 \end{bmatrix}$$

and let us denote by $\mathfrak{L}_6^4$ the hom-Lie algebra $(\mathfrak{L}_6, A)$. We have:

$$\mathfrak{gl}(n,\mathbb{C})_A = \left\{ \begin{bmatrix} t_1 & 0 & t_3 \\ t_5 & t_2 & t_4 \\ 0 & 0 & t_2 \end{bmatrix} \in M(3,\mathbb{C}) : t_1, \ldots, t_5 \in \mathbb{C} \right\}.$$

$$\mathrm{Aut}(\mathfrak{L}_6^4) = \left\{ \begin{bmatrix} 1 & 0 & 0 \\ x & a & 0 \\ 0 & 0 & a \end{bmatrix} \in M(3,\mathbb{C}) : \begin{array}{c} a \in \mathbb{C}^* \\ x \in \mathbb{C} \end{array} \right\}.$$

$$\mathrm{Der}(\mathfrak{L}_6^4) = \left\{ \begin{bmatrix} 0 & 0 & 0 \\ t_2 & t_1 & 0 \\ 0 & 0 & t_1 \end{bmatrix} \in M(3,\mathbb{C}) : t_1, t_2 \in \mathbb{C} \right\} \cong \mathfrak{r}_2(\mathbb{C}).$$

### 8.5.1 $T_{\theta_1, \ldots, \theta_6; \beta_1, \ldots \beta_9}(\mathfrak{L}_6^4)$

Given $D_1(w's)$, $D_2(x's)$ and $D_3(y's)$ in $\mathfrak{gl}(n,\mathbb{C})_A$, we have $\lambda := D_1[\cdot,\cdot] + [D_2\cdot,\cdot] + [\cdot, D_3\cdot]$

$$\lambda := \begin{cases} e1 \cdot e1 = (-x_{2,1} + y_{2,1})\, e2, \\ e1 \cdot e2 = (x_{1,1} + w_{3,3} + y_{3,3})\, e2, \\ e1 \cdot e3 = y_{2,3}\, e2, \\ e2 \cdot e1 = (-y_{1,1} - w_{3,3} - x_{3,3})\, e2, \\ e2 \cdot e3 = -y_{1,3}\, e2, \\ e3 \cdot e1 = -x_{2,3}\, e2, \\ e3 \cdot e2 = x_{1,3}\, e2 \end{cases}$$

By taking $\{\alpha_1 = 0, \alpha_2 = 1, \alpha_3 = 0, \alpha_4 = 0, \alpha_5 = 0, \alpha_6 = 1\}$ and $\{\beta_1 = 1, \beta_2 = -1, \beta_3 = 0, \beta_4 = 1, \beta_5 = -1, \beta_6 = 0, \beta_7 = 1, \beta_8 = -1, \beta_9 = 0,\}$ and then, by replacing $D_2$ and $D_3$ by 0, we have $\mathrm{Ker}(T_{\theta_1,\ldots,\theta_6;\beta_1,\ldots,\beta_9}(\mathfrak{L}_6^4))$ is determined by the equations

$$\{\; w_{3,3} = 0$$

Therefore
$$\mathrm{Dim}\, T_{\theta_1,\ldots,\theta_6;\beta_1,\ldots\beta_9}(\mathfrak{L}_6^4) = 4.$$

### 8.5.2 $\psi_{\alpha,\beta}(\mathfrak{L}_6^4)$

In this case, we have $\psi_{\alpha,\beta}(\mathfrak{L}_6^4) = \mathfrak{t}_{\alpha,\beta}$:

$$\mathfrak{t}_{\alpha,\beta} = \{[e1, e2] = e2, [e1, e3] = \beta\, e2.$$

We note that $\mathfrak{t}_{\alpha,\beta}$ is an almost abelian Lie algebra where $\mathrm{span}_{\mathbb{C}}\{e_2, e_3\}$ is a codimension 1 abelian ideal of $\mathfrak{t}_{\alpha,\beta}$; we can identify $\mathfrak{t}_{\alpha,\beta}$ the matrix

$$B = \begin{bmatrix} 1 & \beta \\ 0 & 0 \end{bmatrix}.$$

Therefore, the Lie algebra $\mathfrak{t}_{\alpha,\beta}$ is isomorphic to $\mathfrak{r}_2(\mathbb{C}) \times \mathbb{C}$.

### 8.5.3 $\phi_\beta(\mathfrak{L}_6^4)$

In this case, we have $\phi_\beta(\mathfrak{L}_6^4) = \widetilde{\mathfrak{t}}_\beta$:

$$\widetilde{\mathfrak{t}}_\beta = \{[e1, e3] = \beta\, e2.$$

We note that $\widetilde{\mathfrak{t}}_\beta$ is an almost abelian Lie algebra where $\mathrm{span}_\mathbb{C}\{e_2, e_3\}$ is a codimension 1 abelian ideal of $\widetilde{\mathfrak{t}}_\beta$; we can identify $\widetilde{\mathfrak{t}}_\beta$ the matrix

$$\widetilde{B} = \begin{bmatrix} 0 & \beta \\ 0 & 0 \end{bmatrix}.$$

The isomorphism class of $\widetilde{\mathfrak{t}}_\beta$ is determined by the eigenvalues of $\widetilde{B}$:

1. $\widetilde{B}$ has repeated eigenvalues if and only if $0 = 0$
   (a) 0 is a repeated eigenvalue of $\widetilde{B}$ if and only if $0 = 0$
       i. $\widetilde{\mathfrak{t}}_\beta$ is the 3-dimensional abelian Lie algebra if and only if $\widetilde{B} = 0$ if and only if $\beta = 0$
       ii. $\widetilde{\mathfrak{t}}_\beta$ is isomorphic to $\mathfrak{n}_3(\mathbb{C})$ if and only if $\widetilde{B} \neq 0$ if and only if $\beta \neq 0$
   (b) $\widetilde{B}$ has a non-zero repeated eigenvalue if and only if $1 = 0$
2. $\widetilde{B}$ has distinct eigenvalues if and only if $1 = 0$

### 8.5.4 $\rho(\mathfrak{L}_6^4)$

In this case, we have $\rho(\mathfrak{L}_6^4) = \widehat{\mathfrak{t}}$:

$$\widehat{\mathfrak{t}} = \{[e1, e3] = e2.$$

which is isomorphic to $\mathfrak{n}_3$

### 8.5.5 Identities

- $A[\cdot, \cdot] = 0$
- $[A\cdot, A\cdot] = 0$

In particular, the hom-Lie algebra $\mathfrak{L}_6^4$ is a multiplicative hom-Lie algebra.

## 8.6 $\mathfrak{L}_6^5$

Let $A$ be the endomorphism of $\mathbb{C}^3$ whose matrix representation with respect to the ordered basis $\{e_1, e_2, e_3\}$ is

$$\begin{bmatrix} 0 & 0 & 0 \\ 0 & 0 & 0 \\ 0 & 1 & 0 \end{bmatrix}$$

and let us denote by $\mathfrak{L}_6^5$ the hom-Lie algebra $(\mathfrak{L}_6, A)$. We have:

$$\mathfrak{gl}(n, \mathbb{C})_A = \left\{ \begin{bmatrix} t_1 & t_5 & 0 \\ 0 & t_2 & 0 \\ t_3 & t_4 & t_2 \end{bmatrix} \in M(3, \mathbb{C}) : t_1, \ldots, t_5 \in \mathbb{C} \right\}.$$

$$\mathrm{Aut}(\mathfrak{L}_6^5) = \left\{ \begin{bmatrix} 1 & 0 & 0 \\ 0 & a & 0 \\ x & 0 & a \end{bmatrix} \in M(3, \mathbb{C}) : \begin{matrix} a \in \mathbb{C}^* \\ x \in \mathbb{C} \end{matrix} \right\}.$$

$$\mathrm{Der}(\mathfrak{L}_6^5) = \left\{ \begin{bmatrix} 0 & 0 & 0 \\ 0 & t_1 & 0 \\ t_2 & 0 & t_1 \end{bmatrix} \in M(3, \mathbb{C}) : t_1, t_2 \in \mathbb{C} \right\} \cong \mathfrak{r}_2(\mathbb{C}).$$

### 8.6.1 $\psi_{\alpha,\beta}(\mathfrak{L}_6^5)$

In this case, we have $\psi_{\alpha,\beta}(\mathfrak{L}_6^5) = \mathfrak{t}_{\alpha,\beta}$:

$$\mathfrak{t}_{\alpha,\beta} = \{[e1,e2] = e2 + \alpha\, e3.$$

We note that $\mathfrak{t}_{\alpha,\beta}$ is an almost abelian Lie algebra where $\operatorname{span}_\mathbb{C}\{e_2, e_3\}$ is a codimension 1 abelian ideal of $\mathfrak{t}_{\alpha,\beta}$; we can identify $\mathfrak{t}_{\alpha,\beta}$ the matrix

$$B = \begin{bmatrix} 1 & 0 \\ \alpha & 0 \end{bmatrix}.$$

Therefore, the Lie algebra $\mathfrak{t}_{\alpha,\beta}$ is isomorphic to $\mathfrak{r}_2(\mathbb{C}) \times \mathbb{C}$.

### 8.6.2 $\phi_\beta(\mathfrak{L}_6^5)$

In this case, we have $\phi_\beta(\mathfrak{L}_6^5) = \widetilde{\mathfrak{t}}_\beta$:

$$\widetilde{\mathfrak{t}}_\beta = \{[e1,e2] = e3.$$

is the Heisenberg Lie algebra, $\mathfrak{n}_3(\mathbb{C})$.

### 8.6.3 $\rho(\mathfrak{L}_6^5)$

In this case, we have $\rho(\mathfrak{L}_6^5) = \widehat{\mathfrak{t}}$ is the 3-dimensional abelian Lie algebra.

### 8.6.4 Identities

- $[A\cdot,\cdot] = 0$

The hom-Lie algebra $\mathfrak{L}_6^5$ is a non-multiplicative hom-Lie algebra.

## 8.7 $\mathfrak{L}_6^6(\lambda)$

Let $A$ be the endomorphism of $\mathbb{C}^3$ whose matrix representation with respect to the ordered basis $\{e_1, e_2, e_3\}$ is

$$\begin{bmatrix} 0 & 0 & 0 \\ 0 & \lambda & \lambda \\ 0 & -\lambda & -\lambda \end{bmatrix}$$

and let us denote by $\mathfrak{L}_6^6(\lambda)$ the hom-Lie algebra $(\mathfrak{L}_6, A)$. We have:

$$\operatorname{Aut}(\mathfrak{L}_6^6(\lambda)) = \left\{ \begin{bmatrix} 1 & 0 & 0 \\ -x & a & 0 \\ x & 0 & a \end{bmatrix} \in M(3,\mathbb{C}) : \begin{array}{l} a \in \mathbb{C}^* \\ x \in \mathbb{C} \end{array} \right\}.$$

$$\operatorname{Der}(\mathfrak{L}_6^6(\lambda)) = \left\{ \begin{bmatrix} 0 & 0 & 0 \\ -t_2 & t_1 & 0 \\ t_2 & 0 & t_1 \end{bmatrix} \in M(3,\mathbb{C}) : t_1, t_2 \in \mathbb{C} \right\} \cong \mathfrak{r}_2(\mathbb{C}).$$

### 8.7.1 $\psi_{\alpha,\beta}(\mathfrak{L}_6^6(\lambda))$

In this case, we have $\psi_{\alpha,\beta}(\mathfrak{L}_6^6(\lambda)) = \mathfrak{t}_{\alpha,\beta}$:

$$\mathfrak{t}_{\alpha,\beta} = \{[e1,e2], (1+\alpha\,\lambda + \beta\,\lambda)\,e2 - \alpha\,\lambda\,e3, [e1,e3], \beta\,\lambda\,e2.$$

We note that $\mathfrak{t}_{\alpha,\beta}$ is an almost abelian Lie algebra where $\mathrm{span}_{\mathbb{C}}\{e_2,e_3\}$ is a codimension 1 abelian ideal of $\mathfrak{t}_{\alpha,\beta}$; we can identify $\mathfrak{t}_{\alpha,\beta}$ the matrix

$$B = \begin{bmatrix} 1 + \alpha\lambda + \beta\lambda & \beta\lambda \\ -\alpha\lambda & 0 \end{bmatrix}.$$

The isomorphism class of $\mathfrak{t}_{\alpha,\beta}$ is determined by the eigenvalues of $B$:

1. $B$ has repeated eigenvalues if and only if $1 + (\alpha-\beta)^2 \lambda^2 + 2(\alpha+\beta)\lambda = 0$

   (a) 0 is a repeated eigenvalue of $B$ if and only if $\{\alpha = -\frac{1}{\lambda}, \beta = 0\}$ or $\{\alpha = 0, \beta = -\frac{1}{\lambda}\}$.
   
   i. $\mathfrak{t}_{\alpha,\beta}$ is the 3-dimensional abelian Lie algebra if and only if $B = 0$ if and only if $1 = 0$
   
   ii. $\mathfrak{t}_{\alpha,\beta}$ is isomorphic to $\mathfrak{n}_3(\mathbb{C})$ if and only if $B \neq 0$ if and only if $\{\alpha = -\frac{1}{\lambda}, \beta = 0\}$ or $\{\alpha = 0, \beta = -\frac{1}{\lambda}\}$.
   
   (b) $B$ has a non-zero repeated eigenvalue if and only if
   $\left\{\alpha = -\frac{s-2\pm\sqrt{s^2-4s}}{2s\lambda}, \beta = \frac{-s+2\pm\sqrt{s^2-4s}}{2s\lambda}\right\}$ for all $s \in \mathbb{C}^\star$.
   
   i. $\mathfrak{t}_{\alpha,\beta}$ is $\mathfrak{r}_{3,1}$ if and only if $B = t\,\mathrm{Id}$ for some $t \neq 0$ if and only if $1 = 0$
   
   ii. $\mathfrak{t}_{\alpha,\beta}$ is isomorphic to $\mathfrak{r}_3$ if and only if $B \neq t\,\mathrm{Id}$ for all $t \in \mathbb{C}$ if and only if
   $\left\{\alpha = -\frac{s-2\pm\sqrt{s^2-4s}}{2s\lambda}, \beta = \frac{-s+2\pm\sqrt{s^2-4s}}{2s\lambda}\right\}$ for any $s \in \mathbb{C}\setminus\{0\}$.

2. $B$ has distinct eigenvalues eigenvalues if and only if $1 + (\alpha-\beta)^2\lambda^2 + 2(\alpha+\beta)\lambda \neq 0$

   (a) $\mathfrak{t}_{\alpha,\beta}$ is isomorphic to $\mathfrak{r}_2(\mathbb{C})\times\mathbb{C}$ if and only if 0 is an eigenvalue of $B$ with algebraic multiplicity 1 if and only if $\{\alpha = -\frac{s-1}{s\lambda}, \beta = 0\}$ or $\{\alpha = 0, \beta = -\frac{s-1}{s\lambda}\}$, for any $s \in \mathbb{C}^\star$.
   
   (b) $\mathfrak{t}_{\alpha,\beta}$ is isomorphic to $\mathfrak{r}_{3,-1}$ if and only if $\mathrm{Trace}(B) = 0$ if and only if
   $\left\{\alpha = \frac{-s\pm\sqrt{s(s-4)}}{2\lambda s}, \beta = \frac{-s\mp\sqrt{s(s-4)}}{2\lambda s}\right\}$ for any $s \in \mathbb{C}^\star$.
   
   (c) $\mathfrak{t}_{\alpha,\beta}$ is isomorphic to $\mathfrak{r}_{3,z}$ for some $z$ such that $z^2 \neq 1$ and $z \neq 0$ if and only if
   $\left\{\alpha = \frac{-s_1 s_2 + s_1 \pm \sqrt{s_1 s_2(s_1 s_2 - 2s_1 + s_2)}}{2s_1 s_2 \lambda}, \beta = \frac{-s_1 s_2 + s_1 \mp \sqrt{s_1 s_2(s_1 s_2 - 2s_1 + s_2)}}{2s_1 s_2 \lambda}\right\}$ for any $s_1, s_2 \in \mathbb{C}^\star$ such that $s_1 - s_2^2 \neq 0$.

### 8.7.2 $\phi_\beta(\mathfrak{L}_6^6(\lambda))$

In this case, we have $\phi_\beta(\mathfrak{L}_6^6(\lambda)) = \widetilde{\mathfrak{t}}_\beta$:

$$\widetilde{\mathfrak{t}}_\beta = \{[e1, e2] = (\lambda + \beta\lambda)\,e2 - \lambda\,e3, [e1, e3] = \beta\lambda\,e2.$$

We note that $\widetilde{\mathfrak{t}}_\beta$ is an almost abelian Lie algebra where $\mathrm{span}_{\mathbb{C}}\{e_2,e_3\}$ is a codimension 1 abelian ideal of $\widetilde{\mathfrak{t}}_\beta$; we can identify $\widetilde{\mathfrak{t}}_\beta$ the matrix

$$\widetilde{B} = \begin{bmatrix} \lambda + \beta\lambda & \beta\lambda \\ -\lambda & 0 \end{bmatrix}.$$

The isomorphism class of $\widetilde{\mathfrak{t}}_\beta$ is determined by the eigenvalues of $\widetilde{B}$:

1. $\widetilde{B}$ has repeated eigenvalues if and only if $\beta = 1$

   (a) 0 is a repeated eigenvalue of $\widetilde{B}$ if and only if $1 = 0$
   
   (b) $\widetilde{B}$ has a non-zero repeated eigenvalue if and only if $\beta = 1$
   
   i. $\widetilde{\mathfrak{t}}_\beta$ is $\mathfrak{r}_{3,1}$ if and only if $\widetilde{B} = t\,\mathrm{Id}$ for some $\lambda \neq 0$ if and only if $\lambda = 0$ (False)
   
   ii. $\widetilde{\mathfrak{t}}_\beta$ is isomorphic to $\mathfrak{r}_3$ if and only if $\widetilde{B} \neq t\,\mathrm{Id}$ for all $\lambda \in \mathbb{C}$ if and only if $\beta = 1$

73

2. $\widetilde{B}$ has distinct eigenvalues if and only if $\beta \neq 1$

   (a) $\widetilde{\mathfrak{t}}_\beta$ is isomorphic to $\mathfrak{r}_2(\mathbb{C}) \times \mathbb{C}$ if and only if 0 is an eigenvalue of $B$ with algebraic multiplicity 1 if and only if $\beta = 0$
   
   (b) $\widetilde{\mathfrak{t}}_\beta$ is isomorphic to $\mathfrak{r}_{3,-1}$ if and only if $\text{Trace}(\widetilde{B}) = 0$ if and only if $\beta = -1$.
   
   (c) $\widetilde{\mathfrak{t}}_\beta$ is isomorphic to $\mathfrak{r}_{3,z}$ for some $z$ such that $z^2 \neq 1$ and $z \neq 0$ if and only if $\beta \neq 1$ and $\beta \neq -1$ and $\beta \neq 0$.

### 8.7.3 $\rho(\mathfrak{L}_6^6(\lambda))$

In this case, we have $\rho(\mathfrak{L}_6^6(\lambda)) = \widehat{\mathfrak{t}}$:

$$\widehat{\mathfrak{t}} = \{[e1, e2] = \lambda\, e2, [e1, e3] = \lambda\, e2.$$

which is isomorphic to $\mathfrak{r}_2 \times \mathbb{C}$.

### 8.7.4 Identities

- $[A\cdot, A\cdot] = 0$

The hom-Lie algebra $\mathfrak{L}_6^7(\lambda)$ is a non-multiplicative hom-Lie algebra.

## 8.8 $\mathfrak{L}_6^7$

Let $A$ be the endomorphism of $\mathbb{C}^3$ whose matrix representation with respect to the ordered basis $\{e_1, e_2, e_3\}$ is

$$\begin{bmatrix} 0 & 0 & 0 \\ 1 & 0 & 0 \\ 0 & 1 & 0 \end{bmatrix}$$

and let us denote by $\mathfrak{L}_6^7$ the hom-Lie algebra $(\mathfrak{L}_6, A)$. We have:

$$\text{Aut}(\mathfrak{L}_6^7) = \left\{ \begin{bmatrix} 1 & 0 & 0 \\ 0 & 1 & 0 \\ x & 0 & 1 \end{bmatrix} \in M(3, \mathbb{C}) : x \in \mathbb{C} \right\}.$$

$$\text{Der}(\mathfrak{L}_6^7) = \left\{ \begin{bmatrix} 0 & 0 & 0 \\ 0 & 0 & 0 \\ t_1 & 0 & 0 \end{bmatrix} \in M(3, \mathbb{C}) : t_1 \in \mathbb{C} \right\} \cong \mathbb{C}.$$

### 8.8.1 $\psi_{\alpha,\beta}(\mathfrak{L}_6^7)$

In this case, we have $\psi_{\alpha,\beta}(\mathfrak{L}_6^7) = \mathfrak{t}_{\alpha,\beta}$:

$$\mathfrak{t}_{\alpha,\beta} = \{[e1, e2] = e2 + \alpha\, e3.$$

We note that $\mathfrak{t}_{\alpha,\beta}$ is an almost abelian Lie algebra where $\text{span}_\mathbb{C}\{e_2, e_3\}$ is a codimension 1 abelian ideal of $\mathfrak{t}_{\alpha,\beta}$; we can identify $\mathfrak{t}_{\alpha,\beta}$ the matrix

$$B = \begin{bmatrix} 1 & 0 \\ \alpha & 0 \end{bmatrix}.$$

Therefore, the Lie algebra $\mathfrak{t}_{\alpha,\beta}$ is isomorphic to $\mathfrak{r}_2(\mathbb{C}) \times \mathbb{C}$.

### 8.8.2 $\phi_\beta(\mathfrak{L}_6^7)$

In this case, we have $\phi_\beta(\mathfrak{L}_6^7) = \widetilde{\mathfrak{t}}_\beta$:

$$\widetilde{\mathfrak{t}}_\beta = \{[e1, e2] = e3.$$

is the Heisenberg Lie algebra, $\mathfrak{n}_3(\mathbb{C})$.

### 8.8.3 $\rho(\mathfrak{L}_6^7)$

In this case, we have $\rho(\mathfrak{L}_6^7) = \widehat{\mathfrak{t}}$ is the 3-dimensional abelian Lie algebra.

### 8.8.4 Identities

- $[A\cdot, \cdot] + [\cdot, A\cdot] = 0$
- $[A\cdot, A\cdot] = 0$
- $A^2[\cdot, \cdot] = 0$

The hom-Lie algebra $\mathfrak{L}_6^7$ is a non-multiplicative hom-Lie algebra.

## 8.9 $\mathfrak{L}_6^8$

Let $A$ be the endomorphism of $\mathbb{C}^3$ whose matrix representation with respect to the ordered basis $\{e_1, e_2, e_3\}$ is

$$\begin{bmatrix} 0 & 0 & 0 \\ 0 & 0 & 1 \\ 1 & 0 & 0 \end{bmatrix}$$

and let us denote by $\mathfrak{L}_6^8$ the hom-Lie algebra $(\mathfrak{L}_6, A)$. We have:

$$\mathrm{Aut}(\mathfrak{L}_6^8) = \left\{ \begin{bmatrix} 1 & 0 & 0 \\ x & 1 & 0 \\ 0 & 0 & 1 \end{bmatrix} \in M(3, \mathbb{C}) : x \in \mathbb{C} \right\}.$$

$$\mathrm{Der}(\mathfrak{L}_6^8) = \left\{ \begin{bmatrix} 0 & 0 & 0 \\ t_1 & 0 & 0 \\ 0 & 0 & 0 \end{bmatrix} \in M(3, \mathbb{C}) : t_1 \in \mathbb{C} \right\} \cong \mathbb{C}.$$

### 8.9.1 $\psi_{\alpha,\beta}(\mathfrak{L}_6^8)$

In this case, we have $\psi_{\alpha,\beta}(\mathfrak{L}_6^8) = \mathfrak{t}_{\alpha,\beta}$:

$$\mathfrak{t}_{\alpha,\beta} = \{[e1, e2] = e2, [e1, e3] = \beta\, e2.$$

We note that $\mathfrak{t}_{\alpha,\beta}$ is an almost abelian Lie algebra where $\mathrm{span}_\mathbb{C}\{e_2, e_3\}$ is a codimension 1 abelian ideal of $\mathfrak{t}_{\alpha,\beta}$; we can identify $\mathfrak{t}_{\alpha,\beta}$ the matrix

$$B = \begin{bmatrix} 1 & \beta \\ 0 & 0 \end{bmatrix}.$$

Therefore, the Lie algebra $\mathfrak{t}_{\alpha,\beta}$ is isomorphic to $\mathfrak{r}_2(\mathbb{C}) \times \mathbb{C}$.

### 8.9.2 $\phi_\beta(\mathfrak{L}_6^8)$

In this case, we have $\phi_\beta(\mathfrak{L}_6^8) = \widetilde{\mathfrak{t}}_\beta$:

$$\widetilde{\mathfrak{t}}_\beta = \{[e1, e3] = \beta\, e2.$$

We note that $\widetilde{\mathfrak{t}}_\beta$ is an almost abelian Lie algebra where $\text{span}_\mathbb{C}\{e_2, e_3\}$ is a codimension 1 abelian ideal of $\widetilde{\mathfrak{t}}_\beta$; we can identify $\widetilde{\mathfrak{t}}_\beta$ the matrix

$$\widetilde{B} = \begin{bmatrix} 0 & \beta \\ 0 & 0 \end{bmatrix}.$$

The isomorphism class of $\widetilde{\mathfrak{t}}_\beta$ is determined by the eigenvalues of $\widetilde{B}$:

1. $\widetilde{B}$ has repeated eigenvalues if and only if $0 = 0$
    (a) $0$ is a repeated eigenvalue of $\widetilde{B}$ if and only if $0 = 0$
        i. $\widetilde{\mathfrak{t}}_\beta$ is the 3-dimensional abelian Lie algebra if and only if $\widetilde{B} = 0$ if and only if $\beta = 0$
        ii. $\widetilde{\mathfrak{t}}_\beta$ is isomorphic to $\mathfrak{n}_3(\mathbb{C})$ if and only if $\widetilde{B} \neq 0$ if and only if $\beta \neq 0$
    (b) $\widetilde{B}$ has a non-zero repeated eigenvalue if and only if $1 = 0$
2. $\widetilde{B}$ has distinct eigenvalues if and only if $1 = 0$

### 8.9.3 $\rho(\mathfrak{L}_6^8)$

In this case, we have $\rho(\mathfrak{L}_6^8) = \widehat{\mathfrak{t}}$:

$$\widehat{\mathfrak{t}} = \{[e1, e3] = e2.$$

which is isomorphic to $\mathfrak{n}_3$

### 8.9.4 Identities

- $A[\cdot, \cdot] = 0$
- $[A\cdot, A\cdot] = 0$
- $[A^2\cdot, \cdot] + [\cdot, A^2\cdot] = 0$

In particular, the hom-Lie algebra $\mathfrak{L}_6^8$ is a multiplicative hom-Lie algebra.

## 8.10 $\mathfrak{L}_6^9(\lambda)$

Let $A$ be the endomorphism of $\mathbb{C}^3$ whose matrix representation with respect to the ordered basis $\{e_1, e_2, e_3\}$ is

$$\begin{bmatrix} 0 & 0 & 0 \\ 1 & \lambda & \lambda \\ 0 & -\lambda & -\lambda \end{bmatrix}$$

and let us denote by $\mathfrak{L}_6^9(\lambda)$ the hom-Lie algebra $(\mathfrak{L}_6, A)$. We have:

$$\text{Aut}(\mathfrak{L}_6^9(\lambda)) = \left\{ \begin{bmatrix} 1 & 0 & 0 \\ x & 1 & 0 \\ -x & 0 & 1 \end{bmatrix} \in M(3, \mathbb{C}) : x \in \mathbb{C} \right\}.$$

$$\text{Der}(\mathfrak{L}_6^9(\lambda)) = \left\{ \begin{bmatrix} 0 & 0 & 0 \\ t_1 & 0 & 0 \\ -t_1 & 0 & 0 \end{bmatrix} \in M(3, \mathbb{C}) : t_1 \in \mathbb{C} \right\} \cong \mathbb{C}.$$

### 8.10.1 $\psi_{\alpha,\beta}(\mathfrak{L}_6^9(\lambda))$

In this case, we have $\psi_{\alpha,\beta}(\mathfrak{L}_6^9(\lambda)) = \mathfrak{t}_{\alpha,\beta}$:

$$\mathfrak{t}_{\alpha,\beta} = \{[e1, e2], (1 + \alpha\,\lambda + \beta\,\lambda)\,e2 - \alpha\,\lambda\,e3, [e1, e3], \beta\,\lambda\,e2.$$

We note that $\mathfrak{t}_{\alpha,\beta}$ is an almost abelian Lie algebra where $\mathrm{span}_{\mathbb{C}}\{e_2, e_3\}$ is a codimension 1 abelian ideal of $\mathfrak{t}_{\alpha,\beta}$; we can identify $\mathfrak{t}_{\alpha,\beta}$ the matrix

$$B = \begin{bmatrix} 1 + \alpha\,\lambda + \beta\,\lambda & \beta\,\lambda \\ -\alpha\,\lambda & 0 \end{bmatrix}.$$

The isomorphism class of $\mathfrak{t}_{\alpha,\beta}$ is determined by the eigenvalues of $B$:

1. $B$ has repeated eigenvalues if and only if $1 + (\alpha - \beta)^2\,\lambda^2 + 2\,(\alpha + \beta)\,\lambda = 0$
   
   (a) 0 is a repeated eigenvalue of $B$ if and only if $\{\alpha = -\frac{1}{\lambda}, \beta = 0\}$ or $\{\alpha = 0, \beta = -\frac{1}{\lambda}\}$.
   
   i. $\mathfrak{t}_{\alpha,\beta}$ is the 3-dimensional abelian Lie algebra if and only if $B = 0$ if and only if $1 = 0$
   
   ii. $\mathfrak{t}_{\alpha,\beta}$ is isomorphic to $\mathfrak{n}_3(\mathbb{C})$ if and only if $B \neq 0$ if and only if $\{\alpha = -\frac{1}{\lambda}, \beta = 0\}$ or $\{\alpha = 0, \beta = -\frac{1}{\lambda}\}$.
   
   (b) $B$ has a non-zero repeated eigenvalue if and only if
   $\left\{\alpha = -\frac{s - 2 \pm \sqrt{s^2 - 4s}}{2s\lambda}, \beta = \frac{-s + 2 \pm \sqrt{s^2 - 4s}}{2s\lambda}\right\}$ for all $s \in \mathbb{C}^\star$.
   
   i. $\mathfrak{t}_{\alpha,\beta}$ is $\mathfrak{r}_{3,1}$ if and only if $B = t\,\mathrm{Id}$ for some $t \neq 0$ if and only if $1 = 0$
   
   ii. $\mathfrak{t}_{\alpha,\beta}$ is isomorphic to $\mathfrak{r}_3$ if and only if $B \neq t\,\mathrm{Id}$ for all $t \in \mathbb{C}$ if and only if
   $\left\{\alpha = -\frac{s - 2 \pm \sqrt{s^2 - 4s}}{2s\lambda}, \beta = \frac{-s + 2 \pm \sqrt{s^2 - 4s}}{2s\lambda}\right\}$ for any $s \in \mathbb{C} \setminus \{0\}$.

2. $B$ has distinct eigenvalues eigenvalues if and only if $1 + (\alpha - \beta)^2\,\lambda^2 + 2\,(\alpha + \beta)\,\lambda \neq 0$
   
   (a) $\mathfrak{t}_{\alpha,\beta}$ is isomorphic to $\mathfrak{r}_2(\mathbb{C}) \times \mathbb{C}$ if and only if 0 is an eigenvalue of $B$ with algebraic multiplicity 1 if and only if $\{\alpha = -\frac{s-1}{s\lambda}, \beta = 0\}$ or $\{\alpha = 0, \beta = -\frac{s-1}{s\lambda}\}$, for any $s \in \mathbb{C}^\star$.
   
   (b) $\mathfrak{t}_{\alpha,\beta}$ is isomorphic to $\mathfrak{r}_{3,-1}$ if and only if $\mathrm{Trace}(B) = 0$ if and only if
   $\left\{\alpha = \frac{-s \pm \sqrt{s(s-4)}}{2\lambda\,s}, \beta = \frac{-s \mp \sqrt{s(s-4)}}{2\lambda\,s}\right\}$ for any $s \in \mathbb{C}^\star$.
   
   (c) $\mathfrak{t}_{\alpha,\beta}$ is isomorphic to $\mathfrak{r}_{3,z}$ for some $z$ such that $z^2 \neq 1$ and $z \neq 0$ if and only if
   $\left\{\alpha = \frac{-s_1 s_2 + s_1 \pm \sqrt{s_1 s_2(s_1 s_2 - 2 s_1 + s_2)}}{2 s_1 s_2 \lambda}, \beta = \frac{-s_1 s_2 + s_1 \mp \sqrt{s_1 s_2(s_1 s_2 - 2 s_1 + s_2)}}{2 s_1 s_2 \lambda}\right\}$ for any $s_1, s_2 \in \mathbb{C}^\star$ such that $s_1 - s_2^2 \neq 0$.

### 8.10.2 $\phi_\beta(\mathfrak{L}_6^9(\lambda))$

In this case, we have $\phi_\beta(\mathfrak{L}_6^9(\lambda)) = \widetilde{\mathfrak{t}}_\beta$:

$$\widetilde{\mathfrak{t}}_\beta = \{[e1, e2] = (\lambda + \beta\,\lambda)\,e2 - \lambda\,e3, [e1, e3] = \beta\,\lambda\,e2.$$

We note that $\widetilde{\mathfrak{t}}_\beta$ is an almost abelian Lie algebra where $\mathrm{span}_{\mathbb{C}}\{e_2, e_3\}$ is a codimension 1 abelian ideal of $\widetilde{\mathfrak{t}}_\beta$; we can identify $\widetilde{\mathfrak{t}}_\beta$ the matrix

$$\widetilde{B} = \begin{bmatrix} \lambda + \beta\,\lambda & \beta\,\lambda \\ -\lambda & 0 \end{bmatrix}.$$

The isomorphism class of $\widetilde{\mathfrak{t}}_\beta$ is determined by the eigenvalues of $\widetilde{B}$:

1. $\widetilde{B}$ has repeated eigenvalues if and only if $\beta = 1$
   
   (a) 0 is a repeated eigenvalue of $\widetilde{B}$ if and only if $1 = 0$

(b) $\widetilde{B}$ has a non-zero repeated eigenvalue if and only if $\beta = 1$

  i. $\widetilde{\mathfrak{t}}_\beta$ is $\mathfrak{r}_{3,1}$ if and only if $\widetilde{B} = t\,\mathrm{Id}$ for some $\lambda \neq 0$ if and only if $\lambda = 0$ (False)

  ii. $\widetilde{\mathfrak{t}}_\beta$ is isomorphic to $\mathfrak{r}_3$ if and only if $\widetilde{B} \neq t\,\mathrm{Id}$ for all $t \in \mathbb{C}$ if and only if $\beta = 1$

2. $\widetilde{B}$ has distinct eigenvalues if and only if $\beta \neq 1$

   (a) $\widetilde{\mathfrak{t}}_\beta$ is isomorphic to $\mathfrak{r}_2(\mathbb{C}) \times \mathbb{C}$ if and only if 0 is an eigenvalue of $B$ with algebraic multiplicity 1 if and only if $\beta = 0$

   (b) $\widetilde{\mathfrak{t}}_\beta$ is isomorphic to $\mathfrak{r}_{3,-1}$ if and only if $\mathrm{Trace}(\widetilde{B}) = 0$ if and only if $\beta = -1$.

   (c) $\widetilde{\mathfrak{t}}_\beta$ is isomorphic to $\mathfrak{r}_{3,z}$ for some $z$ such that $z^2 \neq 1$ and $z \neq 0$ if and only if $\beta \neq 1$ and $\beta \neq -1$ and $\beta \neq 0$.

### 8.10.3 $\rho(\mathfrak{L}_6^9(\lambda))$

In this case, we have $\rho(\mathfrak{L}_6^9(\lambda)) = \widehat{\mathfrak{t}}$:

$$\widehat{\mathfrak{t}} = \{[e1, e2] = \lambda\, e2, [e1, e3] = \lambda\, e2.$$

which is isomorphic to $\mathfrak{r}_2 \times \mathbb{C}$.

### 8.10.4 Identities

- $[A\cdot, A\cdot] = 0$
- $A^2[\cdot, \cdot] = 0$
- $[A^2\cdot, \cdot] + [\cdot, A^2\cdot] = 0$

The hom-Lie algebra $\mathfrak{L}_6^9(\lambda)$ is a non-multiplicative hom-Lie algebra.

## 8.11 $\mathfrak{L}_6^{10}$

Let $A$ be the endomorphism of $\mathbb{C}^3$ whose matrix representation with respect to the ordered basis $\{e_1, e_2, e_3\}$ is

$$\begin{bmatrix} 0 & 1 & 0 \\ 0 & 0 & 0 \\ 0 & 0 & 0 \end{bmatrix}$$

and let us denote by $\mathfrak{L}_6^{10}$ the hom-Lie algebra $(\mathfrak{L}_6, A)$. We have:

$$\mathrm{Aut}(\mathfrak{L}_6^{10}) = \left\{ \begin{bmatrix} 1 & 0 & 0 \\ 0 & 1 & 0 \\ 0 & 0 & a \end{bmatrix} \in M(3, \mathbb{C}) : a \in \mathbb{C}^* \right\}.$$

$$\mathrm{Der}(\mathfrak{L}_6^{10}) = \left\{ \begin{bmatrix} 0 & 0 & 0 \\ 0 & 0 & 0 \\ 0 & 0 & t_1 \end{bmatrix} \in M(3, \mathbb{C}) : t_1 \in \mathbb{C} \right\}.$$

### 8.11.1 $\psi_{\alpha,\beta}(\mathfrak{L}_6^{10})$

In this case, we have $\psi_{\alpha,\beta}(\mathfrak{L}_6^{10}) = \mathfrak{t}_{\alpha,\beta}$:

$$\mathfrak{t}_{\alpha,\beta} = \{[e1, e2] = \alpha\, e1 + e2.$$

which is isomorphic to $\mathfrak{r}_2(\mathbb{C}) \times \mathbb{C}$.

### 8.11.2   $\phi_\beta(\mathfrak{L}_6^{10})$

In this case, we have $\phi_\beta(\mathfrak{L}_6^{10}) = \widetilde{\mathfrak{t}}_\beta$:

$$\widetilde{\mathfrak{t}}_\beta = \{[e1, e2] = e1.$$

which is isomorphic to $\mathfrak{r}_2(\mathbb{C}) \times \mathbb{C}$.

### 8.11.3   $\rho(\mathfrak{L}_6^{10})$

In this case, we have $\rho(\mathfrak{L}_6^{10}) = \widehat{\mathfrak{t}}$ is the 3-dimensional abelian Lie algebra.

### 8.11.4   Identities

- $[A\cdot, \cdot] + [\cdot, A\cdot] = 0$
- $[A\cdot, A\cdot] = 0$

The hom-Lie algebra $\mathfrak{L}_6^{10}$ is a non-multiplicative hom-Lie algebra.

## 8.12   $\mathfrak{L}_6^{11}$

Let $A$ be the endomorphism of $\mathbb{C}^3$ whose matrix representation with respect to the ordered basis $\{e_1, e_2, e_3\}$ is

$$\begin{bmatrix} 0 & 1 & 0 \\ 0 & 0 & 0 \\ 1 & 0 & 0 \end{bmatrix}$$

and let us denote by $\mathfrak{L}_6^{11}$ the hom-Lie algebra $(\mathfrak{L}_6, A)$. We have:

$$\mathrm{Aut}(\mathfrak{L}_6^{11}) = \left\{ \begin{bmatrix} 1 & 0 & 0 \\ 0 & 1 & 0 \\ 0 & 0 & 1 \end{bmatrix} \right\}.$$

$$\mathrm{Der}(\mathfrak{L}_6^{11}) = \left\{ \begin{bmatrix} 0 & 0 & 0 \\ 0 & 0 & 0 \\ 0 & 0 & 0 \end{bmatrix} \right\}.$$

### 8.12.1   $\psi_{\alpha,\beta}(\mathfrak{L}_6^{11})$

In this case, we have $\psi_{\alpha,\beta}(\mathfrak{L}_6^{11}) = \mathfrak{t}_{\alpha,\beta}$:

$$\mathfrak{t}_{\alpha,\beta} = \{[e1, e2] = \alpha\, e1 + e2.$$

which is isomorphic to $\mathfrak{r}_2(\mathbb{C}) \times \mathbb{C}$.

### 8.12.2   $\phi_\beta(\mathfrak{L}_6^{11})$

In this case, we have $\phi_\beta(\mathfrak{L}_6^{11}) = \widetilde{\mathfrak{t}}_\beta$:

$$\widetilde{\mathfrak{t}}_\beta = \{[e1, e2] = e1.$$

which is isomorphic to $\mathfrak{r}_2(\mathbb{C}) \times \mathbb{C}$.

### 8.12.3   $\rho(\mathfrak{L}_6^{11})$

In this case, we have $\rho(\mathfrak{L}_6^{11}) = \widehat{\mathfrak{t}}$ is the 3-dimensional abelian Lie algebra.

### 8.12.4 Identities

- $[A\cdot,\cdot] + [\cdot, A\cdot] = 0$
- $[A\cdot, A\cdot] = 0$
- $[A^2\cdot, \cdot] = 0$

The hom-Lie algebra $\mathfrak{L}_6^{11}$ is a non-multiplicative hom-Lie algebra.

## 8.13 $\mathfrak{L}_6^{12}$

Let $A$ be the endomorphism of $\mathbb{C}^3$ whose matrix representation with respect to the ordered basis $\{e_1, e_2, e_3\}$ is

$$\begin{bmatrix} 0 & 1 & 0 \\ 0 & 0 & 1 \\ 0 & 0 & 0 \end{bmatrix}$$

and let us denote by $\mathfrak{L}_6^{12}$ the hom-Lie algebra $(\mathfrak{L}_6, A)$. We have:

$$\mathrm{Aut}(\mathfrak{L}_6^{12}) = \left\{ \begin{bmatrix} 1 & 0 & 0 \\ 0 & 1 & 0 \\ 0 & 0 & 1 \end{bmatrix} \right\}.$$

$$\mathrm{Der}(\mathfrak{L}_6^{12}) = \left\{ \begin{bmatrix} 0 & 0 & 0 \\ 0 & 0 & 0 \\ 0 & 0 & 0 \end{bmatrix} \right\}.$$

### 8.13.1 $\psi_{\alpha,\beta}(\mathfrak{L}_6^{12})$

In this case, we have $\psi_{\alpha,\beta}(\mathfrak{L}_6^{12}) = \mathfrak{t}_{\alpha,\beta}$:

$$\mathfrak{t}_{\alpha,\beta} = \{[e1, e2] = \alpha\, e1 + e2, [e1, e3] = \beta\, e2.$$

We have that $\mathfrak{t}_{\alpha,\beta}$ is a Lie algebra if and only if $\beta = 0$ or $\alpha = 0$; in such case $\mathfrak{t}_{\alpha,\beta}$ is isomorphic to $\mathfrak{r}_2(\mathbb{C}) \times \mathbb{C}$.

### 8.13.2 $\phi_\beta(\mathfrak{L}_6^{12})$

In this case, we have $\phi_\beta(\mathfrak{L}_6^{12}) = \widetilde{\mathfrak{t}}_\beta$:

$$\widetilde{\mathfrak{t}}_\beta = \{[e1, e2] = e1, [e1, e3] = \beta\, e2.$$

We have that $\widetilde{\mathfrak{t}}_\beta$ is a Lie algebra if and only if $\beta = 0$; in such case $\widetilde{\mathfrak{t}}_\beta$ is isomorphic to $\mathfrak{r}_2(\mathbb{C}) \times \mathbb{C}$.

### 8.13.3 $\rho(\mathfrak{L}_6^{12})$

In this case, we have $\rho(\mathfrak{L}_6^{12}) = \widehat{\mathfrak{t}}$:

$$\widehat{\mathfrak{t}} = \{[e1, e3] = e2.$$

which is isomorphic to $\mathfrak{n}_3$.

### 8.13.4 Identities

- $A^2[\cdot, \cdot] = 0$
- $[A^2\cdot, A^2\cdot] = 0$
- $[A\cdot, A\cdot] + [A^2\cdot, \cdot] + [\cdot, A^2\cdot] = 0$

The hom-Lie algebra $\mathfrak{L}_6^{12}$ is a non-multiplicative hom-Lie algebra.

## 8.14 $\mathfrak{L}_6^{13}(\lambda)$

Let $A$ be the endomorphism of $\mathbb{C}^3$ whose matrix representation with respect to the ordered basis $\{e_1, e_2, e_3\}$ is

$$\begin{bmatrix} 0 & 1 & 0 \\ 0 & \lambda & \lambda \\ 0 & -\lambda & -\lambda \end{bmatrix}$$

and let us denote by $\mathfrak{L}_6^{13}(\lambda)$ the hom-Lie algebra $(\mathfrak{L}_6, A)$. We have:

$$\mathrm{Aut}(\mathfrak{L}_6^{13}(\lambda)) = \left\{ \begin{bmatrix} 1 & 0 & 0 \\ 0 & 1 & 0 \\ 0 & 0 & 1 \end{bmatrix} \right\}.$$

$$\mathrm{Der}(\mathfrak{L}_6^{13}(\lambda)) = \left\{ \begin{bmatrix} 0 & 0 & 0 \\ 0 & 0 & 0 \\ 0 & 0 & 0 \end{bmatrix} \right\}.$$

### 8.14.1 $\psi_{\alpha,\beta}(\mathfrak{L}_6^{13})$

In this case, we have $\psi_{\alpha,\beta}(\mathfrak{L}_6^{13}) = \mathfrak{t}_{\alpha,\beta}$:

$$\mathfrak{t}_{\alpha,\beta} = \{[e1, e2] = \alpha\, e1 + (1 + \alpha\,\lambda + \beta\,\lambda)\, e2 - \alpha\,\lambda\, e3, [e1, e3] = \beta\,\lambda\, e2.$$

We have that $\mathfrak{t}_{\alpha,\beta}$ is a Lie algebra if and only if $\beta = 0$ or $\alpha = 0$. If $\alpha = 0$, then $\mathfrak{t}_{0,\beta}$ is the Lie algebra

$$\mathfrak{t}_{0,\beta} = \{[e1, e2] = (1 + \beta\,\lambda)\, e2, [e1, e3] = \beta\,\lambda\, e2.$$

which is isomorphic to $\mathfrak{r}_2(\mathbb{C}) \times \mathbb{C}$ if and only if $\beta \neq -\frac{1}{\lambda}$; in other case, $\mathfrak{t}_{0,\beta=-\frac{1}{\lambda}}$ is isomorphic to the Heisenberg Lie algebra $\mathfrak{n}_3$. If $\beta = 0$, then $\mathfrak{t}_{\alpha,\beta}$ is the Lie algebra

$$\mathfrak{t}_{\alpha,\beta} = \{[e1, e2] = \alpha\, e1 + (1 + \alpha\,\lambda)\, e2 - \alpha\,\lambda\, e3.$$

which is isomorphic to $\mathfrak{r}_2(\mathbb{C}) \times \mathbb{C}$. In fact, if $\alpha \neq -\frac{1}{\lambda}$ we can consider the change of basis

$$g = \begin{bmatrix} 1 + \alpha\,\lambda & -\alpha & 0 \\ 0 & 1 & 0 \\ \lambda & 0 & 1 \end{bmatrix}$$

and, if $\alpha = -\frac{1}{\lambda}$ we can consider the change of basis given by

$$\widetilde{g} = \begin{bmatrix} 0 & \lambda^{-1} & 0 \\ 1 & 0 & 0 \\ \lambda & 0 & 1 \end{bmatrix}$$

### 8.14.2 $\phi_\beta(\mathfrak{L}_6^{13})$

In this case, we have $\phi_\beta(\mathfrak{L}_6^{13}) = \widetilde{\mathfrak{t}}_\beta$:

$$\widetilde{\mathfrak{t}}_\beta = \{[e1, e2] = e1 + (\lambda + \beta\,\lambda)\, e2 - \lambda\, e3, [e1, e3] = \beta\,\lambda\, e2.$$

We have that $\widetilde{\mathfrak{t}}_\beta$ is a Lie algebra if and only if $\beta = 0$; in such case is isomorphic to $\mathfrak{r}_2(\mathbb{C}) \times \mathbb{C}$. In fact, we can consider the change of basis given by

$$g = \begin{bmatrix} \lambda & -1 & 0 \\ 0 & 1 & 0 \\ \lambda & 0 & 1 \end{bmatrix}.$$

### 8.14.3 $\rho(\mathfrak{L}_6^{13}(\lambda))$

In this case, we have $\rho(\mathfrak{L}_6^{13}(\lambda)) = \widehat{\mathfrak{t}}$:

$$\widehat{\mathfrak{t}} = \{[e1, e2] = \lambda\, e2, [e1, e3] = \lambda\, e2.$$

which is isomorphic to $\mathfrak{r}_2 \times \mathbb{C}$.

#### 8.14.4 Identities

- $[A\cdot, A\cdot] + [A^2\cdot, \cdot] + [\cdot, A^2\cdot] = 0$
- $[A^2\cdot, A^2\cdot] = 0$

The hom-Lie algebra $\mathfrak{L}_6^{13}(\lambda)$ is a non-multiplicative hom-Lie algebra.

## 8.15 Degenerations between hom-Lie algebras ($\mathfrak{L}_6^i$)

If $\mathfrak{L}_i^j \xrightarrow{\deg} \mathfrak{L}_h^k$ then $\mathrm{Der}(\mathfrak{L}_i^j) \leq \mathrm{Der}(\mathfrak{L}_h^k)$. Therefore, we can organize the hom-Lie algebras in the following way:

| Dim(Der) | hom-Lie algebra |
|---|---|
| 0 | $\mathfrak{L}_6^{13}(\lambda), \mathfrak{L}_6^{12}, \mathfrak{L}_6^{11}$ |
| 1 | $\mathfrak{L}_6^9(\lambda), \mathfrak{L}_6^{10}, \mathfrak{L}_6^8, \mathfrak{L}_6^7$ |
| 2 | $\mathfrak{L}_6^6(\lambda), \mathfrak{L}_6^5, \mathfrak{L}_6^4, \mathfrak{L}_6^3$ |
| 3 | $\mathfrak{L}_6^2, \mathfrak{L}_6^1$ |
| 4 | $\mathfrak{L}_6^0$ |

### 8.15.1 Degenerations

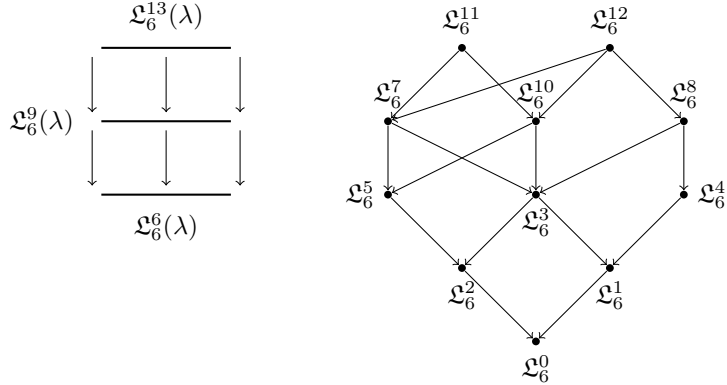

### 8.15.2 Degenerations

1. $\mathfrak{L}_6^{13}(\lambda) \xrightarrow{\deg} \mathfrak{L}_6^9(\lambda)$
   In fact, set

   $$g(t) = \begin{bmatrix} 1 & 0 & 0 \\ \frac{e^t(-1+e^t)}{\lambda} & -\frac{e^t(-1+e^t)^2}{\lambda^2} & 0 \\ -\frac{(-1+e^t)^2}{\lambda} & 0 & -\frac{e^t(-1+e^t)^2}{\lambda^2} \end{bmatrix}, \text{ with } t \in \mathbb{R}.$$

We have $g(t) \in \mathrm{Aut}(\mathfrak{L}_6)$ and $g(t) \cdot \mathfrak{L}_6^{13}$ is the hom-Lie algebra $(\mathfrak{L}_6, A(t))$ with

$$A(t) = \begin{bmatrix} \frac{\lambda}{-1+e^t} & -\frac{\lambda^2}{e^t(-1+e^t)^2} & 0 \\ 1 & \frac{\lambda(-2+e^t)}{-1+e^t} & \lambda \\ 0 & \frac{\lambda}{e^t} - \lambda & -\lambda \end{bmatrix}$$

It is easy to check that $(\mathfrak{L}_6, A(t)) \to \mathfrak{L}_6^9(\lambda)$ as $t$ tends to infinity.

2. $\mathfrak{L}_6^{12} \xrightarrow{\deg} \mathfrak{L}_6^{10}$

   In fact, set

$$g(t) = \begin{bmatrix} 1 & 0 & 0 \\ 0 & 1 & 0 \\ 0 & 0 & e^t \end{bmatrix}, \text{ with } t \in \mathbb{R}.$$

We have $g(t) \in \mathrm{Aut}(\mathfrak{L}_6)$ and $g(t) \cdot \mathfrak{L}_6^{12}$ is the hom-Lie algebra $(\mathfrak{L}_6, A(t))$ with

$$A(t) = \begin{bmatrix} 0 & 1 & 0 \\ 0 & 0 & e^{-t} \\ 0 & 0 & 0 \end{bmatrix}$$

It is easy to check that $(\mathfrak{L}_6, A(t)) \to \mathfrak{L}_6^{10}$ as $t$ tends to infinity.

3. $\mathfrak{L}_6^{12} \xrightarrow{\deg} \mathfrak{L}_6^8$

   In fact, set

$$g(t) = \begin{bmatrix} 1 & 0 & 0 \\ e^{2t} & -e^{3t} & 0 \\ e^t & 0 & -e^{3t} \end{bmatrix}, \text{ with } t \in \mathbb{R}.$$

We have $g(t) \in \mathrm{Aut}(\mathfrak{L}_6)$ and $g(t) \cdot \mathfrak{L}_6^{12}$ is the hom-Lie algebra $(\mathfrak{L}_6, A(t))$ with

$$A(t) = \begin{bmatrix} e^{-t} & -e^{-3t} & 0 \\ 0 & -e^{-t} & 1 \\ 1 & -e^{-2t} & 0 \end{bmatrix}$$

It is easy to check that $(\mathfrak{L}_6, A(t)) \to \mathfrak{L}_6^8$ as $t$ tends to infinity.

4. $\mathfrak{L}_6^{12} \xrightarrow{\deg} \mathfrak{L}_6^7$

   In fact, set

$$g(t) = \begin{bmatrix} 1 & 0 & 0 \\ 0 & e^t & 0 \\ e^t & 0 & -e^{2t} \end{bmatrix}, \text{ with } t \in \mathbb{R}.$$

We have $g(t) \in \mathrm{Aut}(\mathfrak{L}_6)$ and $g(t) \cdot \mathfrak{L}_6^{12}$ is the hom-Lie algebra $(\mathfrak{L}_6, A(t))$ with

$$A(t) = \begin{bmatrix} 0 & e^{-t} & 0 \\ 1 & 0 & -e^{-t} \\ 0 & 1 & 0 \end{bmatrix}$$

It is easy to check that $(\mathfrak{L}_6, A(t)) \to \mathfrak{L}_6^7$ as $t$ tends to infinity.

5. $\mathfrak{L}_6^{11} \xrightarrow{\deg} \mathfrak{L}_6^{10}$

   In fact, set
   $$g(t) = \begin{bmatrix} 1 & 0 & 0 \\ e^{-t} & 1 & 0 \\ e^{-t} & 0 & e^{-2t} \end{bmatrix}, \text{ with } t \in \mathbb{R}.$$

   We have $g(t) \in \text{Aut}(\mathfrak{L}_6)$ and $g(t) \cdot \mathfrak{L}_6^{11}$ is the hom-Lie algebra $(\mathfrak{L}_6, A(t))$ with
   $$A(t) = \begin{bmatrix} -e^{-t} & 1 & 0 \\ -e^{-2t} & e^{-t} & 0 \\ 0 & e^{-t} & 0 \end{bmatrix}$$

   It is easy to check that $(\mathfrak{L}_6, A(t)) \to \mathfrak{L}_6^{10}$ as $t$ tends to infinity.

6. $\mathfrak{L}_6^{11} \xrightarrow{\deg} \mathfrak{L}_6^{7}$

   In fact, set
   $$g(t) = \begin{bmatrix} 1 & 0 & 0 \\ e^{t} & -e^{2t} & 0 \\ -e^{2t} & 0 & e^{t} \end{bmatrix}, \text{ with } t \in \mathbb{R}.$$

   We have $g(t) \in \text{Aut}(\mathfrak{L}_6)$ and $g(t) \cdot \mathfrak{L}_6^{11}$ is the hom-Lie algebra $(\mathfrak{L}_6, A(t))$ with
   $$A(t) = \begin{bmatrix} e^{-t} & -e^{-2t} & 0 \\ 1 & -e^{-t} & 0 \\ 0 & 1 & 0 \end{bmatrix}$$

   It is easy to check that $(\mathfrak{L}_6, A(t)) \to \mathfrak{L}_6^{7}$ as $t$ tends to infinity.

7. $\mathfrak{L}_6^{10} \xrightarrow{\deg} \mathfrak{L}_6^{5}$

   In fact, set
   $$g(t) = \begin{bmatrix} 1 & 0 & 0 \\ 0 & e^{t} & 0 \\ e^{t} & 0 & 1 \end{bmatrix}, \text{ with } t \in \mathbb{R}.$$

   We have $g(t) \in \text{Aut}(\mathfrak{L}_6)$ and $g(t) \cdot \mathfrak{L}_6^{10}$ is the hom-Lie algebra $(\mathfrak{L}_6, A(t))$ with
   $$A(t) = \begin{bmatrix} 0 & e^{-t} & 0 \\ 0 & 0 & 0 \\ 0 & 1 & 0 \end{bmatrix}$$

   It is easy to check that $(\mathfrak{L}_6, A(t)) \to \mathfrak{L}_6^{5}$ as $t$ tends to infinity.

8. $\mathfrak{L}_6^{10} \xrightarrow{\deg} \mathfrak{L}_6^{3}$

   In fact, set
   $$g(t) = \begin{bmatrix} 1 & 0 & 0 \\ e^{t} & -e^{2t} & 0 \\ e^{t} & 0 & 1 \end{bmatrix}, \text{ with } t \in \mathbb{R}.$$

   We have $g(t) \in \text{Aut}(\mathfrak{L}_6)$ and $g(t) \cdot \mathfrak{L}_6^{10}$ is the hom-Lie algebra $(\mathfrak{L}_6, A(t))$ with
   $$A(t) = \begin{bmatrix} e^{-t} & -e^{-2t} & 0 \\ 1 & -e^{-t} & 0 \\ 1 & -e^{-t} & 0 \end{bmatrix}$$

It is easy to check that $(\mathfrak{L}_6, A(t)) \to \mathfrak{L}_6^3$ as $t$ tends to infinity.

9. $\mathfrak{L}_6^9(\lambda) \xrightarrow{\deg} \mathfrak{L}_6^6(\lambda)$
   In fact, set
   $$g(t) = \begin{bmatrix} 1 & 0 & 0 \\ 0 & e^{-t} & 0 \\ 0 & 0 & e^{-t} \end{bmatrix}, \text{ with } t \in \mathbb{R}.$$
   We have $g(t) \in \mathrm{Aut}(\mathfrak{L}_6)$ and $g(t) \cdot \mathfrak{L}_6^9(\lambda)$ is the hom-Lie algebra $(\mathfrak{L}_6, A(t))$ with
   $$A(t) = \begin{bmatrix} 0 & 0 & 0 \\ e^{-t} & \lambda & \lambda \\ 0 & -\lambda & -\lambda \end{bmatrix}$$
   It is easy to check that $(\mathfrak{L}_6, A(t)) \to \mathfrak{L}_6^6(\lambda)$ as $t$ tends to infinity.

10. $\mathfrak{L}_6^8 \xrightarrow{\deg} \mathfrak{L}_6^4$
    In fact, set
    $$g(t) = \begin{bmatrix} 1 & 0 & 0 \\ 0 & e^{-t} & 0 \\ 0 & 0 & e^{-t} \end{bmatrix}, \text{ with } t \in \mathbb{R}.$$
    We have $g(t) \in \mathrm{Aut}(\mathfrak{L}_6)$ and $g(t) \cdot \mathfrak{L}_6^8$ is the hom-Lie algebra $(\mathfrak{L}_6, A(t))$ with
    $$A(t) = \begin{bmatrix} 0 & 0 & 0 \\ 0 & 0 & 1 \\ e^{-t} & 0 & 0 \end{bmatrix}$$
    It is easy to check that $(\mathfrak{L}_6, A(t)) \to \mathfrak{L}_6^4$ as $t$ tends to infinity.

11. $\mathfrak{L}_6^8 \xrightarrow{\deg} \mathfrak{L}_6^3$
    In fact, set
    $$g(t) = \begin{bmatrix} 1 & 0 & 0 \\ 0 & e^{-t} & 0 \\ -e^t & 0 & 1 \end{bmatrix}, \text{ with } t \in \mathbb{R}.$$
    We have $g(t) \in \mathrm{Aut}(\mathfrak{L}_6)$ and $g(t) \cdot \mathfrak{L}_6^8$ is the hom-Lie algebra $(\mathfrak{L}_6, A(t))$ with
    $$A(t) = \begin{bmatrix} 0 & 0 & 0 \\ 1 & 0 & e^{-t} \\ 1 & 0 & 0 \end{bmatrix}$$
    It is easy to check that $(\mathfrak{L}_6, A(t)) \to \mathfrak{L}_6^3$ as $t$ tends to infinity.

12. $\mathfrak{L}_6^7 \xrightarrow{\deg} \mathfrak{L}_6^5$
    In fact, set
    $$g(t) = \begin{bmatrix} 1 & 0 & 0 \\ 0 & e^{-t} & 0 \\ 0 & 0 & e^{-t} \end{bmatrix}, \text{ with } t \in \mathbb{R}.$$

85

We have $g(t) \in \mathrm{Aut}(\mathfrak{L}_6)$ and $g(t) \cdot \mathfrak{L}_6^7$ is the hom-Lie algebra $(\mathfrak{L}_6, A(t))$ with

$$A(t) = \begin{bmatrix} 0 & 0 & 0 \\ \mathrm{e}^{-t} & 0 & 0 \\ 0 & 1 & 0 \end{bmatrix}$$

It is easy to check that $(\mathfrak{L}_6, A(t)) \to \mathfrak{L}_6^5$ as $t$ tends to infinity.

13. $\mathfrak{L}_6^7 \xrightarrow{\deg} \mathfrak{L}_6^3$

In fact, set

$$g(t) = \begin{bmatrix} 1 & 0 & 0 \\ \mathrm{e}^t & 1 & 0 \\ 0 & 0 & -(\mathrm{e}^t)^{-1} \end{bmatrix}, \text{ with } t \in \mathbb{R}.$$

We have $g(t) \in \mathrm{Aut}(\mathfrak{L}_6)$ and $g(t) \cdot \mathfrak{L}_6^7$ is the hom-Lie algebra $(\mathfrak{L}_6, A(t))$ with

$$A(t) = \begin{bmatrix} 0 & 0 & 0 \\ 1 & 0 & 0 \\ 1 & -\mathrm{e}^{-t} & 0 \end{bmatrix}$$

It is easy to check that $(\mathfrak{L}_6, A(t)) \to \mathfrak{L}_6^3$ as $t$ tends to infinity.

14. $\mathfrak{L}_6^5 \xrightarrow{\deg} \mathfrak{L}_6^2$

In fact, set

$$g(t) = \begin{bmatrix} 1 & 0 & 0 \\ \mathrm{e}^t & -\mathrm{e}^t & 0 \\ 0 & 0 & 1 \end{bmatrix}, \text{ with } t \in \mathbb{R}.$$

We have $g(t) \in \mathrm{Aut}(\mathfrak{L}_6)$ and $g(t) \cdot \mathfrak{L}_6^5$ is the hom-Lie algebra $(\mathfrak{L}_6, A(t))$ with

$$A(t) = \begin{bmatrix} 0 & 0 & 0 \\ 0 & 0 & 0 \\ 1 & -\mathrm{e}^{-t} & 0 \end{bmatrix}$$

It is easy to check that $(\mathfrak{L}_6, A(t)) \to \mathfrak{L}_6^2$ as $t$ tends to infinity.

15. $\mathfrak{L}_6^4 \xrightarrow{\deg} \mathfrak{L}_6^1$

In fact, set

$$g(t) = \begin{bmatrix} 1 & 0 & 0 \\ 0 & 1 & 0 \\ \mathrm{e}^t & 0 & -\mathrm{e}^t \end{bmatrix}, \text{ with } t \in \mathbb{R}.$$

We have $g(t) \in \mathrm{Aut}(\mathfrak{L}_6)$ and $g(t) \cdot \mathfrak{L}_6^4$ is the hom-Lie algebra $(\mathfrak{L}_6, A(t))$ with

$$A(t) = \begin{bmatrix} 0 & 0 & 0 \\ 1 & 0 & -\mathrm{e}^{-t} \\ 0 & 0 & 0 \end{bmatrix}$$

It is easy to check that $(\mathfrak{L}_6, A(t)) \to \mathfrak{L}_6^1$ as $t$ tends to infinity.

16. $\mathfrak{L}_6^3 \xrightarrow{\deg} \mathfrak{L}_6^2$

In fact, set
$$g(t) = \begin{bmatrix} 1 & 0 & 0 \\ 0 & e^{-t} & 0 \\ 0 & 0 & 1 \end{bmatrix}, \text{ with } t \in \mathbb{R}.$$

We have $g(t) \in \mathrm{Aut}(\mathfrak{L}_6)$ and $g(t) \cdot \mathfrak{L}_6^3$ is the hom-Lie algebra $(\mathfrak{L}_6, A(t))$ with
$$A(t) = \begin{bmatrix} 0 & 0 & 0 \\ e^{-t} & 0 & 0 \\ 1 & 0 & 0 \end{bmatrix}$$

It is easy to check that $(\mathfrak{L}_6, A(t)) \to \mathfrak{L}_6^2$ as $t$ tends to infinity.

17. $\mathfrak{L}_6^3 \xrightarrow{\deg} \mathfrak{L}_6^1$

In fact, set
$$g(t) = \begin{bmatrix} 1 & 0 & 0 \\ 0 & 1 & 0 \\ 0 & 0 & e^{-t} \end{bmatrix}, \text{ with } t \in \mathbb{R}.$$

We have $g(t) \in \mathrm{Aut}(\mathfrak{L}_6)$ and $g(t) \cdot \mathfrak{L}_6^3$ is the hom-Lie algebra $(\mathfrak{L}_6, A(t))$ with
$$A(t) = \begin{bmatrix} 0 & 0 & 0 \\ 1 & 0 & 0 \\ e^{-t} & 0 & 0 \end{bmatrix}$$

It is easy to check that $(\mathfrak{L}_6, A(t)) \to \mathfrak{L}_6^1$ as $t$ tends to infinity.

18. $\mathfrak{L}_6^2 \xrightarrow{\deg} \mathfrak{L}_6^0$

In fact, set
$$g(t) = \begin{bmatrix} 1 & 0 & 0 \\ 0 & 1 & 0 \\ 0 & 0 & e^{-t} \end{bmatrix}, \text{ with } t \in \mathbb{R}.$$

We have $g(t) \in \mathrm{Aut}(\mathfrak{L}_6)$ and $g(t) \cdot \mathfrak{L}_6^2$ is the hom-Lie algebra $(\mathfrak{L}_6, A(t))$ with
$$A(t) = \begin{bmatrix} 0 & 0 & 0 \\ 0 & 0 & 0 \\ e^{-t} & 0 & 0 \end{bmatrix}$$

It is easy to check that $(\mathfrak{L}_6, A(t)) \to \mathfrak{L}_6^0$ as $t$ tends to infinity.

19. $\mathfrak{L}_6^1 \xrightarrow{\deg} \mathfrak{L}_6^0$

In fact, set
$$g(t) = \begin{bmatrix} 1 & 0 & 0 \\ 0 & e^{-t} & 0 \\ 0 & 0 & 1 \end{bmatrix}, \text{ with } t \in \mathbb{R}.$$

We have $g(t) \in \mathrm{Aut}(\mathfrak{L}_6)$ and $g(t) \cdot \mathfrak{L}_6^1$ is the hom-Lie algebra $(\mathfrak{L}_6, A(t))$ with
$$A(t) = \begin{bmatrix} 0 & 0 & 0 \\ e^{-t} & 0 & 0 \\ 0 & 0 & 0 \end{bmatrix}$$

It is easy to check that $(\mathfrak{L}_6, A(t)) \to \mathfrak{L}_6^0$ as $t$ tends to infinity.

20. $\mathfrak{L}_6^{13}(\lambda) \xrightarrow{\deg} \mathfrak{L}_6^6(\kappa)$ with $(\lambda \neq \kappa)$.

Suppose, contrary to our claim, that $\mathfrak{L}_6^{13}(\lambda) \xrightarrow{\deg} \mathfrak{L}_6^6(\kappa)$ with $(\lambda \neq \kappa)$. In such case, we have

$$\psi_{\alpha,\beta}(\mathfrak{L}_6^{13}(\lambda)) \xrightarrow{\deg} \psi_{\alpha,\beta}(\mathfrak{L}_6^6(\kappa)).$$

By taking $\alpha = 0$ and $\beta = -1/\lambda$, we have from 8.14.1 that $\psi_{\alpha,\beta}(\mathfrak{L}_6^{13}(\lambda))$ is a Lie algebra isomorphic to $\mathfrak{n}_3$, therefore, $\psi_{\alpha,\beta}(\mathfrak{L}_6^6(\kappa))$ is a Lie algebra which is isomorphic to one of the following: $\mathfrak{n}_3$ or $\mathfrak{a}_3$ (by Theorem 1.1). We must have $\psi_{\alpha,\beta}(\mathfrak{L}_6^6(\kappa))$ is isomorphic to $\mathfrak{n}_3$ (from 8.7.1). But, we now apply again 8.7.1, to obtain $\alpha = 0$ or $\beta = -1/\kappa$, and therefore $\lambda = \kappa$; this is a contradiction.

21. $\mathfrak{L}_6^{13}(\lambda) \xrightarrow{\deg} \mathfrak{L}_6^0$.

Suppose, contrary to our claim, that $\mathfrak{L}_6^{13}(\lambda) \xrightarrow{\deg} \mathfrak{L}_6^0$. In such case, we have

$$\psi_{\alpha,\beta}(\mathfrak{L}_6^{13}(\lambda)) \xrightarrow{\deg} \psi_{\alpha,\beta}(\mathfrak{L}_6^0).$$

By taking $\alpha = 0$ and $\beta = -1/\lambda$, we have from 8.14.1 that $\psi_{\alpha,\beta}(\mathfrak{L}_6^{13}(\lambda))$ is a Lie algebra isomorphic to $\mathfrak{n}_3$, therefore, $\psi_{\alpha,\beta}(\mathfrak{L}_6^6(\kappa))$ is a Lie algebra which is isomorphic to one of the following: $\mathfrak{n}_3$ or $\mathfrak{a}_3$ (by Theorem 1.1). But, from 8.1.1, $\psi_{\alpha,\beta}(\mathfrak{L}_6^0)$ is isomorphic to $\mathfrak{r}_2 \times \mathbb{C}$; this is a contradiction.

22. $\mathfrak{L}_6^{12} \xrightarrow{\deg} \mathfrak{L}_6^6(\lambda)$ with $(\lambda \neq 0)$.

Suppose, contrary to our claim, that $\mathfrak{L}_6^{12} \xrightarrow{\deg} \mathfrak{L}_6^6(\lambda)$ with $(\lambda \neq 0)$. In such case, we have

$$\rho(\mathfrak{L}_6^{12}) \xrightarrow{\deg} \rho(\mathfrak{L}_6^6(\lambda)).$$

We have from 8.13.3 that $\rho(\mathfrak{L}_6^{12})$ is a Lie algebra isomorphic to $\mathfrak{n}_3$, therefore, $\rho(\mathfrak{L}_6^6(\lambda))$ is a Lie algebra which is isomorphic to one of the following: $\mathfrak{n}_3$ or $\mathfrak{a}_3$ (by Theorem 1.1). But, from 8.7.3, we have $\rho(\mathfrak{L}_6^6(\lambda))$ is isomorphic to $\mathfrak{r}_2 \times \mathbb{C}$; this is a contradiction.

23. $\mathfrak{L}_6^{11} \xrightarrow{\deg} \mathfrak{L}_6^6(\lambda)$ with $(\lambda \neq 0)$.

Suppose, contrary to our claim, that $\mathfrak{L}_6^{11} \xrightarrow{\deg} \mathfrak{L}_6^6(\lambda)$ with $(\lambda \neq 0)$. In such case, we have

$$\phi_\beta(\mathfrak{L}_6^{11}) \xrightarrow{\deg} \phi_\beta(\mathfrak{L}_6^6(\lambda)).$$

By taking $\beta = -1$, we have from 8.12.2 that $\phi_\beta(\mathfrak{L}_6^{11})$ is a Lie algebra isomorphic to $\mathfrak{r}_2 \times \mathbb{C}$, therefore, $\phi_\beta(\mathfrak{L}_6^6(\lambda))$ is a Lie algebra which is isomorphic to one of the following: $\mathfrak{r}_2 \times \mathbb{C}$, $\mathfrak{n}_3$ or $\mathfrak{a}_3$ (by Theorem 1.1). We must have $\phi_\beta(\mathfrak{L}_6^6(\lambda))$ is isomorphic to $\mathfrak{r}_2 \times \mathbb{C}$ (from 8.7.2). But, we now apply again 8.7.2, to obtain $\beta = 0$; this is a contradiction.

24. $\mathfrak{L}_6^{11} \xrightarrow{\deg} \mathfrak{L}_6^4$.

Suppose, contrary to our claim, that $\mathfrak{L}_6^{11} \xrightarrow{\deg} \mathfrak{L}_6^4$. In such case, we have

$$\rho(\mathfrak{L}_6^{11}) \xrightarrow{\deg} \rho(\mathfrak{L}_6^4).$$

We have from 8.12.3 that $\rho(\mathfrak{L}_6^{11})$ is a Lie algebra isomorphic to $\mathfrak{a}_3$, therefore, $\rho(\mathfrak{L}_6^4)$ is the 3-dimensional abelian Lie algebra. But, from 8.5.4, $\rho(\mathfrak{L}_6^4)$ is isomorphic to $\mathfrak{n}_3$; this is a contradiction.

25. $\mathfrak{L}_6^8 \xrightarrow{\deg} \mathfrak{L}_6^5$.

Suppose, contrary to our claim, that $\mathfrak{L}_6^8 \xrightarrow{\deg} \mathfrak{L}_6^5$. In such case, we have

$$\phi_\beta(\mathfrak{L}_6^8) \xrightarrow{\deg} \phi_\beta(\mathfrak{L}_6^5).$$

By taking $\beta = 0$, we have from 8.9.2 that $\phi_\beta(\mathfrak{L}_6^8)$ is the 3-dimensional Lie algebra algebra, therefore, $\phi_\beta(\mathfrak{L}_6^5)$ is the abelian Lie algebra. But, from 8.6.2, we have $\phi_\beta(\mathfrak{L}_6^5)$ is isomorphic to $\mathfrak{n}_3$; this is a contradiction.

26. $\mathfrak{L}_6^5 \xrightarrow{\deg} \mathfrak{L}_6^1$.

For $\mathfrak{L}_6^5$ satisfies the identity $[A_5\cdot,\cdot] \equiv 0$ but $\mathfrak{L}_6^1$ does not.

27. $\mathfrak{L}_6^4 \xrightarrow{\deg} \mathfrak{L}_6^2$.

Suppose, contrary to our claim, that $\mathfrak{L}_6^4 \xrightarrow{\deg} \mathfrak{L}_6^2$. We have

$$\dim(\operatorname{Ker} T_{0,1,0,0,0,1;1,-1,0,1,-1,0,1,-1,0}(\mathfrak{L}_6^4)) \stackrel{8.5.1}{=} 4 \leq \dim(\operatorname{Ker} T_{0,1,0,0,0,1;1,-1,0,1,-1,0,1,-1,0}(\mathfrak{L}_6^2))$$

From 8.3.1, $\dim(\operatorname{Ker} T_{0,1,0,0,0,1;1,-1,0,1,-1,0,1,-1,0}(\mathfrak{L}_6^2)) = 3$ ; we have a contradiction.

# 9   $\mathfrak{L}_7$: $\mathfrak{so}(3,\mathbb{C}) \cong \mathfrak{sl}(2,\mathbb{C})$

$$\mathfrak{L}_7 := \{[e_1,e_2] = e_3, [e_2,e_3] = e_1, [e_3,e_1] = e_2$$

$\operatorname{Aut}(\mathfrak{L}_7) = \left\{ g \in M(3,\mathbb{C}) : \begin{array}{l} g^{\mathrm{T}}g = \operatorname{Id}, \\ \det(g) = 1 \end{array} \right\} = SO(3,\mathbb{C}).$

$\operatorname{Der}(\mathfrak{L}_7) = \left\{ X \in M(3,\mathbb{C}) : X^{\mathrm{T}} + X = 0 \right\} = \mathfrak{so}(3,\mathbb{C}).$

## 9.1   $\mathfrak{L}_7^0$

Let $A$ be the Zero map and let us denote by $\mathfrak{L}_7^0$ the hom-Lie algebra $(\mathfrak{L}_0, A)$. We have:
$\operatorname{Aut}(\mathfrak{L}_7^0) = \mathfrak{L}_7$.
$\operatorname{Der}(\mathfrak{L}_7^0) = \mathfrak{L}_7$

### 9.1.1   $\psi_{\alpha,\beta}(\mathfrak{L}_7^0)$

In this case, we have $\psi_{\alpha,\beta}(\mathfrak{L}_7^0) = \mathfrak{t}_{\alpha,\beta}$ is the Lie algebra $\mathfrak{so}(3,\mathbb{C})$.

### 9.1.2   $\phi_\beta(\mathfrak{L}_7^0)$

In this case, we have $\phi_\beta(\mathfrak{L}_7^0) = \widetilde{\mathfrak{t}}_\beta$ is the 3-dimensional abelian Lie algebra.

### 9.1.3   $\rho(\mathfrak{L}_7^0)$

In this case, we have $\rho(\mathfrak{L}_7^0) = \widehat{\mathfrak{t}}$ is the 3-dimensional abelian Lie algebra.

## 9.2   $\mathfrak{L}_7^1$

Let $A$ be the endomorphism of $\mathbb{C}^3$ whose matrix representation with respect to the ordered basis $\{e_1, e_2, e_3\}$ is

$$\begin{bmatrix} 0 & 0 & 0 \\ 0 & 1 & \sqrt{-1} \\ 0 & \sqrt{-1} & -1 \end{bmatrix}$$

and let us denote by $\mathfrak{L}_7^1$ the hom-Lie algebra $(\mathfrak{L}_7, A)$. We have:

$$\operatorname{Aut}(\mathfrak{L}_7^1) = \left\{ \begin{bmatrix} 1 & \sqrt{-1}t & -t \\ -\sqrt{-1}t & \frac{1}{2}t^2+1 & \frac{1}{2}\sqrt{-1}t^2 \\ t & \frac{1}{2}\sqrt{-1}t^2 & 1-\frac{1}{2}t^2 \end{bmatrix} \in M(3,\mathbb{C}) : t \in \mathbb{C} \right\}$$

$$\dot{\cup} \left\{ \begin{bmatrix} 1 & -\sqrt{-1}t & t \\ -\sqrt{-1}t & -\frac{1}{2}t^2-1 & -\frac{1}{2}\sqrt{-1}t^2 \\ t & -\frac{1}{2}\sqrt{-1}t^2 & -1+\frac{1}{2}t^2 \end{bmatrix} \in M(3,\mathbb{C}) : t \in \mathbb{C} \right\}.$$

$$\operatorname{Der}(\mathfrak{L}_7^1) = \left\{ \begin{bmatrix} 0 & \sqrt{-1}t_1 & -t_1 \\ -\sqrt{-1}t_1 & 0 & 0 \\ t_1 & 0 & 0 \end{bmatrix} \in M(3,\mathbb{C}) : t_1 \in \mathbb{C} \right\} \cong \mathbb{C}.$$

### 9.2.1   $\psi_{\alpha,\beta}(\mathfrak{L}_7^1)$

In this case, we have $\psi_{\alpha,\beta}(\mathfrak{L}_7^1) = \mathfrak{t}_{\alpha,\beta}$:

$\mathfrak{t}_{\alpha,\beta} = \{[e1, e2] = (i\alpha - i\beta)\,e2 + (1+\beta-\alpha)\,e3, [e1, e3] = (-1-\alpha+\beta)\,e2 + (-i\alpha+i\beta)\,e3, [e2, e3] = e1.$

is a Lie algebra isomorphic to $\mathfrak{so}(3,\mathbb{C})$. In fact, by considering the matrix of the Killing form of $\mathfrak{t}_{\alpha,\beta}$ in the basis $\{e_1, e_2, e_3\}$

$$\begin{bmatrix} -2 & 0 & 0 \\ 0 & -2 - 2\beta + 2\alpha & 2i\alpha - 2i\beta \\ 0 & 2i\alpha - 2i\beta & -2 - 2\alpha + 2\beta \end{bmatrix}$$

we note that it is not degenerate and it follows the conclusion.

### 9.2.2 $\phi_\beta(\mathfrak{L}_7^1)$

In this case, we have $\phi_\beta(\mathfrak{L}_7^1) = \widetilde{\mathfrak{t}}_\beta$:

$$\widetilde{\mathfrak{t}}_\beta = \{[e1, e2] = (i - i\beta)\, e2 + (\beta - 1)\, e3, [e1, e3] = (\beta - 1)\, e2 + (-i + i\beta)\, e3.$$

We note that $\widetilde{\mathfrak{t}}_\beta$ is an almost abelian Lie algebra where $\mathrm{span}_\mathbb{C}\{e_2, e_3\}$ is a codimension 1 abelian ideal of $\widetilde{\mathfrak{t}}_\beta$; we can identify $\widetilde{\mathfrak{t}}_\beta$ the matrix

$$\widetilde{B} = \begin{bmatrix} i - i\beta & \beta - 1 \\ \beta - 1 & -i + i\beta \end{bmatrix}.$$

The isomorphism class of $\widetilde{\mathfrak{t}}_\beta$ is determined by the eigenvalues of $\widetilde{B}$:

1. $\widetilde{B}$ has repeated eigenvalues if and only if $0 = 0$
   (a) $0$ is a repeated eigenvalue of $\widetilde{B}$ if and only if $0 = 0$
      i. $\widetilde{\mathfrak{t}}_\beta$ is the 3-dimensional abelian Lie algebra if and only if $\widetilde{B} = 0$ if and only if $\beta = 1$
      ii. $\widetilde{\mathfrak{t}}_\beta$ is isomorphic to $\mathfrak{n}_3(\mathbb{C})$ if and only if $\widetilde{B} \neq 0$ if and only if $\beta \neq 1$
   (b) $\widetilde{B}$ has a non-zero repeated eigenvalue if and only if $1 = 0$

2. $\widetilde{B}$ has distinct eigenvalues if and only if $1 = 0$

### 9.2.3 $\rho(\mathfrak{L}_7^1)$

In this case, we have $\rho(\mathfrak{L}_7^1) = \widehat{\mathfrak{t}}$:

$$\widehat{\mathfrak{t}} = \{[e1, e2] = -ie2 + e3, [e1, e3] = e2 + ie3.$$

which is isomorphic to $\mathfrak{n}_3$

### 9.2.4 Identities

- $A[\cdot, \cdot] + [A\cdot, \cdot] + [\cdot, A\cdot] = 0$ (Generalized Derivation)
- $[A\cdot, A\cdot] = 0$

The hom-Lie algebra $\mathfrak{L}_7^1$ is a non-multiplicative hom-Lie algebra.

## 9.3 $\mathfrak{L}_7^2$

Let $A$ be the endomorphism of $\mathbb{C}^3$ whose matrix representation with respect to the ordered basis $\{e_1, e_2, e_3\}$ is

$$\begin{bmatrix} 0 & 1 & \sqrt{-1} \\ 1 & 0 & 0 \\ \sqrt{-1} & 0 & 0 \end{bmatrix}$$

and let us denote by $\mathfrak{L}_7^2$ the hom-Lie algebra $(\mathfrak{L}_7, A)$. We have:

$$\mathrm{Aut}(\mathfrak{L}_7^2) = \left\{ \begin{bmatrix} 1 & 0 & 0 \\ 0 & 1 & 0 \\ 0 & 0 & 1 \end{bmatrix} \right\}.$$

$$\mathrm{Der}(\mathfrak{L}_7^2) = \left\{ \begin{bmatrix} 0 & 0 & 0 \\ 0 & 0 & 0 \\ 0 & 0 & 0 \end{bmatrix} \right\}.$$

### 9.3.1 $\psi_{\alpha,\beta}(\mathfrak{L}_7^2)$

In this case, we have $\psi_{\alpha,\beta}(\mathfrak{L}_7^2) = \mathfrak{t}_{\alpha,\beta}$:

$$\mathfrak{t}_{\alpha,\beta} = \{[e1, e2] = (i\alpha - i\beta)\, e1 + e3, [e1, e3] = (-\alpha + \beta)\, e1 - e2, [e2, e3] = e1 + (\alpha - \beta)\, e2 + (i\alpha - i\beta)\, e3.$$

is a Lie algebra isomorphic to $\mathfrak{sl}(2, \mathbb{C})$. In fact, by considering the matrix of the Killing form of $\mathfrak{t}_{\alpha,\beta}$ in the basis $\{e_1, e_2, e_3\}$

$$\begin{bmatrix} -2 & 2\alpha - 2\beta & 2i\alpha - 2i\beta \\ 2\alpha - 2\beta & -2(\alpha-\beta)^2 - 2 & -2i(\alpha-\beta)^2 \\ 2i\alpha - 2i\beta & -2i(\alpha-\beta)^2 & 2(\alpha-\beta)^2 - 2 \end{bmatrix}$$

we note that it is not degenerate and it follows the conclusion.

### 9.3.2 $\phi_\beta(\mathfrak{L}_7^2)$

In this case, we have $\phi_\beta(\mathfrak{L}_7^2) = \widetilde{\mathfrak{t}}_\beta$:

$$\widetilde{\mathfrak{t}}_\beta = \{[e1, e2] = (i - i\beta)\, e1, [e1, e3] = (-1 + \beta)\, e1, [e2, e3] = (1 - \beta)\, e2 + (i - i\beta)\, e3$$

is a unimodular Lie algebra.

We note that $\widetilde{\mathfrak{t}}_\beta$ is 3-dimensional abelian Lie algebra if and only if $\beta = 1$. In other case, $\widetilde{\mathfrak{t}}_\beta$ is isomorphic to $\mathfrak{r}_{3,-1}$ In fact, we can consider the change of basis given by

$$g = \begin{bmatrix} 0 & -i(-1+\beta) & -1+\beta \\ 0 & 1 & 0 \\ 1 & 0 & 0 \end{bmatrix}$$

we have $g \cdot \widetilde{\mathfrak{t}}_\beta = \mathfrak{r}_{3,-1}$.

### 9.3.3 $\rho(\mathfrak{L}_7^2)$

In this case, we have $\rho(\mathfrak{L}_7^2) = \widehat{\mathfrak{t}}$:

$$\widehat{\mathfrak{t}} = \{[e1, e2] = -ie1, [e1, e3] = e1, [e2, e3] = -e2 - ie3.$$

which is isomorphic to $\mathfrak{r}_{3,-1}$. In fact, by considering the change of basis given by

$$g := \begin{bmatrix} 0 & -i & 1 \\ 0 & 1 & 0 \\ 1 & 0 & 0 \end{bmatrix}$$

we have $g \cdot \widehat{\mathfrak{t}} = \mathfrak{r}_{3,-1}$.

### 9.3.4 Identities

- $A[\cdot,\cdot] + [A\cdot,\cdot] + [\cdot,A\cdot] = 0$ (Generalized Derivation)
- $A^2[\cdot,\cdot] = [A\cdot,A\cdot]$

The hom-Lie algebra $\mathfrak{L}_7^2$ is a non-multiplicative hom-Lie algebra.

## 9.4 Degenerations between hom-Lie algebras ($\mathfrak{L}_7^i$)

If $\mathfrak{L}_i^j \xrightarrow{\deg} \mathfrak{L}_h^k$ then $\operatorname{Der}(\mathfrak{L}_i^j) \leq \operatorname{Der}(\mathfrak{L}_h^k)$. Therefore, we can organize the hom-Lie algebras in the following way:

| Dim(Der) | hom-Lie algebra |
|---|---|
| 0 | $\mathfrak{L}_7^2$ |
| 1 | $\mathfrak{L}_7^1$ |
| 3 | $\mathfrak{L}_7^0$ |

### 9.4.1 Degenerations

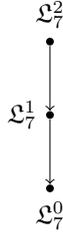

1. $\mathfrak{L}_7^2 \xrightarrow{\deg} \mathfrak{L}_7^1$
   In fact, set
   $$g(t) = \frac{1}{8t^3}\begin{bmatrix} 4t & -8t^3 & 4\sqrt{-1}t \\ 4t^4+4t^2-1 & 4t^2 & -\sqrt{-1}\left(4t^4-4t^2+1\right) \\ \sqrt{-1}\left(-1+4t^4-4t^2\right) & 4\sqrt{-1}t^2 & 1+4t^2+4t^4 \end{bmatrix}, \text{ with } t \in \mathbb{R}.$$
   whose matrix inverse is $g(t)^{\mathrm{T}}$. We have $g(t) \in \operatorname{Aut}(\mathfrak{L}_7)$ and $g(t) \cdot \mathfrak{L}_7^2$ is the hom-Lie algebra $(\mathfrak{L}_7, A(t))$ with
   $$A(t) = \begin{bmatrix} 0 & -t & -\sqrt{-1}t \\ -t & 1 & \sqrt{-1} \\ -\sqrt{-1}t & \sqrt{-1} & -1 \end{bmatrix}$$
   It is easy to check that $(\mathfrak{L}_7, A(t)) \to \mathfrak{L}_7^1$ as $t$ tends to zero.

2. $\mathfrak{L}_7^1 \xrightarrow{\deg} \mathfrak{L}_7^0$
   In fact, set
   $$g(t) = \frac{1}{2t}\begin{bmatrix} 2t & 0 & 0 \\ 0 & -1-t^2 & \sqrt{-1}\left(-1+t^2\right) \\ 0 & -\sqrt{-1}\left(-1+t^2\right) & -1-t^2 \end{bmatrix}, \text{ with } t \in \mathbb{R}.$$
   We have $g(t) \in \operatorname{Aut}(\mathfrak{L}_7)$ and $g(t) \cdot \mathfrak{L}_7^2$ is the hom-Lie algebra $(\mathfrak{L}_7, A(t))$ with
   $$A(t) = \begin{bmatrix} 0 & 0 & 0 \\ 0 & t^2 & \sqrt{-1}t^2 \\ 0 & \sqrt{-1}t^2 & -t^2 \end{bmatrix}$$

It is easy to check that $(\mathfrak{L}_7, A(t)) \to \mathfrak{L}_7^0$ as $t$ tends to zero.

## 9.5 Degenerations between hom-Lie algebras ($\mathfrak{L}^i_j$) with $\mathfrak{L}_j$ being an unimodular Lie algebra

If $\mathfrak{L}^j_i \xrightarrow{\deg} \mathfrak{L}^k_h$ then $\mathfrak{L}_i \xrightarrow{\deg} \mathfrak{L}_h$ and $\mathrm{Der}(\mathfrak{L}^j_i) \leq \mathrm{Der}(\mathfrak{L}^k_h)$.

Therefore, we can organize the hom-Lie algebras in the following way:

| Dim(Der) | | | hom-Lie algebra | | |
|---|---|---|---|---|---|
| 0 | $\mathfrak{L}^2_7$ | | | $\mathfrak{L}^4_1$ | |
| 1 | $\mathfrak{L}^1_7$ | | $\mathfrak{L}^6_4(\lambda), \mathfrak{L}^5_4$ | $\mathfrak{L}^6_1$ | |
| 2 | | | $\mathfrak{L}^4_4(\lambda), \mathfrak{L}^3_4, \mathfrak{L}^2_4$ | $\mathfrak{L}^3_1, \mathfrak{L}^5_1$ | |
| 3 | $\mathfrak{L}^0_7$ | | $\mathfrak{L}^1_4$ | $\mathfrak{L}^2_1$ | $\mathfrak{L}^2_0$ |
| 4 | | | $\mathfrak{L}^0_4$ | $\mathfrak{L}^1_1$ | |
| 5 | | | | | $\mathfrak{L}^1_0$ |
| 6 | | | | $\mathfrak{L}^0_1$ | |
| 9 | | | | | $\mathfrak{L}^0_0$ |

1. $\mathfrak{L}^2_7 \xrightarrow{\deg} \mathfrak{L}^5_4$

   In fact, set
   $$g(t) = \begin{bmatrix} 0 & e^t & \sqrt{-1}e^t \\ e^t & \frac{1}{2}\sqrt{-1}e^{2t} & -\frac{1}{2}e^{2t} \\ \frac{1}{2}\sqrt{-1}e^{2t} & \frac{1}{8}e^t(4+e^{2t}) & \frac{1}{8}\sqrt{-1}e^t(-4+e^{2t}) \end{bmatrix}, \text{ with } t \in \mathbb{R}.$$

   We have $g(t) \cdot \mathfrak{L}^2_7$ is the hom-Lie algebra $(\mu(t), A)$ with
   $$\begin{cases} [e_1, e_2] = -\sqrt{-1}e^{-t}e_1 + e_2, [e_1, e_3] = \sqrt{-1}e^{-t}e_2 - e_3, [e_2, e_3] = -\sqrt{-1}e^{-t}e_3, \\ A = \begin{bmatrix} 0 & 0 & 0 \\ 1 & 0 & 0 \\ 0 & 1 & 0 \end{bmatrix} \end{cases}$$

   It is easy to check that $(\mu(t), A) \to \mathfrak{L}^5_4$ as $t$ tends to infinity.

2. $\mathfrak{L}^1_7 \xrightarrow{\deg} \mathfrak{L}^3_4$

   In fact, set
   $$g(t) = \begin{bmatrix} -\sqrt{-1} & 0 & 0 \\ 0 & -\frac{1}{2}\sqrt{-1}e^t & -\frac{1}{2}e^t \\ 0 & -\sqrt{-1}e^t & e^t \end{bmatrix}, \text{ with } t \in \mathbb{R}.$$

   We have $g(t) \cdot \mathfrak{L}^2_7$ is the hom-Lie algebra $(\mu(t), A)$ with
   $$\begin{cases} [e_1, e_2] = e_2, [e_1, e_3] = -e_3, [e_2, e_3] = e^{-2t}e_1 \\ A = \begin{bmatrix} 0 & 0 & 0 \\ 0 & 0 & 1 \\ 0 & 0 & 0 \end{bmatrix} \end{cases}$$

   It is easy to check that $(\mu(t), A) \to \mathfrak{L}^3_4$ as $t$ tends to infinity.

3. $\mathfrak{L}^1_7 \xrightarrow{\deg} \mathfrak{L}^2_4$

   In fact, set
   $$g(t) = \begin{bmatrix} 0 & e^t & \sqrt{-1}e^t \\ e^t & 0 & -\sqrt{-1}e^t \\ 2\sqrt{-1}e^{2t}+e^t & 2e^{3t}-2\sqrt{-1}e^{2t} & 2\sqrt{-1}e^{3t}+2e^{2t}-\sqrt{-1}e^t \end{bmatrix}, \text{ with } t \in \mathbb{R}.$$

We have $g(t) \cdot \mathfrak{L}_7^1$ is the hom-Lie algebra $(\mu(t), A)$ with

$$\begin{cases} [e_1, e_2] = \frac{1}{2}\mathrm{e}^{-2t}e_1 + e_2, [e_1, e_3] = -\frac{1}{2}\mathrm{e}^{-2t}e_1 - e_3, [e_2, e_3] = \frac{1}{2}\mathrm{e}^{-2t}e_2 + \frac{1}{2}\mathrm{e}^{-2t}e_3 \\ A = \begin{bmatrix} 0 & 0 & 0 \\ 1 & 0 & 0 \\ 1 & 0 & 0 \end{bmatrix} \end{cases}$$

It is easy to check that $(\mu(t), A) \to \mathfrak{L}_4^2$ as $t$ tends to infinity.

4. $\mathfrak{L}_7^0 \xrightarrow{\deg} \mathfrak{L}_4^0$

   In fact, set

   $$g(t) = \begin{bmatrix} 1 & 0 & 0 \\ 0 & \mathrm{e}^t & 0 \\ 0 & 0 & \mathrm{e}^t \end{bmatrix}, \text{ with } t \in \mathbb{R}.$$

   We have $g(t) \cdot \mathfrak{L}_7^0$ is the hom-Lie algebra $(\mu(t), A)$ with

   $$\begin{cases} [e_1, e_2] = e_3, [e_1, e_3] = -e_2, [e_2, e_3] = \mathrm{e}^{-2t}e_1 \\ A = 0 \end{cases}$$

   It is easy to check that $(\mu(t), A) \to \mathfrak{L}_4^0$ as $t$ tends to infinity.

5. $\mathfrak{L}_4^6(\lambda) \xrightarrow{\deg} \mathfrak{L}_1^5$

   In fact, set

   $$g(t) = \begin{bmatrix} -\mathrm{e}^{2t} & 0 & 0 \\ -\frac{1}{2}\frac{\mathrm{e}^{3t}\sqrt{2}}{\sqrt{\lambda}} - \frac{\mathrm{e}^{2t}}{\lambda} & -\mathrm{e}^{2t} & -\mathrm{e}^{2t} \\ 0 & -\frac{1}{2}\frac{\mathrm{e}^{3t}\sqrt{2}}{\sqrt{\lambda}} - \frac{\mathrm{e}^{2t}}{\lambda} & -\frac{1}{2}\frac{\sqrt{2}\mathrm{e}^{3t}}{\sqrt{\lambda}} \end{bmatrix}, \text{ with } t \in \mathbb{R}.$$

   We have $g(t) \cdot \mathfrak{L}_4^6$ is the hom-Lie algebra $(\mu(t), A)$ with

   $$\begin{cases} [e_1, e_2] = \left(\frac{\sqrt{2\lambda}}{\mathrm{e}^t} + \mathrm{e}^{-2t}\right)e_2 + \left(1 + \frac{\sqrt{2}}{\mathrm{e}^t\sqrt{\lambda}}\right)e_3, [e_1, e_3] = -2\mathrm{e}^{-2t}\lambda e_2 - \left(\frac{\sqrt{2\lambda}}{\mathrm{e}^t} + \mathrm{e}^{-2t}\right)e_3 \\ A = \begin{bmatrix} 0 & 0 & 0 \\ 1 & 0 & 0 \\ 0 & 1 & 0 \end{bmatrix} \end{cases}$$

   It is easy to check that $(\mu(t), A) \to \mathfrak{L}_1^5$ as $t$ tends to infinity.

6. $\mathfrak{L}_4^5 \xrightarrow{\deg} \mathfrak{L}_1^5$

   In fact, set

   $$g(t) = \begin{bmatrix} \mathrm{e}^t & 0 & 0 \\ \frac{1}{2}\mathrm{e}^{2t} & \mathrm{e}^t & 0 \\ 0 & \frac{1}{2}\mathrm{e}^{2t} & \mathrm{e}^t \end{bmatrix}, \text{ with } t \in \mathbb{R}.$$

   We have $g(t) \cdot \mathfrak{L}_4^5$ is the hom-Lie algebra $(\mu(t), A)$ with

   $$\begin{cases} [e_1, e_2] = \mathrm{e}^{-t}e_2 + e_3, [e_1, e_3] = -\mathrm{e}^{-t}e_3 \\ A = \begin{bmatrix} 0 & 0 & 0 \\ 1 & 0 & 0 \\ 0 & 1 & 0 \end{bmatrix} \end{cases}$$

   It is easy to check that $(\mu(t), A) \to \mathfrak{L}_1^5$ as $t$ tends to infinity.

7. $\mathfrak{L}_4^4(\lambda) \xrightarrow{\deg} \mathfrak{L}_1^2$

   In fact, set
   $$g(t) = \begin{bmatrix} 0 & \frac{1}{2}\lambda^{-1} & -\frac{1}{2}\lambda^{-1} \\ 0 & 1 & 1 \\ 2\lambda e^{2t} & e^t & e^t \end{bmatrix}, \text{ with } t \in \mathbb{R}.$$

   We have $g(t) \cdot \mathfrak{L}_4^4(\lambda)$ is the hom-Lie algebra $(\mu(t), A)$ with
   $$\begin{cases} [e_1, e_2] = e^{-t}e_2 + e_3, [e_1, e_3] = -e^{-2t}e_2 - e^{-t}e_3, [e_2, e_3] = -\frac{e^{-2t}}{4\lambda^2}e_1 \\ A = \begin{bmatrix} 0 & 1 & 0 \\ 0 & 0 & 0 \\ 0 & 0 & 0 \end{bmatrix} \end{cases}$$

   It is easy to check that $(\mu(t), A) \to \mathfrak{L}_1^2$ as $t$ tends to infinity.

8. $\mathfrak{L}_4^2 \xrightarrow{\deg} \mathfrak{L}_1^2$

   In fact, set
   $$g(t) = \begin{bmatrix} 0 & e^t & e^t \\ 2e^t & 0 & 0 \\ 0 & -2e^{2t} & 2e^{2t} \end{bmatrix}, \text{ with } t \in \mathbb{R}.$$

   We have $g(t) \cdot \mathfrak{L}_4^2$ is the hom-Lie algebra $(\mu(t), A)$ with
   $$\begin{cases} [e_1, e_2] = e_3, [e_2, e_3] = -\frac{1}{4}e^{-2t}e_1 \\ A = \begin{bmatrix} 0 & 1 & 0 \\ 0 & 0 & 0 \\ 0 & 0 & 0 \end{bmatrix} \end{cases}$$

   It is easy to check that $(\mu(t), A) \to \mathfrak{L}_1^2$ as $t$ tends to infinity.

9. $\mathfrak{L}_4^1 \xrightarrow{\deg} \mathfrak{L}_1^1$

   In fact, set
   $$g(t) = \begin{bmatrix} 0 & 0 & 1 \\ e^t & 0 & 0 \\ 0 & e^t & \frac{1}{2}e^t \end{bmatrix}, \text{ with } t \in \mathbb{R}.$$

   We have $g(t) \cdot \mathfrak{L}_4^1$ is the hom-Lie algebra $(\mu(t), A)$ with
   $$\begin{cases} [e1, e2] = e^{-t}e1 + e3, [e2, e3] = e^{-t}e3 \\ A = \begin{bmatrix} 0 & 0 & 0 \\ 0 & 0 & 0 \\ 0 & 1 & 0 \end{bmatrix} \end{cases}$$

   It is easy to check that $(\mu(t), A) \to \mathfrak{L}_1^1$ as $t$ tends to infinity.

10. $\mathfrak{L}_4^0 \xrightarrow{\deg} \mathfrak{L}_1^0$

    In fact, set
    $$g(t) = \begin{bmatrix} e^t & 0 & 0 \\ 0 & 1 & 0 \\ 0 & \frac{1}{2}e^t & 1 \end{bmatrix}, \text{ with } t \in \mathbb{R}.$$

We have $g(t) \cdot \mathfrak{L}_4^0$ is the hom-Lie algebra $(\mu(t), A)$ with

$$\begin{cases} [e_1, e_2] = \mathrm{e}^{-t} e_2 + e_3, [e_1, e_3] = -\mathrm{e}^{-t} e_3 \\ A = 0. \end{cases}$$

It is easy to check that $(\mu(t), A) \to \mathfrak{L}_1^0$ as $t$ tends to infinity.

11. $\mathfrak{L}_1^5 \xrightarrow{\deg} \mathfrak{L}_0^2$

In fact, set

$$g(t) = \begin{bmatrix} 0 & 0 & \mathrm{e}^t \\ 0 & \mathrm{e}^t & 0 \\ \mathrm{e}^t & 0 & 0 \end{bmatrix}, \text{ with } t \in \mathbb{R}.$$

We have $g(t) \cdot \mathfrak{L}_1^5$ is the hom-Lie algebra $(\mu(t), A)$ with

$$\begin{cases} [e_2, e_3] = -\mathrm{e}^{-t} e_1 \\ A = \begin{bmatrix} 0 & 1 & 0 \\ 0 & 0 & 1 \\ 0 & 0 & 0 \end{bmatrix} \end{cases}$$

It is easy to check that $(\mu(t), A) \to \mathfrak{L}_0^2$ as $t$ tends to infinity.

12. $\mathfrak{L}_1^1 \xrightarrow{\deg} \mathfrak{L}_0^1$

In fact, set

$$g(t) = \begin{bmatrix} \mathrm{e}^t & 0 & 0 \\ 0 & 0 & 1 \\ 0 & 1 & 0 \end{bmatrix}, \text{ with } t \in \mathbb{R}.$$

We have $g(t) \cdot \mathfrak{L}_1^1$ is the hom-Lie algebra $(\mu(t), A)$ with

$$\begin{cases} [e_1, e_3] = \mathrm{e}^{-t} e_2 \\ A = \begin{bmatrix} 0 & 0 & 0 \\ 0 & 0 & 1 \\ 0 & 0 & 0 \end{bmatrix} \end{cases}$$

It is easy to check that $(\mu(t), A) \to \mathfrak{L}_0^1$ as $t$ tends to infinity.

13. $\mathfrak{L}_1^0 \xrightarrow{\deg} \mathfrak{L}_0^0$

In fact, set

$$g(t) = \begin{bmatrix} 1 & 0 & 0 \\ 0 & \mathrm{e}^t & 0 \\ 0 & 0 & 1 \end{bmatrix}, \text{ with } t \in \mathbb{R}.$$

We have $g(t) \cdot \mathfrak{L}_1^0$ is the hom-Lie algebra $(\mu(t), A)$ with

$$\begin{cases} [e_1, e_2] = \mathrm{e}^{-t} e_3 \\ A = 0 \end{cases}$$

It is easy to check that $(\mu(t), A) \to \mathfrak{L}_0^0$ as $t$ tends to infinity.

14. $\mathfrak{L}_7^2 \xrightarrow{\deg}\!\!\!\!\!/\ \mathfrak{L}_4^4(\lambda)$ with $(\lambda \neq 0)$.

Suppose, contrary to our claim, that $\mathfrak{L}_7^2 \xrightarrow{\deg} \mathfrak{L}_4^4(\lambda)$ with $(\lambda \neq 0)$. In such case, we have

$$\phi_\beta(\mathfrak{L}_7^2) \xrightarrow{\deg} \phi_\beta(\mathfrak{L}_4^4(\lambda)).$$

By taking $\beta = 1$, we have from 9.3.2 that $\phi_\beta(\mathfrak{L}_7^2)$ is the 3-dimensional Lie abelian Lie algebra therefore, $\phi_\beta(\mathfrak{L}_4^4(\lambda))$ is the Lie algebra $\mathfrak{a}_3$. But, from 6.5.2, $\phi_\beta(\mathfrak{L}_4^4(\lambda))$ is isomorphic to $\mathfrak{r}_{3,1}$; this is a contradiction.

15. $\mathfrak{L}_7^2 \xrightarrow{\deg} \mathfrak{L}_1^3$.

    Suppose, contrary to our claim, that $\mathfrak{L}_7^2 \xrightarrow{\deg} \mathfrak{L}_1^3$. In such case, we have
    
    $$\phi_\beta(\mathfrak{L}_7^2) \xrightarrow{\deg} \phi_\beta(\mathfrak{L}_1^3).$$
    
    By taking $\beta = 1$, we have from 9.3.2 that $\phi_\beta(\mathfrak{L}_7^2)$ is the 3-dimensional abelian Lie algebra, therefore, $\phi_\beta(\mathfrak{L}_1^3)$ is the Lie algebra $\mathfrak{a}_3$. But, from 3.4.2, $\phi_\beta(\mathfrak{L}_1^3)$ is isomorphic to $\mathfrak{r}_{3,1}$; this is a contradiction.

16. $\mathfrak{L}_7^1 \xrightarrow{\deg} \mathfrak{L}_1^5$.

    Suppose, contrary to our claim, that $\mathfrak{L}_7^1 \xrightarrow{\deg} \mathfrak{L}_1^5$. In such case, we have
    
    $$\begin{bmatrix} 0 & 0 & 0 \\ 0 & 1 & \sqrt{-1} \\ 0 & \sqrt{-1} & -1 \end{bmatrix} \xrightarrow{\deg} \begin{bmatrix} 0 & 0 & 0 \\ 1 & 0 & 0 \\ 0 & 1 & 0 \end{bmatrix}.$$
    
    Since the first matrix has nilpotent degree equal to 2 and the second one has nilpotent degree equal to 3, we have a contradiction.

17. $\mathfrak{L}_7^0 \xrightarrow{\deg} \mathfrak{L}_0^1$.

    Suppose, contrary to our claim, that $\mathfrak{L}_7^0 \xrightarrow{\deg} \mathfrak{L}_0^1$. In such case, we have
    
    $$0 \xrightarrow{\deg} \begin{bmatrix} 0 & 0 & 0 \\ 0 & 0 & 1 \\ 0 & 0 & 0 \end{bmatrix}.$$
    
    Since the first matrix has nilpotent degree equal to 1 and the second one has nilpotent degree equal to 2, we have a contradiction.

18. $\mathfrak{L}_4^6(\lambda) \xrightarrow{\deg} \mathfrak{L}_1^3$ with $(\lambda \neq 0)$.

    Suppose, contrary to our claim, that $\mathfrak{L}_4^6(\lambda) \xrightarrow{\deg} \mathfrak{L}_1^3$ with $(\lambda \neq 0)$. In such case, we have
    
    $$\psi_{\alpha,\beta}(\mathfrak{L}_4^6(\lambda)) \xrightarrow{\deg} \psi_{\alpha,\beta}(\mathfrak{L}_1^3).$$
    
    By taking $\alpha = -\beta$ and $\alpha = \pm \frac{\sqrt{-1}}{2\lambda}$, we have from 6.7.1 that $\psi_{\alpha,\beta}(\mathfrak{L}_4^6(\lambda))$ is a Lie algebra isomorphic to $\mathfrak{n}_3$, therefore, $\psi_{\alpha,\beta}(\mathfrak{L}_1^3)$ is a Lie algebra which is isomorphic to one of the following: $\mathfrak{n}_3$ or $\mathfrak{a}_3$ (by Theorem 1.1). We must have $\psi_{\alpha,\beta}(\mathfrak{L}_1^3)$ is isomorphic to $\mathfrak{n}_3$ (from 3.4.1). But, we now apply again 3.4.1, to obtain $\alpha = \beta = 0$; this is a contradiction.

19. $\mathfrak{L}_4^5 \xrightarrow{\deg} \mathfrak{L}_1^3$.

    Suppose, contrary to our claim, that $\mathfrak{L}_4^5 \xrightarrow{\deg} \mathfrak{L}_1^3$. In such case, we have
    
    $$\phi_\beta(\mathfrak{L}_4^5) \xrightarrow{\deg} \phi_\beta(\mathfrak{L}_1^3).$$
    
    By taking $\beta = 1$, we have from 6.6.2 that $\phi_\beta(\mathfrak{L}_4^5)$ is the 3-dimension abelian Lie algebra, therefore, $\phi_\beta(\mathfrak{L}_1^3)$ is the Lie algebra $\mathfrak{a}_3$. But, from 3.4.2, $\phi_\beta(\mathfrak{L}_1^3)$ is isomorphic to $\mathfrak{r}_{3,1}$; this is a contradiction.

20. $\mathfrak{L}_4^4(\lambda) \xrightarrow{\text{deg}} \mathfrak{L}_0^2$.

Suppose, contrary to our claim, that $\mathfrak{L}_4^4(\lambda)\xrightarrow{\text{deg}}\mathfrak{L}_0^2$. In such case, we have

$$\begin{bmatrix} 0 & 0 & 0 \\ 0 & \lambda & \lambda \\ 0 & -\lambda & -\lambda \end{bmatrix} \xrightarrow{\text{deg}} \begin{bmatrix} 0 & 1 & 0 \\ 0 & 0 & 1 \\ 0 & 0 & 0 \end{bmatrix}$$

Since the first matrix has nilpotent degree equal to 2 and the second one has nilpotent degree equal to 3, we have a contradiction.

21. $\mathfrak{L}_4^3 \xrightarrow{\text{deg}} \mathfrak{L}_1^2$.

Suppose, contrary to our claim, that $\mathfrak{L}_4^3\xrightarrow{\text{deg}}\mathfrak{L}_1^2$. In such case, we have

$$\vartheta_{\alpha_0,\ldots,\alpha_{26}}(\mathfrak{L}_4^3)\xrightarrow{\text{deg}}\vartheta_{\alpha_0,\ldots,\alpha_{26}}(\mathfrak{L}_1^2).$$

If $\{\alpha_0 = 0, \alpha_1 = 0, \alpha_2 = 1, \alpha_3 = 0, \ldots, \alpha_{26} = 0\}$ we have $\vartheta_{\alpha_0,\ldots,\alpha_{26}}(\mathfrak{L}_4^3)$ is the pair

$$(\nu_1, A_1) = \begin{cases} \nu_1 = \{e_3 \cdot e_1 = -e_2 \\ A_1 = \begin{bmatrix} 0 & 0 & 0 \\ 0 & 0 & 1 \\ 0 & 0 & 0 \end{bmatrix} \end{cases}$$

and $\vartheta_{\alpha_0,\ldots,\alpha_{26}}(\mathfrak{L}_1^2)$ is the pair

$$(\nu_2, A_2) = \begin{cases} \nu_2 = \{e_2 \cdot e_2 = -e_3 \\ A_2 = \begin{bmatrix} 0 & 1 & 0 \\ 0 & 0 & 0 \\ 0 & 0 & 0 \end{bmatrix} \end{cases}$$

Therefore

$$\dim(\text{Ker}(T_{\theta_1,\ldots,\theta_6;\beta_1,\ldots,\beta_9}(\nu_1, A_1))) \leq \dim(\text{Ker}(T_{\theta_1,\ldots,\theta_6;\beta_1,\ldots,\beta_9}(\nu_2, A_2))). \tag{2}$$

Now, we consider $\{\beta_1 = 1, \beta_2 = -1, \beta_3 = 0, \beta_4 = 1, \beta_5 = -1, \beta_6 = 0, \beta_7 = 1, \beta_8 = -1, \beta_9 = 0,\}$ and $\{\theta_1 = 0, \theta_2 = 1, \theta_3 = 0, \theta_4 = 0, \theta_5 = 0, \theta_6 = 1\}$.

The set of 3-by-3 matrices that commute with $A_1$ is

$$\mathfrak{gl}(n,\mathbb{C})_{A_1} = \left\{ \begin{bmatrix} t_1 & 0 & t_4 \\ t_3 & t_2 & t_5 \\ 0 & 0 & t_2 \end{bmatrix} : t_1, t_2, t_3, t_4, t_5 \in \mathbb{C} \right\}.$$

Given $D_1(w's)$, $D_2(x's)$ and $D_3(y's)$ in $\mathfrak{gl}(n,\mathbb{C})_{A_1}$, with $D_2$ and $D_3$ equal to zero, we have $\lambda_1 := D_1\nu(\cdot,\cdot) + \nu(D_2\cdot,\cdot) + \nu(\cdot,D_3\cdot)$,

$$\lambda_1 := \{\ e_3 \cdot e_1 = (-w_{3,3})\, e_2,$$

and so $\dim(\text{Ker}(T_{\theta_1,\ldots,\theta_6;\beta_1,\ldots,\beta_9}(\nu_1, A_1))) = 4$.

On the other hand, the set of 3-by-3 matrices that commute with $A_2$ is

$$\mathfrak{gl}(n,\mathbb{C})_{A_2} = \left\{ \begin{bmatrix} t_1 & t_3 & t_4 \\ 0 & t_1 & 0 \\ 0 & t_5 & t_2 \end{bmatrix} : t_1, t_2, t_3, t_4, t_5 \in \mathbb{C} \right\}.$$

and given $D_1(w's)$, $D_2(x's)$ and $D_3(y's)$ in $\mathfrak{gl}(n,\mathbb{C})_{A_2}$, with $D_2$ and $D_3$ equal to zero, we have $\lambda_2 := D_1 \nu(\cdot, \cdot) + \nu(D_2 \cdot, \cdot) + \nu(\cdot, D_3 \cdot)$,

$$\lambda_2 := \{ \ [e_2, e_2] = w_{1,3} e_1 + (w_{3,3}) \, e_3$$

and so $\dim(\mathrm{Ker}(T_{\theta_1,\ldots,\theta_6;\beta_1,\ldots,\beta_9}(\nu_2, A_2))) = 3$; this is a contradiction.

22. $\mathfrak{L}_4^3 \xrightarrow{\deg} \mathfrak{L}_0^2$.

Suppose, contrary to our claim, that $\mathfrak{L}_4^3 \xrightarrow{\deg} \mathfrak{L}_0^2$. In such case, we have

$$\begin{bmatrix} 0 & 0 & 0 \\ 0 & 0 & 1 \\ 0 & 0 & 0 \end{bmatrix} \xrightarrow{\deg} \begin{bmatrix} 0 & 1 & 0 \\ 0 & 0 & 1 \\ 0 & 0 & 0 \end{bmatrix}$$

Since the first matrix has nilpotent degree equal to 2 and the second one has nilpotent degree equal to 3, we have a contradiction.

23. $\mathfrak{L}_4^2 \xrightarrow{\deg} \mathfrak{L}_0^2$.

Suppose, contrary to our claim, that $\mathfrak{L}_4^2 \xrightarrow{\deg} \mathfrak{L}_0^2$. In such case, we have

$$\begin{bmatrix} 0 & 0 & 0 \\ 1 & 0 & 0 \\ 1 & 0 & 0 \end{bmatrix} \xrightarrow{\deg} \begin{bmatrix} 0 & 1 & 0 \\ 0 & 0 & 1 \\ 0 & 0 & 0 \end{bmatrix}$$

Since the first matrix has nilpotent degree equal to 2 and the second one has nilpotent degree equal to 3, we have a contradiction.

24. $\mathfrak{L}_4^0 \xrightarrow{\deg} \mathfrak{L}_0^1$.

Suppose, contrary to our claim, that $\mathfrak{L}_4^0 \xrightarrow{\deg} \mathfrak{L}_0^1$. In such case, we have

$$0 \xrightarrow{\deg} \begin{bmatrix} 0 & 0 & 0 \\ 0 & 0 & 1 \\ 0 & 0 & 0 \end{bmatrix}$$

Since the first matrix has nilpotent degree equal to 1 and the second one has nilpotent degree equal to 2, we have a contradiction.

25. $\mathfrak{L}_1^3 \xrightarrow{\deg} \mathfrak{L}_0^2$.

Suppose, contrary to our claim, that $\mathfrak{L}_1^3 \xrightarrow{\deg} \mathfrak{L}_0^2$. In such case, we have

$$\begin{bmatrix} 0 & 0 & 0 \\ 0 & 0 & 1 \\ 0 & 0 & 0 \end{bmatrix} \xrightarrow{\deg} \begin{bmatrix} 0 & 1 & 0 \\ 0 & 0 & 1 \\ 0 & 0 & 0 \end{bmatrix}$$

Since the first matrix has nilpotent degree equal to 2 and the second one has nilpotent degree equal to 3, we have a contradiction.

## 9.6 Degenerations between hom-Lie algebras ($\mathfrak{L}_j^i$) with $j = 6, 1, 0$

If $\mathfrak{L}_i^j \xrightarrow{\text{deg}} \mathfrak{L}_h^k$ then $\mathfrak{L}_i \xrightarrow{\text{deg}} \mathfrak{L}_h$ and $\text{Der}(\mathfrak{L}_i^j) \le \text{Der}(\mathfrak{L}_h^k)$.

Therefore, we can organize the hom-Lie algebras in the following way:

| Dim(Der) | hom-Lie algebra | | | |
|---|---|---|---|---|
| 0 | $\mathfrak{L}_6^{13}(\lambda), \mathfrak{L}_6^{12}, \mathfrak{L}_6^{11}$ | $\mathfrak{L}_1^4$ | | |
| 1 | $\mathfrak{L}_6^9(\lambda), \mathfrak{L}_6^{10}, \mathfrak{L}_6^8, \mathfrak{L}_6^7$ | $\mathfrak{L}_1^6$ | | |
| 2 | $\mathfrak{L}_6^6(\lambda), \mathfrak{L}_6^5, \mathfrak{L}_6^4, \mathfrak{L}_6^3$ | $\mathfrak{L}_1^5, \mathfrak{L}_1^3$ | | |
| 3 | $\mathfrak{L}_6^2, \mathfrak{L}_6^1$ | $\mathfrak{L}_1^2$ | | $\mathfrak{L}_0^2$ |
| 4 | $\mathfrak{L}_6^0$ | $\mathfrak{L}_1^1$ | | |
| 5 | | | | $\mathfrak{L}_0^1$ |
| 6 | | $\mathfrak{L}_1^0$ | | |
| 9 | | | | $\mathfrak{L}_0^0$ |

### 9.6.1 Degenerations

1. $\mathfrak{L}_6^9(\lambda) \xrightarrow{\text{deg}} \mathfrak{L}_1^5$

   In fact, set

   $$g(t) = \begin{bmatrix} e^{2t} & 0 & 0 \\ \frac{e^{2t}(\sqrt{-1}e^t\lambda + \sqrt{\lambda})}{\lambda^{3/2}} & e^{2t} & e^{2t} \\ 0 & \frac{e^{2t}(\sqrt{-1}e^t\lambda + \sqrt{\lambda})}{\lambda^{3/2}} & \frac{\sqrt{-1}e^{3t}}{\sqrt{\lambda}} \end{bmatrix}, \text{ with } t \in \mathbb{R}.$$

   We have $g(t) \cdot \mathfrak{L}_6^9(\lambda)$ is the hom-Lie algebra $(\mu(t), A)$ with

   $$\begin{cases} [e_1, e_2] = -\sqrt{-1}\sqrt{\lambda}e^{-t}e_2 + \left(1 - \frac{\sqrt{-1}}{e^t\sqrt{\lambda}}\right)e_3 \\ [e_1, e_3] = e^{-2t}\lambda\, e_2 + \left(\frac{\sqrt{-1}\sqrt{\lambda}}{e^t} + e^{-2t}\right)e_3 \\ A = \begin{bmatrix} 0 & 0 & 0 \\ 1 & 0 & 0 \\ 0 & 1 & 0 \end{bmatrix} \end{cases}$$

   It is easy to check that $(\mu(t), A) \to \mathfrak{L}_1^5$ as $t$ tends to infinity.

2. $\mathfrak{L}_6^8 \xrightarrow{\text{deg}} \mathfrak{L}_1^5$

   In fact, set

   $$g(t) = \begin{bmatrix} e^t & 0 & 0 \\ -e^{2t} & 0 & e^t \\ 0 & e^t & -e^{2t} \end{bmatrix}, \text{ with } t \in \mathbb{R}.$$

   We have $g(t) \cdot \mathfrak{L}_6^8$ is the hom-Lie algebra $(\mu(t), A)$ with

   $$\begin{cases} [e_1, e_2] = e_3, [e_1, e_3] = e^{-t}e_3 \\ A = \begin{bmatrix} 0 & 0 & 0 \\ 1 & 0 & 0 \\ 0 & 1 & 0 \end{bmatrix} \end{cases}$$

   It is easy to check that $(\mu(t), A) \to \mathfrak{L}_1^5$ as $t$ tends to infinity.

3. $\mathfrak{L}_6^7 \xrightarrow{\deg} \mathfrak{L}_1^5$

   In fact, set
   $$g(t) = \begin{bmatrix} e^t & 0 & 0 \\ e^{2t} & e^t & 0 \\ 0 & e^{2t} & e^t \end{bmatrix}, \text{ with } t \in \mathbb{R}.$$

   We have $g(t) \cdot \mathfrak{L}_6^7$ is the hom-Lie algebra $(\mu(t), A)$ with
   $$\begin{cases} [e_1, e_2] = e^{-t} e_2 + e_3 \\ A = \begin{bmatrix} 0 & 0 & 0 \\ 1 & 0 & 0 \\ 0 & 1 & 0 \end{bmatrix} \end{cases}$$

   It is easy to check that $(\mu(t), A) \to \mathfrak{L}_1^5$ as $t$ tends to infinity.

4. $\mathfrak{L}_6^6(\lambda) \xrightarrow{\deg} \mathfrak{L}_1^2$

   In fact, set
   $$g(t) = \begin{bmatrix} 0 & e^{-t} & 0 \\ 0 & \lambda e^{-t} & \lambda e^{-t} \\ \frac{e^{2t}}{\lambda} & 1 & 1 \end{bmatrix}, \text{ with } t \in \mathbb{R}.$$

   We have $g(t) \cdot \mathfrak{L}_6^6(\lambda)$ is the hom-Lie algebra $(\mu(t), A)$ with
   $$\begin{cases} [e_1, e_2] = e^{-t} e_1 + \lambda e^{-t} e_2 + e_3, [e_1, e_3] = -\lambda e^{-2t} e_1 - \lambda^2 e^{-2t} e_2 - \lambda e^{-t} e_3 \\ A = \begin{bmatrix} 0 & 1 & 0 \\ 0 & 0 & 0 \\ 0 & 0 & 0 \end{bmatrix} \end{cases}$$

   It is easy to check that $(\mu(t), A) \to \mathfrak{L}_1^2$ as $t$ tends to infinity.

5. $\mathfrak{L}_6^3 \xrightarrow{\deg} \mathfrak{L}_1^2$

   In fact, set
   $$g(t) = \begin{bmatrix} 0 & 0 & e^t \\ e^t & 0 & 0 \\ 0 & -(e^t)^2 & (e^t)^2 \end{bmatrix}, \text{ with } t \in \mathbb{R}.$$

   We have $g(t) \cdot \mathfrak{L}_6^3$ is the hom-Lie algebra $(\mu(t), A)$ with
   $$\begin{cases} [e_1, e_2] = e_3, [e_2, e_3] = e^{-t} e_3 \\ A = \begin{bmatrix} 0 & 1 & 0 \\ 0 & 0 & 0 \\ 0 & 0 & 0 \end{bmatrix} \end{cases}$$

   It is easy to check that $(\mu(t), A) \to \mathfrak{L}_1^2$ as $t$ tends to infinity.

6. $\mathfrak{L}_6^2 \xrightarrow{\deg} \mathfrak{L}_1^1$

   In fact, set
   $$g(t) = \begin{bmatrix} 0 & 1 & 0 \\ e^t & 0 & 0 \\ 0 & -e^t & e^t \end{bmatrix}, \text{ with } t \in \mathbb{R}.$$

   103

We have $g(t) \cdot \mathfrak{L}_6^2$ is the hom-Lie algebra $(\mu(t), A)$ with

$$\begin{cases} [e_1, e_2] = -\mathrm{e}^{-t}e_1 + e_3 \\ A = \begin{bmatrix} 0 & 0 & 0 \\ 0 & 0 & 0 \\ 0 & 1 & 0 \end{bmatrix} \end{cases}$$

It is easy to check that $(\mu(t), A) \to \mathfrak{L}_1^1$ as $t$ tends to infinity.

7. $\mathfrak{L}_6^1 \xrightarrow{\deg} \mathfrak{L}_1^1$

   In fact, set

   $$g(t) = \begin{bmatrix} 0 & 0 & 1 \\ \mathrm{e}^t & 0 & 0 \\ 0 & \mathrm{e}^t & \mathrm{e}^t \end{bmatrix}, \text{ with } t \in \mathbb{R}.$$

   We have $g(t) \cdot \mathfrak{L}_6^1$ is the hom-Lie algebra $(\mu(t), A)$ with

   $$\begin{cases} [e_1, e_2] = e_3, [e_2, e_3] = \mathrm{e}^{-t}e_3 \\ A = \begin{bmatrix} 0 & 0 & 0 \\ 0 & 0 & 0 \\ 0 & 1 & 0 \end{bmatrix} \end{cases}$$

   It is easy to check that $(\mu(t), A) \to \mathfrak{L}_1^1$ as $t$ tends to infinity.

8. $\mathfrak{L}_6^0 \xrightarrow{\deg} \mathfrak{L}_1^0$

   In fact, set

   $$g(t) = \begin{bmatrix} \mathrm{e}^t & 0 & 0 \\ 0 & 1 & 0 \\ 0 & \mathrm{e}^t & 1 \end{bmatrix}, \text{ with } t \in \mathbb{R}.$$

   We have $g(t) \cdot \mathfrak{L}_6^0$ is the hom-Lie algebra $(\mu(t), A)$ with

   $$\begin{cases} [e_1, e_2] = \mathrm{e}^{-t}e_2 + e_3 \\ A = 0 \end{cases}$$

   It is easy to check that $(\mu(t), A) \to \mathfrak{L}_1^0$ as $t$ tends to infinity.

9. $\mathfrak{L}_6^{13}(\lambda) \xrightarrow{\deg}\!\!\!\!\!\!\!\!/\ \ \mathfrak{L}_1^3$ with $(\lambda \neq 0)$.

   Suppose, contrary to our claim, that $\mathfrak{L}_6^{13}(\lambda) \xrightarrow{\deg} \mathfrak{L}_1^3$ with $(\lambda \neq 0)$. In such case, we have

   $$\psi_{\alpha,\beta}(\mathfrak{L}_6^{13}(\lambda)) \xrightarrow{\deg} \psi_{\alpha,\beta}(\mathfrak{L}_1^3).$$

   By taking $\alpha = 0$ and $\beta = -1/\lambda$, we have from 8.14.1 that $\psi_{\alpha,\beta}(\mathfrak{L}_6^{13}(\lambda))$ is a Lie algebra isomorphic to $\mathfrak{n}_3$, therefore, $\psi_{\alpha,\beta}(\mathfrak{L}_1^3)$ is a Lie algebra which is isomorphic to one of the following: $\mathfrak{n}_3$ or $\mathfrak{a}_3$ (by Theorem 1.1). We must have $\psi_{\alpha,\beta}(\mathfrak{L}_1^3)$ is isomorphic to $\mathfrak{n}_3$ (from 3.4.1). But, we now apply again 3.4.1, to obtain $\alpha = \beta = 0$; this is a contradiction.

10. $\mathfrak{L}_6^{12} \xrightarrow{\deg}\!\!\!\!\!\!\!\!/\ \ \mathfrak{L}_1^3$.

    Suppose, contrary to our claim, that $\mathfrak{L}_6^{12} \xrightarrow{\deg} \mathfrak{L}_1^3$. In such case, we have

    $$\rho(\mathfrak{L}_6^{12}) \xrightarrow{\deg} \rho(\mathfrak{L}_1^3).$$

    We have from 8.13.3 that $\rho(\mathfrak{L}_6^{12})$ is a Lie algebra isomorphic to $\mathfrak{n}_3$, therefore, $\rho(\mathfrak{L}_1^3)$ is a Lie algebra which is isomorphic to one of the following: $\mathfrak{n}_3$ or $\mathfrak{a}_3$(by Theorem 1.1). But, from 3.4.3, $\rho(\mathfrak{L}_1^3)$ is isomorphic to $\mathfrak{r}_2 \times \mathbb{C}$; this is a contradiction.

11. $\mathfrak{L}_6^{11} \xrightarrow{\text{deg}} \mathfrak{L}_1^3$.

    Suppose, contrary to our claim, that $\mathfrak{L}_6^{11} \xrightarrow{\text{deg}} \mathfrak{L}_1^3$. In such case, we have

    $$\phi_\beta(\mathfrak{L}_6^{11}) \xrightarrow{\text{deg}} \phi_\beta(\mathfrak{L}_1^3).$$

    By taking $\beta = -1$, we have from 8.12.2 that $\phi_\beta(\mathfrak{L}_6^{11})$ is a Lie algebra isomorphic to $\mathfrak{r}_2 \times \mathbb{C}$, therefore, $\phi_\beta(\mathfrak{L}_1^3)$ is a Lie algebra which is isomorphic to one of the following: $\mathfrak{r}_2 \times \mathbb{C}$, $\mathfrak{n}_3$ or $\mathfrak{a}_3$ (by Theorem 1.1). But, from 3.4.2, $\phi_\beta(\mathfrak{L}_1^3)$ is isomorphic to $\mathfrak{r}_{3,-1}$; this is a contradiction.

12. $\mathfrak{L}_6^{10} \xrightarrow{\text{deg}} \mathfrak{L}_1^5$.

    Suppose, contrary to our claim, that $\mathfrak{L}_6^{10} \xrightarrow{\text{deg}} \mathfrak{L}_1^5$. In such case, we have

    $$\begin{bmatrix} 0 & 1 & 0 \\ 0 & 0 & 0 \\ 0 & 0 & 0 \end{bmatrix} \xrightarrow{\text{deg}} \begin{bmatrix} 0 & 0 & 0 \\ 1 & 0 & 0 \\ 0 & 1 & 0 \end{bmatrix}$$

    Since the first matrix has nilpotent degree equal to 2 and the second one has nilpotent degree equal to 3, we have a contradiction.

13. $\mathfrak{L}_6^6(\lambda) \xrightarrow{\text{deg}} \mathfrak{L}_0^2$.

    Suppose, contrary to our claim, that $\mathfrak{L}_6^6(\lambda) \xrightarrow{\text{deg}} \mathfrak{L}_0^2$. In such case, we have

    $$\begin{bmatrix} 0 & 0 & 0 \\ 0 & \lambda & \lambda \\ 0 & -\lambda & -\lambda \end{bmatrix} \xrightarrow{\text{deg}} \begin{bmatrix} 0 & 1 & 0 \\ 0 & 0 & 1 \\ 0 & 0 & 0 \end{bmatrix}$$

    Since the first matrix has nilpotent degree equal to 2 and the second one has nilpotent degree equal to 3, we have a contradiction.

14. $\mathfrak{L}_6^5 \xrightarrow{\text{deg}} \mathfrak{L}_1^2$.

    Suppose, contrary to our claim, that $\mathfrak{L}_6^5 \xrightarrow{\text{deg}} \mathfrak{L}_1^2$. In such case, we have

    $$\varphi_{\alpha_0,\ldots,\alpha_{26}}(\mathfrak{L}_6^5) \xrightarrow{\text{deg}} \varphi_{\alpha_0,\ldots,\alpha_{26}}(\mathfrak{L}_1^2).$$

    If $\{\alpha_0 = 0, \alpha_1 = 0, \alpha_2 = 1, \alpha_3 = 0, \ldots, \alpha_{26} = 0\}$ we have $\varphi_{\alpha_0,\ldots,\alpha_{26}}(\mathfrak{L}_6^5)$ is the Abelian Lie algebra and $\varphi_{\alpha_0,\ldots,\alpha_{26}}(\mathfrak{L}_1^2)$ is the non Lie algebra

    $$\{\ e_2 \cdot e_2 = e_3;$$

    this is a contradiction.

15. $\mathfrak{L}_6^5 \xrightarrow{\text{deg}} \mathfrak{L}_0^2$.

    Suppose, contrary to our claim, that $\mathfrak{L}_6^5 \xrightarrow{\text{deg}} \mathfrak{L}_0^2$. In such case, we have

    $$\begin{bmatrix} 0 & 0 & 0 \\ 0 & 0 & 0 \\ 0 & 1 & 0 \end{bmatrix} \xrightarrow{\text{deg}} \begin{bmatrix} 0 & 1 & 0 \\ 0 & 0 & 1 \\ 0 & 0 & 0 \end{bmatrix}$$

    Since the first matrix has nilpotent degree equal to 2 and the second one has nilpotent degree equal to 3, we have a contradiction.

16. $\mathfrak{L}_6^4 \xrightarrow{\deg} \mathfrak{L}_1^2$.

Suppose, contrary to our claim, that $\mathfrak{L}_6^4 \xrightarrow{\deg} \mathfrak{L}_1^2$. In such case, we have

$$\vartheta_{\alpha_0,\ldots,\alpha_{26}}(\mathfrak{L}_6^4) \xrightarrow{\deg} \vartheta_{\alpha_0,\ldots,\alpha_{26}}(\mathfrak{L}_1^2).$$

If $\{\alpha_0 = 0, \alpha_1 = 0, \alpha_2 = 1, \alpha_3 = 0, \ldots, \alpha_{26} = 0\}$ we have $\vartheta_{\alpha_0,\ldots,\alpha_{26}}(\mathfrak{L}_6^4)$ is the pair

$$(\nu_1, A_1) = \begin{cases} \nu_1 = \{[e_3, e_1] = -e_2 \\ A_1 = \begin{bmatrix} 0 & 0 & 0 \\ 0 & 0 & 1 \\ 0 & 0 & 0 \end{bmatrix} \end{cases}$$

and $\vartheta_{\alpha_0,\ldots,\alpha_{26}}(\mathfrak{L}_1^2)$ is the pair

$$(\nu_2, A_2) = \begin{cases} \nu_2 = \{[e_2, e_2] = e_3 \\ A_2 = \begin{bmatrix} 0 & 1 & 0 \\ 0 & 0 & 0 \\ 0 & 0 & 0 \end{bmatrix} \end{cases}$$

Therefore

$$\dim(\mathrm{Ker}(T_{\theta_1,\ldots,\theta_6;\beta_1,\ldots,\beta_9}(\nu_1, A_1))) \leq \dim(\mathrm{Ker}(T_{\theta_1,\ldots,\theta_6;\beta_1,\ldots,\beta_9}(\nu_2, A_2))). \tag{3}$$

Now, we consider $\{\beta_1 = 1, \beta_2 = -1, \beta_3 = 0, \beta_4 = 1, \beta_5 = -1, \beta_6 = 0, \beta_7 = 1, \beta_8 = -1, \beta_9 = 0,\}$ and $\{\theta_1 = 0, \theta_2 = 0, \theta_3 = 1, \theta_4 = 0, \theta_5 = 0, \theta_6 = 0\}$.

The set of 3-by-3 matrices that commute with $A_1$ is

$$\mathfrak{gl}(n, \mathbb{C})_{A_1} = \left\{ \begin{bmatrix} t_1 & 0 & t_4 \\ t_3 & t_2 & t_5 \\ 0 & 0 & t_2 \end{bmatrix} : t_1, t_2, t_3, t_4, t_5 \in \mathbb{C} \right\}.$$

Given $D_1(w's)$, $D_2(x's)$ and $D_3(y's)$ in $\mathfrak{gl}(n, \mathbb{C})_{A_1}$, with $D_2$ and $D_3$ equal to zero, we have $\lambda_1 := D_1\nu(\cdot, \cdot) + \nu(D_2\cdot, \cdot) + \nu(\cdot, D_3\cdot)$,

$$\lambda_1 := \{ [e_3, e_1] = (-w_{3,3}) e_2,$$

and so $\dim(\mathrm{Ker}(T_{\theta_1,\ldots,\theta_6;\beta_1,\ldots,\beta_9}(\nu_1, A_1))) = 4$.

On the other hand, the set of 3-by-3 matrices that commute with $A_2$ is

$$\mathfrak{gl}(n, \mathbb{C})_{A_2} = \left\{ \begin{bmatrix} t_1 & t_3 & t_4 \\ 0 & t_1 & 0 \\ 0 & t_5 & t_2 \end{bmatrix} : t_1, t_2, t_3, t_4, t_5 \in \mathbb{C} \right\}.$$

and given $D_1(w's)$, $D_2(x's)$ and $D_3(y's)$ in $\mathfrak{gl}(n, \mathbb{C})_{A_2}$, with $D_2$ and $D_3$ equal to zero, we have $\lambda_2 := D_1\nu(\cdot, \cdot) + \nu(D_2\cdot, \cdot) + \nu(\cdot, D_3\cdot)$,

$$\lambda_2 := \{ [e2, e2], w_{1,3}e1 + w_{3,3}e3$$

and so $\dim(\mathrm{Ker}(T_{\theta_1,\ldots,\theta_6;\beta_1,\ldots,\beta_9}(\nu_2, A_2))) = 3$; this is a contradiction.

17. $\mathfrak{L}_6^4 \xrightarrow{\text{deg}} \mathfrak{L}_0^2$.

    Suppose, contrary to our claim, that $\mathfrak{L}_6^4 \xrightarrow{\text{deg}} \mathfrak{L}_0^2$. In such case, we have

    $$\begin{bmatrix} 0 & 0 & 0 \\ 0 & 0 & 1 \\ 0 & 0 & 0 \end{bmatrix} \xrightarrow{\text{deg}} \begin{bmatrix} 0 & 1 & 0 \\ 0 & 0 & 1 \\ 0 & 0 & 0 \end{bmatrix}$$

    Since the first matrix has nilpotent degree equal to 2 and the second one has nilpotent degree equal to 3, we have a contradiction.

18. $\mathfrak{L}_6^3 \xrightarrow{\text{deg}} \mathfrak{L}_0^2$.

    Suppose, contrary to our claim, that $\mathfrak{L}_6^3 \xrightarrow{\text{deg}} \mathfrak{L}_0^2$. In such case, we have

    $$\begin{bmatrix} 0 & 0 & 0 \\ 1 & 0 & 0 \\ 1 & 0 & 0 \end{bmatrix} \xrightarrow{\text{deg}} \begin{bmatrix} 0 & 1 & 0 \\ 0 & 0 & 1 \\ 0 & 0 & 0 \end{bmatrix}$$

    Since the first matrix has nilpotent degree equal to 2 and the second one has nilpotent degree equal to 3, we have a contradiction.

19. $\mathfrak{L}_6^0 \xrightarrow{\text{deg}} \mathfrak{L}_0^1$.

    Suppose, contrary to our claim, that $\mathfrak{L}_6^0 \xrightarrow{\text{deg}} \mathfrak{L}_0^1$. In such case, we have

    $$0 \xrightarrow{\text{deg}} \begin{bmatrix} 0 & 0 & 0 \\ 0 & 0 & 1 \\ 0 & 0 & 0 \end{bmatrix}$$

    Since the first matrix has nilpotent degree equal to 1 and the second one has nilpotent degree equal to 2, we have a contradiction.

## 9.7 Degenerations between hom-Lie algebras ($\mathfrak{L}^i_j$) with $j = 5, 1, 0$

If $\mathfrak{L}^j_i \xrightarrow{\text{deg}} \mathfrak{L}^k_h$ then $\mathfrak{L}_i \xrightarrow{\text{deg}} \mathfrak{L}_h$ and $\text{Der}(\mathfrak{L}^j_i) \leq \text{Der}(\mathfrak{L}^k_h)$.

Therefore, we can organize the hom-Lie algebras in the following way:

| Dim(Der) | hom-Lie algebra | | |
|---|---|---|---|
| 0 | | $\mathfrak{L}^4_1$ | |
| 1 | $\mathfrak{L}^9_5(z, \lambda), \mathfrak{L}^8_5(z), \mathfrak{L}^7_5(z)$ | $\mathfrak{L}^6_1$ | |
| 2 | $\mathfrak{L}^6_5(z, \lambda), \mathfrak{L}^5_5(z), \mathfrak{L}^4_5(z), \mathfrak{L}^3_5(z)$ | $\mathfrak{L}^5_1, \mathfrak{L}^3_1$ | |
| 3 | $\mathfrak{L}^2_5(z), \mathfrak{L}^1_5(z)$ | $\mathfrak{L}^2_1$ | $\mathfrak{L}^2_0$ |
| 4 | $\mathfrak{L}^0_5(z)$ | $\mathfrak{L}^1_1$ | |
| 5 | | | $\mathfrak{L}^1_0$ |
| 6 | | $\mathfrak{L}^0_1$ | |
| 9 | | | $\mathfrak{L}^0_0$ |

### 9.7.1 Degenerations

1. $\mathfrak{L}^9_5(z, \lambda) \xrightarrow{\text{deg}} \mathfrak{L}^5_1$ with $z^2 \neq 1, 0$
   In fact, set

$$g(t) = \begin{bmatrix} e^{2t} & 0 & 0 \\ \frac{ze^{2t} + \sqrt{\lambda(z-1)}e^{3t} - e^{2t}}{\lambda(z-1)} & e^{2t} & e^{2t} \\ 0 & \frac{ze^{2t} + \sqrt{\lambda(z-1)}e^{3t} - e^{2t}}{\lambda(z-1)} & \frac{e^{3t}}{\sqrt{\lambda(z-1)}} \end{bmatrix}, \text{ with } t \in \mathbb{R}.$$

We have $g(t) \cdot \mathfrak{L}^9_5(z, \lambda)$ is the hom-Lie algebra $(\mu(t), A)$ with

$$\begin{cases} [e_1, e_2] = \frac{e^{-t}\left(\sqrt{\lambda(z-1)}ze^{-t} + \lambda z - \lambda\right)}{\sqrt{\lambda(z-1)}} e_2 + \left(\frac{z}{e^t\sqrt{\lambda z - \lambda}} + 1 - \frac{1}{e^t\sqrt{\lambda z - \lambda}}\right) e_3, \\ [e_1, e_3] = -\lambda(z-1)e^{-2t}e_2 + \frac{e^{-t}\left(-\lambda z + e^{-t}\sqrt{\lambda(z-1)} + \lambda\right)}{\sqrt{\lambda(z-1)}} e_3 \\ A = \begin{bmatrix} 0 & 0 & 0 \\ 1 & 0 & 0 \\ 0 & 1 & 0 \end{bmatrix} \end{cases}$$

It is easy to check that $(\mu(t), A) \to \mathfrak{L}^5_1$ as $t$ tends to infinity.

2. $\mathfrak{L}^8_5(z) \xrightarrow{\text{deg}} \mathfrak{L}^5_1$
   In fact, set

$$g(t) = \begin{bmatrix} e^t & 0 & 0 \\ \frac{e^{2t}}{z-1} & 0 & e^t \\ 0 & e^t & \frac{e^{2t}}{z-1} \end{bmatrix}, \text{ with } t \in \mathbb{R}.$$

We have $g(t) \cdot \mathfrak{L}^8_5(z)$ is the hom-Lie algebra $(\mu(t), A)$ with

$$\begin{cases} [e_1, e_2] = e^{-t}ze_2 + e_3, [e_1, e_3] = e^{-t}e_3 \\ A = \begin{bmatrix} 0 & 0 & 0 \\ 1 & 0 & 0 \\ 0 & 1 & 0 \end{bmatrix} \end{cases}$$

It is easy to check that $(\mu(t), A) \to \mathfrak{L}^5_1$ as $t$ tends to infinity.

3. $\mathfrak{L}_5^7(z) \xrightarrow{\text{deg}} \mathfrak{L}_1^5$

   In fact, set
   $$g(t) = \begin{bmatrix} e^t & 0 & 0 \\ -\frac{e^{2t}}{z-1} & e^t & 0 \\ 0 & -\frac{e^{2t}}{z-1} & e^t \end{bmatrix}, \text{ with } t \in \mathbb{R}.$$

   We have $g(t) \cdot \mathfrak{L}_5^7(z)$ is the hom-Lie algebra $(\mu(t), A)$ with
   $$\begin{cases} [e_1, e_2] = e^{-t} e_2 + e_3, [e_1, e_3] = e^{-t} z e_3 \\ \begin{bmatrix} 0 & 0 & 0 \\ 1 & 0 & 0 \\ 0 & 1 & 0 \end{bmatrix} \end{cases}$$

   It is easy to check that $(\mu(t), A) \to \mathfrak{L}_1^5$ as $t$ tends to infinity.

4. $\mathfrak{L}_5^6(z, \lambda) \xrightarrow{\text{deg}} \mathfrak{L}_1^2$

   In fact, set
   $$g(t) = \begin{bmatrix} 0 & 1 & z \\ 0 & -\lambda(z-1) & -\lambda(z-1) \\ (z-1)e^{2t} & \sqrt{-1}e^t(z-1)\sqrt{\lambda} & \sqrt{-1}e^t(z-1)\sqrt{\lambda} \end{bmatrix}, \text{ with } t \in \mathbb{R}.$$

   We have $g(t) \cdot \mathfrak{L}_5^6(z, \lambda)$ is the hom-Lie algebra $(\mu(t), A)$ with
   $$\begin{cases} [e1, e2] = \frac{-\sqrt{-1}(z+1)e^{-t}}{\sqrt{\lambda}(z-1)} e_1 + \sqrt{-1}\sqrt{\lambda}e^{-t}e_2 + e_3, \\ [e1, e3] = -\frac{(z+1)e^{-2t}}{z-1} e_1 + e^{-2t}\lambda e_2 - \sqrt{-1}\sqrt{\lambda}e^{-t}e_3, \\ [e2, e3] = -\frac{ze^{-2t}}{(z^2 - 2z+1)\lambda} e_1 \\ A = \begin{bmatrix} 0 & 1 & 0 \\ 0 & 0 & 0 \\ 0 & 0 & 0 \end{bmatrix} \end{cases}$$

   It is easy to check that $(\mu(t), A) \to \mathfrak{L}_1^2$ as $t$ tends to infinity.

5. $\mathfrak{L}_5^3(z) \xrightarrow{\text{deg}} \mathfrak{L}_1^2$

   In fact, set
   $$g(t) = \begin{bmatrix} 0 & e^t & 0 \\ e^t & 0 & 0 \\ 0 & \frac{(e^t)^2}{z-1} & -\frac{(e^t)^2}{z-1} \end{bmatrix}, \text{ with } t \in \mathbb{R}.$$

   We have $g(t) \cdot \mathfrak{L}_5^3(z)$ is the hom-Lie algebra $(\mu(t), A)$ with
   $$\begin{cases} [e_1, e_2] = -e^{-t}e_1 + e_3, [e_2, e_3] = ze^{-t}e_3 \\ A = \begin{bmatrix} 0 & 1 & 0 \\ 0 & 0 & 0 \\ 0 & 0 & 0 \end{bmatrix} \end{cases}$$

   It is easy to check that $(\mu(t), A) \to \mathfrak{L}_1^2$ as $t$ tends to infinity.

6. $\mathfrak{L}_5^2(z) \xrightarrow{\deg} \mathfrak{L}_1^1$

   In fact, set
   $$g(t) = \begin{bmatrix} 0 & 1 & 0 \\ e^t & 0 & 0 \\ 0 & \frac{e^t}{z-1} & e^t \end{bmatrix}, \text{ with } t \in \mathbb{R}.$$

   We have $g(t) \cdot \mathfrak{L}_5^2(z)$ is the hom-Lie algebra $(\mu(t), A)$ with
   $$\begin{cases} [e_1, e_2] = -e^{-t}e_1 + e_3, [e_2, e_3] = e^{-t}ze_3 \\ A = \begin{bmatrix} 0 & 0 & 0 \\ 0 & 0 & 0 \\ 0 & 1 & 0 \end{bmatrix} \end{cases}$$

   It is easy to check that $(\mu(t), A) \to \mathfrak{L}_1^1$ as $t$ tends to infinity.

7. $\mathfrak{L}_5^1(z) \xrightarrow{\deg} \mathfrak{L}_1^1$

   In fact, set
   $$g(t) = \begin{bmatrix} 0 & 0 & 1 \\ e^t & 0 & 0 \\ 0 & e^t & -\frac{e^t}{z-1} \end{bmatrix}, \text{ with } t \in \mathbb{R}.$$

   We have $g(t) \cdot \mathfrak{L}_5^1(z)$ is the hom-Lie algebra $(\mu(t), A)$ with
   $$\begin{cases} [e_1, e_2] = -e^{-t}ze_1 + e_3, [e_2, e_3] = e^{-t}e_3 \\ \begin{bmatrix} 0 & 0 & 0 \\ 0 & 0 & 0 \\ 0 & 1 & 0 \end{bmatrix} \end{cases}$$

   It is easy to check that $(\mu(t), A) \to \mathfrak{L}_1^1$ as $t$ tends to infinity.

8. $\mathfrak{L}_5^0(z) \xrightarrow{\deg} \mathfrak{L}_1^0$

   In fact, set
   $$g(t) = \begin{bmatrix} -e^t(z-1) & 0 & 0 \\ 0 & 1 & 0 \\ 0 & e^t & 1 \end{bmatrix}, \text{ with } t \in \mathbb{R}.$$

   We have $g(t) \cdot \mathfrak{L}_1^0(z)$ is the hom-Lie algebra $(\mu(t), A)$ with
   $$\begin{cases} [e_1, e_2] = -\frac{e^{-t}}{z-1}e_2 + e_3, [e_1, e_3] = -\frac{e^{-t}z}{z-1}e_3 \\ A = 0 \end{cases}$$

   It is easy to check that $(\mu(t), A) \to \mathfrak{L}_1^0$ as $t$ tends to infinity.

9. $\mathfrak{L}_5^9(z, \lambda) \xrightarrow{\deg} \mathfrak{L}_1^3$ with $(z(z^2-1) \neq 0)$.

   Suppose, contrary to our claim, that $\mathfrak{L}_5^9(z, \lambda) \xrightarrow{\deg} \mathfrak{L}_1^3$ with $(z(z^2-1) \neq 0)$. In such case, we have
   $$\psi_{\alpha, \beta}(\mathfrak{L}_5^9(z, \lambda)) \xrightarrow{\deg} \psi_{\alpha, \beta}(\mathfrak{L}_1^3).$$

   By taking $\alpha = (z+1 \mp \sqrt{z^2+6z+1})/2(-1+z)\lambda$ and $\beta = (z+1 \pm \sqrt{z^2+6z+1})/2(-1+z)\lambda$, we have from 7.10.1 that $\psi_{\alpha, \beta}(\mathfrak{L}_5^9(z, \lambda))$ is a Lie algebra isomorphic to $\mathfrak{n}_3$, therefore, $\psi_{\alpha, \beta}(\mathfrak{L}_1^3)$ is a Lie algebra which is isomorphic to one of the following: $\mathfrak{n}_3$ or $\mathfrak{a}_3$ (by Theorem 1.1). We must have $\psi_{\alpha, \beta}(\mathfrak{L}_1^3)$ is isomorphic to $\mathfrak{n}_3$ (from 3.4.1). But, we now apply again 3.4.1, to obtain $\alpha = \beta = 0$; this is a contradiction.

10. $\mathfrak{L}_5^8 \xrightarrow{\deg} \mathfrak{L}_1^3$.

    Suppose, contrary to our claim, that $\mathfrak{L}_5^8 \xrightarrow{\deg} \mathfrak{L}_1^3$. In such case, we have

    $$\phi_\beta(\mathfrak{L}_5^8) \xrightarrow{\deg} \phi_\beta(\mathfrak{L}_1^3).$$

    By taking $\beta = -z$, we have from 7.9.2 that $\phi_\beta(\mathfrak{L}_5^8)$ is the 3-dimensional abelian Lie algebra, therefore, $\phi_\beta(\mathfrak{L}_1^3)$ is $\mathfrak{a}_3$; this contradicts to 3.4.2.

11. $\mathfrak{L}_5^7 \xrightarrow{\deg} \mathfrak{L}_1^3$.

    Suppose, contrary to our claim, that $\mathfrak{L}_5^7 \xrightarrow{\deg} \mathfrak{L}_1^3$. In such case, we have

    $$\phi_\beta(\mathfrak{L}_5^7) \xrightarrow{\deg} \phi_\beta(\mathfrak{L}_1^3).$$

    By taking $\beta = -1/z$, we have from 7.8.2 that $\phi_\beta(\mathfrak{L}_5^8)$ is the 3-dimensional abelian Lie algebra, therefore, $\phi_\beta(\mathfrak{L}_1^3)$ is $\mathfrak{a}_3$; this contradicts to 3.4.2.

12. $\mathfrak{L}_5^6(z, \lambda) \xrightarrow{\deg} \mathfrak{L}_0^2$.

    Suppose, contrary to our claim, that $\mathfrak{L}_5^6(z, \lambda) \xrightarrow{\deg} \mathfrak{L}_0^2$. In such case, we have

    $$\begin{bmatrix} 0 & 0 & 0 \\ 0 & \lambda & \lambda \\ 0 & -\lambda & -\lambda \end{bmatrix} \xrightarrow{\deg} \begin{bmatrix} 0 & 1 & 0 \\ 0 & 0 & 1 \\ 0 & 0 & 0 \end{bmatrix}$$

    Since the first matrix has nilpotent degree equal to 2 and the second one has nilpotent degree equal to 3, we have a contradiction.

13. $\mathfrak{L}_5^5(z) \xrightarrow{\deg} \mathfrak{L}_1^2$.

    Suppose, contrary to our claim, that $\mathfrak{L}_5^5(z) \xrightarrow{\deg} \mathfrak{L}_1^2$. In such case, we have

    $$\vartheta_{\alpha_0, \ldots, \alpha_{26}}(\mathfrak{L}_5^5(z)) \xrightarrow{\deg} \vartheta_{\alpha_0, \ldots, \alpha_{26}}(\mathfrak{L}_1^2).$$

    If $\{\alpha_0 = 0, \alpha_1 = 0, \alpha_2 = 1, \alpha_3 = 0, \ldots, \alpha_{26} = 0\}$ we have $\vartheta_{\alpha_0, \ldots, \alpha_{26}}(\mathfrak{L})$ is the pair

    $$(\nu_1, A_1) = \begin{cases} \nu_1 = \{e_2 \cdot e_1 = -z e_3 \\ A_1 = \begin{bmatrix} 0 & 0 & 0 \\ 0 & 0 & 0 \\ 0 & 1 & 0 \end{bmatrix} \end{cases}$$

    and $\vartheta_{\alpha_0, \ldots, \alpha_{26}}(\mathfrak{L})$ is the pair

    $$(\nu_2, A_2) = \begin{cases} \nu_2 = \{e_2 \cdot e_2 = e_3 \\ A_2 = \begin{bmatrix} 0 & 1 & 0 \\ 0 & 0 & 0 \\ 0 & 0 & 0 \end{bmatrix} \end{cases}$$

    Therefore

    $$\dim(\mathrm{Ker}(T_{\theta_1, \ldots, \theta_6; \beta_1, \ldots, \beta_9}(\nu_1, A_1))) \leq \dim(\mathrm{Ker}(T_{\theta_1, \ldots, \theta_6; \beta_1, \ldots, \beta_9}(\nu_2, A_2))). \quad (4)$$

    Now, we consider $\{\beta_1 = 1, \beta_2 = -1, \beta_3 = 0, \beta_4 = 1, \beta_5 = -1, \beta_6 = 0, \beta_7 = 1, \beta_8 = -1, \beta_9 = 0,\}$ and $\{\theta_1 = 0, \theta_2 = 1, \theta_3 = 0, \theta_4 = 0, \theta_5 = 0, \theta_6 = 1\}$.

The set of 3-by-3 matrices that commute with $A_1$ is

$$\mathfrak{gl}(n,\mathbb{C})_{A_1} = \left\{ \begin{bmatrix} t_1 & t_3 & 0 \\ 0 & t_2 & 0 \\ t_4 & t_5 & t_2 \end{bmatrix} : t_1, t_2, t_3, t_4, t_5 \in \mathbb{C} \right\}.$$

Given $D_1(w's)$, $D_2(x's)$ and $D_3(y's)$ in $\mathfrak{gl}(n,\mathbb{C})_{A_1}$, with $D_1$ and $D_2$ equal to zero, we have $\lambda_1 := D_1 \nu(\cdot, \cdot) + \nu(D_2 \cdot, \cdot) + \nu(\cdot, D_3 \cdot)$,

$$\lambda_1 := \{ \; e_2 \cdot e_1 = -w_{3,3} z e_3,$$

and so $\dim(\mathrm{Ker}(T_{\theta_1,\ldots,\theta_6;\beta_1,\ldots,\beta_9}(\nu_1, A_1))) = 4$.

On the other hand, the set of 3-by-3 matrices that commute with $A_2$ is

$$\mathfrak{gl}(n,\mathbb{C})_{A_2} = \left\{ \begin{bmatrix} t_1 & t_3 & t_4 \\ 0 & t_1 & 0 \\ 0 & t_5 & t_2 \end{bmatrix} : t_1, t_2, t_3, t_4, t_5 \in \mathbb{C} \right\}.$$

and given $D_1(w's)$, $D_2(x's)$ and $D_3(y's)$ in $\mathfrak{gl}(n,\mathbb{C})_{A_2}$, with $D_1$ and $D_2$ equal to zero, we have $\lambda_2 := D_1 \nu(\cdot, \cdot) + \nu(D_2 \cdot, \cdot) + \nu(\cdot, D_3 \cdot)$,

$$\lambda_2 := \{ \; e_2 \cdot e_2 = w_{1,3} e_1 + w_{3,3} e_3$$

and so $\dim(\mathrm{Ker}(T_{\theta_1,\ldots,\theta_6;\beta_1,\ldots,\beta_9}(\nu_2, A_2))) = 3$; this is a contradiction.

14. $\mathfrak{L}_5^5(z) \xrightarrow{\deg} \mathfrak{L}_0^2$.

Suppose, contrary to our claim, that $\mathfrak{L}_5^5(z) \xrightarrow{\deg} \mathfrak{L}_0^2$. In such case, we have

$$\begin{bmatrix} 0 & 0 & 0 \\ 0 & 0 & 0 \\ 0 & 1 & 0 \end{bmatrix} \xrightarrow{\deg} \begin{bmatrix} 0 & 1 & 0 \\ 0 & 0 & 1 \\ 0 & 0 & 0 \end{bmatrix}$$

Since the first matrix has nilpotent degree equal to 2 and the second one has nilpotent degree equal to 3, we have a contradiction.

15. $\mathfrak{L}_5^4(z) \xrightarrow{\deg} \mathfrak{L}_1^2$.

Suppose, contrary to our claim, that $\mathfrak{L}_5^4(z) \xrightarrow{\deg} \mathfrak{L}_1^2$. In such case, we have

$$\vartheta_{\alpha_0,\ldots,\alpha_{26}}(\mathfrak{L}_5^4(z)) \xrightarrow{\deg} \vartheta_{\alpha_0,\ldots,\alpha_{26}}(\mathfrak{L}_1^2).$$

If $\{\alpha_0 = 0, \alpha_1 = 0, \alpha_2 = 1, \alpha_3 = 0, \ldots, \alpha_{26} = 0\}$ we have $\vartheta_{\alpha_0,\ldots,\alpha_{26}}(\mathfrak{L}_5^4(z))$ is the pair

$$(\nu_1, A_1) = \begin{cases} \nu_1 = \{e_3 \cdot e_1 = -e_2 \\ A_1 = \begin{bmatrix} 0 & 0 & 0 \\ 0 & 0 & 1 \\ 0 & 0 & 0 \end{bmatrix} \end{cases}$$

and $\vartheta_{\alpha_0,\ldots,\alpha_{26}}(\mathfrak{L}_1^2)$ is the pair

$$(\nu_2, A_2) = \begin{cases} \nu_2 = \{e_2 \cdot e_2 = e_3 \\ A_2 = \begin{bmatrix} 0 & 1 & 0 \\ 0 & 0 & 0 \\ 0 & 0 & 0 \end{bmatrix} \end{cases}$$

Therefore

$$\dim(\mathrm{Ker}(T_{\theta_1,\ldots,\theta_6;\beta_1,\ldots,\beta_9}(\nu_1, A_1))) \leq \dim(\mathrm{Ker}(T_{\theta_1,\ldots,\theta_6;\beta_1,\ldots,\beta_9}(\nu_2, A_2))). \quad (5)$$

Now, we consider $\{\beta_1 = 1, \beta_2 = -1, \beta_3 = 0, \beta_4 = 1, \beta_5 = -1, \beta_6 = 0, \beta_7 = 1, \beta_8 = -1, \beta_9 = 0,\}$ and $\{\theta_1 = 0, \theta_2 = 1, \theta_3 = 0, \theta_4 = 0, \theta_5 = 0, \theta_6 = 1\}$.

The set of 3-by-3 matrices that commute with $A_1$ is

$$\mathfrak{gl}(n,\mathbb{C})_{A_1} = \left\{ \begin{bmatrix} t_1 & 0 & t_4 \\ t_3 & t_2 & t_5 \\ 0 & 0 & t_2 \end{bmatrix} : t_1, t_2, t_3, t_4, t_5 \in \mathbb{C} \right\}.$$

Given $D_1(w's)$, $D_2(x's)$ and $D_3(y's)$ in $\mathfrak{gl}(n,\mathbb{C})_{A_1}$, with $D_2$ and $D_3$ equal to 0, we have $\lambda_1 := D_1 \nu(\cdot,\cdot) + \nu(D_2\cdot,\cdot) + \nu(\cdot, D_3\cdot)$,

$$\lambda_1 := \{ \ e_3 \cdot e_1 = (-w_{3,3})\, e_2.$$

and so $\dim(\mathrm{Ker}(T_{\theta_1,\ldots,\theta_6;\beta_1,\ldots,\beta_9}(\nu_1, A_1))) = 4$.

On the other hand, the set of 3-by-3 matrices that commute with $A_2$ is

$$\mathfrak{gl}(n,\mathbb{C})_{A_2} = \left\{ \begin{bmatrix} t_1 & t_3 & t_4 \\ 0 & t_1 & 0 \\ 0 & t_5 & t_2 \end{bmatrix} : t_1, t_2, t_3, t_4, t_5 \in \mathbb{C} \right\}.$$

and given $D_1(w's)$, $D_2(x's)$ and $D_3(y's)$ in $\mathfrak{gl}(n,\mathbb{C})_{A_2}$, with $D_2$ and $D_3$ equal to 0, we have $\lambda_2 := D_1 \nu(\cdot,\cdot) + \nu(D_2\cdot,\cdot) + \nu(\cdot, D_3\cdot)$,

$$\lambda_2 := \{ \ e_2 \cdot e_2 = w_{1,3} e_1 + w_{3,3} e_3$$

and so $\dim(\mathrm{Ker}(T_{\theta_1,\ldots,\theta_6;\beta_1,\ldots,\beta_9}(\nu_2, A_2))) = 3$; this is a contradiction.

16. $\mathfrak{L}_5^4(z) \xrightarrow{\mathrm{deg}} \mathfrak{L}_0^2$.

    Suppose, contrary to our claim, that $\mathfrak{L}_5^4(z) \xrightarrow{\mathrm{deg}} \mathfrak{L}_0^2$. In such case, we have

    $$\begin{bmatrix} 0 & 0 & 0 \\ 0 & 0 & 1 \\ 0 & 0 & 0 \end{bmatrix} \xrightarrow{\mathrm{deg}} \begin{bmatrix} 0 & 1 & 0 \\ 0 & 0 & 1 \\ 0 & 0 & 0 \end{bmatrix}$$

    Since the first matrix has nilpotent degree equal to 2 and the second one has nilpotent degree equal to 3, we have a contradiction.

17. $\mathfrak{L}_5^3(z) \xrightarrow{\mathrm{deg}} \mathfrak{L}_0^2$.

    Suppose, contrary to our claim, that $\mathfrak{L}_5^3(z) \xrightarrow{\mathrm{deg}} \mathfrak{L}_0^2$. In such case, we have

    $$\begin{bmatrix} 0 & 0 & 0 \\ 1 & 0 & 0 \\ 1 & 0 & 0 \end{bmatrix} \xrightarrow{\mathrm{deg}} \begin{bmatrix} 0 & 1 & 0 \\ 0 & 0 & 1 \\ 0 & 0 & 0 \end{bmatrix}$$

    Since the first matrix has nilpotent degree equal to 2 and the second one has nilpotent degree equal to 3, we have a contradiction.

18. $\mathfrak{L}_5^0(z) \xrightarrow{\deg} \mathfrak{L}_0^1$.

    Suppose, contrary to our claim, that $\mathfrak{L}_5^0(z) \xrightarrow{\deg} \mathfrak{L}_0^1$. In such case, we have

    $$0 \xrightarrow{\deg} \begin{bmatrix} 0 & 0 & 0 \\ 0 & 0 & 1 \\ 0 & 0 & 0 \end{bmatrix}$$

    Since the first matrix has nilpotent degree equal to 1 and the second one has nilpotent degree equal to 2, we have a contradiction.

## 9.8 Degenerations between hom-Lie algebras ($\mathfrak{L}^i_j$) with $j = 2, 3, 1, 0$

If $\mathfrak{L}^j_i \xrightarrow{\text{deg}} \mathfrak{L}^k_h$ then $\mathfrak{L}_i \xrightarrow{\text{deg}} \mathfrak{L}_h$ and $\text{Der}(\mathfrak{L}^j_i) \leq \text{Der}(\mathfrak{L}^k_h)$.

Therefore, we can organize the hom-Lie algebras in the following way:

| Dim(Der) | hom-Lie algebra | | | |
|---|---|---|---|---|
| 0 | | | $\mathfrak{L}^4_1$ | |
| 1 | $\mathfrak{L}^6_2(\lambda)$ | | $\mathfrak{L}^6_1$ | |
| 2 | $\mathfrak{L}^5_2(\lambda), \mathfrak{L}^4_2(\lambda), \mathfrak{L}^2_2$ | $\mathfrak{L}^3_3$ | $\mathfrak{L}^5_1, \mathfrak{L}^5_1$ | |
| 3 | $\mathfrak{L}^3_2(\lambda), \mathfrak{L}^1_2$ | $\mathfrak{L}^2_3$ | $\mathfrak{L}^2_1$ | $\mathfrak{L}^2_0$ |
| 4 | $\mathfrak{L}^0_2$ | $\mathfrak{L}^1_3$ | $\mathfrak{L}^1_1$ | |
| 5 | | | | $\mathfrak{L}^1_0$ |
| 6 | | $\mathfrak{L}^0_3$ | $\mathfrak{L}^0_1$ | |
| 9 | | | | $\mathfrak{L}^0_0$ |

### 9.8.1 Degenerations

1. $\mathfrak{L}^6_2(\lambda) \xrightarrow{\text{deg}} \mathfrak{L}^5_1$

   In fact, set
   $$g(t) = \begin{bmatrix} -e^{2t}\lambda & 0 & 0 \\ -e^{3t}\lambda & -e^{2t}\lambda & 0 \\ 0 & -e^{3t}\lambda & -e^{2t} \end{bmatrix}, \text{ with } t \in \mathbb{R}.$$

   We have $g(t) \cdot \mathfrak{L}^6_2(\lambda)$ is the hom-Lie algebra $(\mu(t), A)$ with
   $$\begin{cases} [e_1, e_2] = -\frac{e^{-t}(-\lambda + e^{-t})e_2}{\lambda} + e_3, [e_1, e_3] = -e^{-2t}e_2 - \frac{e^{-t}(e^{-t} + \lambda)e_3}{\lambda} \\ A = \begin{bmatrix} 0 & 0 & 0 \\ 1 & 0 & 0 \\ 0 & 1 & 0 \end{bmatrix} \end{cases}$$

   It is easy to check that $(\mu(t), A) \to \mathfrak{L}^5_1$ as $t$ tends to infinity.

2. $\mathfrak{L}^5_2(\lambda) \xrightarrow{\text{deg}} \mathfrak{L}^2_0$

   In fact, set
   $$g(t) = \begin{bmatrix} 0 & e^t & 0 \\ 0 & 0 & e^t\lambda \\ e^t\lambda & 0 & 0 \end{bmatrix}, \text{ with } t \in \mathbb{R}.$$

   We have $g(t) \cdot \mathfrak{L}^5_2(\lambda)$ is the hom-Lie algebra $(\mu(t), A)$ with
   $$\begin{cases} [e_1, e_3] = -\frac{e^{-t}}{\lambda}e_1, [e_2, e_3] = -\frac{e^{-t}}{\lambda^2}e_1 - \frac{e^{-t}}{\lambda}e_2 \\ A = \begin{bmatrix} 0 & 1 & 0 \\ 0 & 0 & 1 \\ 0 & 0 & 0 \end{bmatrix} \end{cases}$$

   It is easy to check that $(\mu(t), A) \to \mathfrak{L}^2_0$ as $t$ tends to infinity.

3. $\mathfrak{L}^4_2(\lambda) \xrightarrow{\text{deg}} \mathfrak{L}^2_1$

   In fact, set
   $$g(t) = \begin{bmatrix} 0 & 1 & 1 \\ 0 & \lambda & 0 \\ \frac{e^{2t}}{\lambda} & e^t & 0 \end{bmatrix}, \text{ with } t \in \mathbb{R}.$$

115

We have $g(t) \cdot \mathfrak{L}_2^4(\lambda)$ is the hom-Lie algebra $(\mu(t), A)$ with

$$\begin{cases} [e1, e2] = 2\,\mathrm{e}^{-t} e1 + \mathrm{e}^{-t}\lambda\, e2 + e3, [e1, e3] = -2\lambda\,\mathrm{e}^{-2t} e1 - \lambda^2 \mathrm{e}^{-2t} e2 - \mathrm{e}^{-t}\lambda\, e3, [e2, e3] = \mathrm{e}^{-2t} e1 \\ A = \begin{bmatrix} 0 & 1 & 0 \\ 0 & 0 & 0 \\ 0 & 0 & 0 \end{bmatrix} \end{cases}$$

It is easy to check that $(\mu(t), A) \to \mathfrak{L}_1^2$ as $t$ tends to infinity.

4. $\mathfrak{L}_2^3 \xrightarrow{\deg} \mathfrak{L}_3^1$

In fact, set

$$g(t) = \begin{bmatrix} 0 & 0 & 1 \\ 0 & \lambda^{-1} & 0 \\ \mathrm{e}^t & 0 & -\mathrm{e}^t \end{bmatrix}, \text{ with } t \in \mathbb{R}.$$

We have $g(t) \cdot \mathfrak{L}_2^3$ is the hom-Lie algebra $(\mu(t), A)$ with

$$\begin{cases} [e1, e2] = e2, [e1, e3] = -\mathrm{e}^{-t} e1 - \frac{\mathrm{e}^{-t} e2}{\lambda} + e3, [e2, e3] = -\mathrm{e}^{-t} e2 \\ A = \begin{bmatrix} 0 & 0 & 0 \\ 1 & 0 & 0 \\ 0 & 0 & 0 \end{bmatrix} \end{cases}$$

It is easy to check that $(\mu(t), A) \to \mathfrak{L}_3^1$ as $t$ tends to infinity.

5. $\mathfrak{L}_2^3(\lambda) \xrightarrow{\deg} \mathfrak{L}_0^1$

In fact, set

$$g(t) = \begin{bmatrix} \mathrm{e}^t & 0 & 0 \\ 0 & 1 & 0 \\ 0 & 0 & \lambda \end{bmatrix}, \text{ with } t \in \mathbb{R}.$$

We have $g(t) \cdot \mathfrak{L}_2^3$ is the hom-Lie algebra $(\mu(t), A)$ with

$$\begin{cases} [e_1, e_2] = \mathrm{e}^{-t} e_2, [e_1, e_3] = \frac{\mathrm{e}^{-t} e_2}{\lambda} + \mathrm{e}^{-t} e_3 \\ A = \begin{bmatrix} 0 & 0 & 0 \\ 0 & 0 & 1 \\ 0 & 0 & 0 \end{bmatrix} \end{cases}$$

It is easy to check that $(\mu(t), A) \to \mathfrak{L}_0^1$ as $t$ tends to infinity.

6. $\mathfrak{L}_2^2 \xrightarrow{\deg} \mathfrak{L}_1^2$

In fact, set

$$g(t) = \begin{bmatrix} 0 & 0 & \mathrm{e}^t \\ \mathrm{e}^t & 0 & 0 \\ 0 & -(\mathrm{e}^t)^2 & 0 \end{bmatrix}, \text{ with } t \in \mathbb{R}.$$

We have $g(t) \cdot \mathfrak{L}_2^2$ is the hom-Lie algebra $(\mu(t), A)$ with

$$\begin{cases} [e_1, e_2] = -\mathrm{e}^{-t} e_1 + e_3, [e_2, e_3] = \mathrm{e}^{-t} e_3 \\ A = \begin{bmatrix} 0 & 1 & 0 \\ 0 & 0 & 0 \\ 0 & 0 & 0 \end{bmatrix} \end{cases}$$

It is easy to check that $(\mu(t), A) \to \mathfrak{L}_1^2$ as $t$ tends to infinity.

7. $\mathfrak{L}_2^1 \xrightarrow{\text{deg}} \mathfrak{L}_1^1$

   In fact, set
   $$g(t) = \begin{bmatrix} 0 & 0 & -1 \\ e^t & 0 & 0 \\ 0 & e^t & 0 \end{bmatrix}, \text{ with } t \in \mathbb{R}.$$

   We have $g(t) \cdot \mathfrak{L}_2^1$ is the hom-Lie algebra $(\mu(t), A)$ with
   $$\{ \ [e_1, e_2] = -e^{-t} e_1 + e_3, [e_2, e_3] = e^{-t} e_3$$

   It is easy to check that $(\mu(t), A) \to \mathfrak{L}_1^1$ as $t$ tends to infinity.

8. $\mathfrak{L}_2^1 \xrightarrow{\text{deg}} \mathfrak{L}_3^1$

   In fact, set
   $$g(t) = \begin{bmatrix} 1 & 0 & 0 \\ 0 & 1 & 0 \\ 0 & 0 & e^t \end{bmatrix}, \text{ with } t \in \mathbb{R}.$$

   We have $g(t) \cdot \mathfrak{L}_2^1$ is the hom-Lie algebra $(\mu(t), A)$ with
   $$\begin{cases} [e_1, e_2] = e_2, [e_1, e_3] = e^{-t} e_2 + e_3 \\ A = \begin{bmatrix} 0 & 0 & 0 \\ 1 & 0 & 0 \\ 0 & 0 & 0 \end{bmatrix} \end{cases}$$

   It is easy to check that $(\mu(t), A) \to \mathfrak{L}_3^1$ as $t$ tends to infinity.

9. $\mathfrak{L}_2^0 \xrightarrow{\text{deg}} \mathfrak{L}_1^0$

   In fact, set
   $$g(t) = \begin{bmatrix} -e^{2t} & 0 & 0 \\ 0 & 1 & 0 \\ 0 & e^t & 1 \end{bmatrix}, \text{ with } t \in \mathbb{R}.$$

   We have $g(t) \cdot \mathfrak{L}_2^0$ is the hom-Lie algebra $(\mu(t), A)$ with
   $$\begin{cases} [e_1, e_2] = \left(e^{-t} - e^{-2t}\right) e_2 + e_3, [e_1, e_3] = -e^{-2t} e_2 + \left(-e^{-2t} - e^{-t}\right) e_3 \\ A = 0 \end{cases}$$

   It is easy to check that $(\mu(t), A) \to \mathfrak{L}_1^0$ as $t$ tends to infinity.

10. $\mathfrak{L}_2^0 \xrightarrow{\text{deg}} \mathfrak{L}_3^0$

    In fact, set
    $$g(t) = \begin{bmatrix} 1 & 0 & 0 \\ 0 & 1 & 0 \\ 0 & 0 & e^t \end{bmatrix}, \text{ with } t \in \mathbb{R}.$$

    We have $g(t) \cdot \mathfrak{L}_2^0$ is the hom-Lie algebra $(\mu(t), A)$ with
    $$\begin{cases} [e_1, e_2] = e_2, [e_1, e_3] = e^{-t} e_2 + e_3 \\ A = 0 \end{cases}$$

    It is easy to check that $(\mu(t), A) \to \mathfrak{L}_3^0$ as $t$ tends to infinity.

11. $\mathfrak{L}_3^3 \xrightarrow{\deg} \mathfrak{L}_0^2$

    In fact, set
    $$g(t) = \begin{bmatrix} 0 & 0 & e^t \\ 0 & e^t & 0 \\ e^t & 0 & 0 \end{bmatrix}, \text{ with } t \in \mathbb{R}.$$

    We have $g(t) \cdot \mathfrak{L}_3^3$ is the hom-Lie algebra $(\mu(t), A)$ with
    $$\begin{cases} [e_1, e_3] = -e^{-t}e_1, [e_2, e_3] = -e^{-t}e_2 \\ A = \begin{bmatrix} 0 & 1 & 0 \\ 0 & 0 & 1 \\ 0 & 0 & 0 \end{bmatrix} \end{cases}$$

    It is easy to check that $(\mu(t), A) \to \mathfrak{L}_0^2$ as $t$ tends to infinity.

12. $\mathfrak{L}_3^1 \xrightarrow{\deg} \mathfrak{L}_0^1$

    In fact, set
    $$g(t) = \begin{bmatrix} 0 & 0 & 1 \\ 0 & e^t & 0 \\ e^t & 0 & 0 \end{bmatrix}, \text{ with } t \in \mathbb{R}.$$

    We have $g(t) \cdot \mathfrak{L}_3^1$ is the hom-Lie algebra $(\mu(t), A)$ with
    $$\begin{cases} [e_1, e_3] = -e^{-t}e_1, [e_2, e_3] = -e^{-t}e_2 \\ A = \begin{bmatrix} 0 & 0 & 0 \\ 0 & 0 & 1 \\ 0 & 0 & 0 \end{bmatrix} \end{cases}$$

    It is easy to check that $(\mu(t), A) \to \mathfrak{L}_0^1$ as $t$ tends to infinity.

13. $\mathfrak{L}_3^0 \xrightarrow{\deg} \mathfrak{L}_0^0$

    In fact, set
    $$g(t) = \begin{bmatrix} e^t & 0 & 0 \\ 0 & 1 & 0 \\ 0 & 0 & 1 \end{bmatrix}, \text{ with } t \in \mathbb{R}.$$

    We have $g(t) \cdot \mathfrak{L}_3^0$ is the hom-Lie algebra $(\mu(t), A)$ with
    $$\begin{cases} [e_1, e_2] = e^{-t}e_2, [e_1, e_3] = e^{-t}e_3 \\ A = 0 \end{cases}$$

    It is easy to check that $(\mu(t), A) \to \mathfrak{L}_3^0$ as $t$ tends to infinity.

14. $\mathfrak{L}_2^6(\lambda) \xrightarrow{\deg} \mathfrak{L}_1^3$ with $\lambda \neq 0$.

    Suppose, contrary to our claim, that $\mathfrak{L}_2^6(\lambda) \xrightarrow{\deg} \mathfrak{L}_1^3$ with $\lambda \neq 0$. In such case, we have
    $$\psi_{\alpha,\beta}(\mathfrak{L}_2^6(\lambda)) \xrightarrow{\deg} \psi_{\alpha,\beta}(\mathfrak{L}_1^3).$$

    By taking $\alpha = (-1 \mp \sqrt{2})/\lambda$ and $\beta = (-1 \pm \sqrt{2})/\lambda$, we have from 4.7.1 that $\psi_{\alpha,\beta}(\mathfrak{L}_2^6(\lambda))$ is a Lie algebra isomorphic to $\mathfrak{n}_3$, therefore, $\psi_{\alpha,\beta}(\mathfrak{L}_1^3)$ is a Lie algebra which is isomorphic to one of the following: $\mathfrak{n}_3$ or $\mathfrak{a}_3$ (by Theorem 1.1). We must have $\psi_{\alpha,\beta}(\mathfrak{L}_1^3)$ is isomorphic to $\mathfrak{n}_3$ (from 3.4.1). But, we now apply again 3.4.1, to obtain $\alpha = \beta = 0$; this is a contradiction.

15. $\mathfrak{L}_2^6(\lambda) \xrightarrow{\deg} \mathfrak{L}_3^0$ with $(\lambda \neq 0)$.

    Suppose, contrary to our claim, that $\mathfrak{L}_2^6(\lambda)\xrightarrow{\deg}\mathfrak{L}_3^0$ with $(\lambda \neq 0)$. In such case, we have
    $$\psi_{\alpha,\beta}(\mathfrak{L}_2^6(\lambda))\xrightarrow{\deg}\psi_{\alpha,\beta}(\mathfrak{L}_3^0).$$
    By taking $\alpha = (-1 \mp \sqrt{2})/\lambda$ and $\beta = (-1 \pm \sqrt{2})/\lambda$, we have from 4.7.1 that $\psi_{\alpha,\beta}(\mathfrak{L}_2^6(\lambda))$ is a Lie algebra isomorphic to $\mathfrak{n}_3$, therefore $\psi_{\alpha,\beta}(\mathfrak{L}_3^0)$ is a Lie algebra which is isomorphic to one of the following: $\mathfrak{n}_3$ or $\mathfrak{a}_3$ (by Theorem 1.1). But, from 5.1.1, we have $\psi_{\alpha,\beta}(\mathfrak{L}_3^0)$ is isomorphic to $\mathfrak{r}_{3,1}$; this is a contradiction.

16. $\mathfrak{L}_2^5(\lambda) \xrightarrow{\deg} \mathfrak{L}_3^2$.

    Suppose, contrary to our claim, that $\mathfrak{L}_2^5(\lambda)\xrightarrow{\deg}\mathfrak{L}_3^2$. In such case, we have
    $$\psi_{\alpha,\beta}(\mathfrak{L}_2^5(\lambda))\xrightarrow{\deg}\psi_{\alpha,\beta}(\mathfrak{L}_3^2).$$
    From 4.6.1, we have $\psi_{\alpha=-\frac{1}{2\lambda},\beta=-\frac{1}{2\lambda}}(\mathfrak{L}_2^5(\lambda))$ is a Lie algebra isomorphic to $\mathfrak{r}_{3,1}$, therefore, $\psi_{\alpha=-\frac{1}{2\lambda},\beta=-\frac{1}{2\lambda}}(\mathfrak{L}_3^2)$ is a Lie algebra which is isomorphic to one of the following: $\mathfrak{r}_{3,1}$, $\mathfrak{a}_3$ (by Theorem 1.1). We must have $\psi_{\alpha=-\frac{1}{2\lambda},\beta=-\frac{1}{2\lambda}}(\mathfrak{L}_3^2)$ is isomorphic to $\mathfrak{r}_{3,1}$ (from 5.3.1). But, we now apply again 5.3.1, to obtain $\alpha = -\beta$ but $\alpha = \beta$, and therefore $\alpha = 0 = -\frac{1}{2\lambda}$; this is a contradiction.

17. $\mathfrak{L}_2^5(\lambda) \xrightarrow{\deg} \mathfrak{L}_1^0$ with $(\lambda \neq 0)$.

    Suppose, contrary to our claim, that $\mathfrak{L}_2^5(\lambda)\xrightarrow{\deg}\mathfrak{L}_1^0$ with $(\lambda \neq 0)$. In such case, we have
    $$\psi_{\alpha\neq 0,\beta=\frac{-\alpha}{\lambda(1+\alpha)}}(\mathfrak{L}_2^5(\lambda))\xrightarrow{\deg}\psi_{\alpha\neq 0,\beta=\frac{-\alpha}{\lambda(1+\alpha)}}(\mathfrak{L}_1^0).$$
    From 4.6.1, we have $\psi_{\alpha\neq 0,\beta=\frac{-\alpha}{\lambda(1+\alpha)}}(\mathfrak{L}_2^5(\lambda))$ is a Lie algebra isomorphic to $\mathfrak{r}_{3,1}$, but $\psi_{\alpha\neq 0,\beta=\frac{-\alpha}{\lambda(1+\alpha)}}(\mathfrak{L}_1^0)$ is isomorphic to $\mathfrak{n}_3(\mathbb{C})$; this contradicts theorem 1.1.

18. $\mathfrak{L}_2^4(\lambda) \xrightarrow{\deg} \mathfrak{L}_0^2$.

    Suppose, contrary to our claim, that $\mathfrak{L}_2^4(\lambda)\xrightarrow{\deg}\mathfrak{L}_0^2$. In such case, we have
    $$\begin{bmatrix} 0 & 0 & 0 \\ 0 & 0 & 0 \\ 0 & \lambda & 0 \end{bmatrix} \xrightarrow{\deg} \begin{bmatrix} 0 & 1 & 0 \\ 0 & 0 & 1 \\ 0 & 0 & 0 \end{bmatrix}$$
    Since the first matrix has nilpotent degree equal to 2 and the second one has nilpotent degree equal to 3, we have a contradiction.

19. $\mathfrak{L}_2^2 \xrightarrow{\deg} \mathfrak{L}_3^2$.

    Suppose, contrary to our claim, that $\mathfrak{L}_2^2\xrightarrow{\deg}\mathfrak{L}_3^2$. In such case, we have
    $$\phi_\beta(\mathfrak{L}_2^2)\xrightarrow{\deg}\phi_\beta(\mathfrak{L}_3^2).$$
    By taking $\beta \neq -1$, we have from 4.3.2 that $\phi_\beta(\mathfrak{L}_2^2)$ is the 3-dimensional abelian Lie algebra, therefore, $\phi_\beta(\mathfrak{L}_3^2)$ is the Lie algebra $\mathfrak{a}_3$. But, from 5.3.2, we have $\phi_\beta(\mathfrak{L}_3^2)$ is isomorphic to $\mathfrak{n}_3$; this is a contradiction.

20. $\mathfrak{L}_2^2 \xrightarrow{\deg} \mathfrak{L}_0^2$.

    Suppose, contrary to our claim, that $\mathfrak{L}_2^2\xrightarrow{\deg}\mathfrak{L}_0^2$. In such case, we have
    $$\begin{bmatrix} 0 & 0 & 0 \\ 0 & 0 & 0 \\ 1 & 0 & 0 \end{bmatrix} \xrightarrow{\deg} \begin{bmatrix} 0 & 1 & 0 \\ 0 & 0 & 1 \\ 0 & 0 & 0 \end{bmatrix}$$

Since the first matrix has nilpotent degree equal to 2 and the second one has nilpotent degree equal to 3, we have a contradiction.

21. $\mathfrak{L}_2^0 \xrightarrow{\deg} \mathfrak{L}_0^1$.

    Suppose, contrary to our claim, that $\mathfrak{L}_2^0 \xrightarrow{\deg} \mathfrak{L}_0^1$. In such case, we have

    $$0 \xrightarrow{\deg} \begin{bmatrix} 0 & 0 & 0 \\ 0 & 0 & 1 \\ 0 & 0 & 0 \end{bmatrix}$$

    Since the first matrix has nilpotent degree equal to 1 and the second one has nilpotent degree equal to 2, we have a contradiction.




\clearpage
\begin{thebibliography}{BML}

\bibitem{alvacartes}
\'Alvarez A., Cartes F.:  
\emph{Cohomology and deformations for the Heisenberg Hom-Lie algebras.} 
Linear and Multilinear Algebra
\textbf{67}-11 (2019), 1--21.

\bibitem{AEM}
Ammar F., Ejbehi Z., Makhlouf A.: 
\emph{Cohomology and deformations of Hom-algebras.} 
Journal Lie Theory
\textbf{21}-4 (2011), 813--836.

\bibitem{BB}
Bene\v{s} T., Burde D.: 
\emph{Classification of orbit closures in the variety of three-dimensional Novikov algebras.} 
Journal of Algebra and Its Applications
\textbf{13}-2 (2014), p. 1350081 (33 pages)

\bibitem{Bianchi}
Bianchi L.: 
\emph{Lezioni sulla Teoria dei gruppi continui finiti di trasformazioni.} 
Enrico Spoerri, Libraio-Editore. Pisa, Italy (1918)


\bibitem{Bianchi2}
Bianchi L.: 
\emph{On the Three-Dimensional Spaces Which Admit a Continuous Group of Motions.} 
General Relativity and Gravitation
\textbf{33}-12 (2001), 2171--2253

\bibitem{Borel1}
Borel A.: 
\emph{Groupes Lineaires Algebriques.} 
Annals of Mathematics
\textbf{64}-1 (1956), 20--82

\bibitem{Borel2}
Borel A.: 
\emph{Linear Algebraic Groups.} 
Graduate Texts in Mathematics
\textbf{126}.
Springer-Verlag New York (1991)


\bibitem{Borel3}
Borel A., Harish-Chandra: 
\emph{Arithmetic Subgroups of Algebraic Groups.} 
Annals of Mathematics
\textbf{75}-3 (1962), 485--535


\bibitem{B2}
Burde D., Steinhoff C.: 
\emph{Classification of Orbit Closures of 4-Dimensional Complex Lie Algebras.} 
Journal of Algebra
\textbf{214}-2 (1999), 729--739.


\bibitem{B}
Burde D.: 
\emph{Contractions of Lie algebras and algebraic groups.} 
Archivum Mathematicum
\textbf{43}-5 (2007), 321--332.


\bibitem{B3}
Burde D., Ceballos M.: 
\emph{Abelian Ideals of Maximal Dimension for Solvable Lie Algebras.} 
Journal of Lie Theory
\textbf{22}-3 (2012), 741--756

\bibitem{C}
Carles R.: 
\emph{Sur la structure des alg\`{e}bres de Lie rigides.} 
Annales de l{'}institut Fourier
\textbf{34}-3 (1984), 65--82.

\bibitem{F}
Fern\'andez-Culma E. A.: 
\emph{On the isomorphism problem for Lie algebras.} 
\textbf{Work in progress.}

\bibitem{FialowskiKhudoyberdiyevOmirov}
Fialowski A., Khudoyberdiyev A. K., Omirov B. A.: 
\emph{A Characterization of Nilpotent Leibniz Algebras.} 
Algebras and Representation Theory
\textbf{16}-5 (2013), 1489--1505.

\bibitem{Filippov1}
Filippov V.T.:
\emph{On $\delta$-derivations of Lie algebras.}  
Siberian Mathematical Journal
\textbf{39}-6 (1998), 1218--1230.
Translated from Sibirskii Matematicheskii Zhurnal
\textbf{39}-6 (1998), 1409--1422.


\bibitem{Filippov2}
Filippov V.T.:
\emph{On $\delta$-derivations of prime Lie algebras.}  
Siberian Mathematical Journal
\textbf{40}-1 (1999), 174--184.
Translated from Sibirskii Matematicheskii Zhurnal
\textbf{40}-1 (1999), 201--213.


\bibitem{Gabriel}
Gabriel P.: 
\emph{Finite representation type is open.} 
Lecture Notes in Mathematics
\textbf{488}. Representations of Algebras; Proceedings of the International Conference Ottawa 1974.
Edited by V. Dlab and P. Gabriel.
Springer-Verlag Berlin Heidelberg (1975)


\bibitem{G-DSS-V}
Garc\'ia-Delgado R., Salgado G., S\'anchez-Valenzuela O. A.: 
\emph{On 3-dimensional complex Hom-Lie algebras.} 
\textbf{arXiv preprint} arXiv:1902.08569 (2019).

\bibitem{Gerstenhaber2}
Gerstenhaber M.: 
\emph{On nilalgebras and linear varieties of nilpotent matrices, III.} 
Annals of Mathematics
\textbf{70}-1 (1959), 167--205.


\bibitem{Gerstenhaber}
Gerstenhaber M.: 
\emph{On the Deformation of Rings and Algebras.} 
Annals of Mathematics
\textbf{79}-1 (1964), 59--103.

\bibitem{GrassbergerKingTirao}
Grassberger J., King A., Tirao P.: 
\emph{On the homology of free $2$-step nilpotent Lie algebras.}
Journal of Algebra
\textbf{254}-2 (2002), 213--225.

\bibitem{GO}
Grunewald F., O'Halloran J.: 
\emph{Deformations of Lie algebras.} 
Journal of Algebra
\textbf{162}-1 (1993), 210--224.

\bibitem{J1}
Jacobson N.: 
\emph{A note on automorphisms and derivations of Lie algebras.}
In Nathan Jacobson Collected Mathematical Papers, Volume \textbf{2} (1947--1965).
Contemporary Mathematicians, Gian-Carlo Rota Editor. Birkh\"{a}user Boston, Basel, Berlin  (1989)

\bibitem{J2}
Jacobson N.: 
\emph{A note on automorphisms and derivations of Lie algebras.}
Proceedings of the American Mathematical Society
\textbf{6} (1955), 281--283.

\bibitem{HLS}
Hartwig J. T., Larsson D., Silvestrov S. D.: 
\emph{Deformations of Lie algebras using $\sigma$-derivations.} 
Journal of Algebra
\textbf{295}-2 (2006), 314--361.


\bibitem{Harvey}
Harvey A.: 
\emph{Automorphisms of the Bianchi model Lie groups.} 
Journal of Mathematical Physics
\textbf{20}-2 (1979), 251--253.


\bibitem{Heintze}
Heintze E.:
\emph{On homogeneous manifolds of negative curvature.}
Mathematische Annalen
\textbf{211}-1 (1974), 23--34.

\bibitem{Hesselink}
Hesselink W.: 
\emph{Singularities in the Nilpotent scheme of a classical group.} 
Transactions of the American Mathematical Society
\textbf{222} (1976), 1--32.

\bibitem{HJ}
Hrivn\'{a}k J., Novotn\'{y} P.: 
\emph{Twisted cocycles of Lie algebras and corresponding invariant functions.} 
Linear Algebra and its Applications
\textbf{430}-4 (2009), 1384--1403.

\bibitem{Hrivnak}
Hrivn\'{a}k J.
\emph{Invariants of Lie algebras.} 
Doctoral Thesis. Czech Technical University in Prague (2007). 
Czech Republic.

\bibitem{Jantzen}
Jantzen R.: 
\emph{Editor's Note: On the Three-Dimensional Spaces which Admit a Continuous Group of Motions.} 
General Relativity and Gravitation
\textbf{33}-12 (2001), 2157--2170

\bibitem{JL}
Jin Q., Li X.: 
\emph{Hom-Lie algebra structures on semi-simple Lie algebras.} 
Journal of Algebra
\textbf{319}-4 (2008), 1398--1408.

\bibitem{Knapp}
Knapp A. W.: 
\emph{Lie groups beyond an introduction.}
Progress in Mathematics
\textbf{140}.
Springer Science+Business Media, LLC (Birkh\"{a}user). Boston, Massachusetts (2002).

\bibitem{Larsson}
Larsson D.: 
\emph{Quasi-Lie Algebras and Quasi-Deformations.} 
Doctoral Theses in Mathematical Sciences 2006:1.
Lund University (2006). Sweden.

\bibitem{Larsson2}
Larsson D.: 
\emph{Equivariant hom-Lie algebras and twisted derivations on (arithmetic) schemes.} 
Journal of Number Theory 
\textbf{{176}} (2017) 249--278


\bibitem{LegerLuks}
Leger G. F., Luks E. M.: 
\emph{Generalized Derivations of Lie Algebras.} 
Journal of Algebra
\textbf{{228}}-1 (2000), 165--203.

\bibitem{LiLi}
Li X-Ch., Li Y-f.:
\emph{Classification of 3-dimensional multiplicative Hom-Lie algebras.}
Journal of Xinyang Normal University, Natural Science Edition
\textbf{25}-4 (2012), 427--430 $+$ 455.

\bibitem{Lie}
Lie S.: 
\emph{Theorie der Transformationsgruppen; Dritter und letzter abschnitt (Vol. 3).} 
Unter mitwirkung von  Prof. Dr. Friedrich Engel.
B. G. Teubner Verlag. Leipzig, Germany (1893)

\bibitem{MS1}
Makhlouf A., Silvestrov S.: 
\emph{Hom-algebra structures.} 
Journal of Generalized Lie Theory and Applications
\textbf{2}-2 (2008), 51--64.


\bibitem{MS2}
Makhlouf A., Silvestrov S.: 
\emph{Notes on 1-parameter formal deformations of Hom-associative and Hom-Lie algebras.} 
Forum Mathematicum
\textbf{22}-4 (2010), 715--739.


\bibitem{Pasha}
Makhlouf A., Zusmanovich P.: 
\emph{Ado theorem for nilpotent Hom-Lie algebras.} 
International Journal of Algebra and Computation.
\textbf{Ready Online}. Pages 23.


\bibitem{mazzola}
Mazzola G.: 
\emph{The algebraic and geometric classification of associative algebras of dimension five.} 
Manuscripta Mathematica
\textbf{27}-1 (1979), 81--101.

\bibitem{Milnor13}
Milnor J.:
\emph{Curvatures of left invariant metrics on Lie groups.}
Advances in mathematics
\textbf{21}-3 (1976), 293--329.

\bibitem{Moens}
Moens W. A.: 
\emph{A Characterisation of Nilpotent Lie Algebras by Invertible Leibniz-Derivations.} 
Communications in Algebra
\textbf{41}-7 (2013), 2427--2440.


\bibitem{M}
Mumford D.: 
\emph{The Red Book of Varieties and Schemes.} 
Lecture Notes in Mathematics
\textbf{1358}, Springer-Verlag, Berlin, Heidelberg .
Second Expanded Edition. (1999)

\bibitem{NP}
Nesterenko M., Popovych R.: 
\emph{Contractions of low-dimensional Lie algebras.} 
Journal of Mathematical Physics
\textbf{47}-12 (2006), p. 123515.

\bibitem{NN13}
Nikolayevsky Y., Nikonorov Y. G.: 
\emph{On solvable Lie groups of negative Ricci curvature.}
Mathematische Zeitschrift
\textbf{280}-1-2 (2015), 1--16.


\bibitem{NR1}
Nijenhuis A., Richardson R. W.: 
\emph{Cohomology and deformations of algebraic structures}. 
Bulletin of the American Mathematical Society
\textbf{70}-3 (1964), 406--411.


\bibitem{NR2}
Nijenhuis A., Richardson R. W.: 
\emph{Cohomology and deformations in graded Lie algebras}. 
Bulletin of the American Mathematical Society
\textbf{72}-1 (1966), 1--29.

\bibitem{NR3}
Nijenhuis A., Richardson R. W.: 
\emph{Deformations of Lie Algebra Structures}. 
Journal of Mathematics and Mechanics
\textbf{17}-1 (1967), 89--105.

\bibitem{NH}
Novotn\'{y} P., Hrivn\'{a}k J.:
\emph{On $(\alpha,\beta,\gamma)$-derivations of Lie algebras and corresponding invariant functions.}
Journal of Geometry and Physics
\textbf{58}-2 (2008), 208--217.



\bibitem{OHalloran}
O'Halloran J.: 
\emph{A simple proof of the Gerstenhaber-Hesselink theorem for nilpotent matrices.} 
Communications in Algebra
\textbf{15}-10 (1987), 2017--2023.

\bibitem{ORS}
Ongong'a E., Richter J., Silvestrov S.: 
\emph{Classification of 3-dimensional Hom-Lie algebras.}
IOP Conference Series: Journal of Physics.
\textbf{1194}-1 (2019), 1--10.

\bibitem{Popovych1}
Popovych R., Boyko V., Nesterenko M., Lutfullin M.: 
\emph{Realizations of real low-dimensional Lie algebras.} 
Journal of Physics A: Mathematical and General
\textbf{36} (2003), 7337--7360.

\bibitem{Popovych2}
Popovych R., Boyko V., Nesterenko M., Lutfullin M.: 
\emph{Realizations of real low-dimensional Lie algebras.} 
Extended version. arXiv:math-ph/0301029v7 (2005), 1--39.

\bibitem{Remm}
Remm Elisabeth.: 
\emph{3-Dimensional Skew-symmetric Algebras and the Variety of Hom-Lie Algebras.} 
Algebra Colloquium
\textbf{25}-4 (2018), 547--566


\bibitem{Rojas}
Rojas N.: 
\emph{Degenerations of symplectic structures on Lie groups} 
\textbf{Work in progress.}


\bibitem{R}
Richardson R. W.: 
\emph{On the rigidity of semi-direct products of Lie algebras.} 
Pacific Journal of Mathematics
\textbf{22}-2 (1967), 339--344.

\bibitem{Scott}
Scott N.: 
\emph{A new canonical form for complex symmetric matrices.} 
Proceedings: Mathematical and Physical Sciences,
\textbf{441}-1913 (1993), 625--640

\bibitem{Serre}
Serre J.-P.: 
\emph{G\'eom\'etrie alg\'ebrique et g\'eom\'etrie analytique.} 
Annales de l'institut Fourier,
\textbf{6} (1956), 1--42.

\bibitem{Sheng}
Sheng Y.: 
\emph{Representations of hom-Lie algebras.} 
Algebras and representation theory
\textbf{15}-6 (2012), 1081--1098.

\bibitem{SX}
Sheng Y., Xiong Z.: 
\emph{On Hom-Lie algebras.} 
Linear and Multilinear Algebra
\textbf{63}-12 (2015), 2379--2395.

\bibitem{Silvestrov}
Silvestrov S.: 
\emph{Paradigm of quasi-Lie and quasi-Hom-Lie algebras and quasi-deformations.}
In: New techniques in Hopf algebras and graded ring theory,
Koninklijke Vlaamse Academie van Belgi\"e voor Wetenschappen en Kunsten (KVAB) 2006.
S. Caenepeel and F. Van Oystaeyen (eds.). Brussels, Belgium (2007). 165--177.


\bibitem{Tauvel}
Tauvel P., Yu R.W.T.: 
\emph{Lie algebras and algebraic groups.} 
Springer Monographs in Mathematics.
Springer-Verlag Berlin Heidelberg (2005)


\bibitem{V}
Vergne M.: 
\emph{Cohomologie des alg\`{e}bres de Lie nilpotentes. Application \`{a} l{'}\'{e}tude de la vari\'{e}t\'{e} des alg\`{e}bres de Lie nilpotentes.} 
Bulletin de la Soci\'{e}t\'{e} math\'{e}matique de France,
\textbf{98}-2 (1970), 81--116.

\bibitem{Weil}
Weil A.: 
\emph{Introduction \`{a} l'\'etude des vari\'et\'es k\"ahl\'{e}riennes.} 
Actualit\'{e}s scientifiques et industrielles \textbf{1267}.
Publications de l'{}Institut de Math\'{e}matique de l'{}Universit\'{e} de Nancago
\textbf{VI}.
Hermann - 6, Hue de la Sorbonne, Paris V (1958). New corrected version (1971).


\bibitem{XJL}
Xie W., Jin Q., Liu W.: 
\emph{Hom-structures on semi-simple Lie algebras.} 
Open Mathematics
\textbf{13}-1 (2015), 617--630.

\bibitem{Yau}
Yau D.: 
\emph{Hom-Algebras and Homology.} 
Journal of Lie Theory
\textbf{19}-2 (2009), 409--421.

\end{thebibliography}
\end{document}